\documentclass[12pt, english]{report}
\usepackage{amsmath}
\usepackage{amsthm}
\usepackage{upgreek}
\usepackage{amsfonts, amssymb}
\usepackage{amssymb}
\usepackage{dsfont}
\usepackage{yfonts}
\usepackage{bbm}
\usepackage[all]{xy}
\usepackage{a4,amsmath,amssymb,amsfonts,exscale,stmaryrd}
\usepackage{verbatim}
\usepackage{leftidx}
\usepackage[breaklinks=true,pdftex,bookmarks]{hyperref}
\NoComputerModernTips
\newdir^{ (}{{}*!/-3pt/\dir^{(}}
\newdir_{ (}{{}*!/-3pt/\dir_{(}}

\newcommand{\Id}{\mathop{\rm Id}\nolimits}
\newcommand{\ovset}[2]{\overset{#1}{#2}}
\newcommand{\unset}[2]{\underset{#1}{#2}}
\def\Ab{\mathop{\rm \mathfrak{Ab}}\nolimits}
\def\Scheme{\mathop{\rm \mathfrak{Sch}}\nolimits}
\def\Module{\mathop{\rm \mathfrak{Mod}}\nolimits}
\def\Div{\mathop{\rm \mathfrak{Div}}\nolimits}
\def\BTE{\mathop{\rm \mathfrak{BT}}\nolimits}
\def\Alt{\mathop{\rm Alt}\nolimits}
\def\Anti{\mathop{\rm Antisym}\nolimits}
\def\Sym{\mathop{\rm Sym}\nolimits}
\def\Mult{\mathop{\rm Mult}\nolimits}

\def\Hom{\mathop{\rm Hom}\nolimits}
\def\Mor{\mathop{\rm Mor}\nolimits}
\def\Ens{\mathop{\rm Ens}\nolimits}

\def\Tr{\mathop{\rm Tr}\nolimits}
\def\Gal{\mathop{\rm Gal}\nolimits}
\def\End{\mathop{\rm End}\nolimits}

\def\Spec{\mathop{\bf Spec}\nolimits}
\def\Spf{\mathop{\bf Spf}\nolimits}

\def\kernel{\mathop{\rm Ker}\nolimits}
\def\cokernel{\mathop{\rm Coker}\nolimits}
\def\image{\mathop{\rm Im}\nolimits}

\def\rank{\mathop{\rm rank}\nolimits}

\def\Tr{{\rm Tr}}

\def\w{{\rm w}}

\def\id{{\rm id}}

\def\Def{{\rm \emph{Def}}}

\def\sgn{\mathop{\rm sgn}\nolimits}
\def\res{\mathop{\rm Res}\nolimits}
\def\innHom{\underline{\Hom}}
\def\innMult{\underline{\Mult}}
\def\innSym{\underline{\Sym}}
\def\innAlt{\underline{\Alt}}
\let\phi\varphi
\let\epsilon\varepsilon
\let\setminus\smallsetminus

\newtheorem{thm}{Theorem}[section]

\newtheorem{cor}[thm]{Corollary}

\newtheorem{lem}[thm]{Lemma}

\newtheorem{prop}[thm]{Proposition}
\newtheorem{Ques}[thm]{Question}

\newenvironment{myequation}
                {\addtocounter{thm}{1}\begin{equation}}
                {\end{equation}}
                
\theoremstyle{definition}
\newtheorem{dfn}[thm]{Definition}                
\newtheorem{ex}[thm]{Example}              
\newtheorem{rem}[thm]{Remark}

\newtheorem{cons}[thm]{Construction}
\newtheorem{notation}[thm]{Notations}

\newcommand{\BA}{{\mathbb{A}}}

\newcommand{\BD}{{\mathbb{D}}}
\newcommand{\BE}{{\mathbb{E}}}
\newcommand{\BF}{{\mathbb{F}}}
\newcommand{\BG}{{\mathbb{G}}}
\newcommand{\BH}{{\mathbb{H}}}

\newcommand{\BN}{{\mathbb{N}}}

\newcommand{\BQ}{{\mathbb{Q}}}

\newcommand{\BW}{{\mathbb{W}}}

\newcommand{\BZ}{{\mathbb{Z}}}

\newcommand{\Fa}{{\mathfrak{a}}}

\newcommand{\Fm}{{\mathfrak{m}}}

\newcommand{\Fp}{{\mathfrak{p}}}
\newcommand{\Fq}{{\mathfrak{q}}}

\newcommand{\FF}{{\mathfrak{F}}}
\newcommand{\FG}{{\mathfrak{G}}}
\newcommand{\FH}{{\mathfrak{H}}}

\newcommand{\FL}{{\mathfrak{L}}}

\newcommand{\FR}{{\mathfrak{R}}}

\newcommand{\FX}{{\mathfrak{X}}}

\newcommand{\CA}{{\cal A}}
\newcommand{\CB}{{\cal B}}
\newcommand{\CC}{{\cal C}}
\newcommand{\CD}{{\cal D}}
\newcommand{\CE}{{\cal E}}
\newcommand{\CF}{{\cal F}}
\newcommand{\CG}{{\cal G}}

\newcommand{\CL}{{\cal L}}
\newcommand{\CM}{{\cal M}}
\newcommand{\CN}{{\cal N}}
\newcommand{\CO}{{\cal O}}
\newcommand{\CP}{{\cal P}}

\newcommand{\CS}{{\cal S}}
\newcommand{\CT}{{\cal T}}

\newcommand{\CY}{{\cal Y}}

\newcommand{\lbb}{\llbracket}
\newcommand{\rbb}{\rrbracket}
\newcommand{\epR}{\underset{R}{ \bigwedge}}
\newcommand{\ep}{{ \bigwedge}}
\newcommand{\epO}{\underset{\mathcal{O}}{ \bigwedge}}
\newcommand{\eps}{{\epsilon}}
\newcommand{\ol}{\overline}
\newcommand{\ul}{\underline}
\newcommand{\udel}{\updelta}
\newcommand{\sbar}{{\bar{s}}}
\newcommand{\kbar}{{\bar{k}}}
\newcommand{\ebar}{{\bar{E}}}
\def\longto{\longrightarrow}
\def\into{\hookrightarrow}
\let\onto\twoheadrightarrow
\def\isoto{\arrover{\cong}}
\def\longinto{\lhook\joinrel\longrightarrow}

\newbox\mybox
\def\arrover#1{\mathrel{
       \setbox\mybox=\hbox spread 1.4em
              {\hfil$\scriptstyle#1\vphantom{g}$\hfil}
       \vbox{\offinterlineskip\copy\mybox
             \hbox to\wd\mybox{\rightarrowfill}}}}             
\def\invlim{\mathop{\vtop{\hbox{\rm lim}\vskip-8pt
        \hbox{\hskip1pt$\scriptstyle\longleftarrow$}\vskip-1pt}}}
\def\dirlim{\mathop{\vtop{\hbox{\rm lim}\vskip-8pt
        \hbox{\hskip1pt$\scriptstyle\longrightarrow$}\vskip-1pt}}}
\def\ontoover#1{\mathrel{
       \setbox\mybox=\hbox spread 1.4em{\hfil$\scriptstyle#1$\hfil}
       \vbox{\offinterlineskip\copy\mybox
             \hbox to\wd\mybox{\rightarrowfill\hskip-2.8mm
                               $\rightarrow$}}}}

\def\to{\rightarrow}

\def\uset{\underset}

\author{\textsf{S. Mohammad Hadi Hedayatzadeh}}
\title{\textsc{Exterior Powers of Barsotti-Tate Groups}}

\begin{document}
\makeatletter
\@addtoreset{thm}{chapter}
\makeatother

\begin{center}
\Large{\textsc{Exterior Powers of Barsotti-Tate Groups}}\\[45pt]
\normalsize
Mohammad Hadi Hedayatzadeh\footnote{Dept. of Mathematics, ETH Z\"urich, 8092 Z\"urich, Switzerland, \href{mailto:hedayatzadeh@math.ethz.ch}{hedayatzadeh@math.ethz.ch}} \\[65pt]
\date{\today}
\end{center}

\textbf{Abstract.} Let $ \CO $ be the ring of integers of a non-Archimedean local field and $ \pi $ a fixed uniformizer of $ \CO $. We establish three main results. The first one states that the exterior powers of a $ \pi $-divisible $ \CO $-module scheme of dimension at most 1 over a field exist and commute with algebraic field extensions. The second one states that the exterior powers of a $p$-divisible group of dimension at most 1 over arbitrary base exist and commute with arbitrary base change. The third one states that when $ \CO $ has characteristic zero, then the exterior powers of $ \pi $-divisible groups with scalar $ \CO $-action and dimension at most 1 over a locally Noetherian base scheme exist and commute with arbitrary base change. We also calculate the height and dimension of the exterior powers in terms of the height of the given $p$-divisible group or $ \pi $-divisible $ \CO $-module scheme.\\

\textbf{R\'esum\'e.} Soient $ \CO $ l'anneau des entiers d'un corps local non-archim\'edien et $ \pi $ une uniformisante de $ \CO $. On d\'emontre trois r\'esultats principaux. Le premier affirme que les puissances ext\'erieures d'un sch\'ema en $ \CO $-modules $\pi$-divisible de dimension au plus 1 sur un corps existent et commutent avec extensions alg\'ebriques de corps. Ensuite on \'etablit que les puissances ext\'erieures d'un groupe $p$-divisible de dimension au plus 1 sur une base quelconque existent et qu'elles commutent avec changements de base arbitraires. Finalement on d\'emontre que si $ \CO $ est de charact\'eristique z\'ero, alors les puissances ext\'erieures d'un sch\'ema en $ \CO $-modules $ \pi $-divisible avec une $ \CO $-action scalaire et de dimension au plus 1 sur un sch\'ema de base localement noetherian existent et commutent avec changements de base arbitraires. De m\^eme, on calcule la hauteur et la dimension des puissances ext\'erieures en termes de la hauteur du groupe $p$-divisible ou du sch\'ema en $ \CO $-modules $ \pi $-divisible donn\'e.

\newpage

\tableofcontents

\newpage

\addtocounter{chapter}{-1}

\chapter{Introduction}

As its title suggest, in this thesis, we are dealing with the ``exterior powers'' of $p$-divisible groups or more generally ``$\pi$-divisible groups''. So, let us explain what we mean by exterior power. Let $\CC$ be a category containing all finite products, and let $G_0,G_1,\dots, G_r$ be Abelian group objects, or more generally, $R$-module objects (with $R$ a commutative ring with 1). The latter is an object $X$ of $\CC$ such that the contravariant functor $$h_X:=\Mor_{\CC}(\_,X):\CC\to \Ens$$ factors through the forgetful functor $R$-$\mathfrak{Mod}\to \Ens$.\\

By an \emph{$R$-multilinear} morphism $\phi:G_1\times\dots\times G_r\to G_0$, we mean a morphism such that for every object $T$ of $\CC$, the induced map \[\phi(T):G_1(T)\times \dots G_r(T)\to G_0(T)\] is an $R$-multilinear morphism, where $G_i(T)$ stands for the $R$-module $h_{G_i}(T)$. Denote by $\Mult_{\CC}^R(G_1\times\dots\times G_r,G_0)$ the $R$-module of all $R$-multilinear morphisms from $G_1\times\dots\times G_r$ to $G_0$. If $r=1$, we use $\Hom_{\CC}^R$ instead of $\Mult_{\CC}^R$. In the same fashion, we define \emph{alternating} $R$-multilinear (or simply alternating) morphisms and denote by $\Alt_{\CC}^R(G^r,H)$ the $R$-module of all such morphisms from $G^r$  to $H$.\\

An obvious example is when $\CC$ is the category of sets, and we obtain the accustomed notion of $R$-multilinear morphisms and alternating morphisms of $R$-modules. A more interesting example is when $\CC$ is the category of schemes over a base scheme $S$ and $R$ is the ring of rational integers $\BZ$. The $R$-module objects of $\CC$ are then commutative group schemes over $S$. Multilinear morphisms are the natural generalization of group scheme homomorphisms, and they appear quite naturally. Examples of such morphisms are the Weil pairing of the torsion points of an Abelian variety, or the Cartier duality of a finite flat commutative group scheme. Drinfeld modules, or more generally Anderson modules provide non-trivial examples of $R$-module schemes, where $R$ is a ring of functions on some affine curve defined over a finite field.\\

Given an $R$-module object $G$, we call an object $\epR^rG$, the $r^{\text{th}}$-exterior power of $G$ in a subcategory $\CD$ of $\CC$, consisting of $R$-module objects of $\CC$, if there exists an alternating $R$-multilinear morphism $\lambda:G^r\to \epR^rG$, such that the following universal property is satisfied:\\

For every object $H$ of $\CD$ and every alternating $R$-multilinear morphism $\phi:G^r\to H$, there exists a unique $R$-linear morphism $\bar{\phi}:\epR^rG\to H$ such that $\bar{\phi}\circ \lambda=\phi$. In other words, the homomorphism \[\Hom_{\CC}^R(\epR^rG,H)\to \Alt_{\CC}^R(G^r,H)\] induced by $\lambda$ is an isomorphism.\\

If we drop the adjective ``alternating'' in the above description, we get the notion of \emph{tensor product}. Having tensor products and exterior powers enables us to translate multilinear and alternating morphisms into the language of the category of $R$-module objects we are working with, i.e., we will have ($R$-linear) morphisms instead of multilinear ones.\\

In the category of $R$-modules, the exterior powers are the usual exterior powers, and we know that they always exist. However, the question of existence of such objects in a given category is a subtle one and the main challenge is to construct them in the given category.\\

In this thesis we consider only the case, where $\CC$ is a category of schemes in a wider sense, e.g., the category of $p$-divisible or $\pi$-divisible groups, where the objects are built from schemes. Since we will ultimately be interested in alternating morphisms and exterior powers, let us concentrate on them.\\

Let $\CC$ be the category of finite flat commutative group schemes over a base scheme $S$, and let $G$ be an object of $\CC$. From now on, group schemes are assumed to be commutative. We can always define the contravariant functor:\[\innAlt_{S}(G^r,\BG_m):\Scheme_S\to \Ab,\] \[T\mapsto \Alt_T(G_T^r,\BG_{m,T}).\] One can show, using Weil restriction, that this functor is representable by a  group scheme of finite type and affine over $S$, which abusing the notation, will be denoted by $\innAlt_S(G^r,\BG_m)$. Assume for a moment that $\innAlt_S(G^r,\BG_m)$ is finite and flat over $S$. Then its Cartier dual, denoted by $\Lambda^r$, is a finite flat group scheme over $S$ and we have a canonical isomorphism \[\alpha:\innHom_S(\Lambda^r,\BG_m)\to \innAlt_S(G^r,\BG_m).\] One will deduce from this isomorphism that $\Lambda^r$ is the $r^{\text{th}}$-exterior power of $G$ in $\CC$. This observation shows that $\ep^rG$ exists in $\CC$, whenever $\innAlt_S(G^r,\BG_m)$ is finite and flat over $S$. Of course, we may relax the finiteness condition and ask whether $\ep^rG$ exists in a larger category, e.g., in the category of  group schemes over $S$. In fact, if $S$ is the spectrum of a field, using the fact that $\innAlt_S(G^r,\BG_m)$ is an affine group scheme of finite type over $S$ and a devissage argument, it is shown in \cite{P}, that the Cartier dual of $\innAlt_S(G^r,\BG_m)$ (with an adequate definition) is the $r^{\text{th}}$-exterior power of $G$ in the category of  group schemes over $S$. It is always a profinite group scheme, but not necessary a finite one. Because of its existence and the fact that its construction commutes with arbitrary base change, the group scheme $\innAlt_S(G^r,\BG_m)$ will play a significant role in this thesis, especially when it comes to the question of the existence of $\ep^rG$ and its base change properties.\\

Let us now define the other main ingredient of this writing, namely $\pi$-divisible groups, or more precisely $\pi$-divisible $\CO$-module schemes. These are the generalization of $p$-divisible groups. Let $\CO$ be the ring of integers of a non-Archimedean local field, i.e., a complete discrete valuation ring with finite residue field (say $\BF_q$, of characteristic $p$). Fix a uniformizer $\pi$ of $\CO$. A \emph{$\pi$-divisible $\CO$-module scheme} over a base scheme $S$ is a formal scheme $\CM$ over $S$ with an action of the ring $\CO$, such that 1) the multiplication by $\pi$ is an isogeny (``divisibility'' by $\pi$) , i.e., an epimorphism with finite and flat kernel, 2) $\CM$ is $\pi$-power torsion, i.e., $\CM$ is the union of the $\CM_i$ ($i\geq 1$), where $\CM_i$ is the kernel of multiplication by $\pi^i$ on $\CM$. It will be shown that a $\pi$-divisible $\CO$-module scheme is a smooth formal scheme and that there exists a natural number $h$, called the height of $\CM$, such that for all $i\geq 1$, the group schemes $\CM_i$ are finite of order $q^{ih}$. As we said earlier, these are generalizations of $p$-divisible groups. If $\CO=\BZ_p$, the ring of $p$-adic integers, then $p$-divisible $\CO$-module schemes are the same as $p$-divisible groups. In fact if $\CO$ is any mixed characteristic complete discrete valuation ring with finite residue field, it can be easily shown that a $\pi$-divisible $\CO$-module scheme is a $p$-divisible group. Other important examples are Lubin-Tate groups or formal completion of a Drinfeld or Anderson module.\\

As for $p$-divisible groups, a morphism $f:\CM\to\CN$ of $\pi$-divisible $\CO$-module schemes is defined as a system of homomorphisms on finite levels, compatible with projections $\pi:\CM_{i+1}\to \CM_i$ and $\pi:\CN_{i+1}\to \CN_i$. In other words, it is an element of the $\CO$-module $\invlim\,\Hom_S(\CM_i,\CN_i)$, where the transition morphisms are induced by the above projections. By definition, an alternating $\CO$-multilinear morphism from $\CM^r$ to $\CN$, is an element of the inverse limit $\invlim\,\Alt_S(\CM_i^r,\CN_i)$ with the transition morphisms again induced by the above projections. This defines the notion of the exterior power of a $\pi$-divisible $\CO$-module scheme.  Note that if $r>1$ and instead of projections, we require compatibility with inclusions $\CM_i\into \CM_{i+1}$ and $\CN_i\into \CN_{i+1}$, then any alternating morphism would be the zero one.\\

The main question in this thesis is the following:

\begin{Ques}
\label{question 1}
Let $\CM$ be a $\pi$-divisible $\mathcal{O}$-module scheme of height $h$ and dimension at most one. Do the exterior powers of $\CM$ exist? If they do, what are their height and dimension? How do the exterior powers behave with respect to base change?
\end{Ques}

Before a discussion of motivations behind this problem, and our approach to it, let us first make a few remarks. First of all, for technical reasons, we have to assume that the prime number $p$ is different from $2$. Secondly, the condition on the dimension is essential for our proof of the existence of exterior powers. The case of $p$-divisible groups over perfect fields of characteristic $p$ illustrates this point. One way of seeing it, is to look at the slopes of the $F$-crystal associated to $p$-divisible groups. Let $G$ a $p$-divisible group and let $M$ be its $F$-crystal. It is natural to expect that the $F$-crystal associated to $\ep^rG$, if it exists, is isomorphic to $\ep^rM$. However, as soon as $r>1$ and $M$ has a slope larger than $\frac{1}{2}$, $\ep^rM$ will have a slope larger than 1 and thus cannot be the $F$-crystal associated to a $p$-divisible group, as the slopes of such $F$-crystals are always between 0 and 1. This phenomenon is also reflected in the impossibility of defining a Frobenius $F$ on $\ep^rM$ in such a way that $ FV=VF=p$.\\

The logarithm of the order of finite (flat) group schemes behaves much like the rank of finitely generated projective or free modules over rings. So, as for the finite free modules over rings, the expectation is that the height of $ \epR^r\CM $ is equal to $ \binom{h}{r} $, where $ h $ is the height of $ \CM $. In fact, we will prove this statement and will show that the dimension of $ \epR^r\CM$ is equal to $ \binom{h-1}{r-1} $.\\

Let us now see how an affirmative answer to the above question might be interesting. Let $ \CM $ be a Lubin-Tate group of dimension 1 and height $ h $, over some base scheme. Then, the highest exterior power of $ \CM $, i.e., $ \epR^h\CM $ is a Lubin-Tate group of dimension and height 1. In this way, we obtain a ``determinant map" from the moduli space of Lubin-Tate groups of height h (and dimension 1) to the moduli space of Lubin-Tate groups of height 1 (and dimension 1).\\

Let $ \CM $ be a $ \pi $-divisible $ \CO $-module scheme of dimension 1, and let $ \rho $ be the Galois representation attached to it, or more precisely, the Galois representation associated to its Tate module (defined exactly as for $ p $-divisible groups). For  $ p $-divisible groups over $ p $-adic rings (i.e., rings of integers of characteristic zero non-Archimedean local fields) this Galois representation is a crystalline representation of Hodge-Tate weight 0 or 1. For Drinfeld modules, this is the Galois representation one wishes to study. A positive answer to the above question provides a conceptual and precise reason why the determinant of these Galois representations is the Galois representation of an object of the same kind.\\

Another motivation is the following. One would like to have tensor constructions for Abelian varieties, or locally (at prime $ p $), for their $ p $-divisible groups. Again, due to the slope constraint, this is not possible in general, if one needs to stay in the framework of Abelian varieties or $ p $-divisible groups. However, one can capture the data encoded in the part of the exterior powers that has slopes between 0 and 1. For elliptic curves, this problem does not occur, since the dimension of the corresponding $ p $-divisible group is at most 1.\\

Apart from the mathematical delight this enterprise has engendered, the above discussions were motivating enough to pursue  an answer to the question. We hope the reader would feel alike.\\

In this thesis, we will prove that the exterior powers of $ p $-divisible groups of dimension at most 1 over any base scheme exist and that their construction commutes with arbitrary base change. We will also show that the height and dimension of the $ r^{\rm th} $-exterior power of a $ p $-divisible group of height $h$ and dimension one are respectively equal to $ \binom{h}{r} $ and $ \binom{h-1}{r-1} $. If the ring $ \CO $ is a $ p $-adic ring, we will explain how to generalize and adapt the proofs so as to have the results for $ \pi $-divisible $ \CO $-module schemes. We proceed in this way, because the generalization rests upon a result from the so far unpublished Ph.D. thesis of Tobias Ahsendorf (cf. \cite{TA}). If $ \CO $ is of characteristic $ p $ and the base scheme is the spectrum of a field, then we will show that the exterior powers of $ \pi $-divisible $ \CO $-module exist and we will obtain the similar result on their height and dimension.\\

Let us explain a brief sketch of our proof for the above statements. We hope that this will justify the organization of the chapters and the order in which we undertake the proofs. In this sketch, we focus on $ p $-divisible groups.\\

Let $ G $ be a $ p $-divisible group of dimension 1 over a base scheme $ S $, and denote by $ G_n $ the truncated Barsotti-Tate group of level $n$.  One main idea is that, although not necessary from the definition, we construct the exterior powers of individual $ G_n $, and show that this construction commutes with base change. When $ S $ is the spectrum of a perfect field of characteristic $ p $, we use Pink's multilinear theory of finite commutative group schemes (cf. \cite{P}) to compute the covariant Dieudonn\'e module of $ \ep^rG_n $. We show that this Dieudonn\'e module is isomorphic, as expected,  to $ \ep^rD_*(G_n) $, where $ D_*(G_n) $ is the Dieudonn\'e module of $ G_n $. This allows us to compute the order of $ \ep^rG_n $ and consequently, to obtain short exact sequences \[ 0\to \ep^rG_n\to \ep^rG_{n+m}\to \ep^rG_m\to 0 \] induced from the exact sequences \[ 0\to G_n\to G_{n+m}\to G_m\to 0 .\]  Hence, the system $ \{\ep^rG_n\}_{n\geq 1} $ is a Barsotti-Tate group, and its direct limit, $ \ep^rG $, is a $p$-divisible group. The above calculation will also imply that the Dieudonn\'e module of $ \ep^rG $ is isomorphic to $ \ep^rD_*(G) $, and from this isomorphism, we can compute the height and dimension of $ \ep^rG $.\\

Next, let $ S $ be the spectrum of a local Artin ring $R$, with residue characteristic $ p $. An important ingredient in this case, is the theory of displays over $R$. This theory is a generalization of Dieudonn\'e theory, and in particular is equivalent to Dieudonn\'e theory when $R$ is a perfect field. Since displays are linear algebraic objects, it makes sense to talk about their exterior powers. With little effort and under the dimension 1 condition, we prove that the exterior powers of a display are again displays, and that the construction of the exterior powers commutes with base change. After defining multilinear morphisms of displays, we construct a homomorphism \[ \beta:\Mult(\CP_1\times\dots\times\CP_r,\CP_0)\to \Mult(BT_{\CP_1}\times\dots\times BT_{\CP_r},BT_{\CP_0}),\] where $ \CP_i $ are displays over $R$ and $ BT_{\CP_i} $ are their associated $p$-divisible groups. This map preserves alternating morphisms and commutes with base change.\\

Let $ \CP $  be the display of $G$, and denote by $ \Lambda^r $ the $p$-divisible group of $ \ep^r\CP $, i.e., the $p$-divisible group $BT_{\ep^r\CP} $. The universal alternating morphism $ \CP^ r\to \ep^r\CP$ gives rise, via $ \beta $, to an alternating morphism $ \lambda:G^r\to \Lambda^r $. This alternating morphism is our candidate for the exterior power of $ G $. This morphism induces (by definition) an alternating morphism $ \lambda_n:G_n^r\to\Lambda^r_n $ for every $n$. For every group scheme $X$ over $R$, the morphism  $ \lambda_n $ induces a homomorphism \[ \ul{\lambda}_n^*(X):\innHom(\Lambda^r_n,X)\to \innAlt(G_n^r,X). \]  One of the main results of this thesis is that $ \beta $ is an isomorphism when $ R $ is a perfect field. Together with what we know over fields, this implies that the homomorphism $ \ul{\lambda}_n^*(X) $ is an isomorphism over  $ L $-rational points, for every perfect field $L$ over $R$. We will then explain that it follows that $ \ul{\lambda}_n^*(\BG_m) $ is an isomorphism and finally, that these results are sufficient to prove that $ \Lambda^r_n $ is the $ r^{\text{th}} $ exterior power of $ G_n $. Then the base change property will be proved, using the isomorphism $\ul{\lambda}_n^*(\BG_m)$.\\

When $ R $ is a complete local Noetherian ring with residue characteristic $p$, an approximation argument combined with the universal property of exterior powers will provide the exterior powers of $ G $ over $R$.  In particular we have the exterior powers of the universal deformation of a fixed connected $p$-divisible group of dimension 1 (defined over $ \BF_p $) over the universal deformation ring $ \BZ_p\lbb x_1\dots,x_{h-1}\rbb $. The base change property follows from this property over truncations of $R$.\\

A result of Lau states the following. Let $ G_0 $ over $ \BF_p $ be a connected $p$-divisible group of dimension 1 and height $h$, and $ \CG $ over $ R:=\BZ_p\lbb x_1,\dots,x_{h-1}\rbb $ be the universal deformation of $G_0$. Let $H$ be a truncated Barsotti-Tate group of level $n\geq 1$ and of height $h$ over a $ \BZ_{(p)} $-scheme $X$. We assume that the fibers of $H$ over the points of $X$ of characteristic $p$  have dimension one. Then there exist morphisms \[ X\ovset{\phi}{\longleftarrow} Y \ovset{\psi}{\longrightarrow} \Spec R\] with $ \phi $ faithfully flat and affine, such that $ \phi^*H\cong \psi^*\CG_n.$\\

We prove faithfully flat descent results (descent of objects and morphisms) and in conjunction with Lau's result, we construct the exterior powers of truncated Barsotti-Tate groups $ G_n $, when $ S $ is defined over $ \BZ_{(p)} $. We then show that these exterior powers sit in exact sequences, making them a Barsotti-Tate group, or their direct limit a $p$-divisible group. That this construction commutes with base change follows from faithfully flat descent lemmas we prove, together with this property for $ \CG $ over $ \BZ_p\lbb x_1,\dots,x_{h-1}\rbb  $.\\

For any base scheme $S$, using the \'etale ``dictionary", we translate the question of existence of exterior powers of \'etale $p$-divisible groups to the same question for continuous representations of \'etale fundamental group of $S$. Since these objects form a tensor category, we can solve our problem rather easily.\\

Finally, let $S$ be any scheme. We have a faithfully flat covering $$ S':=S[\frac{1}{p}]\coprod S_{(p)}\to S  ,$$ where $S[\frac{1}{p}]$ and $ S_{(p)} $ are respectively the pullbacks of $S\to\Spec(\BZ)$ via the morphism  $ \Spec(\BZ[\frac{1}{p}])\to \Spec(\BZ) $ and $ \Spec(\BZ_{(p)})\to \Spec(\BZ) $. By base change, we then obtain $p$-divisible groups $ G[\frac{1}{p}] $ over $S[\frac{1}{p}]$ and $ G_{(p)} $ over $ S_{(p)} $. Since $p$ is invertible on $S[\frac{1}{p}]$, the $p$-divisible group $G[\frac{1}{p}]$ is \'etale, and thus $\ep^rG[\frac{1}{p}]$ exists. Over $S_{(p)}$, we also have constructed the exterior power $\ep^rG_{(p)}$. These $p$-divisible groups glue together to produce the $ r^{\text{th}} $ exterior power of $G'$ (the pullback of $G$) over $S'$. Again, using the faithfully flat descent, we get the $p$-divisible group $\ep^rG$ over $S$. The fact that it commutes with arbitrary base change is again a consequence of the lemmas we prove on faithfully flat descent, and the fact that $\ep^rG'$ commutes with base change.\\

When the base scheme $S$ is locally Noetherian, there is an elementary way to avoid Lau's result and construct the exterior powers of $G$, by proving that $ \innAlt_S(G_n^r,\BG_m) $ is finite and flat over $S$. We prove this statement, by reducing to the case of a local Artin base, where we know it is true ($\ul{\lambda}_n^*(\BG_m)$ is an isomorphism).\\

We now give a quick overview of the chapters. In chapter 1, we define notions and prove results from algebraic geometry that will be used later. In chapter 2, we introduce the category of $R$-module schemes over a base scheme and define their multilinear (symmetric and alternating) morphisms and also the refined notion of pseudo-$R$-multilinear (symmetric and alternating) morphisms (cf. Definitions \ref{def 5} and \ref{def 18}). We also define the presheaf of multilinear morphisms of $R$-module schemes (cf. Definition \ref{def 4}). We then prove an adjunction result, that will help us later, mainly for induction arguments (cf. Proposition \ref{prop 3}). In chapter 3, we review results from \cite{P} on the relations between multilinear morphisms of group schemes and morphisms between their Dieudonn\'e modules, and explain how to generalize them to $R$-module schemes. For any $r >0$ and any finite $R$-module schemes $M_1, \dots, M_r$ and $M$ of $p$-power torsion, we prove, giving an explicit morphism, that the $R$-module $ \Mult^R(M_1\times\dots\times M_r,M) $ is isomorphic to the $R$-module $L^R(D_1\times\dots\times D_r,D) $ consisting of $R$-multilinear morphisms satisfying certain conditions involving the actions of Frobenius and Verschiebung, called the $F$ and $V$-conditions (cf. Definition \ref{def07}). Pink's explicit constructions incorporate both covariant and contravariant Dieudonn\'e theory. We translate these constructions into (purely) covariant theory, something that in some situations makes calculations easier.\\

In the first section of chapter 4, we define tensor products, symmetric powers and exterior powers of $R$-module schemes, by universal properties (cf. Definition \ref{def 7}). Again, we generalize Pink's argument to show that these objects exist, when the base is a field and the $R$-modules are finite over the base (cf. Theorem \ref{thm41}). In the second section, we examine the base change properties of these constructions. We prove that these constructions commute with base change, in either of the following two cases: 1) the base is a perfect field and we change it to an algebraic field extension (cf. Proposition \ref{prop 41}), or 2) the base is a field and we change it to a finite field extension (cf. Proposition \ref{lem8}). For the first case we use Galois descent and for the second one we use Weil restriction. In the third section, we explore functorial and ``categorical" properties of exterior powers. In particular, we explain how the construction of exterior powers behaves with respect to short exact sequences. Main results of this section are Lemma \ref{lem 4}, Theorem \ref{thm 1} and Proposition \ref{prop 7}. Finally, in the fourth section, generalizing Pink's results on Dieudonn\'e modules of finite $p$-group schemes over perfect fields and using these results, we find sufficient conditions under which, taking the Dieudonn\'e module and exterior power of a finite $R$-module scheme of $p$-power order commute (cf. Lemma \ref{lem 11}).\\

The main result of chapter 5 is Theorem \ref{thmoveranyfield}, which will serve as a ``basis" for the Main Theorem of this thesis, in the sense that the general case will be reduced to this situation. We begin the chapter with the definition of $ \pi $-divisible $ \CO $-module schemes (cf. Definition \ref{piBT}) and prove some expected properties, those that they share with $p$-divisible groups, e.g., that they are formally smooth. We also define the height of a $ \pi $-divisible $ \CO $-module and show that it is an integer. Then, we define multilinear and pseudo-multilinear morphisms of $ \pi $-divisible modules and also symmetric and alternating morphisms (cf. Definition \ref{def41}). Using these definitions, we define the notion of exterior powers of $ \pi $-divisible modules (cf. Definition \ref{def42}). Next, we prove that over any base scheme, exterior powers of \'etale $ \pi $-divisible modules and finite \'etale $ \CO $-module schemes exist and their construction commutes with arbitrary base change (cf. Proposition \ref{prop025}). When the base scheme is the spectrum of a perfect field of characteristic $p$, we prove that the Dieudonn\'e module of a $ \pi $-divisible $ \CO $-module of height $h$ is a free $ W(k)\widehat{\otimes}_{\BZ_p}\CO $-module of rank $ h $. When a $ \pi $-divisible module is connected and has dimension one, we exhibit a basis of its Dieudonn\'e module over $ W(k)\widehat{\otimes}_{\BZ_p}\CO $. Using this basis, we construct morphisms $ \Phi $ and $ \Upsilon $ on the exterior powers of the Dieudonn\'e module, such that $ \Phi\circ\Upsilon=p=\Upsilon\circ \Phi $ and prove that the exterior powers of the Dieudonn\'e module are Dieudonn\'e modules with their Frobenius and Verschiebung being $ \Phi $ and $ \Upsilon $ respectively. Let $ \CM $ be a connected $ \pi $-divisible $ \CO $-module scheme of dimension 1 and height $h$ over a perfect field $k$. Denote by $ D_n $ the Dieudonn\'e module of $ \CM_n$. Using $ \Phi $ and $ \Upsilon $, we define morphisms $ \phi $ and $ \upsilon $ on $ \ep^rD_n $ and show that $ \ep^rD_n $ is canonically isomorphic to the Dieudonn\'e module of $ \epO^r\CM_n $ with Frobenius and Verschiebung acting through $ \phi $ and respectively $ \upsilon $ (cf. Corollary \ref{cor41}). This implies that the order of $ \epO^r\CM_n $ is equal to $ q^{n\binom{h}{r}} $. When $k$ is not necessarily perfect, using the base change properties of exterior powers, proved in chapter 4, we explain how we can use these quantitative results and the exact sequences from chapter 4 (stated above) to show that $ \epO^r\CM_n $ form an inductive system, which is a $ \pi $-divisible $ \CO $-module scheme of height $ \binom{h}{r} $. Finally, we show that the inductive system above, seen as a $ \pi $-divisible module, is the $ r^{\rm th} $-exterior power of $ \CM $, its dimension is $ \binom{h-1}{r-1} $ (cf. Proposition \ref{prop 11} and Theorem \ref{thm 4}) and when $k$ is perfect, its Dieudonn\'e module is canonically isomorphic to $ \ep^rD(\CM) $. Lastly, we combine the results in the \'etale case and the case over fields of characteristic $p$, to show that over any ground field, the exterior powers of $ \pi $-divisible modules of dimension at most 1 exist (cf. Theorem \ref{thmoveranyfield}). A by-product of this chapter is the Corollary \ref{cordieudonnpidiv}, which will be useful later.\\

In the first section of chapter 6, we recall elements from Zink's theory of Displays, that are needed for this thesis. In the second section, we are dealing with multilinear theory of Displays. For $ i=0,\dots,r $, let $ \CP_i=(P_i,Q_i,F,V^{-1}) $ be $3n$-displays over a ring $R$. We define the group of all multilinear morphisms from $ \CP_1\times\dots\times \CP_r $ to $\CP_0 $, and denote it by $ \Mult(\CP_1\times\dots\times \CP_r,\CP_0) $. These are morphisms preserving the action of $ V^{-1} $ (cf. Definition \ref{def05}). Similarly, we define symmetric (multilinear) morphisms and alternating (multilinear) morphisms. We then show natural base change properties of multilinear morphisms (cf. Lemma \ref{lem0 18}). We then construct the homomorphism \[ \beta:\Mult(\CP_1\times\dots\times \CP_r,\CP_0)\to \Mult(BT_{\CP_1}\times\dots\times BT_{\CP_r},BT_{\CP_0})\] (cf. Construction \ref{cons05}). We prove that $ \beta $ preserves symmetric and alternating morphisms and that it commutes with base change (cf. Proposition \ref{prop0 16} and Corollary \ref{cor beta}). In the third section, we construct exterior powers of a $3n$-display of rank one using a normal decomposition (cf. Construction \ref{cons03}) and prove that this construction is independent from the choice of the normal decomposition and it commutes with base change, that these exterior powers are $3n$-displays and that they are nilpotent, when the given $3n$-display is nilpotent (cf. Lemma \ref{lem0 21}). Finally, we prove that the exterior powers of a $3n$-display satisfy the universal property of exterior powers (cf. Proposition \ref{prop ext. disp.}).\\

In chapter 7, we give explicit isomorphisms between the Cartier module, the Dieudonn\'e module and the display of a connected $p$-divisible group over a perfect field of characteristic $p$. That these linear algebraic gadgets are isomorphic is known to experts, but the author was not able to find, in the literature, the isomorphism between the Cartier module and the Dieudonn\'e module. According to \cite{B}, the isomorphism between the Cartier module and the Dieudonn\'e module of a connected $p$-divisible group over a perfect field of characteristic $p$ is due to W. Messing. The construction of the morphism from the Dieudonn\'e module to the Cartier module (Construction \ref{cons06}) was inspired by that given in \cite{B}.\\

In chapter 8, we prove the Main Theorem of the thesis (for $p$-divisible groups), i.e. Theorem \ref{thm07}. In the first section, we prove some technical statements using the isomorphisms constructed in chapter 7. These will allow us to make explicit calculations and prove Theorem \ref{thm04} and Corollary \ref{cortechresult}, which are the key results for the proof of the Main Theorem. In the second section, we construct the morphism \[  \ul{\lambda}_n^*(X):\innHom(\Lambda^r_n,X)\to \innAlt(G_n^r,X)\] mentioned above (cf. Construction \ref{cons07}). When $p$ is nilpotent in $R$ and $ \CP $ is nilpotent, we show that $ \ul{\lambda}_n^*(\BG_m) $ is an isomorphism (cf. Proposition \ref{prop013}) and conclude from it that $ \ul{\lambda}_n^*(X) $ is an isomorphism when $X$ is finite and flat over $R$ (cf. Propositions \ref{prop014}). Using this result, we prove that if $R$ is a complete local Noetherian ring with residue characteristic $p$ and the special fiber of $G$ is connected, then the exterior powers of $G$ exist (cf Proposition \ref{prop026}) and they commute with arbitrary base change (cf. Corollary \ref{cor04}). In the last section, we use the results from previous sections to construct the exterior powers of $p$-divisible groups over arbitrary base. We begin by stating Lau's result mentioned above (and giving his proof). We then prove Lemma \ref{lem027} on faithfully flat descent for truncated Barsotti-Tate groups and $p$-divisible groups. We then show that if $S$ is a scheme over $ \BZ_{(p)} $ and $G$ over $S$ is a $p$-divisible group whose fibers at points of characteristic $p$ have dimension 1, then the exterior powers of $G$ exist and they commute with arbitrary base change (cf. Lemma \ref{lem029}). We prove this lemma by proving the similar result for truncated Barsotti-Tate groups (cf. Lemma \ref{lem028}). Finally, as we explained above, we glue these results to prove the Main Theorem \ref{thm07}.\\

In chapter 9, we prove the Main Theorem of the thesis for $ \pi $-divisible $ \CO $-module schemes, where $ \CO $ is a $p$-adic ring and the action of $ \CO $ on their Lie algebra is by scalar multiplication. In the first section, we briefly define ramified Witt vectors and state (without proof) their main properties. Then, we define ramified $3n$-displays over $ \CO $-algebras. These are natural generalizations of Zink's $3n$-displays, with $ (q,\pi) $ replacing $ (p,p) $. This generalization is the work of T. Ahsendorf (cf. \cite{TA}). We follow the constructions of \cite{TA} and state the results of \cite{TA} that we will use. Then, we explain how our constructions from chapter 6 can be generalized to this new setting of ramified displays. In the second section, we construct and define the ramified Dieudonn\'e module of a $ \pi $-divisible module over a perfect field $k$ (cf. Construction \ref{consramdieudonne1}) and endow it with a ramified $3n$-display structure. Next, we construct an equivalence of categories, $ \BH $, between the category of Dieudonn\'e modules over $k$ with a ``scalar" $ \CO $-action and the category of ramified Dieudonn\'e modules over $k$ with ``scalar" $ \CO $-action. A scalar action, is an action which on the tangent space is given via the scalar multiplication (cf. Lemma \ref{equivalence}). We then prove that this equivalence preserves  multilinear morphisms satisfying the $ V $-condition (cf. Lemma \ref{isomult}). We show that if $ \CP $ is a display over $k$ with scalar $ \CO $-action, and $ \BH(\CP) $ is the corresponding ramified display, then the associated $p$-divisible group to $ \CP $ and $ \pi $-divisible module to $ \BH(\CP) $ are isomorphic as formal $ \CO $-modules (cf. Proposition \ref{isoBTgroups}). A key technical result in this part is Proposition \ref{betadiagramramdisplay}. This proposition together with what we proved in chapter 8 imply Corollary \ref{corbetaisoramdis}, which states that the homomorphism $ \beta $ is an isomorphism, also in the framework of ramified displays. Having this crucial result, we can proceed as in chapter 8, and construct the exterior powers of an infinitesimal $ \pi $-divisible module of dimension 1 over local Artin $ \CO $-algebra. Let $ \CM $ be a $ \pi $-divisible $ \CO $-module  over a base scheme $S$ and of dimension at most 1. In this chapter, instead of generalizing Lau's result and when $S$ is locally Noetherian, we prove that $ \widetilde{\innAlt}_{S}^{\CO}(\CM_n^r,\BG_m) $ is a finite flat $ \CO $-module scheme over $S$, where the symbol $ \widetilde{\Alt} $ means pseudo-$ \CO $-multilinear alternating. We prove this statement, by showing some elementary lemmas and reducing to the case of a local Artin $ \CO $-algebra. We then explain that this statement implies the existence of the exterior power $ \epO^r\CM_n $ (cf. Proposition \ref{propextlNstBT}). As in chapter 8, we prove that the system $ \{\epO^r\CM_n\}_n $ is a $ \pi $-divisible $ \CO $-module scheme over $S$ (cf. Proposition \ref{propextlNspiBT}). Finally, we show that this system is the $ r^{\rm th} $-exterior power of $ \CM $ and that for every $S$-scheme $T$, we have a canonical isomorphism $ (\epO^r\CM)_T\cong \epO^r(\CM_T) $ (cf. Theorem \ref{thm07ram}). We have similar statement about the height and dimension of the exterior powers.\\

In chapter 10, we exhibit some examples.\\

\textbf{\Large{Acknowledgements.}}\\

First and foremost, I thank Allah (subhaana wa ta'aalaa) for His uncountable benefactions. I am especially thankful for His help in completing this work, though it is just one example from a multitude. \\

Many people contributed to this thesis, directly or indirectly, mathematically or non-mathematically. I'm unfortunately not able to remember all of them. First of all, I should thank my Ph.D. advisor Prof. Dr. Richard Pink for suggesting the subject of this thesis and for his helpfulness, support and advice. It has been a great pleasure to work with such a distinguished mathematician. I wish I could have learned more from Richard, and I'm truly happy that he agreed to supervise me.\\

I evince my gratefulness to my treasured grandmother. She has taught me, from my early childhood to the present time, to think rigorously and logically and every time I see her, I learn something new. I thank my beloved parents, though it is not possible to put into words my gratitude toward them. There is no way to measure their love, help and support, which is expanded in countless dimensions. I owe them all good things which I have. My mother is the best mother I can imagine. My father has always shown me, by his great achievements, and taught me, that one can overcome all difficulties and problems, and I finally learned from him that nothing is impossible. Also, I would like to thank my adored spouse, who was always there for me during the happy and also difficult times of my Ph.D. Her encouragement kept me working despite the frustrations. Not only did she believe in me, but in fact she knew I could succeed. I believe that without her I couldn't have finished this work. I express my thanks to my dearest sisters for the happy and joyful moments they create for me.\\

I offer my most sincere gratitude and profound acknowledgement to my mathematics teacher, Mr. Mohammad Hossein Nazari. He showed me the beauty and the charm of mathematics. I'm sure that I would have been a surgeon, or something of the kind, if I had not encountered him. It is also my pleasure to express my indebtedness and recognition to Prof. Dr. Claudio Rea. I was still a schoolboy when he accepted, with enthusiasm, to teach me mathematics at his home or at his office in the University of Rome ``Tor Vergata". I never forget his skill of making complicated and hard mathematics easy to understand, and I definitely cannot stop thinking about the marvelous time I spent with him as a disciple.\\

I spent the period April-June 2010 at the University of Pennsylvania and worked with Prof. Dr. Ching-Li Chai. My time at UPenn was very productive and working with Prof. Chai was a fantastic and unique experience. He was so patient with my questions and he gave generously of his time and expertise, not only while I was in Philadelphia, but also when I returned to Zurich. He also accepted the uninteresting job of being a co-referee of this dissertation. I send him my sincere thanks. The staff at the Mathematics Department of UPenn was very kind and helpful to me, in particular Monica Dalin Pallanti, who did all the administrative job of inviting me and made my stay in Philadelphia a pleasant one. \\

I also spent a week at the University of Bielefeld visiting Prof. Dr. Thomas Zink. Though it was a short time, it happened to be highly productive. Prof. Zink has also accepted to read my thesis and to be present at my thesis defense. I'm very thankful to him for his invitation, for the constructive conversations we had, and for being my co-referee. In Bielefeld, I met Dr. Eike Lau and Tobias Ahsendorf. Apart from the interesting conversations that I had with Dr. Lau, he told me about an unpublished result of his and very kindly allowed me to include this result and its proof in my thesis. Tobias Ahsendorf generously provided me with his, not yet defended, Ph.D. thesis, which turned out to be very useful for Chapter 9 of this dissertation.\\

I was invited by Prof. Dr. H\'el\`ene Esnault to spend a week in Essen. I also worked with Dr. Kay R\"ulling while I was at the University of Duisburg-Essen. I thank both of them very much.\\

I had wonderful professors at Sharif University of Technology, \'Ecole Polythechnique F\'ed\'erale de Lausanne and Eidgen\"ossische Technische Hochschule Z\"urich.  I acknowledge, with deep gratitude and appreciation, the inspiration, encouragement, valuable time and guidance they gave me. At SUT, Prof. Dr. Saeed Akbari was the first to teach me how to do rigorous mathematics. Prof. Dr. Arash Rastegar introduced to me the wonderland of Algebraic Geometry.  Prof. Dr. Yahya Tabesh was very kind and helpful and I learned excellent mathematics from Prof. Dr. Siavash Shahshahani, primarily in his office at IPM.\\

At EPFL, Prof. Dr. Eva Bayer Fluckiger made it possible for me to pursue Arithmetic Geometry further and recommended me to Prof. Pink for a Masters and then a Ph.D. thesis. Without her help, I couldn't have done this thesis. Prof. Dr. Anthony C. Davison made my transition to EPFL easy and smooth and during my studies and while he was the dean of the Mathematics Department, helped me solve all kinds of administrative problems. Prof. Dr.  Kathryn Hess Bellwald, warmly welcomed me in her group, taught me beautiful mathematics and supervised two projects of mine. Prof. Dr. Michel Matthey, was a good friend and I enjoyed a lot the conversations we had and the enthusiastic courses he gave; I can't think of $K$-Theory, without thinking about him. Prof. Dr. Manuel Ojanguren was an excellent professor. I could write pages about him and how deeply he influenced me. He is more than a professor; he is a very good friend. His office door was always open to me, and I remember the hours I spent talking to him, almost everyday and about almost every thing. He helps me when I need help. I cannot express enough my appreciation for his friendship. I'm especially grateful to Prof. Dr. Farhad Rachidi, a professor at the Electrical Engineering Department, who helped me continue my studies at EPFL and his outstanding scientific character and wonderful personality made him an example for me to follow. I enjoyed the interesting courses and conversations I had with Prof. Dr. Donna Testerman, particularly, I enjoyed the project I did under her supervision, which was extremely fruitful; I learned what a group scheme is. I'm very happy to have such a good professor and friend.\\

In Zurich, I had the chance to have courses with Prof. Dr. Andrew Kresch from the University of Zurich, and to organize a reading course on \'etale cohomology under his supervision. He was open to all my math-related problems. Apart from taking interesting classes with Dr. Joseph Ayoub at the University of Zurich, I spent a lot of time in his office asking him questions and receiving clear answers and lucid explanations.\\

During my studies as a math student, I had the opportunity to make many many good friends, most of them mathematicians, in particular, Borna Barkhordar, Dr. Salman Abolfath Beigi, Dr. habil. Cristiana Bertolin, Prasenjit Bhowmik, C\'edric Bujard, Dr. Maurice Cochand, Dr. Eaman Eftekhary, Dr. Jean Fasel, Dr. Giordano Favi, Christian Graf, Amir Hossein Hamdavi, Dr. Keyvan Mallahi, Aleksander Momot, Simon Schieder and Dr. Olivier Siegenthaler. I thank all of them for the happy and agreeable times we shared.  I offer special thanks to my dear friend Dr. Mohsen Sharifitabar, who helps me any time I need help, with mathematical and chiefly non-mathematical problems. This is also a good opportunity for me to say thank you to my childhood companion and savvy cousin Roohollah Honarvar, for teaching me skills I can't find in books.\\

Last but not least, I express my deep gratitude to Prof. Dr. Brian Conrad and Prof. William Messing. Whenever I e-mail Prof. Conrad a question, he promptly replies to my e-mail and with a comprehensive answer. I have learned a lot from my correspondence with him. I had valuable conversations and e-mail correspondence with Prof. Messing.\\


\textbf{\Large{Conventions.}}\\

Throughout the article, unless otherwise specified, rings are commutative with $1$. A geometric point of a scheme is a morphism from the spectrum of an algebraic closed field to the scheme. The group schemes are assumed to be commutative. By dimension of a finite group scheme over a field, we mean the vector space dimension of its Lie algebra. An exact sequence of group schemes is an exact sequence of sheaves on the fppf site over the base. In the theory we are dealing with, there are subtle existence problems; so, if we talk about an object, we are always implicitly assuming that it exists, without every time expressing it.\\

Individual chapters, sections or parts of them may require specific conventions, and we tried to introduce these conventions in the preamble of the corresponding chapters or sections. We therefore ask the reader to refer to the beginning of each chapter and section, in order to find the adequate conventions for that part.\\

\textbf{\Large{Notations.}}\\

\begin{itemize}
 \item $ \BN=\{0,1,2,\dots,\} $
 \item $ \BN_{+}=\{1,2,3,\dots,\}  $
 \item For any $ n\in\BN_+ $, we denote by $ \omega(n) $ the number of distinct prime factors of $n$.
 \item $ \mu:\BN_{+}\to \{-1,0,1\} $ is the M\"obius function defined as follows: \[ \mu(n):=\begin{cases}(-1)^{\omega(n)} & \text{if}\,\, n\,\, \text{is square-free},\\ 0 & \text{otherwise.}\end{cases}\]
 \item $ \BQ_{\geq 0} $ is the set of non-negative rational numbers.
 \item $p$ is a prime number.
 \item $ q $ is a power of $p$ and $ \BF_q $ is the finite field with $ q $ elements.
 \item If $R$ is a ring and $r$ is an element of $R$, we denote by $ R/r $ the quotient ring $ R/rR$.
 \item $ \BZ_{(p)} $ is the localization of $ \BZ $ at the prime ideal $ (p) $.
 \item $ \BG_{m} $ is the multiplicative group scheme over any given base scheme.
 \item $ \BG_a $ is the additive group scheme over any given base scheme.
 \item For natural numbers $m$ and $n$, the binomial symbol $\binom{n}{m}$ is defined to be zero when $m>n$ and if $ n\geq m $ it is defined as usual.
 \item Let $ \ul{x}=(x_1,\dots,x_r) $ be an element of $ \BZ^r $. We denote by $ \max\ul{x} $ respectively $ \min\ul{x} $ the integer $ \max \{x_1,\dots, x_r\} $ respectively $\min \{x_1,\dots, x_r\}$.
 \item $ \BZ_0^r:=\{\underline{d}=(d_1,\dots,d_r)\in \BZ^r \vert \min \ul{d}=0\}. $ We denote by $ \BZ^r_{0,<M} $ the subset of $ \BZ^r_0 $ consisting of vectors $ \ul{d} $ with $ \max \ul{d}<M $.
 \item For integers $ a\leq b $, we set $ \lbb a,b\rbb :=[a,b]\cap \BZ.$
 \item Let $r$ be a natural number and $ S_1,\dots, S_r $ non-empty sets. Choose an element $z_i\in S_i$ for some $i$  and let $ a:\lbb1,r\rbb\setminus\{i\}\to \coprod_{j=1}^rS_j$ be a map such that $ a(j)\in S_j $ for every $ j\neq i$. We denote by $ (a(1),\dots,\uset{\uparrow}{z_i},\dots,a(r)) $ the element $ (a(1),\dots,a(i-1),z_i,a(i+1),\dots,a(r))\in S_1\times\dots\times S_r $. (This strange notation is used in the rare occasions, where the notations are very heavy and carrying all the indices reduces the readability).
 \item If $R$ is a ring and $M$ is an $R$-module, we denote by $ \ell_R(M) $ the length of $M$ over $R$.
 \item We denote the kernel of a homomorphism of group functors $\phi:\CF\to \CG$, by $\CF[\phi]$.
 \item Let $ X $ be a group scheme over a base scheme $S$. We will identify the sheaf $ h_X:=\Hom_S(\_,X) $ on the fppf site of $S$ with the scheme $X$. So, if $T$ is an $S$-scheme, $ X(T) $ will denote the set $ \Hom_S(T,X) $.
 \item Let $X$ be a scheme over a base scheme $S$ and $ f:T\to S $ a morphism. We denote by $ X_T $ the fiber product $ X\times_S T $. If $\CF$ is a sheaf on a Grothendieck site over $S$ (e.g. the fppf site), we denote by $ f^*\CF $ the pullback of $ \CF $ along $ f $. So, if $X$ is a scheme over $S$, the pullback of $X$, $ f^*X $, regarded as an fppf sheaf and the scheme $ X_T $ are identified.
 \item Let $ \CF $ and $ \CG $ be sheaves on a Grothendieck site. We denote by $ \innHom(\CF,\CG) $ the ``sheaf hom" from $ \CF $ to $ \CG $, in other words, this is the sheaf that sends an object $ U $ of the site to the set $ \Hom_U(\CF_{|_U},\CG_{|_U}) $.
 \item Let $G$ be a finite flat group scheme over a base scheme $S$. The Cartier dual of $G$, i.e., the finite flat group scheme $ \innHom_S(G,\BG_{m,S}) $ is denoted by $ G^* $.
 \item Let $R$ be a ring. We denote by $W(R)$ and $ CW(R) $ (respectively $ CW^u(R) $), the ring of Witt vectors and the ring of Witt covectors (respectively unipotent covectors) with respect to $p$ and with coefficients in $R$. Sometimes, when the risk of confusion is minor, we will write $W$ (respectively $ CW $) instead of $W(R)$ (respectively $ CW(R) $).  We write a vector of $ W(R) $ in the form $ \ul{x}=(x_0,x_1,\dots) $ and a covector of $ CW(R) $ in the form $ \ul{x}=(\dots,x_1,x_0) $.
 \item $ \BZ_q $ is the unramified extension $ W(\BF_q) $ of $ \BZ_p $ with residue field $ \BF_q $.
 \item For any scheme $S$, we denote by $W_S$ and $CW_S$ (respectively $ CW^u_S $) the ring scheme of Witt vectors and  Witt covectors (respectively unipotent covectors) over $S$. The Frobenius and Verschiebung will be denoted by $F$ and $V$.
 \item Denote by $ W_m $ the cokernel of the morphism $ V^m:W\to W $, i.e., the ring scheme of Witt vectors of length $m$ and denote by $ W_{m,n}$ the group scheme $ W_m[F^n] $, i.e., the kernel of $ F^n:W_m\to W_m $.
 \item Let $k$ be a perfect field of characteristic $p$. Denote by $ \widetilde{W} $ the direct system $ W_1\arrover{v}W_2\arrover{v}\dots $, viewed as an ind-object of the category of commutative group schemes over $k$. Thus, for any commutative group scheme $G$ over $k$, we have by definition $ \Hom(G,\widetilde{W})=\uset{n}{\dirlim}\,\Hom(G,W_n).$ Note that $  \widetilde{W} $ is canonically isomorphic to $ CW^u $ (cf. \cite{F}, Chapitre II, \S1).
 \item For all $m$ and $n$ consider the morphism (of schemes) \[ \tau_{m,n}:W_{m,n}\into W,\quad (x_0,\dots,x_{m-1})\mapsto (x_0,\dots,x_{m-1},0,0,\dots). \] Denote by $ \widehat{W} $ the formal group scheme $ \bigcup_{m,n}\tau_{m,n}(W_{m,n}) .$ To avoid heavy notations, when confusion is unlikely, we write $\tau$ instead of $\tau_{m,n}$. Note that $ \widehat{W} $ is sub-ind-object of $ \widetilde{W} $.
 \item Denote by $ \BW $ the inverse limit, $ \invlim W_{m,n} $ with transition morphisms the projections $ r:W_{m+1,n}\onto W_{m,n} $ (the truncation) and $ f:W_{m,n+1}\onto W_{m,n} $ (the Frobenius). For every $n$, denote by $ \pi_n $ the projection $ \BW\onto \invlim W_{m,n}=W[F^n] $.
 \item Let $R$ be a ring with a distinguished element $\pi$, and $G$ an $R$-module functor (i.e. a functor from a category to the category of $R$-modules). We denote by $G_n$ the functor $ G[\pi^n] $, i.e., the kernel of $ \pi^n:G\to G $. For instance, if $G$ is a $p$-divisible group and $R:=\BZ$ with $ \pi:=p $, then $G_n$ denotes the kernel of the multiplication by $p^n$. If there are more than one $R$-module functors $ G_0, G_1, G_2, \dots $, indexed by natural numbers, then for every $i$, the kernel of multiplication by $ \pi^n $ on $ G_i $ will be denoted by $ G_{i,n} $.
 \item Let $k$ be a perfect field of characteristic $p$. We denote by $\BE_k$ the Dieudonn\'e ring over $k$, i.e., the non-commutative polynomial ring $$\dfrac{W(k)[F,V]}{(FV-VF, VF-p)}$$ with $F\xi=\xi^{\sigma}F$ and $V\xi^{\sigma}=\xi V$ for all $\xi\in W(k)$, where $^{\sigma}:W(k)\to W(k)$ is the Frobenius morphism of $W(k)$.
 \item Let $k$ be a perfect field of characteristic $p$. We denote by $ \widehat{\BE}_k $ the $ (F,V) $-adic completion of $ \BE_k $, i.e., we have \[ \widehat{\BE}_k= \dfrac{W(k)\lbb F,V\rbb}{(FV-VF, VF-p)}.\]
 \item Let $k$ be a perfect field of characteristic $p$ and $G$ a finite group scheme over $k$ of $p$-power order. The contravariant Dieudonn\'e module of $G$, denoted by $ D^*(G) $, is the $ \BE_k $-module $ \Hom(G,CW^u)\cong\Hom(G,\widetilde{W}) $. The covariant Dieudonn\'e module of $G$, denoted by $ D_*(G) $, is the $ \BE_k $-module $ D^*(G^*) $. If $G$ is local-local, this module is canonically isomorphic to $ \Hom(\BW,G) $. For details refer to \cite{F} or \cite{P}.
 \item For any ring $R$, denote by $ \ul{\Lambda}_R $ the affine group scheme over $R$, which associates to every $R$-algebra $A$, the multiplicative group $1+t\cdot A\lbb t\rbb$ of formal power series in $A$ with constant term $1$. Seen as a functor from schemes over $R$ to Abelian groups, $ \ul{\Lambda}_R $ is isomorphic to the product $ \prod_{\BN_+}\BA^1_R $ (cf. \cite{D}).
 \item Set $ F(t):=\prod_{p\nmid n}(1-t)^{\frac{\mu(n)}{n}}\in 1+t\cdot \BZ_{(p)}\lbb t\rbb $ (cf. \cite{D}).
 \item The Artin-Hasse exponential is the following morphism \[ E:W_{\BZ_{(p)}}\to \ul{\Lambda}_{\BZ_{(p)}},\qquad \ul{x}\mapsto E(\ul{x},t):=\prod_{n\in\BN}F(x_n\cdot t^{p^n})\] (cf. \cite{D}).
 \end{itemize}

\chapter{Algebraic Geometry Results}

In this chapter we define some general notions and prove some auxiliary results from algebraic geometry, that will be used later. These results may hold under weaker conditions. We did not intend to prove the most general statements. Rather, we tried to devise a compromise between simplicity of proofs and generality of statements.

\begin{lem}
\label{lemhensel}
Let $ X=\Spec(A) $ with $A$ a complete local Noetherian ring and let $ f:Y\to X $ be a separated morphism with the following property: for every local Artin ring $R$ and every morphism $ \Spec(R)\to X $, the base change of $f$ to $R$, $ f_R:Y_R\to \Spec(R) $, is a finite and flat morphism. Then $f$ is a finite and flat morphism.
\end{lem}

\begin{proof}
The hypothesis on $f$ implies that it is a quasi-finite morphism. Since $A$ is a local Henselian ring, by Theorem 4.2, p. 32 of \cite{JM}, $f$ is a finite morphism. Thus, in particular, $f$ is affine and we can write $ Y=\Spec(B) $ with $B$ a finite $A$-algebra. Let us denote by $ \Fm $ the maximal ideal of $A$, by $ A_n $ the local Artin ring $ A/\Fm^{n+1} $ and by $ B_n $, the finite $ A_n $-algebra $ B\otimes_AA_n $. Then, by assumption, for every natural number $n$, $ B_n $ is a finite flat $ A_n $-algebra. It follows from the local flatness criterion (cf. Theorem 22.3, p. 174 of \cite{Matsumura}) that $B$ is flat over $A$.
\end{proof}

\begin{lem}
\label{lem01}
Let $\phi:X\to Y$ be a surjective morphism of schemes over a base scheme $S$. Denote by $ f:X\to S $ and  $ g:Y\to S $ the structural morphism of $X$, respectively of $Y$. If $f$ is finite and $g$ is separated and of finite type, then $g$ is finite.
\end{lem}

\begin{proof}
As $\phi$ is surjective, the fibers of $f$ surject onto the fibers of $g$ and since $f$ is finite and thus quasi-finite, the fibers of $g$ are finite too (as sets) and therefore, $g$ is a quasi-finite morphism. If we show that $g$ is proper, then it will be finite (``proper"$+$``quasi-finite"$ \Rightarrow $ ``finite"). Since $g$ is already by assumption separated, we only need to show the universal closedness. Note that since $f$ is finite, it is proper as well. As the properties of being proper and surjective are preserved under base change, in order to show that $g$ is a universally closed morphism, it is sufficient to show that it is a closed map of topological spaces. Let $F\subseteq Y$ be a closed subset. Since $\phi$ is surjective, we have $ F=\phi(\phi^{-1}(F)) $ and thus \[ g(F)=g\circ \phi (\phi^{-1}(F))= f(\phi^{-1}(F)),\] which is a closed subset of $S$, because $f$ is proper and therefore a closed map and $ \phi^{-1}(F) $ is a closed subset of $X$.
\end{proof}

\begin{dfn}
\label{def011}
Let $X$ be a separated scheme over a base scheme $S$ and let $s\in S(L)$ be an $L$-valued point, with $L$ a field. Denote by $X_s$ the base extension of $X$ with respect to the morphism $ s:\Spec(L)\to S $. By the \emph{order} of $X_s$ over $s$, we mean the dimension over $L$ of the $L$-vector space $ \Gamma(X_s,\CO_{X_s}) $. In particular, if $s\in S$ is a point, the order of the fiber $X_s$ over $s$, is the dimension of the $\kappa(s)$-vector space $ \Gamma(X_s,\CO_{X_s}) $.
\end{dfn}

\begin{lem}
\label{lem023}
Let $X,Y$ be affine schemes over $S=\Spec(R)$, where $R$ is a local ring. Assume furthermore that $X$ is finite and flat and $Y$ is of finite type over $S$, that the fibers of $X$ and $Y$ have the same order over every point of $S$, and that we have a morphism $\phi:X\to Y$ over $S$ which is an isomorphism on the special fiber.  Then $\phi$ is an isomorphism.
\end{lem}

\begin{proof}
We show at first that $ \phi $ is a closed embedding. Set $ A:=\Gamma(X,\CO_X) $ and $ B:=\Gamma(Y,\CO_Y) $. By assumption, we have $ X=\Spec(A) $ and $ Y=\Spec(B) $, with  $A$ a flat and finite $R$-algebra (i.e., finite as $R$-module). The morphism $\phi:X\to Y$ corresponds to a ring homomorphism $ f:B\to A $. We want to show that $f$ is surjective. Write $ C $ for the cokernel of $f$ and denote by $ k $ the residue field of $R$. Tensoring the exact sequence of $R$-modules $ B\arrover{f} A\to C\to 0 $ with $ k $ over $R$, we obtain the exact sequence \[ B\otimes_Rk\arrover{f\otimes_R \Id_k} A\otimes_Rk\to C\otimes_Rk\to 0.\] By hypothesis, $ B\otimes_Rk\arrover{f\otimes_R \Id_k} A\otimes_Rk $ is an isomorphism, and therefore, $ C\otimes_Rk $ is the zero $k$-vector space. As $A$ is a finitely generated $R$-module and $C$ is a quotient, we can apply the Nakayama's lemma to $C$ and deduce that $ C=0 $. This shows that $f$ is surjective. Write $K$ for the kernel of $f$, i.e., we have a short exact sequence $ 0\to K\to B\arrover{f}A\to 0 $. As $A$ is flat and finitely generated and $R$ is local, it is free. This implies that the above short exact sequence is split (as $R$-modules) and we can write $ B\cong K\oplus A $, and so $ B\otimes_Rk\cong (K\otimes_Rk)\oplus (A\otimes_Rk) $. Again, since by assumption $ f $ is an isomorphism after tensoring with $ k $ over $R$, we have $ K\otimes_Rk =0$. Assume for the moment that $B$ is a finitely generated $R$-module. Then, $K$ being a quotient of $B$, is also finitely generated and we can apply once again Nakayama's lemma and conclude that $K=0$, which achieves the proof of the proposition. So, we have to show that $B$ is a finitely generated $R$-module or equivalently, that $ Y $ is a finite $S$-scheme. Fix a point $ s\in S $. As $ \phi:X\to Y $ is a closed embedding, the induced morphism $ \phi_s:X_s\to Y_s $ is a closed embedding as well. By assumption, $X_s$ is finite over $s$ and the fibers $X_s$ and $Y_s$ have the same order over $s$, which should be then finite. This shows that the embedding $\phi_s $ is in fact an isomorphism (a surjective map of vector spaces of the same finite dimension is an isomorphism). Consequently, the morphism $ \phi $ is surjective as a map between topological spaces. We can now apply Lemma \ref{lem01} and conclude that $ Y $ is a finite scheme over $S$.
\end{proof}

\begin{prop}
\label{prop018}
Let $S$ be a base scheme and $ \phi:X\to Y $ a morphism of $S$-schemes with $X$ finite and flat and $Y$ of finite type and separated over $S$. Assume that for every geometric point $ s $ of $S$, $ X_s $ and $ Y_s $ have the same order over $s$ and that $ \phi $ is an isomorphism over every closed point of $ S $. Then $\phi$ is an isomorphism. 
\end{prop}

\begin{proof}
Denote by $ f:X\to S $, respectively $ g:Y\to S $ the structural morphisms of $X$, respectively of $Y$. Assume that we have shown the proposition for $S$, $X$ and $Y$ affine. Let $ S=\bigcup_{\alpha\in \Lambda}S_{\alpha} $, and $ Y=\bigcup_{\alpha\in \Lambda}Y_{\alpha}$ be open affine coverings, such that $ g(Y_{\alpha})\subseteq S_{\alpha} $. Set $ X_{\alpha}:=\phi^{-1}(Y_{\alpha}) $, and therefore we have also $ f(X_{\alpha})\subseteq S_{\alpha} $. Since by hypothesis, $f$ is finite and thus affine and $g$ is separated, by \cite{EGAII} I.6.2 (v), we know that $ \phi $ is an affine morphism and therefore $ \bigcup_{\alpha\in \Lambda}X_{\alpha} $ is an open affine covering of $X$. Denote by $ \phi_{\alpha}:X_{\alpha}\to Y_{\alpha} $ the restriction of $ \phi $. We know that for all $ \alpha\in \Lambda $, the morphism $ \phi_{\alpha} $ is an isomorphism, and so, it follows that $ \phi $ is as isomorphism. So, it is enough to show the statement in the affine case. Set $ A:=\Gamma(X,\CO_X) $, $ B:=\Gamma(Y,\CO_Y) $ and $ R:=\Gamma(S,\CO_S) $ and denote by $ h:B\to A $ the ring homomorphism corresponding to $ \phi:X\to Y $. We want to show that for every maximal ideal $ \Fm $ of $ R $, the localization $ h_{\Fm}:B_{\Fm}\to A_{\Fm} $ of $h$ is an isomorphism. It follows then that $ h $ is an isomorphism. We can therefore assume further that $ R $ is a local ring. Let $s$ be a point of $S$, and $ \bar{\kappa}$ an algebraic closure of $\kappa(s)$. Since by assumption the $\bar{\kappa}$-vector spaces $ B\otimes_R\kappa(s)\otimes_{\kappa(s)}\bar{\kappa}$ and $ A\otimes_R\kappa(s)\otimes_{\kappa(s)}\bar{\kappa} $ have the same finite dimension, the $ \kappa(s) $-vector spaces $B\otimes_R\kappa(s)$ and $A\otimes_R\kappa(s)$ have the same dimension too. This shows that the fibers of $ X $ and $ Y $ have the same order. We also know by assumption that $ \phi $ is an isomorphism over the special fiber. We can now apply the previous lemma, and conclude that $ \phi $ is an isomorphism.
\end{proof}

\begin{rem}
\label{rem0 17}
Assume that $X$, $Y$ and $S$ are like in the previous proposition and $\phi:X\to Y$ is an isomorphism over every geometric point of $S$, then the hypotheses of the previous proposition are satisfied and we can draw the same conclusion.
\end{rem}

\begin{prop}
\label{prop02}
Let  $\psi:G\to H$ be a homomorphism of affine group schemes of finite type over $\Spec(k)$, where $k$ is a field. Assume that for every finite group scheme $I$ over $k$, the induced homomorphism of groups $$\psi_*(I):\Hom(I,G)\to\Hom(I,H)$$ is an isomorphism and also the induced homomorphism on $\bar{k}$-valued points, $\psi(\overline{k}):G(\overline{k})\to H(\overline{k})$, is an isomorphism. Then $\psi$ is an isomorphism.
\end{prop}

\begin{proof}
We show at first that $ \psi $ is a monomorphism. Denote by $K$ the kernel of $ \psi $. Since $G$ is of finite type over $k$, its closed subgroup $K$ is also of finite type over $k$. The sequence \[ 0\longrightarrow K(\kbar)\longrightarrow G(\kbar)\arrover{\psi(\kbar)}H(\kbar) \] is exact, but by assumption, $ \psi(\kbar) $ is injective and therefore, $ K(\kbar)=0 $. It follows that $K$ is a finite group scheme over $k$. By assumption the homomorphism $$ \psi_*(K):\Hom(K,G)\to \Hom(K,H) $$ is injective. It implies that the inclusion $ K\into G $ is the zero homomorphism and thus $K=0$.\\

In order to show that $\psi$ is an epimorphism, we consider the problem over fields of positive characteristic and characteristic zero separately. First, the case when $k$ has positive characteristic $p$. Assume at first that $H$ is connected. Let $ H[F^n] $ denote the kernel of the homomorphism $ F_H^n:H\to H^{(p^n)} $. As $H$ is a scheme of finite type over $k$, the subgroup schemes $ H[F^n] $ are finite over $k$, for every $n$. It follows from the assumption that the inclusion $ H[F^n]\into H $ factors through the inclusion $ G\into H$. Denote by $ I_H $ the augmentation ideal of $H$ and by $ J $ the ideal in $ \CO(H) $ (the coordinate ring of $H$) defining $G$, i.e., we have $ \CO(G)\cong\CO(H)/J $. Since $G$ contains the kernel of all powers of the Frobenius morphism of $H$, we have $ J\subseteq \bigcap_{n=1}^{\infty}I_H^n $. But this intersection is trivial, because $H$ is connected. Hence $ \CO(G)\cong\CO(H) $ and $ G\cong H $. In the general case, denote by $ H^0 $ the connected component of $H$, containing the zero section, and by $ G_0 $ the intersection $ G\cap H^0 $. The hypotheses of the proposition hold for the induced homomorphism $ \psi_{|_{G_0}}:G_0\to H^0 $ and since $H^0$ is connected, by the above arguments, we have $ G_0=H^0 $. This shows that $ G $ contains $ H^0 $. As $H$ is of finite type over $k$, it has finitely many connected components. Thus, the quotient $ H/H^0 $ is a finite \'etale group scheme. This finite quotient surjects onto the quotient $ H/G $, which implies that $ H/G $ is a finite \'etale group scheme over $k$. Consider the following short exact sequence: \[0\to G\arrover{\psi} H\longrightarrow H/G\to 0.\] Taking the $\bar{k}$-valued points, we obtain the following short exact sequence:\[0\to G(\bar{k})\arrover{\psi(\bar{k})} H(\bar{k})\longrightarrow H/G(\bar{k})\to 0.\] Since by assumption $ \psi(\bar{k}) $ is an isomorphism, we have that $ (H/G)(\bar{k})$ is trivial. As $ H/G $ is \'etale and is trivial on $\bar{k}$-valued points, the group scheme $ H/G $ is trivial as well. Hence $ \psi$ is an isomorphism.\\

Now assume that $k$ is a field of characteristic zero. Since $ \psi(\kbar) $ is surjective, $ \psi $ is a dominant morphism. It follows that the kernel of the ring homomorphism $ \psi^{\sharp}:\CO(H)\to \CO(G) $ between the Hopf algebras of $H$ and $G$ is nilpotent. Since $k$ is of characteristic zero, $H$ is reduced (Cartier's theorem) and therefore the kernel of $ \psi^{\sharp} $ is zero, which means that $ \psi:G\to H $ is an epimorphism.
\end{proof}

\begin{rem}
\label{rem111}
$  $
\begin{itemize}
\item[1)] If characteristic of $k$ is positive, the homomorphism $ \psi $ is a monomorphism even without the assumption on the $ \kbar $-valued points. Indeed denote by $K$ the kernel of $ \psi $. Since $G$ is a scheme of finite type over $k$, the closed subscheme $K$ is also of finite type and therefore $ K[F] $, the kernel of the Frobenius morphism of $K$, is a finite group scheme over $k$. As $K[F]$ lies inside the kernel of $\psi$, the composition $ K[F]\into G\arrover{\psi}H $ is the trivial homomorphism. It follows from the assumption that $ K[F]\into G $ is the zero homomorphism. Hence $ K[F]=0 $. This implies that $K$ is an \'etale group scheme of finite type over $k$ and therefore is finite. Again, the composition $ K\into G\arrover{\psi} H $ is trivial, which implies that $K=0$.\\

When the characteristic of $k$ is zero, however, the assumption on the $\kbar$-valued points is necessary, for there are affine group schemes of finite type which don't have any non-trivial finite subgroups (e.g. $ \BG_a^{\oplus n} $), and therefore, any group homomorphism between them satisfies the first assumption.

\item[2) ]The author believes that when $k$ is of positive characteristic, the first assumption alone is enough to guarantee that $ \psi $ is an isomorphism. However, since we will not use the more general statement in the sequel, we content ourselves with the weaker statement.
\end{itemize}
\end{rem}

Now, we give the definition of a truncated Barsotti-Tate group over a base scheme, given in \cite{I}. For more details on (truncated) Barsotti-Tate groups and their properties we refer to l.c. and \cite{M}.

\begin{dfn}
\label{TBT}
Let $S$ be a scheme, $n$ a positive natural number and $G$ an fppf sheaf of Abelian groups over $S$. We call $G$ a \emph{truncated Barsotti-Tate group of level $n$ over $S$} if $G$ fulfills the following conditions (i) and (ii), and the extra condition (iii) when $n=1$:
\begin{itemize}
\item[(i)] $G$ is annihilated by $p^n$ and is a flat sheaf of $ \BZ/p^n\BZ $-modules.
\item[(ii)] the kernel of $p:G\to G$ is representable by a finite locally free group scheme over $S$. 
\item[(iii)] ($n=1$) If $S_0:=V(p)\subseteq S$ (the closed subscheme of $S$ where $p$ is zero) and $G_0:=G\times_SS_0$, the sequence \[ G_0\arrover{F} G_0\arrover{V} G_0 \] is exact.
\end{itemize}
The rank of $ G[p]=\kernel (p.:G\to G) $ is of the form $ p^h $, where $h:S\to \BN$ is a locally constant function, called \emph{the height} of $G$.
\end{dfn}

\begin{dfn}
\label{BT}
Let $S$ be a scheme and $G$ an fppf sheaf of Abelian groups over $S$. We call $G$ a \emph{Barsotti-Tate group} or \emph{$p$-divisible group over $S$} if the following conditions are satisfied:
\begin{itemize}
\item[(i)] $G$ is $p$-divisible, i.e., the homomorphism $p:G\to G$ is an epimorphism.
\item[(ii)] $G$ is $p$-torsion, i.e., the canonical homomorphism $ \uset{n}{\dirlim}\, G[p^n]\to G $ is an isomorphism.
\item[(iii)] $ G[p] $ is representable by a finite locally free group scheme over $S$.
\end{itemize}
The rank of $ G[p]$ is of the form $ p^h $, where $h:S\to \BN$ is a locally constant function, called \emph{the height} of $G$.
\end{dfn}

\begin{dfn}
\label{def09}
A $p$-divisible group $G$ over a scheme $S$ is said to be \emph{infinitesimal}, if for every $s\in S $, the fiber $ G_s $  is a connected (or equivalently formal) $p$-divisible group over the residue field $ \kappa(s) $ at $s$.
\end{dfn}

\begin{dfn}
\label{def013}
Let $G$ be a $p$-divisible group over a base scheme $S$. The \emph{dimension} of $G$ is the set-theoretic map $ \dim (G):S\to\BN $, which sends a point $s\in S$ to the dimension of the $p$-divisible group $ G_s $ over the residue field $ \kappa(s) $ at $s$.
\end{dfn}

\begin{dfn}
\label{def014}
Let $(A,\Fm)$ be a complete local Noetherian ring and denote by $ \FX $ and $ \FX_n $ the formal scheme $ \Spf(A) $ respectively the affine scheme $ \Spec(A/\Fm^n) $. We also set $X:=\Spec(A)$. 
\begin{itemize}
\item[(i)] A \emph{truncated Barsotti-Tate group of level $i$} over  $ \FX $ is a system $ \FG=(G(n))_{n\geq 1} $ of truncated Barsotti-Tate groups of level $i$ over $ \FX_n $ endowed with isomorphisms $ G(n+1)|_{\FX_{n}}\cong G(n) $, where $ G(n+1)|_{\FX_{n}} $ is the base change of $G(n+1)$ to $ \FX_n $. A \emph{homomorphism} $ \phi:\FG\to\FH $ between two truncated Barsotti-Tate groups of level $i$ over $ \FX $ is a system $ (\phi(n))_{n\geq 1} $ of homomorphisms $ \phi(n):G_n\to H_n $ over $ \FX_n $, such that for all $n$, $ \phi(n+1)|_{\FX_{n}}=\phi(n) $. We denote by $\BTE_i/\FX $ (respectively by $\BTE_i/X $) the category of truncated Barsotti-Tate groups of level $i$ over $\FX$ (respectively over $ X$). Multilinear, symmetric and alternating morphisms of truncated Barsotti-Tate groups of level $i$ over $ \FX $ are defined similarly.
\item[(ii)] A \emph{$p$-divisible group} over $ \FX $ is a system $ \FG=(G(n))_{n\geq 1} $ of $p$-divisible groups $ G(n) $ over $ \FX_n $ endowed with isomorphisms $ G(n+1)|_{\FX_{n}}\cong G(n) $, where $ G(n+1)|_{\FX_{n}} $ is the base change of $G(n+1)$ to $ \FX_n $. A \emph{homomorphism} $ \phi:\FG\to\FH $ between two $p$-divisible groups over $ \FX $ is a system $ (\phi(n))_{n\geq 1} $ of homomorphisms $ \phi(n):G_n\to H_n $ over $ \FX_n $, such that for all $n$, $ \phi(n+1)|_{\FX_{n}}=\phi(n) $. We denote by $ p$-$\Div/\FX $ (respectively by $ p$-$\Div/X $) the category of $p$-divisible groups over $\FX$ (respectively over $ X$). Multilinear, symmetric and alternating morphisms of $p$-divisible groups over $ \FX $ are defined similarly.
\end{itemize}
\end{dfn}

Let $ G $ be an object of $p$-$ \Div/X $ (respectively of $ \BTE_i/X  $) and denote by $ G(n) $ the pullback of $ G $ to $ \FX_n $. We have canonical isomorphisms $ G(n+1)|_{\FX_n}\cong G(n) $ and therefore, the system $ (G(n))_{n\geq 1} $ defines a $p$-divisible group (respectively a truncated Barsotti-Tate groups of level $i$) over $ \FX $ that we denote by $ \FF(G) $.

For the proof of the following proposition, we refer to \cite{M}, Ch. II, lemma 4.16, p. 75, or \cite{dJ}, lemma 2.4.4, p. 17.

\begin{prop}
\label{prop023}
The functors \[ \FF:\BTE_i/X \longrightarrow \BTE_i/\FX \] and \[ \FF:p\text{-}\Div/X \longrightarrow p\text{-}\Div/\FX \] are an equivalences of categories.
\end{prop}

\chapter{$R$-Multilinear Theory of $R$-Module Schemes}
\label{section R-modules}

\section{$R$-module schemes}

\begin{dfn}
\label{def 1}
Let $M$ be a commutative group scheme over  $S$. If there is a ring homomorphism $\alpha_M:R\to \End_S(M)$, then $ M $ together with $ \alpha_M $  is called an \emph{$ R $-modules scheme over $S$}, and $\alpha_M$ is called the \emph{module structure} of $M$. Equivalently, an $R$-module structure on $M$ is a factorization of the representable functor $ h_M=\Hom_S(\_,M):\Scheme/S\to\Ab $ through the forgetful functor $ R $-$ \Module \to\Ab$ from the category of $R$-modules to the category of Abelian groups. By abuse of terminology, we call $M$ an $R$-module scheme (over $S$). For simplicity, we write $r\cdot:M\to M$ or simply $r:M\to M$ for $\alpha(r).$
\end{dfn}

\begin{dfn}
\label{def 2}
Let $M$ and $N$ be $R$-module schemes over $S$. An \emph{$R$-linear homomorphism} over $S$ or an \emph{$R$-module homomorphism }$\phi$ over $S$ from $M$ to $N$ is a group scheme homomorphism $\phi:M\to N$ over $S$, such that $\alpha_N(r)\circ \phi=\phi\circ \alpha_M(r)$ for every $r\in R$. We denote by $\Hom^R(M,N)$ the group of all $R$-linear homomorphisms from $M$ to $N$, which is in fact an $R$-module using the action of $R$ on $M$ or $N$. If the ring $R$ is understood from the context and there is little risk of confusion with the group of all homomorphisms (not necessarily $R$-linear) from $M$ to $N$, we denote this module by $\Hom(M,N)$. 
\end{dfn}

\begin{rem}
\label{rem 1}
$  $
\begin{itemize}
\item[1)] A sequence \[ 0\to M'\longrightarrow M\longrightarrow M''\to 0 \] of $R$-module schemes is exact, if it is exact as a sequence of commutative group schemes.
\item[2)] If $T$ is an $S$-scheme, then the $R$-module structure of $M$ gives an $R$-module structure of the base extension $M_T$.
\item[3)]
If $M_1,\dots,M_r$ are $R$-module schemes over $S$, then the product $M_1\times_SM_2\times_S\dots\times_SM_r$ is again an $R$-module scheme over $S$.
\item[4)] A commutative group scheme over $S$ is a $\BZ$-module scheme over $S$. So, we can think of the theory of $R$-module schemes as a generalization of the theory of commutative group schemes.
\item[5)] Let $M$ be an $R$-module scheme over $S$ and $G$ a group scheme over $S$. Then the group $ \Hom_S(M,G) $ has a natural structure of $R$-module through the action of $R$ on $M$.
\item[6)] Let $M$ be a finite flat $R$-module scheme over $S$. The Cartier dual of $M$, i.e., the group scheme $\innHom(M,\BG_{m,S})$ has a natural $R$-module scheme structure given by the action of $R$ on $M$.
\item[7)] Let $M$ be a finite flat $R$-module scheme over $\Spec(A)$, where $A$ is a Henselian local ring. We have the connected-\'etale sequence of $M$ as a group scheme over $\Spec(A)$ \[0\to M^0\to M\to M^{\text{\'et}}\to 0.\] The functoriality of this sequence implies that the action of $R$ on $M$ induces actions on connected and \'etale factors, i.e., for every $r\in R$, we have the following commutative diagram:
\[\xymatrix{
0\ar[r] & M^0\ar[d]^{r\cdot}\ar[r] & M\ar[d]^{r\cdot}\ar[r] & M^{\text{\'et}}\ar[d]^{r\cdot}\ar[r] & 0\\
0\ar[r] & M^0\ar[r] & M\ar[r] & M^{\text{\'et}}\ar[r] & 0.
}\]  Therefore, the connected and \'etale factors of $M$ have natural structures of $R$-module schemes and the connected-\'etale sequence of $M$ is an exact sequence of $R$-module schemes over $\Spec(A)$.
\end{itemize}
\end{rem}

\begin{dfn}
\label{def 3}
Let $M$ and $N$ be $R$-module schemes over $S$. Define a contravariant functor $\innHom^R(M,N)$ from
the category of schemes over $S$ to the category of $R$-modules as
follows:
$$T\mapsto \innHom^R(M,N)(T):=\Hom^R_T(M_T,N_T).$$
If this functor is representable by a group scheme over $S$, that
group scheme, which is in fact an $R$-module scheme is also denoted by $\innHom^R(M,N)$ and is called the
\emph{inner $\Hom$ from $M$ to $N$.}
\end{dfn}

\begin{rem}
\label{rem 2}
Note that the condition of a homomorphism to be $R$-linear, is a closed condition, therefore, if $\innHom(M,N)$ exists as a group scheme, then $\innHom^R(M,N)$ exists and is a closed subscheme of $\innHom(M,N)$. So, we can apply the existence results that we have for group schemes, to the new setting of $R$-module schemes. For instance, if $M$ is finite flat over $S$, and $N$ is affine, then $\innHom^R(M,N)$ is representable by an affine $R$-module scheme (cf. theorem 1.3.5 in \cite{P} \footnote{Note that since the paper \cite{P} is not yet published and is in preparation, its numbering is subject to change. In this treatise, we will use the numbering of the last version available so far.}). If in addition, $N$ is of finite type over $S$, then $\innHom^R(M,N)$ is of finite type too.
\end{rem}

It is known that exact sequences of group schemes are stable under base change, and therefore, the same holds for exact sequences of $R$-module schemes. However, we give a proof of this fact in the following special case:

\begin{lem}
\label{ESBC}
Suppose that $0\to K\arrover{i}N\arrover{p}Q\to 0$ is a short
exact sequence of affine $R$-module schemes over a field $k$ and let $T=\Spec C$ be a $k$-scheme. Then the sequence \[ 0\to K_T\arrover{i_T}N_T\arrover{p_T}Q_T\to 0 \] obtained by base change is exact.
\end{lem}

\begin{proof}
Denote by $A, B$ the Hopf algebras representing $N, Q$ and by $I_B$ the
augmentation ideal of $B$. Then the Hopf algebra representing $K$ is
$A/(I_B\cdot A)$. Since $C$ is flat over
$k$, we have an injection $B\otimes_kC\into A\otimes_kC$ and
therefore $N_T\arrover{p_T}Q_T$ is a quotient morphism. We also have $(I_B\cdot A)\otimes_kC=(I_B\otimes_kC)\cdot
(A\otimes_kC)$ and so by flatness we have $$(A/(I_B\cdot
A))\otimes_kC\cong A\otimes_kC/((I_B\cdot
A)\otimes_kC)=A\otimes_kC/(I_B\otimes_kC)(A\otimes_kC).$$ It implies
that $K_T$ is the kernel of $N_T\arrover{\pi_T}Q_T$. Consequently
the short sequence $0\to K_T\arrover{i_T}N_T\arrover{p_T}Q_T\to 0$
is exact.
\end{proof}

\begin{prop}
\label{prop 1}
Let $M$ be an affine $R$-module scheme over a field $k$. Then the functors $\innHom^R(-,M)$ and $\innHom^R(M,-)$ from the category of affine $R$-module schemes over $k$ to the category of presheaves of $R$-modules are left exact.
\end{prop}

\begin{proof}[\textsc{Proof}. ]
Let $0\to K\arrover{i}N\arrover{p}Q\to 0$ be a short exact sequence of $R$-module schemes over $k$. We have to show that the sequence
$$0\to\innHom^R(Q,M)\arrover{p^*}\innHom^R(N,M)\arrover{i^*}\innHom^R(K,M)$$
is exact. It is equivalent to the exactness of the sequence
$$0\to\innHom^R(Q,M)(C)\arrover{p^*}\innHom^R(N,M)(C)\arrover{i^*}\innHom^R(K,M)(C)$$
for every $k$-algebra $C$, i.e., the exactness of the sequence
$$0\to\Hom^R_C(Q_C,M_C)\arrover{p^*}\Hom^R_C(N_C,M_C)\arrover{i^*}\Hom^R_C(K_C,M_C).$$

By previous lemma, the $R$-morphism $N_C\arrover{p_C}Q_C$ is the cokernel of the injection $K_C\into N_C$ in the category of affine $R$-module schemes. So, for any $R$-homomorphism $\phi:N_C\to M_C$ such that $\phi\circ i_C=0$, there exists a unique $R$-homomorphism $\psi:Q_C\to M_C$ with $\phi=\psi\circ p_C$, i.e., the following diagram is commutative:
$$\xymatrix{
0\ar[r]&K_C\ar[r]^{i_C}\ar[dr]^0 & N_C\ar[r]^{p_C}\ar[d]^{\phi} & Q_C\ar[r]\ar[dl]^{\exists
!\ \psi}&0\\
&& M_C.&&}$$
The exactness now follows; indeed, pick an $R$-morphism $f:Q_C\to M_C$ with $f\circ p_C=0$, then putting $\phi:=0$ the zero morphism, there are two $R$-morphisms $Q_C\to M_C$, namely $f$ and the zero morphism, whose composition with $p_C$ are $\phi$ and from the above observation they should be equal. This shows the injectivity of $$\Hom^R_C(Q_C,M_C)\arrover{p^*_C}\Hom^R_C(N_C,M_C).$$ Clearly we have $\image p^*_C\subset \kernel i^*_C$. Let $g:N_C\to M_C$ be an element of $\kernel i^*_C$, i.e., $g\circ i^*_C=0$, then according to what we said above, there is a $\psi:Q_C\to M_C$ with $p_C\circ\psi=g$, or in other words $g=p^*_C(\psi)$ and thus $\kernel i^*_C\subset\image p^*_C$.\\

Similarly, the fact that $K_C\arrover{i_C}N_C$ is the kernel of the quotient morphism $N_C\arrover{p_C}Q_C$ implies that given any $R$-homomorphism $\phi:M_C\to N_C$ with trivial composition $p_C\circ\phi$ there is a unique $R$-homomorphism $\psi:M_C\to K_C$ such that the following diagram is commutative
$$\xymatrix{
0\ar[r]&K_C\ar[r]^{i_C} & N_C\ar[r]^{p_C}&Q_C\ar[r]&0\\
&&M_C\ar[ul]^{\exists!\ \psi}\ar[u]^{\phi}\ar[ur]^0.&&}$$
And this implies, as above, the exactness of the following short sequence
$$0\to \Hom^R_C(M_C, K_C)\arrover{i^*_C}\Hom^R_C(M_C,N_C)\arrover{p^*_C}\Hom^R_C(M_C,Q_C)$$
for every $k$-algebra $C$, and consequently the following sequence of $R$-module schemes is exact
$$0\to\innHom^R(M,K)\arrover{i^*}\innHom^R(M,N)\arrover{p^*}\innHom^R(M,Q).$$
\end{proof}

\section{$R$-multilinear morphisms}

Let $M$ be a presheaf on the fppf site of a base scheme $S$. For any positive integer $r$ we denote by $M^r$ the product of $r$ copies of $M$, and for any $1\leq i<j\leq r$ we let $$\Delta^r_{ij}: M^{r-1}\to M^r$$  denote the generalized diagonal embedding equating the $i^{\text{th}}$ and $j^{\text{th}}$ components.

\begin{dfn}
\label{def 5}
Let $M_1,\dots,M_r, M$ and $N$ be presheaves of $R$-modules on the fppf site of the scheme $S$.
\begin{itemize}
\item[(i)] An \emph{$R$-multilinear} or simply \emph{multilinear} morphism from $M_1\times\dots\times M_r$ to $N$ is a presheaves (of sets) morphism, which is $R$-linear in each factor or equivalently, if for every $S$-scheme $T$, the induced morphism $M_1(T)\times\dots\times M_r(T)\to N(T)$ is $R$-multilinear. The $R$-module of all $R$-multilinear morphisms from $M_1\times\dots\times M_r$ to $N$ is denoted by $\Mult^R(M_1\times\dots\times M_r,N)$.

\item[(ii)] An $R$-multilinear morphism $M^r\to N$ is called \emph{symmetric} if it is invariant under permutation of the factors. Equivalently, a multilinear morphism is symmetric if for every $S$-scheme $T$, the induced morphism $M(T)^r\to N(T)$ is symmetric. The $R$-module of all such symmetric multilinear morphisms is denoted by $\Sym^R(M^r,N)$.

\item[(iii)] An $R$-multilinear morphism $M^r\to N$ is called \emph{alternating} if its composition with $ \Delta^r_{ij} $ is trivial for all $1\leq i<j\leq r$. Equivalently, a multilinear morphism is alternating if for every $S$- scheme $T$, the induced morphism $M(T)^r\to N(T)$ is alternating. The $R$-module of all such alternating multilinear morphisms is denoted by $\Alt^R(M^r,N)$.
\end{itemize}
\end{dfn}

There is a weaker notion of multilinearity which will be useful when we want to map to group schemes rather than $R$-module schemes.

\begin{dfn}
\label{def 18}
Let $M_1\dots,M_r$ and $M$ be presheaves of $R$-modules  and $G$ a presheaf of Abelian groups.
\begin{itemize}
\item[(i)] We denote by $\widetilde{\Mult}^R(M_1\times\dotsb\times M_r,G)$ the group of morphisms $\phi:M_1\times\dotsb\times M_r \to G$ which are multilinear, when $ M_i $ are regarded as presheaves of Abelian groups and has the following weaker property than $R$-linearity:\\
for every $S$-scheme $T$, every tuple $(m_1,\dotsb, m_r)\in  M_1(T)\times\dotsb\times M_r(T)$, every $a\in R$ and every $i\in \{2,3,\dotsb,r\}$, we have
$$\phi(a\cdot m_1,m_2,\dotsb, m_r)=\phi(m_1,\dotsb, m_{i-1},a\cdot m_i,m_{i+1},\dotsb,m_r).$$ The elements of $\widetilde{\Mult}^R(M_1\times\dotsb\times M_r,G)$ are called \emph{pseudo-$R$-multilinear}.

\item[(ii)] We denote by $\widetilde{\Sym}^R(M^r,G)$ the subgroup of $\widetilde{\Mult}^R(M^r,G)$ consisting of symmetric morphisms.

\item[(iii)] We denote by $\widetilde{\Alt}^R(M^r,G)$ the subgroup of $\widetilde{\Mult}^R(M^r,G)$ consisting of alternating morphisms.
\end{itemize}
\end{dfn}

\begin{rem}
\label{rem200}
Note that the group $ \widetilde{\Mult}^r(M_1\times\dots\times M_r,G) $ has a natural structure of $R$-module through the action of $R$ on one of the factors $ M_1, M_2\dots M_{r-1} $ or $ M_r $, and this is independent of the factor we choose. Similarly, there is a natural $R$-module structure on the groups $  \widetilde{\Sym}^R(M^r,G) $ and $ \widetilde{\Alt}^R(M^r,G) $.
\end{rem}

\begin{dfn}
\label{def 4}
Let $M_1,\dots,M_r, M$ and $N$ be presheaves of $R$-modules and $G$ a presheaf of Abelian groups. Define contravariant functors from the category of $R$-module schemes over $S$ to the category of $R$-modules as follows:
\begin{itemize}
\item[(i)] $$T \mapsto\innMult^R(M_1\times\dotsb\times M_r, N)(T):=\Mult_T^R(M_{1,T}\times\dotsb\times M_{r,T}, N_T)$$ and respectively $$T \mapsto\widetilde{\innMult}^R(M_1\times\dotsb\times M_r, G)(T):=\widetilde{\Mult}_T^R(M_{1,T}\times\dotsb\times M_{r,T}, G_T).$$ If these functors are representable by group schemes over $S$, we will also denote those group schemes, which are $R$-module schemes, by $\innMult^R(M_1\times\dotsb\times M_r, N)$ and respectively $\widetilde{\innMult}^R(M_1\times\dotsb\times M_r, G)$.
\item[(ii)] $$T \mapsto\innSym^R(M^r,N)(T):=\Sym_T^R(M_{T}^r, N_T)$$ and respectively $$T \mapsto\widetilde{\innSym}^R(M^r, G)(T):=\widetilde{\Sym}_T^R(M_{T}^r, G_T).$$ If these functors are representable by group schemes over $S$, we will also denote those group schemes, which are $R$-module schemes, by \\$\innSym^R(M^r, N)$ and respectively $\widetilde{\innSym}^R(M^r, G)$.
\item[(iii)] $$T \mapsto\innAlt^R(M^r,N)(T):=\Sym_T^R(M_{T}^r, N_T)$$ and respectively $$T \mapsto\widetilde{\innAlt}^R(M^r, G)(T):=\widetilde{\Alt}_T^R(M_{T}^r, G_T)).$$ If these functors are representable by group schemes over $S$, we will also denote those group schemes, which are $R$-module schemes, by $\innAlt^R(M^r, N)$ and respectively $\widetilde{\innAlt}^R(M^r, G)$.
\end{itemize}
\end{dfn}



\begin{rem}
\label{rem 3}
$  $
\begin{itemize}
\item[1)] $\Mult^R(M_1\times\dots\times M_r,N)$ and respectively $\widetilde{\Mult}^R(M_1\times\dots\times M_r,G)$ are subgroups of $\Mult(M_1\times\dots\times M_r,N)$ and respectively $\Mult(M_1\times\dots\times M_r,G)$ (the group of multilinear morphisms). The conditions of being $R$-multilinear or pseudo-$R$-multilinear are closed conditions, and thus the functors $\innMult^R(M_1\times\dots\times M_r,N)$ and $\widetilde{\innMult}^R(M_1\times\dots\times M_r,G)$ are closed subschemes of $\innMult(M_1\times\dots\times M_r,N)$ and respectively $\innMult(M_1\times\dots\times M_r,G)$ if they are representable (cf. Remark \ref{rem 2}).
\item[2)] $\Sym^R(M^r,N)$ and respectively $\widetilde{\Sym}^R(M^r,G)$ are subgroups of \\$\Sym(M^r,N)$ and respectively $\Sym(M^r,G)$ (the group of symmetric multilinear morphisms). Thus, $\innSym^R(M^r,N)$ and $\widetilde{\innSym}^R(M^r,G)$ are closed subschemes of $\innSym(M^r,N)$ and respectively $\innSym(M^r,G)$ if they are representable.
\item[3)] $\Alt^R(M^r,N)$ and respectively $\widetilde{\Alt}^R(M^r,G)$ are subgroups of $\Alt(M^r,N)$ and respectively $\Alt(M^r,G)$ (the group of symmetric multilinear morphisms). Therefore, $\innAlt^R(M^r,N)$ and $\widetilde{\innAlt}^R(M^r,G)$ are closed subschemes of $\innAlt(M^r,N)$ and respectively $\innAlt(M^r,G)$ if they are representable.
\item[4)] We have a natural action of the symmetric group $S_r$ on $M^r$. This action induces an action on the $R$-module $\Mult^R(M^r,N)$ (and respectively $\widetilde{\Mult}^R(M^r,G)$). Its submodule $\Sym^R(M^r,N)$ (respectively $\widetilde{\Sym}^R(M^r,G)$) is precisely the submodule of fixed points, i.e. $$\Sym^R(M^r,N)=\Mult^R(M^r,N)^{S_r}$$ (respectively $\widetilde{\Sym}^R(M^r,G)=\widetilde{\Mult}^R(M^r,G)^{S_r}$).
\end{itemize}
\end{rem}

We are now going to prove a general proposition on multilinear morphisms which will be used throughout the paper, but we first establish two lemmas:

\begin{lem}
\label{lem 19}
Let $M_1,\dots, M_r, M$ and $N$ be $R$-module schemes over $S$ and $G$ a group scheme over $S$. There are natural isomorphisms of $R$-modules
$$\Mult^R(M_1\times\dotsb\times M_r,\innHom^R(M,N))\cong\Mult^R(M_1\times\dotsb\times M_r\times M,N)$$ and
\[ \Mult^R(M_1\times\dotsb\times M_r,\innHom(M,G))\cong\widetilde{\Mult}^R(M_1\times\dotsb\times M_r\times M,G) \]
functorial in all arguments.
\end{lem}

\begin{proof}[\textsc{Proof}.]
We show the first isomorphism, and the second one is proved similarly. However, one has to note that in the second isomorphism, on the left hand side, we have the $R$-module of $R$-multilinear morphisms and not merely pseudo-$R$-multilinear ones.\\

By the definition of $\innHom^R(M,N)$, giving a morphism of schemes $$\phi:M_1\times\dotsb\times M_r\to \innHom^R(M,N)$$ is equivalent to giving a morphism of schemes $$\widetilde{\phi}:M_1\times\dotsb\times M_r\times M\to N$$ which is $R$-linear in $M$. Since the $R$-module structure of $\innHom^R(M,N)$ is induced by that of $M$, one sees easily that $\phi$ is $R$-linear in $M_i$ if and only if $\widetilde{\phi}$ is $R$-linear in $M_i$. This completes the proof.
\end{proof}

Now, we give an ``underlined" version of this lemma in order to show our general result of this type:

\begin{lem}
\label{lem 20} Let us use the notations of the previous lemma. We have natural isomorphisms
$$\innMult^R(M_1\times\dotsb\times M_r, \innHom^R(M,N))\cong\innMult^R(M_1\times\dotsb\times M_r\times M, N)$$ and $$\innMult^R(M_1\times\dotsb\times M_r, \innHom(M,G))\cong\widetilde{\innMult}^R(M_1\times\dotsb\times M_r\times M, G)$$
functorial in all arguments.
\end{lem}

\begin{proof}[\textsc{Proof}.]
If we establish the isomorphism (as functors), the representability will follow directly from it, because if two functors are naturally isomorphic and one is representable, the other is representable too. The second isomorphism can be proved similarly to the first one, and so we only show the first isomorphism. We show at first that for any $R$-module schemes $M$ and $N$ over $S$ and any $S$-scheme $T$, we have $\innHom^R(M_T,N_T)\cong\innHom^R(M, N)_T$. Indeed, if $X$ is any $T$-scheme, then $$\innHom^R(M_T, N_T)(X)=\Hom^R_X((M_T)_X,(N_T)_X)\cong\Hom^R_X(M_X,N_X)$$ $$=\innHom^R(M, N)(X)=\innHom^R(M,N)_T(X).$$
Now, we have $$\innMult^R(M_1\times M_2\times\dotsb\times M_r,N)(T)=\Mult^R(M_{1,T}\times M_{2,T}\times\dotsb\times M_{r,T},N_T)$$ and by Lemma \ref{lem 19} this is isomorphic to $$\Mult^R(M_{1,T}\times M_{2,T}\times\dotsb\times M_{r-1,T},\innHom^R(M_{r,T},N_T)).$$ By the above discussion, it is isomorphic to
$$\Mult^R(M_{1,T}\times M_{2,T}\times\dotsb\times M_{r-1,T},\innHom^R(M_r,N)_T)=$$
$$\innMult^R(M_1\times M_2\times\dotsb\times M_{r-1},\innHom^R(M_r, N))(T).$$
This achieves the proof.
\end{proof}

Here is the desired result:

\begin{prop}
\label{prop 2}
Let $M_1, \dots, M_r, N_1, \dots, N_s$ and $P$ be $R$-module schemes over $S$ and $G$ a group scheme over $S$. We have natural isomorphisms
$$\Mult^R(M_1\times\dotsb\times M_r,\innMult^R(N_1\times\dotsb\times N_s,P))\cong$$
$$\Mult^R(M_1\times\dotsb\times M_r\times N_1\times\dotsb\times N_s,P)$$ and $$\Mult^R(M_1\times\dotsb\times M_r,\widetilde{\innMult}^R(N_1\times\dotsb\times N_s,G))\cong$$
$$\widetilde{\Mult}^R(M_1\times\dotsb\times M_r\times N_1\times\dotsb\times N_s,G)$$
functorial in all arguments.
\end{prop}

\begin{proof}[\textsc{Proof}.]
As before, we only prove the first isomorphism. We prove this proposition by induction on $s$. If $s=1$, then it is exactly the Lemma \ref{lem 19}. So assume that $s>1$ and that the proposition is true for $s-1$. We have a series of isomorphisms:
$$\Mult^R(M_1\times\dotsb\times M_r\times N_1\times\dotsb\times N_s,P)\overset{\ref{lem 19}}{\cong}$$
$$\Mult^R(M_1\times\dotsb\times M_r\times N_1\times\dotsb\times N_{s-1},\innHom^R(N_s,P))\overset{\text{ind.}}{\underset{\text{hyp.}}{\cong}}$$
$$\Mult^R(M_1\times\dotsb\times M_r,\innMult^R(N_1\times\dotsb\times N_{s-1},\innHom^R(N_s,P)))\overset{\ref{lem 20}}{\cong}$$
$$\Mult^R(M_1\times\dotsb\times M_r,\innMult^R(N_1\times\dotsb\times N_s,P)).$$
\end{proof}

\begin{rem}
\label{rem 25}
Let $M_1\dotsc,M_r,N_1\dotsc,N_s,M,N$ and $P$ be $R$-module schemes over a base scheme $S$. There is a natural action of the symmetric group $S_n$ on $N^n$ that induces an action on the $R$-module scheme $\innMult^R(N^n,P)$ which itself induces an action on the $R$-module
$$\Mult^R(M_1\times\dotsb\times M_r,\innMult^R(N^n,P)).$$
We also have a natural action of this group on the $R$-module
$$\Mult^R(M_1\times\dotsb\times M_1\times N^n,N).$$
One checks that the isomorphism in the proposition is invariant under the action of $S_n$. Similarly, we have an action of the symmetric group $S_m$ on
$$\Mult^R(M^m,\innMult(N_1\times\dotsb\times N_s,P))\quad \text{and}\quad \Mult^R(M^m\times N_1\times\dotsb\times N_s,P)$$
induced by its action on $M^m$. Again, one can easily verify that the isomorphism in the proposition is invariant under this action of $S_m$.\\

We have the same remark for the pseudo-$R$-multilinear morphisms.
\end{rem}

In the same way that Lemma \ref{lem 20} follows from Lemma \ref{lem 19}, the following proposition can be deduced from Proposition \ref{prop 2}; we will thus omit its proof:

\begin{prop}
\label{prop 3} 
Let $M_1, \dots, M_r, N_1, \dots, N_s$ and $P$ be $R$-module schemes over $S$ and $G$ a group scheme over $S$. We have natural isomorphisms $$\innMult^R(M_1\times\dotsb\times M_r,\innMult^R(N_1\times\dotsb\times N_s,P))\cong$$
$$\innMult^R(M_1\times\dotsb\times M_r\times N_1\times\dotsb\times N_s,P)$$ and $$\innMult^R(M_1\times\dotsb\times M_r,\widetilde{\innMult}^R(N_1\times\dotsb\times N_s,G))\cong$$
$$\widetilde{\innMult}^R(M_1\times\dotsb\times M_r\times N_1\times\dotsb\times N_s,G)$$
functorial in all arguments.\qed
\end{prop}

\begin{rem}
\label{rem 26}
$  $
\begin{itemize}
\item[1)] Assume that $M_1,\dotsc,M_r$ are finite and flat and $N$ (respectively $G$) is affine over $S$. We can show by induction on $r$ that $\innMult^R(M_1\times\dotsb\times M_r,N)$ (respectively $\widetilde{\innMult}^R(M_1\times\dotsb\times M_r,G)$) is representable by an affine $R$-module scheme, and this scheme is of finite type, if moreover, $N$ (respectively $G$) is of finite type. We explain the $R$-multilinear case and drop the similar case of pseudo-$R$-multilinear morphisms. Indeed, if $r=1$ then this is what we explained in Remark \ref{rem 2}. So let $r>1$ and suppose that the statement is true for $r-1$. By the induction hypothesis, $\innMult^R(M_1\times\dotsb\times M_{r-1},\innHom^R(M_r,N))$ is representable and is affine. From Lemma \ref{lem 20}, it follows that $$\innMult^R(M_1\times\dotsb\times M_{r-1},\innHom^R(M,N))\cong \innMult^R(M_1\times\dotsb\times M_r,N).$$
Hence, the right hand side is representable and affine. A similar argument implies the property of being of finite type.
\item[2)] Let $M$ be finite and flat and $N$ (respectively $G$) affine over $S$. By Definition \ref{def 4}, it is clear that the functors $\innSym^R(M^r,N)$ and $\innAlt^R(M^r,N)$ (respectively $\widetilde{\innSym}^R(M^r,G)$ and $\widetilde{\innAlt}^R(M^r,G)$) are subfunctors of the representable functor $\innMult^R(M^r,N)$ (respectively $\widetilde{\innMult}^R(M^r,G)$). Since the conditions defining these subfunctors are closed conditions (given by equations), they are represented by closed subgroup schemes, and therefore are affine and if $N$ (respectively $G$) is of finite type, they are also of finite type.
\end{itemize}
\end{rem}

\begin{lem}
\label{lem 21} Let $M,N$ be $R$-module schemes over a base scheme $S$ and let $\Gamma$ be a finite group acting on $M$. Then we have a natural isomorphism
$$\Hom^R(N,M)^{\Gamma}\cong\Hom^R(N,M^{\Gamma}),$$
where $M^{\Gamma}$ is the submodule scheme of fixed points, in other words, $M^{\Gamma}(T)=M(T)^{\Gamma}$ for any $S$-scheme $T$, where $M(T)^{\Gamma}$ is the submodule of fixed points of the  $R[\Gamma]$-module $M(T)$ ($R[\Gamma]$ being the group ring) and the action of $\Gamma$ on the $R$-module $\Hom^R(N,M)$ is induced by its action on $M$. More precisely, the image of the inclusion $\Hom^R(N,M^{\Gamma})\into \Hom^R(N,M)$ is the module of fixed points $\Hom^R(N,M)^{\Gamma}$.
\end{lem}

\begin{proof}[\textsc{Proof}.]
Let $\phi:N\to M^{\Gamma}$ and $\gamma\in\Gamma$ be given. The image of $\phi$ under the inclusion in the lemma is the composition $N\arrover{\phi}M^{\Gamma}\into M$ and under the action of $\gamma$ on $\Hom^R(N,M)$ it maps to the morphism $N\arrover{\phi} M^{\Gamma}\into M\arrover{\gamma\cdot} M$. But by definition of $M^{\Gamma}$, we have that the composition $M^{\Gamma}\into M\arrover{\gamma\cdot} M$ is the same as the inclusion $M^{\Gamma}\into M$ and hence the composition $N\arrover{\phi}M^{\Gamma}\into M$ is an element of $\Hom^R(N,M)^{\Gamma}$. We have thus an inclusion $\Hom^R(N,M^{\Gamma})\subset \Hom^R(N,M)^{\Gamma}$, where we have identified $\Hom^R(N,M^{\Gamma})$ with its image.

Now, assume that we have a morphism $\psi:N\to M$ which lies inside the module of fixed points. This means that the composition $\gamma\circ\psi$ for any $\gamma\,\cdot:M\to M$ is equal to $\psi$ and therefore $\psi$ must factor through $M^{\Gamma}$. This gives the inclusion $\Hom^R(N,M)^{\Gamma}\subset \Hom^R(N,M^{\Gamma})$ and the lemma is proved.
\end{proof}

We are now going to apply this lemma to the particular case, where the acting group is the symmetric group $S_n$ which acts on $\innMult^R(M^n,P)$,  where $M$ and $P$ are two $R$-module schemes.

\begin{prop}
\label{prop 12}
Let $M, N$ and $P$ (respectively $G$) be $R$-module schemes (group scheme) over a base scheme $S$, then for every natural number $n$ we have natural isomorphisms
$$\Hom^R(N,\innSym^R(M^n,P))\cong\Mult^R(N\times M^n,P)^{S_n}$$ (respectively $$\Hom^R(N,\widetilde{\innSym}^R(M^n,G))\cong\widetilde{\Mult}^R(N\times M^n,G)^{S_n})$$
functorial in all arguments.
\end{prop}

\begin{proof}[\textsc{Proof}.]
We prove the statement for the $R$-multilinear morphisms and drop the similar proof for the pseudo-$R$-multilinear morphisms. Lemma \ref{lem 21} states that we have an isomorphism
$$\Hom^R(N,\innMult^R(M^n,P)^{S_n})\cong \Hom^R(M,\innMult^R(M^n,P))^{S_n}.$$

By Definitions \ref{def 5} and \ref{def 4} and Remark \ref{rem 3}, $\innSym^R(M^n,P)$ is exactly the module of fixed points $\innMult^R(M^n,P)^{S_n}$, and therefore we can rewrite the last isomorphism as
\begin{myequation}
\label{sym}
\Hom^R(N,\innSym^R(M^n, P))\cong \Hom^R(N,\innMult^R(M^n,P))^{S_n}.
\end{myequation}

We now apply Proposition \ref{prop 2} and Remark \ref{rem 25}: taking the fixed points of both sides of the isomorphism in Proposition \ref{prop 2}, we will again get an isomorphism. We can thus apply it to our situation, and obtain the isomorphism: $$\Hom^R(N,\innMult^R(M^n,P))^{S_n}\cong\Mult^R(N\times M^n,P)^{S_n}.$$
Combining this with (\ref{sym}), we obtain the desired isomorphism.
\end{proof}

\begin{rem}
\label{rem 27}
$  $
\begin{itemize}
\item[1)] We recall that the action of $S_n$ on the right hand side consists of permuting the factors of $M^n$ and consequently, the group $\Mult^R(N\times M^n,P)^{S_n}$ (respectively $\widetilde{\Mult}^R(N\times M^n,G)^{S_n}$) consists of  $R$-multilinear (respectively pseudo-$R$-multilinear) morphisms  from $N\times M^n$ to $P$ (respectively $G$) that are symmetric in $M^n$.
\item[2)]Note that the functoriality of this isomorphism in $N$ implies that the group scheme $\innSym^R(M^n,P)$ (respectively $\widetilde{\innSym}^R(M^n,G)$) represents the functor $\Mult^R(-\times M^n,P)^{S_n}$ (respectively  $\widetilde{\Mult}^R(-\times M^n,G	)^{S_n}$) from the category of $R$-module schemes to the category of $R$-modules.
\item[3)]It is clear that if we change $N\times M^n$ to $M^n\times N$ the proposition remains valid; we have thus another natural and functorial isomorphism
$$\Hom^R(N,\innSym^R(M^n,P))\cong\Mult^R(M^n\times N,P)^{S_n}$$ (respectively $$\Hom^R(N,\widetilde{\innSym}^R(M^n,G))\cong\widetilde{\Mult}^R(M^n\times N,G)^{S_n}).$$
\end{itemize}
\end{rem}

Similar arguments prove the following proposition:

\begin{prop}
\label{prop 13} Let $M, N_1, \dots, N_s$ and $P$ be $R$-module schemes over $S$ and $G$ a group scheme over $S$. We have natural isomorphisms
$$\Mult^R(N_1\times\dots\times N_r\times M^n,P)^{S_n}\cong\Mult^R(N_1\times\dots\times N_r,\innSym^R(M^n,P))$$ and $$\widetilde{\Mult}^R(N_1\times\dots\times N_r\times M^n,G)^{S_n}\cong\Mult^R(N_1\times\dots\times N_r,\widetilde{\innSym}^R(M^n,P)).$$\qed
\end{prop}

We can show, with slight modifications of arguments, similar results concerning the group of alternating multilinear morphisms and in particular the following proposition:

\begin{prop}
\label{prop 4} Let  $M_1,\,\dotsc,\,M_r,\,N, P$ and respectively $G$ be $R$-module schemes and respectively a group scheme over $S$. We have natural isomorphisms
$$\Alt^R(M_1\times\dots\times M_r\times N^n,P)\cong\Mult^R(M_1\times\dots\times M_r,\innAlt^R(N^n,P))$$ and respectively $$\widetilde{\Alt}^R(M_1\times\dots\times M_r\times N^n,G)\cong\Mult^R(M_1\times\dots\times M_r,\widetilde{\innAlt}^R(N^n,G))$$
where the modules on the left hand side are the modules of $R$-multilinear and respectively pseudo-$R$-multilinear morphisms that are alternating in $N^n$.\qed
\end{prop}

\chapter{$R$-Multilinear Covariant Dieudonn\'e Theory}

In addition to the notations at the beginning, we use the following notations in this chapter.

\begin{notation}
\label{notations01}
$ $
 \begin{itemize}
 \item $R$ is a fixed ring.
 \item Unless otherwise specified, all schemes are defined over $k$, where $k$ is a perfect field of characteristic $p$.
 \item Let $G$ be a local-local $p$-divisible group over $k$. Then for every positive natural number $n$, the finite group $G_n$ is local-local, and there exists a natural number $m_G(n)$ such that for all $m\geq m_G(n)$ we have $ F^{m}G_n=0=V^mG_n $ (cf. \cite{D}).  
 \end{itemize}
\end{notation}

\begin{rem}
\label{rem31}
Let $M$ be a finite $p$-torsion $R$-module scheme over $k$. By functoriality of the Dieudonn\'e functor (covariant or contravariant), the $R$-module structure on $M$ induces an $R$-module structure on $D(M)$, where $D(M)$ is the Dieudonn\'e module of $M$: \[R\to \End(M)\cong\End(D(M)).\] It follows that $ D(M) $ has a natural action of $ \BE_k\otimes_{\BZ}R $.
\end{rem}

\begin{rem}
\label{rem0 10}
Let $G$ be a finite local group scheme over $k$, then the inclusion $ \Hom(G,\widehat{W})\into \Hom(G,\widetilde{W}) $ induced by the inclusion $ \widehat{W}\into \widetilde{W} $ is an isomorphism. Indeed, for every $n$ we have $ \Hom(G,W_n)\cong\uset{m}{\dirlim}\,\Hom(G,W_{n,m}) $, because $ G $ is annihilated by a power of Frobenius, and therefore\[ \Hom(G,\widetilde{W})=\uset{n}{\dirlim}\,\Hom(G,W_{n})\cong\uset{n,m}{\dirlim}\,\Hom(G,W_{n,m})=\Hom(G,\widehat{W}).\]
\end{rem}

\begin{cons}
\label{cons01}
Fix a natural number $r$ and an element $\underline{d}\in \BZ^r_0$. For every natural number $n$, the composition \[ \BW^r\arrover{\pi_{n+d_1}\times\dots\times\pi_{n+d_r}} W[F^{n+d_1}]\times \dots\times W[F^{n+d_r}]\into W\times\dots\times W\arrover{\text{mult}} W\] has image inside the subgroup scheme $ W[F^n] $, because $ \min \underline{d}=0 $ and Frobenius is a ring homomorphism. Therefore, for every $n$, we have a multilinear morphism \[\zeta_{\ul{d},n}:= \BW^r\to W[F^n] \] and these morphisms (for all $n$) are compatible with respect to the projections $F: W[F^{n+1}]\onto W[F^n]$, and thus, they induce a multilinear morphism \[ \zeta_{\ul{d}}:=\BW^r\to \BW \] with the property that for all $n$, $ \pi_n\circ\zeta_{\ul{d}}=\zeta_{\ul{d},n} . $
\end{cons}

We cite the Proposition-Definition 4.4.2, p. 40 of \cite{P} in the following definition:

\begin{dfn}
\label{def01}
For any $ r\geq 2 $ and $ \ul{d}\in\BZ^r_0 $ there exists a unique multilinear morphism $ \Phi_{\ul{d}}:\BW^r\times \widetilde{W}\to \BG_m $ such that for all $ n\geq 0 $, all $ x_i\in\BW $, and $ y\in W_n $, we have \[ \Phi_{\ul{d}}(x_1,\dots, x_r,\epsilon(y))=E(\pi_{n+d_1}(x_1)\dots\pi_{n+d_r}(x_r)\cdot \tau(y);1).\]
\end{dfn}

\begin{prop}
\label{prop0 19}
Consider the following composition: \[ \Mult(\BW^r\times \widetilde{W},\BG_m)\arrover{f}\Mult(\BW^r\times\widehat{W},\BG_m)\uset{\cong}{\arrover{g}}\] \[\Mult(\BW^r,\innHom(\widehat{W},\BG_m)) \uset{\cong}{\arrover{h}} \Mult(\BW^r,\BW), \] where $f$ is induced by the inclusion $ \widehat{W}\into \widetilde{W} $ and the isomorphism $ h $ is induced by the duality between $ \widehat{W} $ and $ \BW$, given by the Artin-Hasse exponential. Under this composition and for all $ \ul{d}\in\BZ^r_0 $, the element $ \Phi_{\ul{d}} $ maps to $ \zeta_{\ul{d}} $.
\end{prop}

\begin{proof}
For every $ m\geq 1 $, let us denote by $ a_m $ the isomorphism \[W[F^m]\cong \uset{n}{\invlim}\,W_{n,m}\cong \uset{n}{\invlim}\,W_{m,n}^*=\uset{n}{\invlim}\,\innHom(W_{m,n},\BG_m)\cong\innHom(\uset{n}{\dirlim}\,W_{m,n},\BG_m).\] Then, the isomorphism $ a:\BW\arrover{\cong}\innHom(\widehat{W},\BG_m) $, given by the Artin-Hasse exponential, is the inverse limit over $m$ of $ a_m $. It means that if $ \xi=(\xi_m)$ is an element of $\BG $ with $ \xi_m\in W[F^m] $, then for all $ y\in \uset{n}{\dirlim}\,W_{m,n} $, we have
\begin{myequation}
\label{Artin}
a(\xi)(y)=E(\xi_m\cdot \tau(y);1),
\end{myequation}
where $ E(\_\,;1) $ denotes the Artin-Hasse exponential. Now take an element $ \vec{x}\in\BW^r $ and set $ \xi=(\xi_m):=h\circ g\circ f (\Phi_{\ul{d}})(\vec{x})\in\BW $. We have $ a(\xi)=g\circ f(\Phi_{\ul{d}})(\vec{x}) $ and so for all $ y\in\uset{n}{\dirlim}\,W_{m,n}  $ we have \[ a(\xi)(y)=  \Phi_{\ul{d}}(\vec{x},\epsilon(y))=E(\pi_{m+d_1}(x_1)\dots\pi_{m+d_r}(x_r)\cdot \tau(y);1)=E(\zeta_{\ul{d},m}\cdot \tau(y);1).\] The latter is equal to $a(\zeta_{\ul{d}})$ by (\ref{Artin}). Thus, for all $m$ and all $ y\in \uset{n}{\dirlim}\,W_{m,n}$ we have $ a(\xi)(y)=a(\zeta_{\ul{d}})(y) $. It follows that $ a(\xi)=a(\zeta_{\ul{d}})$ and since $ a $ is a bijection, this implies that $ \xi=\zeta_{\ul{d}} $, finishing the proof.
\end{proof}

\begin{prop}
\label{prop03}
Let $H$ be a finite group scheme. Then for every element $v\in \Hom(\BW,H)$, the following diagram is commutative: \[ \xymatrix{\BW^r\times H^*\ar[d]_{\Id\times v^*}\ar[rr]^{\zeta_{\ul{d}}\times\Id}&&\BW\times H^*\ar[d]^{v\times\Id}\\ \BW^r\times \widehat{W}\ar[d]_{\Id\times\iota }&&H\times H^*\ar[d]^{\text{pairing}}\\ \BW^r\times\widetilde{W}\ar[rr]_{\Phi_{\ul{d}}}&&\BG_m} \] where $ v^*:H^*\to \widehat{W} $ is the dual morphism to $ v:\BW\to H $ (using the Artin-Hasse exponential, the group functors $ \BW $ and $ \widehat{W} $ are in duality), $ \iota:\widehat{W}\into \widetilde{W} $ is the inclusion, and $ H\times H^*\to \BG_m $ is the perfect pairing putting $ H, H^* $ Cartier dual one of the other.
\end{prop}

\begin{proof}
By the definition of $ v^* $, the following diagram commutes:\[ \xymatrix{\BW\ar[d]_{v}\ar[rr]^{\cong\quad}&&\innHom(\widehat{W},\BG_m)\ar[d]^{(\_)\circ v^*}\\ H\ar[rr]^{\cong\quad}&&\innHom(H^*,\BG_m).}\] Applying the functor $ \Mult(\BW^r,\_\,) $ on this diagram, we obtain the following commutative diagram 
\begin{myequation}
\label{richard-hadi}
\xymatrix{\Mult(\BW^r,\BW)\ar[r]^{\cong \text{\qquad\quad  }}\ar[d]&\Mult(\BW^r,\innHom(\widehat{W},\BG_m))\ar[d]\ar[r]^{\cong}&\Mult(\BW^r\times\widehat{W},\BG_m)\ar[d]\\\Mult(\BW^r,H)\ar[r]^{\cong \text{\qquad\quad }}&\Mult(\BW^r,\innHom(H^*,\BG_m))\ar[r]^{\cong}&\Mult(\BW^r\times H^*,\BG_m).}
\end{myequation}\\

Now, consider the two compositions \[ \alpha:\Mult(\BW^r\times\widetilde{W},\BG_m)\to\Mult(\BW^r\times \widehat{W},\BG_m)\to \Mult(\BW^r\times H^*,\BG_m) \] induced by $ v^* $ and \[ \beta:\Mult(\BW^r,\BW)\to \Mult(\BW^r,H)\arrover{\cong}\Mult(\BW^r\times H^*,\BG_m)\] induced by $v$. The commutativity of the diagram in the statement of the proposition is equivalent to the equality $ \alpha(\Phi_{\ul{d}}) =\beta(\zeta_{\ul{d}}).$ This equality follows from the last proposition and the commutativity of diagram (\ref{richard-hadi}).
\end{proof}

\begin{rem}
\label{rem0 14}
It follows from the previous proposition, that for every finite group scheme $H$ and every $ \ul{d}\in\BZ^r_0 $, the following diagram commutes: \[ \xymatrix{\Hom(\BW,H)\ar[rr]^{(\_)\circ\zeta_{\ul{d}}}\ar[d]_{\cong}&&\Mult(\BW^r,H)\ar[d]^{\cong}\\\Hom(H^*,\widehat{W})\ar[rr]_{\Phi_{\ul{d}}^*\quad}&&\Mult(\BW^r\times H^*,\BG_m),} \] where by $ \Phi_{\ul{d}}^* $ we mean the map that sends an element $ u\in \Hom(H^*,\widehat{W}) $ to the element $ \Phi_{\ul{d}}\circ (\Id\times\dots\times\Id\times u) $. In fact, the statement of the proposition is equivalent to the commutativity of this diagram.
\end{rem}

The following theorem, is a direct consequence of theorem 4.4.5, p. 41 of \cite{P}, when taking into account the presence of $R$.

\begin{thm}
\label{thm01}
For any unipotent $R$-module scheme $M$ over $k$ and any $r >1$, respectively for any profinite local-local $R$-module scheme $M$ over $k$ and any $r > 0$, the following morphism is an isomorphism: \[ \theta_{M}:\bigoplus_{\ul{d}\in\BZ^r_0}\Hom(M,\widetilde{W})\longrightarrow \Mult(\BW^r\times M,\BG_m), \] \[ (u_{\ul{d}})_{\ul{d}}\longmapsto \prod_{\ul{d}\in\BZ^r_0}\Phi_{\ul{d}}\circ (\Id\times\dots\times\Id\times u_{\ul{d}}). \]
\end{thm}

\begin{rem}
\label{rem0 13}
$  $
\begin{itemize}
\item[1)] Once we have the theorem for group schemes (i.e., $R=\BZ$, and which is the result in \cite{P}), then this theorem follows from the fact that the homomorphism $ \Theta_M $ preserves the scalar multiplication of $R$.
\item[2)] Using Remark \ref{rem0 10}, in the previous theorem, we can replace $ \Hom(M,\widetilde{W}) $ by $ \Hom(M,\widehat{W}) $.
\end{itemize}
\end{rem}

\begin{prop}
\label{prop04}
For any finite local $R$-module scheme $M$ and any $r > 1$, the following morphism is an isomorphism: \[ \Delta_{M}:\bigoplus_{\ul{d}\in\BZ^r_0}\Hom(\BW,M)\longrightarrow \Mult(\BW^r,M), \] \[ (f_{\ul{d}})_{\ul{d}}\longmapsto \sum_{\ul{d}\in\BZ^r_0}f_{\ul{d}}\circ\zeta_{\ul{d}}.\]
\end{prop}

\begin{proof}
It is clear that the morphism $ \Delta_M $ preserves the $R$-module structure. It is therefore sufficient to prove that it is a bijection. The diagrams (for all $ \ul{d} $ in $ \BZ_0^r $) in the last remark give rise to the following commutative diagram:  \[ \xymatrix{\bigoplus_{\ul{d}\in\BZ^r_0}\Hom(\BW,M)\ar[rr]^{\Delta_{M}=\sum(\_)\circ\zeta_{\ul{d}}}\ar[d]_{\cong}&&\Mult(\BW^r,M)\ar[d]^{\cong}\\\bigoplus_{\ul{d}\in\BZ^r_0}\Hom(M^*,\widehat{W})\ar[rr]_{\sum\Phi_{\ul{d}}^*\quad}&&\Mult(\BW^r\times M^*,\BG_m).} \] Now using Theorem \ref{thm01} (note that $ M^*$ is unipotent) and Remark \ref{rem0 13}, we know that the homomorphism $ \sum\Phi_{\ul{d}}^* $ is an isomorphism and since the vertical homomorphisms are also isomorphisms, we conclude that the homomorphism $ \Delta_{H} $ is an isomorphism as well.
\end{proof}

\begin{dfn}
\label{def07}
Let $M_1, . . . , M_r, M, N$ be left $\BE_k\otimes_{\BZ}R$-modules.
\begin{itemize}
\item[1)] We let $L^R(M_1\times\dots\times M_r, N)$ denote the group of $W(k)\otimes_{\BZ}R$-multilinear maps $\ell:M_1\times\dots\times M_r\to N$ which satisfy the following conditions for all $m_i \in M_i$: \[ \ell(Vm_1,\dots,Vm_r)=V\ell(m_1,\dots,m_r),\] \[ \ell(Fm_1,m_2,\dots, m_r)=F\ell(m_1,Vm_2,\dots,Vm_r), \] \[ \vdots \] \[ \ell(m_1,\dots,m_{r-1}, Fm_r)=F\ell(Vm_1,\dots,Vm_{r-1},m_r). \]
\item[2)] Let $ L^R_{\text{sym}}(M^r,N) $ denote the submodule of $ L^r(M^r,N) $ consisting of symmetric morphisms.
\item[3)] Let $ L^R_{\text{alt}}(M^r,N) $ denote the submodule of $ L^r(M^r,N) $ consisting of alternating morphisms.
\end{itemize}
\end{dfn}

\begin{rem}
\label{rem03}
For any $r>0$ and any sheaves of $R$-modules $M, N$ over $k$, the group $ \Mult(\BW^r\times M, N) $ has a multilinear left action of $ \widehat{\BE}_k^r\otimes_{\BZ}R $ by \[(e_1,\dots,e_r)\otimes r \cdot \phi:=\phi\circ (e_1^*\times\dots\times e_r^*\times r.),\] where $ (\_)^*$ is the natural anti-automorphism of $ \widehat{\BE}_k $, being identity on $ W(k) $ and interchanging $F$ and $V$.
\end{rem}

The following proposition is a direct generalization of proposition 4.5.3, p. 48 of \cite{P} and its proof is the same as the proof of proposition 4.5.3, p. 48 of \cite{P} with the slight and easy modifications due to $R$-linearity and $R$-multilinearity, and therefore we omit the proof of the proposition. The following proposition is:

\begin{prop}
\label{prop05}
For any $r >1$, any finite local-local $R$-module schemes $M_1,\dots, M_r$ and any unipotent $R$-module scheme $M$ the following map is a well-defined isomorphism, where $D_1,\dots, D_r$ are respectively the covariant Dieudonn\'e modules of $M_1,\dots,M_r$: $$L^R(D_1\times\dots\times D_r,D^*(M))\arrover{\theta} \Mult_{\widehat{\BE}_k^r\otimes_{\BZ}R}(D_1\times\dots\times D_r,\Mult(\BW^r\times M,\BG_m)),$$ \[ \ell\mapsto \theta({\ell}):(u_1,\dots,u_r)\mapsto \theta_M((\ell(V^{d_1}u_1,\dots,V^{d_r}u_r))_{\ul{d}})=\] \[\prod_{\ul{d}\in\BZ^r_0}\Phi_{\ul{d}}\circ (\Id\times\dots\times\Id\times\ell(V^{d_1}u_1,\dots,V^{d_r}u_r)). \]\\
\end{prop}

Let $H$ be a finite group scheme. Take an element $ (w_1,\dots, w_r) \in\BW^r$, a homomorphism $ v:H^*\to \widetilde{W} $ (i.e., an element of $ D^*(H^*) $) and an element $ \ul{d}\in \BZ^r_0 $. The homomorphism $ \Phi_{\ul{d}}(w_1,\dots,w_r,v(\_)):H^*\to \BG_m $ can be seen as a section of $ H$ under the identification $ \innHom(H^*,\BG_m)\cong H $. We have thus for any $ v\in D^*(H^*)$ and any $ \ul{d}\in \BZ^r_0 $ a multilinear morphism $$ \Phi_{\ul{d}}\circ(\Id\times\dots\times\Id\times v(\_)):\BW^r\to H $$ which corresponds to the multilinear morphism $ \Phi_{\ul{d}}\circ (\Id\times\dots\times\Id\times v)\in \Mult(\BW^r\times H^*,\BG_m) $ under the canonical isomorphism $ \Mult(\BW^r,H) \cong \Mult(\BW^r\times H^*,\BG_m) $.

\begin{prop}
\label{prop022}
For any $r >1$, any finite local-local $R$-module schemes $M_1,\dots, M_r$ and any finite local $R$-module scheme $M$ the following morphism is a well-defined isomorphism, where $D_1,\dots, D_r$ and $D$ are respectively the covariant Dieudonn\'e modules of $M_1,\dots,M_r$ and $M$: $$L^R(D_1\times\dots\times D_r,D)\arrover{\Delta_{(M_1,\dots,M_r;M)}} \Mult_{\widehat{\BE}_k^r\otimes_{\BZ}R}(D_1\times\dots\times D_r,\Mult(\BW^r,M)),$$ \[ \ell\mapsto \Delta_{(M_1,\dots,M_r;M)}({\ell}):(u_1,\dots,u_r)\mapsto  \] \[\sum_{\ul{d}\in \BZ^0}\Phi_{\ul{d}}\circ (\Id\times\dots\times \Id\times \ell(V^{d_1}u_1,\dots,V^{d_r}u_r)(\_)). \]
\end{prop}

\begin{proof}
This is a direct consequence of the previous proposition, in virtue of the following facts: $ M^* $ is a unipotent $R$-module scheme, the covariant Dieudonn\'e module of $M$ is canonically isomorphic to the contravariant Dieudonn\'e module of $M^*$, and the two groups $ \Mult(\BW^r,M) $ and $ \Mult(\BW^r\times M^*,\BG_m) $ are isomorphic.
\end{proof}

In the case, when $M$ is also local-local, we can give a more direct isomorphism of the two groups $L^R(D_1\times\dots\times D_r,D)$ and $ \Mult_{\widehat{\BE}_k^r\otimes_{\BZ}R}(D_1\times\dots\times D_r,\Mult(\BW^r,M))$, without a detour to the Cartier duality. We have:

\begin{prop}
\label{prop06}
For any $r >1$ and any finite local-local $R$-module schemes $M_1,\dots, M_r$ and $M$ the following morphism is a well-defined isomorphism, where $D_1,\dots, D_r$ and $D$ are respectively the covariant Dieudonn\'e modules of $M_1,\dots\\,M_r$ and $M$: $$L^R(D_1\times\dots\times D_r,D)\arrover{\Delta_{(M_1,\dots,M_r;M)}} \Mult_{\widehat{\BE}_k^r\otimes_{\BZ}R}(D_1\times\dots\times D_r,\Mult(\BW^r,M)),$$ \[ \ell\mapsto \Delta_{(M_1,\dots,M_r;M)}({\ell}):(u_1,\dots,u_r)\mapsto \Delta_M((\ell(V^{d_1}u_1,\dots,V^{d_r}u_r))_{\ul{d}})=\] \[\sum_{\ul{d}\in\BZ^r_0}\ell(V^{d_1}u_1,\dots,V^{d_r}u_r)\circ\zeta_{\ul{d}}. \]
\end{prop}

\begin{proof}
The proposition follows at once from Proposition \ref{prop05} and Remark \ref{rem0 14}.
\end{proof}

Again, the following proposition is the ``$R$-generalization" of the proposition 4.5.2, p.47 of \cite{P} and we omit its proof:
\begin{prop}
\label{prop07}
For any $r >0$, any finite local-local $R$-module schemes $M_1,\dots, M_r$ and any sheaves of $R$-modules $M, N$, the following morphism is an isomorphism, where $D_1,\dots, D_r$ are respectively the covariant Dieudonn\'e modules of $M_1,\dots,M_r$: \[ \Mult^R(M_1\times\dots\times M_r\times M, N)\to \Mult_{\widehat{\BE}_k^r\otimes_{\BZ}R}(D_1\times\dots\times D_r,\Mult(\BW^r\times M, N)) \] \[ \phi\mapsto ((u_i)\mapsto\phi\circ(u_1\times\dots\times u_r\times \Id)). \]
\end{prop}

\begin{prop}
\label{prop08}
For any $r >0$ and any finite local-local $R$-module schemes $M_1,\dots, M_r$ and any sheaf of $R$-modules $M$, the following morphism is an isomorphism, where $D_1,\dots, D_r$ are respectively the covariant Dieudonn\'e modules of $M_1,\dots,M_r$: \[\Mult^R(M_1\times\dots\times M_r,M)\arrover{\nabla_{(M_1,\dots,M_r;M)}} \Mult_{\widehat{\BE}_k^r\otimes_{\BZ} R}(D_1\times\dots\times D_r,\Mult(\BW^r,M)) \] \[ \phi\mapsto ((f_i)\mapsto\phi\circ(f_1\times\dots\times f_r)). \]
\end{prop}

\begin{proof}
This proposition follows from the previous proposition by replacing in that proposition, $ M $ with the trivial sheaf of $R$-modules and $ H $ by the given $R$-module scheme (in the statement of this proposition) seen as a sheaf of $R$-modules.
\end{proof}

\begin{cor}
\label{cor03}
For any $r >0$ and any finite $R$-module schemes $M_1, \dots, M_r$ and $M$, of $p$-power torsion, there exists a natural isomorphism:
\[ L^R(D_1\times\dots\times D_r,D)\longrightarrow  \Mult^R(M_1\times\dots\times M_r,M),\]  where $D_1,\dots, D_r$ and $D$ are respectively the covariant Dieudonn\'e modules of $M_1,\dots,M_r$ and $M$. This isomorphism is functorial in all arguments. When $M$ is local and $ M_i $ are local-local, this isomorphism is given by  \[ \nabla_{(M_1,\dots,M_r;M)}^{-1}\circ \Delta_{(M_1,\dots,M_r;M)}. \]
\end{cor}

\begin{proof}
If $r=1$, then this is the classical Dieudonn\'e theory. If $r>1$, $M$ is local and $ M_i $ are local-local, then the corollary is a direct consequence of Propositions \ref{prop08} and \ref{prop022}. Otherwise, the same arguments as in the proof of Proposition 4.5.9, p. 52 in \cite{P} imply the required result.
\end{proof}

\begin{rem}
\label{rem33}
$  $
\begin{itemize}
\item[1)] In later chapters, we only need the explicit isomorphism of the Corollary in the case, when $ M_i $ are local-local and $ M $ is local. This is why we didn't reproduce the proof in all cases.
\item[2)] Let $M$ and $N$ be $R$-module schemes of $p$-power torsion. According to Theorem 5.4.2, p. 66 in \cite{P}, the submodule $ L^R_{\text{sym}}(D_*(M)^r,D_*(N)) $ of $ L^R(D_*(M),D_*(N)) $ is mapped, under the above isomorphism, bijectively onto the submodule $ \Sym^R(M^r,N) $ of $ \Mult^R(M^r,N) $. Similarly, when $p>2$, the submodule $ L^R_{\text{alt}}(D_*(M)^r,D_*(N)) $ of $ L^R(D_*(M),D_*(N)) $ is mapped, under the above isomorphism, bijectively onto the submodule $ \Alt^R(M^r,N) $ of $ \Mult^R(M^r,N) $.
\end{itemize}
\end{rem}

\chapter{Tensor Product and Related Constructions}

\section{Basic constructions}

\begin{dfn}
\label{def 7}
Let $M_1\cdots,M_r,M$ be $R$-module schemes and $M'$ a group scheme over
$S$.
\begin{itemize}
\item[(i)] A pseudo-$R$-multilinear morphism
$\tau:M_1\times\dotsb\times M_r\to M'$, or by abuse of terminology,
the group scheme $M'$, is called a \emph{tensor product of}
$M_i, (i=1,\dotsb, r)$ 
if, for all group schemes $N$ over $S$,
the induced morphism
$$\tau^*:\Hom(M',N)\to\widetilde{\Mult}^R(M_1\times\dotsb\times M_r,N),\quad
\psi\mapsto\psi\circ\tau,$$ is an isomorphism. If such $M'$ and
$\tau$ exist, we write $M_1\otimes_R\dotsb\otimes_R M_r$ for $M'$ and call $\tau$ \emph{the universal multilinear morphism} defining $ M_1\otimes_R\dotsb\otimes_R M_r $.

\item[(ii)] A symmetric pseudo-$R$-multilinear morphism $\sigma:M^r\to M'$,
or by abuse of terminology, the group scheme $M'$, is called an
\emph{$r^{\text{th}}$ symmetric power} of $M$ over $R$, if for all
group schemes $N$ over $S$, the induced morphism
$$\sigma^*:\Hom(M',N)\to \widetilde{\Sym}^R(M^r,N),\quad \psi\mapsto \psi\circ\sigma,$$ is an
isomorphism. If such $M'$ and $\sigma$ exist, we write $\underset{R}{S}^rM$ for
$M'$ and call $ \sigma $ \emph{the universal symmetric morphism} defining $\underset{R}{S}^rM$.

\item[(iii)] An alternating pseudo-$R$-multilinear morphism $\lambda:M^r\to M'$,
or by abuse of terminology, the group scheme $M'$, is called an
\emph{$r^{\text{th}}$ exterior power} of $M$ over $R$, if for all
group schemes $N$ over $S$, the induced morphism
$$\lambda^*:\Hom(M',N)\to \widetilde{\Alt}^R(M^r,N),\quad \psi\mapsto \psi\circ\lambda,$$ is an
isomorphism. If such $M'$ and $\lambda$ exist, we write $\underset{R}{\bigwedge}^rM$
for $M'$ and call $ \lambda $ \emph{the universal alternating morphism} defining $\underset{R}{\bigwedge}^rM$.
\end{itemize}
\end{dfn}

\begin{rem}
\label{rem 4}
If $M_1\otimes_R\dotsb\otimes_R M_r$ respectively $\underset{R}{S}^rM$ respectively $\underset{R}{\bigwedge}^rM$ exists, it with the pseudo-$R$-multilinear morphism $\tau:M_1\times\dotsb\times M_r\to M_1\otimes_R\dotsb\otimes_R M_r$ respectively $\sigma:M^r\to \underset{R}{S}^rM$, respectively $\lambda:M^r\to \underset{R}{\bigwedge}^rM$, is
unique up to unique isomorphism and so, in the sequel, we will say ``the tensor product",``the symmetric power" and ``the exterior power".
\end{rem}

\begin{dfn}
\label{def4111}
Let $k$ be a field and $M$ an $R$-module scheme over $k$. We set $ M^*:=\invlim_{M'}M'^* $, where $M'$ runs through all finite subgroup schemes of $M$.
\end{dfn}

\begin{thm}
\label{thm41}
Let $k$ be a field and $ M_1,\dots,M_r  $ and $M$ profinite $R$-module schemes over $k$. 
\begin{itemize}
\item[1)] $ M_1\otimes_R\dots\otimes_RM_r $ exists and is profinite. If $M_i$ are finite, then \[M_1\otimes_R\dots\otimes_RM_r\cong \widetilde{\innMult}^R(M_1\times\dots\times M_r,\BG_m)^*.\]
\item[2)] $ \epR^jM $ exists and is profinite. If $M$ is finite, then \[ \epR^jM\cong\widetilde{\innAlt}^R(M^j,\BG_m)^*. \]
\item[3)] $ \uset{R}{S}^jM $ exists and is profinite. If $M$ is finite, then \[ \uset{R}{S}^jM\cong\widetilde{\innSym}^R(M^j,\BG_m)^*. \]
\end{itemize}
\end{thm}

\begin{proof}
Theorems 2.1.6 (p. 21) and 2.3.3 (p. 24) in \cite{P} are the same as this theorem when $R=\BZ$, i.e., this result for group schemes over $k$. Little modification of their proof, due to the presence of $R$, will give a proof of this theorem and therefore, we will not prove the theorem. However, one can also prove this theorem using those two theorems in \cite{P}. Indeed, these tensor objects ``over $R$" (tensor product, exterior and symmetric powers) are quotients of the same tensor objects ``over $\BZ$" by the subgroups generated by the obvious relations due to scalar multiplication of $R$, e.g., for the tensor product, we should mod out the relations \[am_1\otimes\dots\otimes m_r=m_1\otimes am_2\otimes\dots\otimes m_r=\dots=m_1\otimes\dots\otimes m_{r-1}\otimes am_r\] for every $a\in R$ and every $ m_i\in M_i $. To be more precise, for every $a\in R$, let $a_i:M_1\otimes\dots\otimes M_r\to M_1\otimes\dots\otimes M_r$ denote the morphism \[ m_1\otimes\dots\otimes m_r\mapsto m_1\otimes\dots\otimes m_{i-1}\otimes am_i\otimes m_{i+1}\otimes\dots\otimes m_r\] and denote by $ N_{i,j}(a) $ the image of the morphism $ a_i-a_j $ and finally set \[N':=\uset{i,j\leq r, a\in R}{\sum}N_{i,j}(a).\] Then $ M_1\otimes_R\dots\otimes_R M_r=M_1\otimes\dots\otimes M_r/N' $.
\end{proof}

\begin{rem}
\label{rem 7}
Let $ \phi_i:M_i\to N_i$ for $ i=1,\cdots,r $ and $\phi:M\to N $ be $R$-linear morphisms between $R$-module schemes over $S$.
\begin{itemize}
\item[1)] The morphism $ \phi_1\times\cdots\times\phi_r:M_1\times\cdots\times M_r\to N_1\times\cdots\times N_r$ composed with the tensor product $\tau':N_1\times\cdots\times N_r\to N_1\otimes_R\dotsb\otimes_R N_r$ is pseudo-$R$-multilinear and therefore, by the universal property of the tensor product we have a unique homomorphism $\phi_1\otimes_R\cdots\otimes_R\phi_r:M_1\otimes_R\dotsb\otimes_R M_r\to N_1\otimes_R\dotsb\otimes_R N_r$ such that the following diagram commutes:
\[ \xymatrix{
M_1\times\cdots\times M_r\ar[rr]^{\phi_1\times\cdots\times\phi_r}\ar[d]_{\tau}&&N_1\times\cdots\times N_r\ar[d]^{\tau'}\\
M_1\otimes_R\dotsb\otimes_R M_r\ar@{-->}[rr]_{\phi_1\otimes_R\cdots\otimes_R\phi_r}&&N_1\otimes_R\dotsb\otimes_R N_r.
 }\]
\item[2)] For all positive natural numbers $n$, we have a morphism $\phi^n:M^n\to N^n$ whose composition with the $n^{\text{th}}$ symmetric power $\sigma':N^n\to \underset{R}{S}^nN$ is a symmetric pseudo-$R$-multilinear morphism, and therefore, from the universal property of the symmetric power, we obtain a unique homomorphism $ \underset{R}{S}^n\phi:\underset{R}{S}^nM\to \underset{R}{S}^nN $ such that the following diagram is commutative:
\[\xymatrix{
M^n\ar[r]^{\phi^n}\ar[d]_{\sigma}&N^n\ar[d]^{\sigma'}\\
\underset{R}{S}^nM\ar@{-->}[r]_{\underset{R}{S}^n\phi}&\underset{R}{S}^nN.
}\]
\item[3)] Likewise to the symmetric power, we obtain for every positive natural number $n$, a homomorphism $\underset{R}{\bigwedge}^n\phi:\underset{R}{\bigwedge}^nM\to \underset{R}{\bigwedge}^nN$ such that the following diagram commutes:
\[\xymatrix{
M^n\ar[r]^{\phi^n}\ar[d]_{\lambda}&N^n\ar[d]^{\lambda'}\\
\underset{R}{\bigwedge}^nM\ar@{-->}[r]_{\underset{R}{\bigwedge}^n\phi}&\underset{R}{\bigwedge}^nN.
}\]
\end{itemize}
\end{rem}

\begin{rem}
\label{rem 8}
$  $
\begin{itemize}
\item[1)] From the definition of the tensor product, it follows that if the tensor product of $R$-module schemes $M_1,\cdots, M_n$ exists, then it possesses an $R$-module structure, which is given by the morphism \[\Id_{M_1}\otimes_R\cdots\otimes_R\Id_{M_{i-1}}\otimes_R (r.) \otimes_R\Id_{M_{i+1}}\otimes_R\cdots\otimes_R\Id_{M_n}\] where $r$ is an element of $R$ and $i$ is any index between 1 and $n$ (changing $i$, doesn't change the morphism). More precisely, the action of the element $r$ on the tensor product is equal to this morphism.
\item[2)] Similarly, let $M$ be an $R$-module scheme and $r\in R$. Then the following diagram is commutative:
\[\xymatrix{
M\times\cdots\times M\ar[rrr]^{\Id_M\times\cdots\times(r.)\times\cdots\times\Id_M}\ar[d]_{\sigma}&&&M\times\cdots\times M\ar[d]^{\sigma}\\
\underset{R}{S}^nM\ar[rrr]_{r.}&&&\underset{R}{S}^nM}\]where in the first row, the morphism $(r.)$ is at the $i^{\text{th}}$ place, with $i$ any index between 1 and $n$.
\item[3)] Again, with the above notations, we have a commutative diagram:
\[\xymatrix{
M\times\cdots\times M\ar[rrr]^{\Id_M\times\cdots\times(r.)\times\cdots\times\Id_M}\ar[d]_{\lambda}&&&M\times\cdots\times M\ar[d]^{\lambda}\\
\underset{R}{\bigwedge}^nM\ar[rrr]_{r.}&&&\underset{R}{\bigwedge}^nM}\]where in the first row, the morphism $(r.)$ is at the $i^{\text{th}}$ place, with $i$ any index between 1 and $n$.
\end{itemize}
\end{rem}

\begin{rem}
\label{rem 28}
It follows from Remark \ref{rem 8}, that if the tensor product $M_1\otimes_R\dotsb\otimes_R M_r$ exists, then it satisfies the following universal property as well:\\
let $N$ be an $R$-module scheme, then the morphism \[\Hom^R(M_1\otimes_R\cdots\otimes_RM_r,N)\to \Mult^R(M_1\times\dotsb\times M_r,N), \phi\mapsto \phi\circ\tau\]
is an isomorphism. In fact we have the following commutative diagram:
\[\xymatrix{
\Hom^R(M_1\otimes_R\cdots\otimes_RM_r,N)\ar@{^{ (}->}[d]\ar[rr]^{\cong}&&\Mult^R(M_1\times\dotsb\times M_r,N)\ar@{^{ (}->}[d]\\
\Hom(M_1\otimes_R\cdots\otimes_RM_r,N)\ar[rr]_{\cong}&&\widetilde{\Mult}^R(M_1\times\dotsb\times M_r,N).}\]
But, it is \textit{a priori} possible that the tensor product which satisfies the bottom isomorphism exists, whereas the tensor product as in the Definition \ref{def 7} does not.
\end{rem}

\begin{prop}
\label{prop 8}
Let $R\onto R'$ be a surjective ring homomorphism, and let $M, M_1,\cdots,M_r$ be $R'$-module schemes over $S$. Then we have canonical $R$-linear isomorphisms:
\begin{itemize}
\item[a)] $M_1\otimes_R\cdots\otimes_RM_r\cong M_1\otimes_{R'}\cdots\otimes_{R'}M_r$.
\item[b)] $\underset{R}{S}^nM\cong\underset{R'}{S}^nM$.
\item[c)] $\epR^nM\cong\underset{R'}{\bigwedge}^nM$.
\end{itemize}
\end{prop}

\begin{proof}[\textsc{Proof}. ]
We show the first isomorphism; the other two can be shown similarly.\\
We prove that the tensor product $$\tau':M_1\times\cdots\times M_r\to M_1\otimes_{R'}\cdots\otimes_{R'}M_r$$ has the universal property of the tensor product of $M_i$ over $R$. So, take an $R$-module scheme $N$ and an $R$-multilinear morphism $\phi:M_1\times\cdots\times M_r\to N$. Let $I$ be the kernel of the homomorphism $R\to R'$ and define the following submodule scheme of $N$:\[K:=\bigcap_{r\in I}\kernel(r.:N\to N).\]
It is clear from the definition of $K$ that it is an $R'$-module scheme.
Since the $R$-module structure of $M_i$ is given by the "restriction of scalars" $R\to R'$, and so $I\cdot M_i=0$, the $R$-multilinear morphism $\phi$ factors (in a unique way) through $\iota:K\into N$:
\[\xymatrix{
& M_1\times\cdots\times M_r\ar@{-->}[dl]_{\overline{\phi}}\ar[dr]^{\phi}&\\
K\ar@{^{ (}->}[rr]_{\iota}&& N.
}\]
Now, the morphism $\overline{\phi}$ is $R'$-multilinear and so there exists a unique $R'$-linear morphism $\phi':M_1\otimes_{R'}\cdots\otimes_{R'}M_r\to K$ such that $\phi'\circ \tau'=\overline{\phi}$. Putting this with the last diagram, we obtain:
\[\xymatrix{
& M_1\times\cdots\times M_r\ar[dl]_{\tau'} \ar[d]^{\overline{\phi}}\ar[dr]^{\phi}&\\
M_1\otimes_{R'}\cdots\otimes_{R'} M_r \ar[r]_{\phi'}& K \ar@{^{ (}->}[r]_{\iota}& N.
}\]
So, the composition $\iota\circ \phi'$ is $R$-linear and its composition with $\tau'$ is equal to the given multilinear morphism $\phi$. This shows the existence part of the universal property of the tensor product. The uniqueness part follows from the uniqueness of $\phi'$ and $\overline{\phi}$ and the fact that the homomorphism $\iota:K\to N$ is a monomorphism (note that since $M_1\otimes_{R'}\cdots\otimes_{R'} M_r$ is an $R'$-module scheme, every $R$-homomorphism from $M_1\otimes_{R'}\cdots\otimes_{R'} M_r$ to $N$ factors in a unique way through $\iota:K\into N$). 
\end{proof}

\section{Base change and Weil restriction}

\begin{dfn}
\label{def 15}
Let $ M $ and $ M_1, \cdots, M_r $ be $ R $-module schemes over a base scheme $ S $ and let $ T $ be an $ S $-scheme and $ r $ a positive natural number. Then we have a natural isomorphism $ M_{1_T}\times_T\cdots\times_TM_{r_T}\cong (M_1\times_S\cdots\times_SM_r)_T $.
\begin{itemize}
\item[(i)] The two universal multilinear morphisms defining $M_1\otimes_R\dotsb\otimes_R M_r$ and $M_{1_T}\otimes_R\dotsb\otimes_R M_{r_T}$ give rise to the following diagram
\[\xymatrix{
M_{1_T}\times_T\cdots\times_TM_{r_T}\ar[rr]^{\cong}\ar[d]_{\tau} && (M_1\times_S\cdots\times_SM_r)_T\ar[d]^{\tau_T} \\
M_{1_T}\otimes_R\dotsb\otimes_R M_{r_T} && (M_1\otimes_R\dotsb\otimes_R M_r)_T.
}\]
Now, the universal property of the tensor product $ M_{1_T}\otimes_R\dotsb\otimes_R M_{r_T} $ completes the diagram by a unique morphism 
\[\tau_{T/S}:M_{1_T}\otimes_R\dotsb\otimes_R M_{r_T} \to (M_1\otimes_R\dotsb\otimes_R M_r)_T\]
which we call \emph{the base change homomorphism of tensor product}.

\item[(ii)] The two universal symmetric multilinear morphisms defining $S^n_RM$ and $S^n_R(M_T)$ give rise to the following diagram
\[\xymatrix{
M_{T}^n\ar[rr]^{\cong}\ar[d]_{\sigma} && (M^n)_T\ar[d]^{\sigma_T} \\
S^n_R(M_T) && (S^n_RM)_T.
}\]
Now, the universal property of the symmetric power $ S^n_R(M_T) $ completes the diagram by a unique morphism 
\[\sigma_{T/S}:S^n_R(M_T) \to (S^n_RM)_T\]
which we call \emph{the base change homomorphism of symmetric power}.

\item[(ii)] The two universal alternating multilinear morphisms defining $\epR^nM$ and $\epR^n(M_T)$ give rise to the following diagram
\[\xymatrix{
M_{T}^n\ar[rr]^{\cong}\ar[d]_{\lambda} && (M^n)_T\ar[d]^{\lambda_T} \\
\epR^n(M_T) && (\epR^nM)_T.
}\]
Now, the universal property of the exterior power $ \epR^n(M_T) $ completes the diagram by a unique morphism 
\[\lambda_{T/S}:\epR^n(M_T) \to (\epR^nM)_T\]
which we call \emph{the base change homomorphism of exterior power}.
\end{itemize}
\end{dfn}

\begin{rem}
\label{rem 20}
The base change homomorphisms in the last definition need not be isomorphisms. However, as we will show (cf. Propositions \ref{prop 41} and \ref{lem8}), if $ S=\Spec E$ and $T=\Spec L$, where $L/E$ is either a separable or a finite field extension, then the three base change homomorphisms are isomorphisms.
\end{rem}

\begin{prop}
\label{prop 41}
Let $E$ be a field and $L/E$ a separable field extension. Then the three base change homomorphisms \[M_{1,L}\otimes_R\dots\otimes_R M_{r,L}\to (M_1\otimes_R\dots\otimes_R M_r)_L, \]  \[ \epR^r(M_L)\to (\epR^rM)_L \] and \[ \uset{R}{S}^r(M_L)\to (\uset{R}{S}^rM)_L \] are isomorphisms.
\end{prop}

\begin{proof}
We prove the statement just for the tensor product, and drop the similar proof of the exterior and symmetric powers. For simplicity, denote by $ T(M_L) $ and $ T(M) $ the tensor products $ M_{1,L}\otimes_R\dots\otimes_RM_{r,L} $ and  $ M_1\otimes_R\dots\otimes_RM_r $ respectively. By Theorem \ref{thm41} we know that $$ T(M_L)\cong\widetilde{\innMult}^R(M_{1,L}\times\dots\times M_{r,L},\BG_m)^* $$ and $$ T(M)_L\cong \big(\widetilde{\innMult}^R(M_1\times\dots\times M_r,\BG_m)^*\big)_L .$$ By definition, we have \[ \widetilde{\innMult}^R(M_1\times\dots\times M_r,\BG_m)^*=\uset{H}{\invlim}\,H^*, \] where $H$ runs through all finite subgroup schemes of $ G:=\widetilde{\innMult}^R(M_1\times\dots\times M_r,\BG_m) $. Consequently, we have \[\big(\widetilde{\innMult}^R(M_1\times\dots\times M_r,\BG_m)^*\big)_L \cong  (\uset{H}{\invlim}\,H^*)_L\cong  \uset{H}{\invlim}\,(H_L^*)  \] since the transition homomorphisms are epimorphisms and $L/E$ is flat. Let $ H $ be a finite subgroup scheme of $G$. Then $ H_L $ is a finite subgroup scheme of $$G_L=\widetilde{\innMult}^R(M_1\times\dots\times M_r,\BG_m)_L\cong  \widetilde{\innMult}^R(M_{1,L}\times\dots\times M_{r,L},\BG_m) .$$ Therefore, there exists a natural homomorphism $$\pi: \widetilde{\innMult}^R(M_{1,L}\times\dots\times M_{r,L},\BG_m)^*\to \widetilde{\innMult}^R(M_1\times\dots\times M_r,\BG_m)^*_L$$ which corresponds to the base change homomorphism $ b:T(M_L)\to T(M)_L $. Again, since the transition homomorphisms are epimorphisms, this homomorphism ( ``projection to sublimit") is an epimorphism.\\

Let us assume at first that $L$ is a separable closure of $E$ and denote by $\Gamma$ the absolute Galois group of $E$. We would like to show that $\pi$ is an isomorphism. It is sufficient to prove that the system of finite subgroups of $ G_L$, which are of the form $ H_L $ for a finite subgroup $H$ of $ G $, is cofinal in the system of all finite subgroups of $ G_L $. Let $ \bar{H} $ be a finite subgroup of $ G_L $. So, elements of $\bar{H}$ are closed points of $ G_L $, in other words, $ \bar{H} $ is a finite subgroup of the group $ G_L(L)=G(L) $ of $ L $-rational points of $G$. It follows that $ \Gamma\cdot \bar{H} $ (the Galois conjugate of $ \bar{H} $) is again a finite subset of $ G(L) $ (note that $ \Gamma $ acts canonically on $ G_L $ and $ G(L) $). The subgroup of $ G(L) $ generated by $ \Gamma\cdot \bar{H} $ is therefore finite, because $G$ is Abelian. This subgroup, denoted by $ H_{\Gamma} $, is a finite subgroup scheme of $ G_L $, which is invariant under the action of $ \Gamma $. By Galois descent, there exists a finite subgroup $ H $ of $ G $ such that $ H_L=H_{\Gamma} $. As $ \bar{H}\subseteq H_{\Gamma} $, we have that $ \bar{H}\subseteq H_L $, proving the claim.\\

Now, assume that $L$ is a separable extension of $E$ and thus subextension of a separable closure $ E_s $, of $E$. The base change homomorphisms yield the following commutative triangle \[ \xymatrix{&T(M_{E_s})\ar[dr]^{b_{E_s/E}}\ar[dl]_{b_{E_s/L}}&\\ T(M_L)_{E_s}\ar[rr]_{b_{E_s}}&&T(M)_{E_s}.} \] By the above discussion, $ b_{E_s/E} $ is an isomorphism and $ b_{E_s/L} $ is an epimorphism. It follows that $ b_{E_s/L} $ is also a monomorphism and thus an isomorphism. This implies that $ b_{E_s} $  is an isomorphism. This homomorphism is the extension to $ E_s $ of the base change homomorphism $b: T(M_L)\to T(M)_L $. As this extension is an isomorphism, we conclude that $ b $ is an isomorphism as well.
\end{proof}

Recall that if $E$ is a field and $L$ is a finite extension of $E$, then the Weil restriction $ \res_{L/E} $ is the right adjoint of the base change functor from the category of affine schemes over $E$ to the category of affine schemes over $L$, i.e., for every affine $E$-scheme $X$ and every affine $L$-scheme $Y$ we have a bijection of sets: \[ \Mor_L(X_L,Y)\cong \Mor_E(X,\res_{L/E}Y). \] In fact, the Weil restriction is the restriction, to the full subcategory of affine schemes, of the pushforward functor from the category of fppf sheaves over $\Spec L$ to the category of fppf sheaves over $\Spec E$. Explicitly, this functor sends an fppf sheaf $\CF$ over $\Spec L$ to the sheaf \[ T\mapsto \res_{L/E}\CF(T):=\CF(T_L).\]

Recall also that the Weil restriction preserves the group objects, i.e., $ \res_{L/E} H $ is an affine group scheme over $E$ if $H$ is an affine group scheme over $L$ and the above bijection restricts to a group isomorphism \[ \Hom_L(G_L,H)\cong \Hom_E(G,\res_{L/E}H) \] for every affine group scheme $G$ over $E$. It also follows from the adjunction that if $H$ is an affine $R$-module scheme over $L$, then $ \res_{L/E}H $ is an affine $R$-module scheme over $E$ and the above isomorphism induces an $R$-module isomorphisms, when restricted to the subgroup of $R$-linear homomorphisms. Finally, recall that the Weil restriction commutes with base change, that is, if $T$ is an $E$-scheme, then there exists a canonical isomorphism $$ (\res_{L/E}Y)_T\cong \res_{L\times_ET/T}(Y\times_ET). $$

\begin{lem}
\label{lem7}
Let $ E $ be a field and $ L/E $ a finite field extension. Let $M$ be an affine $R$-module scheme over $E$ and $ N $ an fppf sheaf of $R$-modules over $\Spec L$. Then there exists a canonical sheaf isomorphism \[ \res_{L/E}\innHom_L(M_L,N)\cong \innHom_E(M,\res_{L/E}N) \]  which is the ``sheafified" version of the Weil restriction, i.e., the global sections of this isomorphism deliver the usual Weil restriction.
\end{lem}

\begin{proof}
Let $T$ be a scheme over $E$. We have the following isomorphisms \[ \innHom_E(M,\res_{L/E}N)(T)=\Hom_T(M_T,(\res_{L/E}N)_T)\cong \] \[\Hom_T(M_T,\res_{T_L/T}(N\times_E T)\cong \Hom_T(M_T,\res_{T_L/T}(N\times_L T_L))\cong\] \[ \Hom_{T_L}((M_T)_{T_L},N\times_L T_L))\cong \Hom_{T_L}((M_L)_{T_L},N\times_L T_L))=\] \[  \innHom_L(M_L,N)(T_L)\cong \res_{L/E}(\innHom_L(M_L,N))(T),\] where the first isomorphism (not equality) follows from the fact that the Weil restriction commutes with base change. The second isomorphisms is induced by the canonical isomorphism $ N\times_E T\cong N\times_L T_L $. The third isomorphism follows from the adjunction property of the Weil restriction. The fourth isomorphism follows from the canonical isomorphism $ (M_T)_{T_L}\cong (M_L)_{T_L}$. Finally, the last isomorphism is given by the ``definition" of the Weil restriction.
\end{proof}

\begin{prop}
\label{prop 49}
Let $ E $ be a field and $ L/E $ a finite field extension. Let $ M_1, M_2,\dots, M_r$  and $M$ be affine $R$-module schemes over $E$ and $N$ an fppf sheaf of Abelian groups over $L$.
\begin{itemize}
\item[a)] The bijection \[ \Mor_L(M_{1,L}\times\dots\times M_{r,L}, N)\cong \Mor_E(M_1\times\dots\times M_r, \res_{L/E}N) \] restricts to an  isomorphism \[ \widetilde{\Mult}^R_L(M_{1,L}\times\dots\times M_{r,L}, N)\cong \widetilde{\Mult}^R_E(M_1\times\dots\times M_r, \res_{L/E}N).\] 
\item[b)] The  isomorphism $ \widetilde{\Mult}^R_L(M_L^r,N)\cong \widetilde{\Mult}^R_E(M,\res_{L/E}N) $ of part a) restricts to an isomorphism \[ \widetilde{\Alt}^R_L(M_L^r,N)\cong \widetilde{\Alt}^R_E(M,\res_{L/E}N).\]
\item[c)] The isomorphism $ \widetilde{\Mult}^R_L(M_L^r,N)\cong \widetilde{\Mult}^R_E(M,\res_{L/E}N) $ of part a) restricts to an isomorphism \[ \widetilde{\Sym}^R_L(M_L^r,N)\cong \widetilde{\Sym}^R_E(M,\res_{L/E}N).\]
\end{itemize}
\end{prop}

\begin{proof}
\begin{itemize}
\item[a)] We will proceed by induction on $r$. If $r=1$, then the statement follows from the adjunction of the Weil restriction or more generally of the pushforward, discussed before Lemma \ref{lem7}. So, assume that $r>1$ and that we know the result for smaller integers. We have \[ \widetilde{\Mult}^R_L(M_{1,L}\times\dots\times M_{r,L}, N)\cong \] \[\widetilde{\Mult}^R_L(M_{1,L}\times\dots\times M_{r-1,L}, \innHom_L(M_{r,L},N))\ovset{\text{ind'n}}{\cong} \] \[ \widetilde{\Mult}^R_E(M_{1}\times\dots\times M_{r-1}, \res_{L/E}\innHom_L(M_{r,L},N))\ovset{\ref{lem7}}{\cong} \] \[ \widetilde{\Mult}^R_E(M_{1}\times\dots\times M_{r-1}, \innHom_E(M_{r},\res_{L/E}N))\cong \] \[ \widetilde{\Mult}^R_E(M_1\times\dots\times M_r, \res_{L/E}N). \]
\item[b)] By the adjunction property, we know that there exists a unique $R$-linear homomorphism $ \sigma_N:(\res_{L/E}N)_L\to N $ with the following universal property: for every $E$-scheme $X$ the map \[ \Mor_E(X,\res_{L/E}N)\to\Mor_L(X_L,N) \] induced by base change to $L$ and composing with $ \sigma_N $ is an isomorphism. It follows that if $ \phi:M^r\to \res_{L/E}N $ is alternating, the morphism $ \sigma_N\circ\phi_L $ is alternating too and therefore, the $R$-linear isomorphism of part $a)$ restricts to an $R$-linear monomorphism \[ \widetilde{\Alt}^R_E(M^r,\res_{L/E}N)\into \widetilde{\Alt}^R_L(M_L^r,N).\] We have to show that this is surjective too. Take an alternating morphism $ \phi:M_L^r\to N $. Because of part $a)$, we know that there exists a pseudo-$R$-multilinear morphism $ \psi:M^r\to\res_{L/EN} $ such that  $ \phi=\sigma_N\circ \psi_L $. We want to show that $ \psi $ is alternating. For any $ 1\leq i<j\leq r $, let $$\Delta_{i,j}^r:M^{r-1}\to M^r,\quad (m_1,\dots,m_{r-1})\mapsto (m_1,\dots,m_{j-1},m_i,m_j,\dots,m_r)$$ denote the generalized diagonal embedding equating the $ i^{\text{th}} $ and $ j^{\text{th}} $ components. By functoriality of the Weil restriction, the following diagram commutes: \[ \xymatrix{\widetilde{\Mult}^R_E(M^r,\res_{L/E}N)\ar[rr]^{\cong}\ar[d]_{\Delta_{i,j}^{*}}&&\widetilde{\Mult}^R_L(M_L^r,N)\ar[d]^{\Delta_{i,j}^{*}}\\\widetilde{\Mult}^R_E(M^{r-1},\res_{L/E}N)\ar[rr]_{\cong}&&\widetilde{\Mult}^R_L(M_L^{r-1},N).}\] Since $ \phi $ is mapped to zero under the homomorphism $$ \widetilde{\Mult}^R_L(M_L^r,N)\to\widetilde{\Mult}^R_L(M^{r-1}_L,N) $$ (because it is alternating), the morphism $ \psi $ lies in the kernel of the homomorphism $$ \widetilde{\Mult}^R_E(M^r,\res_{L/E}N)\to\widetilde{\Mult}^R_E(M^{r-1},\res_{L/E}N),$$ and this holds for every pair $ i<j $. Hence $ \psi $ is alternating.
\item[c)] Similar arguments prove the desired isomorphism.
\end{itemize}
\end{proof}

\begin{prop}
\label{lem8}
Let $ E $ be a field and $ L/E $ a finite field extension. Let $ M_1, M_2,\dots, M_r$  and $M$ be finite $R$-module schemes over $E$. Then the three base change homomorphisms \[M_{1,L}\otimes_R\dots\otimes_R M_{r,L}\to (M_1\otimes_R\dots\otimes_R M_r)_L, \]  \[ \epR^r(M_L)\to (\epR^rM)_L \] and \[ \uset{R}{S}^r(M_L)\to (\uset{R}{S}^rM)_L \] are isomorphisms.
\end{prop}

\begin{proof}
We prove the statement for the tensor product and drop the proofs for the exterior and symmetric power, as they can be similarly proved.\\

Let $N$ be an affine group scheme over $L$. By functoriality of the Weil restriction, we have the following commutative diagram: \[ \xymatrix{\Hom_E(M_1\otimes_R\dots\otimes_R M_r,\res_{L/E}N)\ar[d]_{\cong}\ar[r]^{\cong}&\Hom_L((M_1\otimes_R\dots\otimes_R M_r)_L,N)\ar[d]\\\widetilde{\Mult}_E^R(M_1\times_R\dots\times_R M_r,\res_{L/E}N)\ar[r]_{\cong}&\widetilde{\Mult}_L^R(M_{1,L}\times_R\dots\times_R M_{r,L},N)} \] where the vertical homomorphisms are induced by the universal $R$-multilinear morphisms. The left vertical morphism is an isomorphism by the definition of $ M_1\otimes_R\dots\otimes_R M_r $. The horizontal morphisms are isomorphism by the Weil restriction. Thus the right vertical homomorphism is an isomorphism as well. The base change morphism induces a commutative triangle \[ \xymatrix{\Hom_L((M_1\otimes_R\dots\otimes_R M_r)_L,N)\ar[r]\ar[d]&\Hom_L(M_{1,L}\otimes_R\dots\otimes_R M_{r,L},N)\ar[dl]\\ \widetilde{\Mult}_L^R(M_{1,L}\times_R\dots\times_R M_{r,L},N)& }\] where the vertical and oblique morphisms are isomorphisms. In fact the vertical arrow is an isomorphism because of the last diagram and the oblique arrow is an isomorphism by the definition of $ M_{1,L}\otimes_R\dots\otimes_R M_{r,L} $. Consequently, the homomorphism \[ \Hom_L((M_1\otimes_R\dots\otimes_R M_r)_L,N)\to \Hom_L(M_{1,L}\otimes_R\dots\otimes_R M_{r,L},N) \] is an isomorphism, and this holds for every affine group scheme $N$ over $L$. It follows that the base change homomorphism is an isomorphism.
\end{proof}

\section{Main properties of exterior powers}

\begin{prop}
\label{prop 5} Let $M$ be an $R$-module scheme over $S$. If
$\underset{R}{\bigwedge}^nM=0$ then we have $\underset{R}{\bigwedge}^mM=0$ for all $m\geq n$.
\end{prop}

\begin{proof}[\textsc{Proof}. ]
We show that $\underset{R}{\bigwedge}^{n+1}M=0$; the result follows immediately by induction. Let $N$ be an $R$-module scheme. By the definition, we have an isomorphism $$\Hom^R(\underset{R}{\bigwedge}^{n+1}M,N)\cong\Alt^R(M^{n+1},N).$$ Under the isomorphism
$$\Mult(M^{n+1},H)\cong\Mult(M^n,\innHom(M,H))$$
given in Lemma \ref{lem 19}, the image of the sub-$R$-module $\Alt^R(M^{n+1},N)$ lies in the sub-$R$-module $\Alt^R(M^n,\innHom^R(M,N))$ of $\Mult^R(M^n,\innHom^R(M,N))$, i.e., we have an isomorphism
\[\Alt^R(M^{n+1},N)\cong \Alt^R(M^n,\innHom^R(M,N)).\]
Again by definition, we have
$$\Alt^R(M^n,\innHom^R(M,N))\cong\Hom^R(\underset{R}{\bigwedge}^nM,\innHom^R(M,N)).$$
The latter $R$-module is trivial by hypothesis and thus, we have $$\Hom^r(\underset{R}{\bigwedge}^{n+1}M,N)\cong \Alt^R(M^{n+1},N)\cong \Alt^R(M^n,\innHom^R(M,N))=0$$ for all $R$-module schemes $N$, which implies that $\underset{R}{\bigwedge}^{n+1}M=0$.
\end{proof}

\begin{prop}
\label{prop 6} Let $M_1, M_2$ and $P$ be $R$-module schemes over $S$
and $r_1,r_2$ two positive integers. We have a natural isomorphism
$$\lambda_{1,2}^*:\Alt^R(M_1^{r_1}\times M_2^{r_2},P)\cong\Hom^R(\underset{R}{\bigwedge}^{r_1}M_1\otimes_R\underset{R}{\bigwedge}^{r_2}M_2,P).$$
\end{prop}

\begin{proof}[\textsc{Proof}. ]
Using Proposition \ref{prop 4} we have a natural isomorphism
$$\Alt^R(M_1^{r_1}\times M_2^{r_2},P)\cong\Alt^R(M_1^{r_1},\innAlt^R(M_2^{r_2},P))$$
and by definition of $\underset{R}{\bigwedge}^{r_1}M_1$ this is isomorphic to $\Hom^R(\underset{R}{\bigwedge}^{r_1}M_1,\innAlt^R(M_2^{r_2},P))$ which is again by Proposition \ref{prop 4} isomorphic to $$\Alt^R(\underset{R}{\bigwedge}^{r_1}M_1\times M_2^{r_2},P)\cong\Alt^R(M_2^{r_2},\innHom^R(\underset{R}{\bigwedge}^{r_1}M_1,P))\cong$$ $$\Hom^R(\underset{R}{\bigwedge}^{r_2}M_2,\innHom^R(\underset{R}{\bigwedge}^{r_1}M_1,P))\cong\Mult^R(\underset{R}{\bigwedge}^{r_1}M_1\times\underset{R}{\bigwedge}^{r_2}M_2,P)\cong$$ $$\Hom^R(\underset{R}{\bigwedge}^{r_1}M_1\otimes\underset{R}{\bigwedge}^{r_2}M_2,P).$$
\end{proof}

\begin{rem}
\label{rem 5}
$  $
\begin{itemize}
\item[1)] Let $\lambda_{1,2}:M_1^{r_1}\times M_2^{r_2}\to\underset{R}{\bigwedge}^{r_1}M_1\otimes_R
\underset{R}{\bigwedge}^{r_2}M_2$ be the multilinear morphism in $\Alt^R(M_1^{r_1}\times M_2^{r_2},\underset{R}{\bigwedge}^{r_1}M_1\otimes_R\underset{R}{\bigwedge}^{r_2}M_2)$ that maps to the identity of\\
$\underset{R}{\bigwedge}^{r_1}M_1\otimes_R\underset{R}{\bigwedge}^{r_2}M_2$ by the isomorphism given in the Proposition \ref{prop 6}. Then, one can easily see that the $R$-module scheme $\underset{R}{\bigwedge}^{r_1}M_1\otimes_R\underset{R}{\bigwedge}^{r_2}M_2$ has the following universal property: Given any $R$-multilinear morphism $\phi:M_1^{r_1}\times M_2^{r_2}\to P$ which is alternating in $M_1^{r_1}$ and $M_2^{r_2}$, there exists a unique $R$-linear homomorphism $$\widetilde{\phi}:\underset{R}{\bigwedge}^{r_1}M_1\otimes_R\underset{R}{\bigwedge}^{r_2}M_2\to P$$ making the following diagram commute:
$$\xymatrix{M_1^{r_1}\times
M_2^{r_2}\ar[rr]^{\phi}\ar[dr]_{\lambda_{1,2}}&&P\\
&\underset{R}{\bigwedge}^{r_1}M_1\otimes_R\underset{R}{\bigwedge}^{r_2}M_2.\ar[ur]_{\exists!\widetilde{\phi}}&
}$$
\item[2)] We can generalize the Proposition \ref{prop 6}, i.e., if $M_1,\cdots,M_n$ are  $R$-module schemes and $r_1,\cdots,r_n$ are positive integers, then there is a multilinear morphism
$$\lambda_{1,\cdots,n}:M_1^{r_1}\times\cdots\times M_n^{r_n}\to\underset{R}{\bigwedge}^{r_1}M_1\otimes_R\cdots\otimes_R\underset{R}{\bigwedge}^{r_n}M_n$$
alternating in each $M_i^{r_i}$ such that the homomorphism
$$\Hom^R(\underset{R}{\bigwedge}^{r_1}M_1\otimes_R\cdots\otimes_R\underset{R}{\bigwedge}^{r_n}M_n,P)\to\Alt(M_1^{r_1}\times\cdots\times M_n^{r_n},P)$$
$$\phi\mapsto\phi\circ\lambda_{1,\cdots,n}$$is an isomorphism.
\item[3)] Let $\phi_i:M_i\to M$, $i=1,2$ be $R$-linear morphisms and $r_1,r_2$ two positive natural numbers. We have a morphism $\phi_1^{r_1}\times\phi_2^{r_2}:M_1^{r_1}\times M_2^{r_2}\to M^{r_1+r_2}$ whose composition with $\lambda:M^{r_1+r_2}\to\epR^{r_1+r_2}M$ lies in the module $\Alt^{R}(M_1^{r_1}\times M_2^{r_2},\epR^{r_1+r_2}M)$. Using the isomorphism of Proposition \ref{prop 6} with $P$ replaced by $\epR^{r_1+r_2}M$, we obtain an $R$-homomorphism $$\ep^{r_1}\phi_1\wedge\ep^{r_2}\phi_2:\epR^{r_2}M_1\otimes_R\epR^{r_2}M_2\to \epR^{r_1+r_2}M.$$
That is, we have the following commutative diagram:
\[\xymatrix{
M_1^{r_1}\times M_2^{r_2}\ar[d]_{\lambda_{1,2}}\ar[rrr]^{\phi_1^{r_1}\times\phi_2^{r_2}}&&&M^{r_1+r_2}\ar[d]^{\lambda}\\
\epR^{r_1}M_1\otimes_R\epR^{r_2}M_2\ar@{-->}[rrr]_{\ep^{r_1}\phi_1\wedge\ep^{r_2}\phi_2}&&&\epR^{r_1+r_2}M.}\] If $r_1$ (respectively $r_2$) is equal to $1$, we write $\phi_1\wedge\ep^{r_2}\phi_2$ instead of $\ep^{1}\phi_1\wedge\ep^{r_2}\phi_2$ \ (respectively $\ep^{r_1}\phi_1\wedge\phi_2$ instead of $\ep^{r_1}\phi_1\wedge\ep^{1}\phi_2$). 
\item[4)] Let $P$ be an $R$-module scheme over $S$. The homomorphism $\ep^{r_1}\phi_1\wedge\ep^{r_2}\phi_2$ induces an $R$-module homomorphism \[ \Hom^R(\epR^{r_1+r_2}M,P)\to \Hom^R(\epR^{r_1}M_1\otimes_R\epR^{r_2}M_2,P) \] that together with the isomorphisms of Proposition \ref{prop 6} and Definition \ref{def 7} $(iii)$ gives the following commutative diagram:
\[\xymatrix{
\Alt^R(M^{r_1+r_2},P)\ar[d]_{\lambda^*}^{\cong}\ar[rr]&&\Alt^R(M_1^{r_1}\times M_2^{r_2},P)\ar[d]^{\lambda_{1,2}^*}_{\cong}\\
\Hom^R(\epR^{r_1+r_2}M,P)\ar[rr]&&\Hom^R(\epR^{r_1}M_1\otimes_R\epR^{r_2}M_2,P).
}\]
\end{itemize}
\end{rem}

\begin{lem}
\label{lem 5}
Let $\phi_i:M_i\to M,\ i=1,2$ be $R$-epimorphisms of $R$-module schemes over $S$ and $r_1, r_2$ two positive natural numbers. Then the morphism (cf. Remark \ref{rem 5} 3))
\[\ep^{r_1}\phi_1\wedge\ep^{r_2}\phi_2:\epR^{r_1}M_1\otimes_R\epR^{r_2}M_2\to \epR^{r_1+r_2}M\]
is an epimorphism. In particular, we have a canonical epimorphism
\[\epR^{r_1}M\otimes_R\epR^{r_2}M\onto \epR^{r_1+r_2}M.\]
\end{lem}

\begin{proof}[\textsc{Proof}]
The morphism $ \phi_1^{r_1}\times \phi_2^{r_2}:M_1^{r_2}\times M_2^{r_2}\to M^{r_1+r_2}$ is an epimorphism. Therefore, for every $R$-module scheme $P$, the induced morphism \[\Alt^R(M^{r_1+r_2},P)\to \Alt^R(M_1^{r_2}\times M_2^{r_2},P)\]
is injective.\\

Using the diagram of Remark \ref{rem 5} 4), we conclude that the induced morphism \[ \Hom^R(\epR^{r_1+r_2}M,P)\to \Hom^R(\epR^{r_1}M_1\otimes_R\epR^{r_2}M_2,P) \] is injective for all $R$-module schemes $P$. This proves the first part of the statement. The second part follows from the first part by replacing $\phi_i$ with the identity morphism of $M$. 
\end{proof}

\textbf{Notations.} Let $M',M''$ be sub-$R$-module schemes of $M$ and $P$ and $N$ $R$-module schemes. By $\Alt^R(M'^r\times M''^s\times P^t,N)$ we
mean the module of $R$-multilinear morphisms that are alternating in
$M'^r, M''^s$ and $(M'\cap M'')^{r+s}$ (as a submodule scheme of $M^{r+s})$) and in $P^{t}$. When we say that a
multilinear morphism $$M'^r\times M''^s\times P^t\to N$$ is
alternating, we mean that it belongs to the module $\Alt^R(M'^r\times
M''^s\times P^t,N)$. Likewise, we define the module
$$\Alt^R(M_1^{r_1}\times \cdots\times M_n^{r_n}\times P_1^{s_1}\times
\cdots\times P_m^{s_m},N)$$ with $M_i$ sub-$R$-module schemes of $M$ and
$P_j$'s arbitrary $R$-module schemes.

\begin{lem}
\label{lem 1}
Let $\pi:M\onto M''$ be an epimorphism and let $\phi:M''^r\to H$ be an $R$-multilinear morphism such that the composition $\phi\circ \pi^r:M^r\to H$ is alternating. Then $\phi$ is alternating as well.
\end{lem}

\begin{proof}[\textsc{Proof}]
The morphism $\pi$ induces a morphism $\Delta^r \pi:\Delta^r M\to\Delta^r M''$ between diagonals and since the morphism $\pi$ is an epimorphism, the morphism $\Delta^r\pi$ is an epimorphism too.

Similarly, we have an induced epimorphism between $\Delta_{ij}^rM\subset M^r$ and $\Delta_{ij}^r M''\subset M''^r$ for all $1\leq i < j\leq r$, which we denote by $\Delta_{ij}^r \pi$. In order to show that $\phi$ is alternating, we must show that for any $1\leq i<j\leq r$ the composition $\Delta_{ij}^rM''\into M''^r\arrover{\phi}N$ is trivial. But we have a commutative diagram
$$\xymatrix{
\Delta^r_{ij}M\ar[r]^{\Delta\pi}\ar@{^{
(}->}[d]_{\iota}&\Delta^r_{ij}M''\ar@{^{ (}->}[d]^{\iota''}&\\
M^r\ar[r]_{\pi^r}&M''^r\ar[r]_{\phi}&N.}$$
and since the composite $\phi\circ\pi^r$ is alternating, the composition $\phi\circ\pi^r\circ\iota$ is trivial, and so is the composition $\phi\circ\iota''\circ\Delta\pi$. The morphism $\Delta\pi$ is an epimorphism, which implies that $\phi\circ\iota''$ is trivial.
\end{proof}

\begin{rem}
\label{rem 6}
Let $M'$ be a sub-$R$-module scheme of $M$ and $\pi:M\onto M''$ an epimorphism. It can be shown in the same fashion that if the composition of a multilinear morphism $M''^r\times M^s\times M'^t\to N$ with the epimorphism $$\pi^r\times\Id_{M's}\times\Id_{M'^t}:M^r\times M^s\times M'^t\to M''^r\times M^s\times M'^t$$ is alternating (in the sense of the Notations above), then this multilinear morphism is also alternating.
\end{rem}

\begin{lem}
\label{lem 2} Let $M_1\cdots,M_r$ be $R$-module schemes and $\psi:M_1\times\cdots\times M_r\to N$ an $R$-multilinear morphism. Assume that for some $1\leq i\leq r$ we have an exact sequence $M'_i\arrover{\iota} M_i\arrover{\pi} M''_i\to 0$. If the restriction $\psi|_{M_1\times \cdots\times M'_i\times\cdots\times M_r}$ is zero, then there is a unique multilinear morphism $$\psi': M_1\times \cdots\times M''_i\times\cdots\times M_r\to N$$ such that $\psi=\psi'\circ(\Id_{M_1}\times\cdots\times\pi\times \cdots\times\Id_{M_r})$ with $\pi$ at the $i^{\text{th}}$ place.
\end{lem}

\begin{proof}[\textsc{Proof}]
By functoriality of the isomorphism in Proposition \ref{prop 2} we have a commutative diagram (we omit the superscript $R$ to avoid heavy notations, and so the homomorphisms and multilinear morphisms are all $R$-linear or $R$-multilinear):
$$\xymatrix{
\Mult(M_1\times\cdots\times M_r,N)\ar[r]^{\cong\qquad\qquad\quad}\ar[d]^{\overline{\pi}}&\Hom(M''_i,\innMult(M_1\times\cdots\times\check{M_i}\times\cdots\times M_r,N))\ar[d]^{\pi^*}\\
\Mult(M_1\times\cdots\times M_r,N)\ar[d]^{\overline{\iota}}\ar[r]^{\cong\qquad\qquad\quad}&\Hom(M_i,\innMult(M_1\times\cdots\times\check{M_i}\times\cdots\times M_r,N))\ar[d]^{\iota^*}\\
\Mult(M_1\times\cdots\times M_r,N)\ar[r]^{\cong\qquad\qquad\quad}&\Hom(M'_i,\innMult(M_1\times\cdots\times\check{M_i}\times\cdots\times M_r,N))}$$
where the indicated morphisms are the obvious ones and $\check{M_i}$ means that this factor is omitted. The right column is exact and $\pi^*$ is injective, because the sequence $0\to M'_i\arrover{\iota} M_i\arrover{\pi} M''_i\to 0$ is exact and the functor $\Hom^R(\_ ,P)$ is left exact for any $R$-module scheme $P$. Therefore, the left column is exact too and $\overline{\pi}$ is injective. The morphism $\psi$ is an element of $\Mult^R(M_1\times\cdots\times M_r,N)$ which goes to zero under the morphism $\overline{\iota}$ (restriction morphism). By exactness, there is a unique multilinear morphism $\psi'\in\Mult^R(M_1\times\cdots\times M''_i\times\cdots\times M_r,N)$ which is mapped to $\psi$ under $\overline{\pi}$. This proves the lemma.
\end{proof}

\begin{lem}
\label{lem 3} Let $M_1\arrover{\iota} M_2\arrover{\pi} M_3\to 0$ be an exact sequence of $R$-module schemes over $S$ and $n$ a positive natural number. Then the sequence $$0\to\Alt^R(M_3^n,N)\to\Alt^R(M_2^n,N)\to \Alt^R(M_1\times M_2^{n-1},N)$$
is exact.
\end{lem}

\begin{proof}[\textsc{Proof}]
Since $ \pi:M_2\to M_3 $ is an epimorphism, the morphism $ \pi^n:M_2^n\to M_3^n$ is also an epimorphism and so the induced morphism $ \Alt^R(M_3^n,N)\to \Alt^R(M_2^n,N) $ is injective.\\

Let $\phi:M_2^n\to N$ be an alternating morphism and assume that the restriction $\phi|_{M_1\times M_2^{n-1}}$ is zero. We will show that
there is a multilinear morphism $\phi':M_3^n\to N$ such that $\phi=\phi'\circ\pi^n$. The result will then follow, since by Lemma \ref{lem 2}, $\phi'$ is also alternating, and it is thus inside the $R$-module $\Alt^R(M_3^n,N)$. This will prove the exactness at the middle of the sequence (the fact that the composite $$\Alt^R(M_3^n,N)\to\Alt^R(M_2^n,N)\to \Alt^R(M_1\times M_2^{n-1},N)$$ is zero follows directly from the fact that the composition $ \pi\circ \iota $ is zero) and the proof will be achieved.\\

Note that since $\phi$ is alternating and the restriction $\phi|_{M_1\times M_2^{n-1}}$ is zero, the restrictions $\phi|_{M_2^i\times M_1\times M_2^{n-i-1}}$ are zero for any $1\leq i\leq n-1$. Set $\phi_0=\phi$. We show by induction on $0\leq i\leq n$ that there is a multilinear morphism $\phi_i:M_3^i\times M_2^{n-i}\to N$ such that $\phi=\phi_i\circ(\pi^i\times\Id_{M_2^{n-i}})$. This is clear for $i=0$, so let $i>0$ and assume that we have $\phi_{i-1}$ with the stated property. Consider the following commutative diagram
$$\xymatrix{
M_2^{i-1}\times M_1\times M_2^{n-i}\ar[d]_{\widehat{\pi}}\ar@{^{(}->}[r]^{\rho}&M_2^{i-1}\times M_2\times M_2^{n-i}\ar[d]^{\widetilde{\pi}}\ar[dr]^{\phi}&\\
M_3^{i-1}\times M_1\times M_2^{n-i}\ar@{^{(}->}[r]^{\rho'}&M_3^{i-1}\times M_2\times M_2^{n-i}\ar[r]^{\qquad \quad\phi_{i-1}}&N}$$
where
$\widehat{\pi}=\pi^{i-1}\times\Id_{M_1}\times\Id_{M_2^{n-i}}$,
$\widetilde{\pi}=\pi^{i-1}\times\Id_{M_2}\times \Id_{M_2^{n-i}}$ and $\rho,\rho'$ are the inclusion morphisms. We know that $0=\phi\circ\rho=\phi_{i-1}\circ\widetilde{\pi}\circ \rho$, which implies that $\phi_{i-1}\circ\rho'\circ\widehat{\pi}=0$. The morphism $\widehat{\pi}$ is an epimorphism and so
$\phi_{i-1}\circ\rho'$, the restriction of $\phi_{i-1}$ to $ M_3^{i-1}\times M_1\times M_2^{n-i} $, is zero. We can therefore apply Lemma \ref{lem 2} and obtain a multilinear morphism $\phi_i:M_3^i\times M_2^{n-i}\to N$ such that $$\phi_i=\phi_{i-1}\circ(\Id_{M_3^{i-1}}\times\pi\times\Id_{M_2^{n-i}}).$$ We have thus $\phi=\phi_i\circ(\pi^i\times \Id_{M_2^{n-i}})$.\\

Now put $i=n$, the statement says that there is a multilinear morphism $\phi_n:M_3^n\to N$ with $\phi=\phi_n\circ\pi^n$. This $\phi_n$ is the required $\phi'$.
\end{proof}

For the following lemma, which is in some sense the ``linearized'' version of the previous one, we use notations introduced in Remark \ref{rem 5} $3)$.

\begin{lem}
\label{lem 4} 
Let $ M_1\arrover{\alpha} M_2\arrover{\beta} M_3\to 0 $ be an exact sequence of $R$-module schemes and $ n $ a positive natural number. Then the sequence
\[ M_1\otimes_R\underset{R}{\bigwedge}^{n-1}M_2\arrover{\alpha\wedge\ep^{n-1}\Id_{M_2}} \underset{R}{\bigwedge}^nM_2\arrover{\bigwedge^n\beta}\underset{R}{\bigwedge}^nM_3\to 0 \] is exact.
\end{lem}

\begin{proof}[\textsc{Proof}]
The statement follows essentially from Lemma \ref{lem 3} and Proposition \ref{prop 6}. Indeed, the exactness of this sequence is equivalent to the exactness of the sequence obtained by applying the contravariant functor $\Hom^R(-,N)$ on it for all $R$-module schemes $N$ over $S$. So, let $N$ be an arbitrary $R$-module scheme over $S$ and consider the following diagram:
\[\xymatrix{
0\ar[r]&\Hom(\epR^{n}M_3,N)\ar[d]_{\cong}^{\lambda_3^*}\ar[r]&\Hom(\epR^{n}M_2,N)\ar[d]_{\cong}^{\lambda_2^*}\ar[r]&\Hom(M_1\otimes_R\underset{R}{\bigwedge}^{n-1}M_2,N)\ar[d]_{\cong}^{\lambda_{1,2}^*}\\
0\ar[r]&\Alt^R(M_3^{n},N)\ar[r]&\Alt^R(M_2^{n},N)\ar[r]&\Alt^R(M_1\times M_2^{n-1},N)
}\]where the first two vertical morphisms, i.e., $\lambda_2^*,\lambda_3^*$ are given by the universal property of the exterior powers and the last one by Proposition \ref{prop 6}. The commutativity of the first square follows from Remark \ref{rem 7} 3) and that of the second square from Remark \ref{rem 5} 4). The second row is exact by Lemma \ref{lem 3}. The exactness of the first row follows immediately and this achieves the proof.
\end{proof}

\begin{thm}
\label{thm 1} 
Assume that $0\to M_1\arrover{\iota} M_2\arrover{\pi}
M_3\to 0$ is a short exact sequence of $R$-module schemes over $S$.
Let $n_2$ be a non-negative integer and write $n_2=n_1+n_3$ for non-negative integers $n_1$ and $n_3$. Consider the diagram
$$\xymatrix{
\Alt^R(M_2^{n_2},N)\ar[r]^{\rho\qquad}&\Alt^R(M_1^{n_1}\times M_2^{n_3},N)\\
&\Alt^R(M_1^{n_1}\times M_3^{n_3},N)\ar@{^{ (}->}[u]_{\pi^*}}$$ where
$\rho$ is the restriction morphism.
\begin{itemize}
\item[\emph{(a)}] If $\underset{R}{\bigwedge}^{n_3+1}M_3=0$, then $\rho$ is injective.
\item[\emph{(b)}] If $\underset{R}{\bigwedge}^{n_1+1}M_1=0$, then $\rho$ factors through $\pi^*$.
\item[\emph{(c)}] If both conditions hold, then there is a natural epimorphism
$$\zeta:\underset{R}{\bigwedge}^{n_1}M_1\otimes_R\underset{R}{\bigwedge}^{n_3}M_3\onto\underset{R}{\bigwedge}^{n_2}M_2.$$
\item[\emph{(d)}] If furthermore the sequence is split, then the epimorphism
$\zeta$ is an isomorphism.
\end{itemize}
\end{thm}

\begin{proof}[\textsc{Proof}]
If $n_2=0$ then $n_1=0=n_3$ and all statements are trivially true, so
assume $n_2>0$. We prove each point of the proposition separately.
\begin{itemize}
\item[(a)] Fix $n_2$. We show by induction on $ n_1 $, with $0\leq n_1\leq n_2$, that the restriction morphism gives an injective morphism
$$\Alt^R(M_2^{n_2},N)\into \Alt^R(M_2^{n_3},\innMult^R(M_1^{n_1},N)).$$

If $n_1=0$ then $n_3=n_2$, and $\rho$ is the identity morphism, so there is
nothing to show. So assume that $0<n_1\leq n_2$ and that the statement
is true for $n_1-1$ and $n_3+1$ in place of $n_1$ and $n_3$. Then
$\epR^{n_3+1}M_3=0$ implies $\epR^{n_3+2}M_3=0$ by proposition
\ref{prop 5}; so by the induction hypothesis we have an injection
$$\Alt^R(M_2^{n_2},N)\into \Alt^R(M_2^{n_3+1},\innMult^R(M_1^{n_1-1},N)).$$ Since by hypothesis we have
$\epR^{n_3+1}M_3=0$ we can use Lemma \ref{lem 3} (note that $\epR^{n_3+1}M_3=0$ implies that $ \Alt^R(M_3^{n_3+1},P)=0 $ for every $R$-module scheme $P$), and we have thus
an injection
$$\Alt^R(M_2^{n_3+1},\innMult^R(M_1^{n_1-1},N))\into\Alt^R(M_2^{n_3}\times
M_1,\innMult^R(M_1^{n_1-1},N)).$$ The latter $R$-module is inside the $R$-module
$$\Alt^R(M_2^{n_3},\innHom^R(M_1,\innMult^R(M_1^{n_1-1},N))).$$ By proposition
\ref{prop 3}, $\innHom^R(M_1,\innMult^R(M_1^{n_1-1},N))\cong
\innMult^R(M_1^{n_1},N)$. Putting these together, we conclude that there is an injection
$$\theta:\Alt^R(M_2^{n_2},N)\into \Alt^R(M_2^{n_3},\innMult^R(M_1^{n_1},N)).$$ Following through the above isomorphisms and inclusions, one verifies that this injection
is induced by the restriction morphism. Under the isomorphism $$\Mult^R(M_2^{n_3},\innMult^R(M_1^{n_1},N))\cong
\Mult^R(M_2^{n_3}\times M_1^{n_1},N)$$ given by Proposition \ref{prop 2},
the image of $\Alt^R(M_2^{n_2},N)$ by $\theta$ lies inside the submodule
$\Alt^R(M_2^{n_3}\times M_1^{n_1},N)$ of $ \Mult^R(M_2^{n_3}\times M_1^{n_1},N) $ and we can easily see that the
injection $$\Alt^R(M_2^{n_2},N)\into \Alt^R(M_2^{n_3}\times M_1^{n_1},N)$$ thus
obtained is given by the restriction morphism.

\item[(b)] Choose an alternating multilinear morphism $\phi:M_2^{n_2}\to N$ and write
$\phi_0$ for the restriction $\phi|_{M_1^{n_1}\times M_2^{n_3}}$. For
any $0\leq j\leq n_3-1$ the restriction of $\phi_0$ to the submodule
scheme $M_1^{n_1}\times M_2^j\times M_1\times M_2^{n_3-j-1}$ belongs to the
$R$-module
$$\Alt^R(M_1^{n_1}\times M_2^j\times M_1\times M_2^{n_3-j-1},N)$$ which injects into $$\Alt^R(M_1^{n_1+1},\innMult^R(M_2^{n_3-1},N)).$$ The latter $R$-module is isomorphic to
$$\Hom^R(\epR^{n_1+1}M_1,\innMult^R(M_2^{n_3-1},N)),$$ which is zero by
assumption. So, the restriction $\phi_0|_{M_1^{n_1}\times
M_2^j\times M_1\times M_2^{n_3-j-1}}$ is zero.

Now we show by induction on $0\leq i\leq n_3$, that there exists a
multilinear morphism $\phi_i:M_3^i\times M_2^{n_3-i}\times M_1^{n_1}\to
N$ such that the composition $$M_2^i\times M_2^{n_3-i}\times
M_1^{n_1}\arrover{\overline{\pi}}M_3^i\times M_2^{n_3-i}\times
M_1^{n_1}\arrover{\phi_i}N$$ is $\phi_0$, where
$\overline{\pi}=\pi^i\times \Id_{M_2^{n_3-i}}\times \Id_{M_1^{n_1}}$. If
$i=0$ then we have nothing to show, so let $i<n_3$ and assume that
we have constructed $\phi_i$ with the desired property and we
construct $\phi_{i+1}$. Consider the following commutative diagram:
$$\xymatrix{
M_2^i\times M_1\times M_2^{n_3-i-1}\times
M_1^{n_1}\ar@{->>}[d]_{\widehat{\pi}}\ar@{^{ (}->}[r]^{\quad
\rho}&M_2^i\times
M_2^{n_3-i}\times M_1^{n_1}\ar@{->>}[d]^{\overline{\pi}}\ar[dr]^{\phi_0}&\\
M_3^i\times M_1\times M_2^{n_3-i-1}\times M_1^{n_1}\ar@{^{
(}->}[r]^{\quad \rho'}&M_3^i\times M_2^{n_3-i}\times
M_1^{n_1}\ar[r]^{\qquad \quad\phi_i}&N.}$$ As we have said above, the
restriction of $\phi_0$, $\phi_0\circ \rho$, is zero. By the
induction hypothesis, we have $\phi_0=\phi_i\circ\overline{\pi}$ and
therefore, $0=\phi_0\circ
\rho=\phi_i\circ\overline{\pi}\circ\rho=\phi_i\circ\rho'\circ\widehat{\pi}$.
The morphism $\widehat{\pi}$ being an epimorphism, we conclude that the
restriction of $\phi_i$, i.e., $\phi_i\circ \rho'$ is zero. This
allows us to use Lemma \ref{lem 2} in order to find a multilinear
morphism $$\phi_{i+1}:M_3^{i+1}\times M_2^{n_3-i-1}\times M_1^{n_1}\to N$$
such that
$$\phi_i=\phi_{i+1}\circ(\Id_{M_3^i}\times\pi\times\Id_{M_2^{n_3-i-1}}\times\Id_{M_1^{n_1}}).$$
It follows at once that $\phi_0=\phi_{i+1}\circ(\pi^{i+1}\times
\Id_{M_2^{n_2-i-2}}\times \Id_{M_1})$.

Put $i=n_3$, then the statement says that there is a multilinear
morphism $\phi_{n_3}:M_3^{n_3}\times M_1^{n_1}\to H$ such that
$\phi_0=\phi_{n_3}\circ(\pi^{n_3}\times \Id_{M_1^{n_1}})$. Since
$\phi_0$ is alternating, by Remark \ref{rem 6}, $\phi_{n_3}$ is also
alternating.

\item[(c)] If both conditions hold, then by (a), $\rho$ is
injective and therefore the homomorphism
$\Alt^R(M_2^{n_2},H)\to\Alt(M_1^{n_1}\times M_3^{n_3},H)$ defined in (b) is
injective as well. So we obtain
$$\Hom^R(\epR^{n_2}M_2,N)\cong\Alt(M_2^{n_2},N)\into\Alt^R(M_1^{n_1}\times
M_3^{n_3},N)\overset{\ref{prop 6}}{\cong}$$
$$\Hom^R(\epR^{n_1}M_1\otimes\epR^{n_3}M_3,N)$$ which is
natural, in other words we have a natural injection of functors
$$\tau:\Hom^R(\epR^{n_2}M_2,-)\into\Hom^R(\epR^{n_1}M_1\otimes\epR^{n_3}M_3,-).$$
It is a known fact that any natural transformation between such
functors is induced by a unique morphism
$$\zeta:\epR^{n_1}M_1\otimes\epR^{n_3}M_3\to \epR^{n_2}M_2,$$ in
fact, this morphism is the image of the identity morphism of
$\epR^{n_2}M_2$ under this transformation. This means that for any
$R$-module scheme $N$,
$$\tau_N:\Hom^R(\epR^{n_2}M_2,N)\to\Hom^R(\epR^{n_1}M_1\otimes\epR^{n_3}M_3,N)$$
sends a morphism $f:\epR^{n_2}M_2\to N$ to the morphism $f\circ \zeta$.
The injectivity of $\tau$ implies that $\zeta$ is an epimorphism.

\item[(d)]
Let $s:M_3\to M_2$ be a section of $\pi$, i.e., $\pi\circ s=\Id_{M_3}$
and $r:M_2\to M_1$ the corresponding retraction of $\iota$, that is,
$r\circ\iota=\Id$ and that the short sequence $$0\to
M_3\arrover{s}M_2\arrover{r}M_1\to0$$ is exact. Then we show that the
morphism \[\mu:\Alt^R(M_2^{n_2},N)\to \Alt^R(M_1^{n_1}\times M_3^{n_3},N)\] whose
composition with $\pi^*$ is $\rho$ (given by (b)) is induced by the
inclusion $$j:=\iota^{n_1}\times s^{n_3}:M_1^{n_1}\times M_3^{n_3}\into
M_2^{n_2}.$$ Indeed, given a morphism $f\in\Alt^R(M_2^{n_2},N)$, we have
$\rho(f)=\pi^*(\mu(f))$, or in other words, $\mu(f)\circ
(\Id_{M_1}^{n_1}\times \pi^{n_3})=f\circ (\iota^{n_1}\times
\Id_{M_2}^{n_3})$. Hence the following diagram is commutative
$$\xymatrix{
M_1^{n_1}\times M_3^{n_3}\ar@{^{ (}->}[d]_{\widetilde{s}}\ar@/_3pc/[dd]_{\Id}\ar@{^{ (}->}_{j\,}[dr]\ar[drrr]^{\mu(f)}&&\\
M_1^{n_1}\times M_2^{n_3}\ar@{->>}[d]_{\widetilde{\pi}}\ar@{^{ (}->}[r]^{\quad i}&M_2^{n_2}\ar[rr]^{f\quad}&&N\\
M_1^{n_1}\times M_3^{n_3}\ar[urrr]_{\mu(f)}&&}$$ where $\widetilde{s},
\widetilde{\pi}$ and $i$ are respectively the morphisms
$\Id_{M_1}^{n_1}\times s^{n_3}, \Id_{M_1}^{n_1}\times \pi^{n_3}$ and the
inclusion $\iota^{n_1}\times \Id_{M_2}^{n_3}$. Consequently,
$\mu(f)=f\circ j$. This shows that $\mu$ is induced by $j$ as we
claimed.

Now define a morphism $\omega:M_2^{n_2}\to M_1^{n_1}\times M_3^{n_3}$ as
follows: for any $k$-algebra $A$ $\omega$ sends an element
$(g_1,\cdots,g_{n_2})\in M_2(A)^{n_2}$ to
$$\sum_{\sigma}\text{sgn}(\sigma,\tau)(r(g_{\sigma(1)}),\cdots,r(g_{\sigma(n_1)}),\pi(g_{\tau(1)}),\cdots,\pi(g_{\tau(n_3)}))$$
where the sum runs over all length $n_1$ subsequences
$\sigma=(\sigma(1),\cdots,\sigma(n_1))$ of $(1,2,\cdots,n_2)$ with
complementary subsequences $\tau=(\tau(1),\cdots,\tau(n_3))$ and
$\text{sgn}(\sigma,\tau)$ is the signature of $(\sigma,\tau)$ as a
permutation of $n_2$ elements. This morphism induces a homomorphism
$$\omega^*:\Alt^R(M_1^{n_1}\times M_3^{n_3},N)\to \Mult^R(M_2^{n_2},N)$$
and it is straightforward to see that in fact the image lies inside the
submodule $\Alt^R(M_2^{n_2},N)$. We also denote by $\omega^*$ the
homomorphism $$\Alt^R(M_1^{n_1}\times M_3^{n_3},H)\to \Alt^R(M_2^{n_2},N)$$
obtained by restricting the codomain of $\omega^*$. Since the
composites $r\circ s$ and $\pi\circ\iota$ are trivial and
$r\circ\iota$ and $\pi\circ s$ are the identity morphisms, we see
that the composition $\omega\circ j$ is the identity morphism of
$M_1^{n_1}\times M_3^{n_3}$. Therefore the composite
$\mu\circ\omega^*$ is the identity homomorphism. Consequently, the
homomorphism $\mu:\Alt^R(M_2^{n_2},N)\to\Alt^R(M_1^{n_1}\times M_3^{n_3},N)$ is
an epimorphism. We know from $(c)$ that it is a monomorphism, and
hence it is an isomorphism. We obtain thus
$$\Hom^R(\epR^{n_2}M_2,N)\cong\Alt^R(M_2^{n_2},N)\isoto$$
$$\Alt^R(M_1^{n_1}\times
M_3^{n_3},N)\cong\Hom^R(\epR^{n_1}M_1\otimes\epR^{n_3}M_3,N).$$ As
we know, this homomorphism is induced by the morphism
$$\zeta:\epR^{n_1}M_1\otimes\epR^{n_3}M_3\to \epR^{n_2}M_2.$$
Since it is an isomorphism, the morphism $\zeta$ must be an
isomorphism as well.
\end{itemize}
\end{proof}

\begin{lem}
\label{lem 6}
Let $M$ be an $R$-module scheme over $S$ and $r\in R$. Then for every positive natural number $n$ we have a commutative diagram
\[\xymatrix{
M\otimes_R\epR^{n-1}M\ar[rr]^{(r.)\otimes_R\Id}\ar@{->>}[d]_{\Id\wedge\ep^{n-1}\Id}&&M\otimes_R\epR^{n-1}M\ar@{->>}[d]^{\Id\wedge\ep^{n-1}\Id}\\
\epR^{n}M\ar[rr]_{r.}&&\epR^nM. }\]
\end{lem}

\begin{proof}[\textsc{Proof}]
Remark \ref{rem 8} 3) gives us the following commutative diagram:
\begin{myequation}
\label{diagram 1}
\xymatrix{
M\times\cdots\times M\ar[rrr]^{(r.)\times\Id_M\times\cdots\times\Id_M}\ar[d]_{\lambda}&&&M\times\cdots\times M\ar[d]^{\lambda}\\
\underset{R}{\bigwedge}^nM\ar[rrr]_{r.}&&&\underset{R}{\bigwedge}^nM}.
\end{myequation} 
Using Remark \ref{rem 5} 1), we can split this diagram into two squares:
\[\xymatrix{
M\times\cdots\times M\ar[d]_{\lambda_{1,2}}\ar[rr] ^{(r.)\times\Id_M^{n-1}}&& M\times\cdots\times M\ar[d]^{\lambda_{1,2}}\\
M\otimes_R\epR^{n-1}M\ar[d]_{\Id\wedge\ep^{n-1}\Id}\ar[rr]^{(r.)\otimes_R\Id} && M\otimes_R\epR^{n-1}M\ar[d]^{\Id\wedge\ep^{n-1}\Id}\\
\epR^nM \ar[rr]_{r.}&& \epR^nM
}\]
with the top square commutative.\\

The commutativity of the bottom square follows from the commutativity of the top diagram, diagram \ref{diagram 1}
and the universal property of the morphism $$\lambda_{1,2}:M\times\cdots\times M\to M\otimes_R\epR^{n-1}M$$ given in Remark \ref{rem 5} 1); more precisely, we have:
\[(\Id\wedge\ep^{n-1}\Id)\circ ((r.)\otimes_R\Id)\circ \lambda_{1,2}= (\Id\wedge\ep^{n-1}\Id)\circ \lambda_{1,2}\circ (r.)\times \Id^{n-1}=\]
\[(r.)\circ (\Id\wedge\ep^{n-1}\Id)\circ \lambda_{1,2}\]and the universal property of $\lambda_{1,2}$ implies that \[(r.)\circ (\Id\wedge\ep^{n-1}\Id)=(\Id\wedge\ep^{n-1}\Id)\circ ((r.)\otimes_R\Id).\] The fact that the vertical arrows in the diagram of the lemma are epimorphisms has been proved in Lemma \ref{lem 5}.
\end{proof}

\begin{prop}
\label{prop 7}
Let $M$ be an $R$-module scheme over $S$, and $Q$ the cokernel of multiplication by an element $r\in R$, i.e., we have an exact sequence\\
 $M\arrover{r.}M\arrover{p}Q\to 0$. Then, for any positive natural number $n$ the following sequence is exact:
\[ \epR^nM\arrover{r.}\epR^nM\arrover{\ep^np}\epR^nQ\to 0. \]
\end{prop}

\begin{proof}[\textsc{Proof}]
It follows from Remarks \ref{rem 8} and \ref{rem 5} that the composition\\ $(\Id\wedge\ep^{n-1}\Id)\circ ((r.)\otimes_R\Id)$ as in previous lemma is equal to $(r.)\wedge\ep^{n-1}\Id$. So, the diagram of the previous lemma can be rewritten as: \begin{myequation}
\label{diagram 2}
\xymatrix{M\otimes_R\epR^{n-1}M\ar[rr]^{(r.)\wedge\ep^{n-1}\Id}\ar@{->>}[dr]_{\Id\wedge\ep^{n-1}\Id}&&\epR^nM\\
&\epR^nM\ar[ur]_{r.}.&}
\end{myequation}Now, apply Lemma \ref{lem 4} to the exact sequence $M\arrover{r.}M\arrover{p}Q\to 0$ in order to obtain the following exact sequence:
\[M\otimes_R\underset{R}{\bigwedge}^{n-1}M\arrover{(r.)\wedge\ep^{n-1}\Id_{M}} \underset{R}{\bigwedge}^nM\arrover{\bigwedge^np}\underset{R}{\bigwedge}^nQ\to 0.\]

Using diagram \ref{diagram 2}, we can factorize the first morphism of the sequence, so that the following diagram is commutative with exact row
\[\xymatrix{
M\otimes_R\epR^{n-1}M\ar[rr]^{(r.)\wedge\ep^{n-1}\Id}\ar@{->>}[dr]_{\Id\wedge\ep^{n-1}\Id} && \epR^nM\ar[rr]^{\ep^np} && \epR^nQ \ar[rr]&& 0\\
&\epR^nM\ar[ur]_{r.}.&
}\]
Since the morphism $\Id\wedge\ep^{n-1}\Id:M\otimes_R\epR^{n-1}M\to \epR^nM$ is an epimorphism, we conclude that the sequence
\[ \epR^nM\arrover{r.}\epR^nM\arrover{\ep^np}\epR^nQ\to 0. \]
is exact as well, and the proof is achieved.
\end{proof}

\section{Dieudonn\'e modules}

We are going to study the covariant Dieudonn\'e modules of tensor products and exterior powers of $R$-module schemes over perfect fields of characteristic $p>2$ and let $k$ denote such a field for the rest of this section.

\begin{dfn}
\label{def 17}
Let $ R $ be a ring and $ P $ an $R$-module endowed with four set-theoretic maps $ F,V:P\to P $ and $ \phi, \upsilon:\epR^jP\to \epR^jP$. Denote by $\lambda$ the alternating morphism $ \lambda:P\times\cdots\times P\to \epR^jP $ which sends $ (x_1,\cdots,x_j) $ to $ x_1\wedge\cdots\wedge x_j$.
\begin{itemize}
\item[(i)] The following diagram is called the \emph{$F$-diagram associated to $\phi$}
\[\xymatrix{
P\times P\times\dots\times
P\ar[rr]^{\qquad\quad\lambda}&&\epR^jP\ar[dd]^{\phi}\\
P\times P\times\dots\times
P\ar[d]_{F\times\Id\times\cdots\times\Id}\ar[u]^{\Id\times V\times\cdots\times V}&&\\
P\times P\times\dots\times P\ar[rr]_{\qquad\lambda}&& \epR^jP.}\]
\item[(ii)] The following diagram is called the \emph{$V$-diagram associated to $\upsilon$}
\[\xymatrix{P\times P\times\cdots\times P\ar[d]_{V\times\cdots\times V}\ar[rr]^{\qquad\quad\lambda}&&\epR^jP\ar[d]^{\upsilon}\\
P\times P\times\cdots\times P\ar[rr]_{\qquad\quad\lambda}&&\epR^jP.}\]
\end{itemize}
\end{dfn}

\begin{dfn}
\label{def 13}
Let $n$ be a positive natural number.
\begin{itemize}
\item[(i)] Let $ M_1,\cdots, M_n $ be left $ \BE_k\otimes_{\BZ}R $-modules. Consider the tensor product $$ \BE_k\otimes_WM_1\otimes_{W\otimes_{\BZ}R}\cdots\otimes_{W\otimes_{\BZ}R}M_n $$
which uses the action $ W $ on $ \BE_k $ by right multiplication. This is a left $ \BE_k\otimes_{\BZ}R $-module with respect to left multiplication of $\BE_k$ on the first factor and the action of $R$ on the other factors. Define $ T(M_1\times\cdots\times M_n) $ to be its quotient by the $ \BE_k $-submodule generated by the elements
\[V\otimes m_1\otimes\cdots\otimes m_n-1\otimes Vm_1\otimes\cdots\otimes Vm_n,\]
\[F\otimes m_1\otimes Vm_2\otimes\cdots\otimes Vm_n-1\otimes Fm_1\otimes m_2\otimes\cdots\otimes m_n,\]
\[\vdots\]
\[F\otimes Vm_1\otimes Vm_2\otimes\cdots\otimes Vm_{n-1}\otimes m_n-1\otimes m_1\otimes m_2\otimes\cdots\otimes m_{n-1}\otimes Fm_n\]
for all $ m_i\in M_i $

\item[(ii)] Let $ M_1,\cdots, M_n $ be profinite topological left $ \BE_k\otimes_{\BZ}R $-modules. Then each $ M_i $ is the inverse limit of all its finite quotients $ M''_i $ by open $ \BE_k\otimes_{\BZ}R $-submodules. Define $ \hat{T}(M_1\times\cdots\times M_n) $ to be the inverse limit of all finite $ \BE_k\otimes_{\BZ}R $-module quotients of $ T(M''_1\times\cdots\times M''_n) $ for all $ M''_i $, which is again profinite topological left $ \BE_k\otimes_{\BZ}R $-module.

\item[(iii)] Let $M$ be a left $ \BE_k\otimes_{\BZ}R $-module. Define $ T_{\text{sym}}(M^n) $ to be the quotient of $ T(M^n) $ by the $ \BE_k\otimes_{\BZ}R $-submodule generated by the elements:
\[[1\otimes m_1\otimes\cdots\otimes m_n]-[1\otimes m_{\varrho(1)}\otimes\cdots\otimes m_{\varrho(n)}]\]
for all $ m_i\in M $ and all permutations $ \varrho\in S_n $.\\

Similarly, define $T_{\text{antisym}}(M^n)$ to be the quotient of $ T(M^n) $  by the $ \BE_k\otimes_{\BZ}R $-submodule generated by the elements:
\[[1\otimes m_1\otimes\cdots\otimes m_n]-\text{sgn}(\varrho)\cdot[1\otimes m_{\varrho(1)}\otimes\cdots\otimes m_{\varrho(n)}].\]

Finally, define $  T_{\text{weakalt}}(M^n)$ to be the quotient of $T_{\text{antisym}}(M^n)$ by the $ \BE_k\otimes_{\BZ}R $-submodule generated by the elements $ [1\otimes m_1\otimes\cdots\otimes m_n] $ for all $ m_i\in M $ of which at least two coincide and lie in the the submodule $ VM$.

\item[(iv)] Let $M$ be a profinite topological left $ \BE_k\otimes_{\BZ}R $-module. For any $ *\in\{$sym, antisym, weakalt$\}$, define $ \hat{T}_*(M^n) $ to be the inverse limit of finite $ \BE_k\otimes_{\BZ}R $-module quotients of $ T_*(M''^n) $ for all finite quotients $ M'' $ of $M$.
\end{itemize}
\end{dfn}

\begin{rem}
\label{rem43}
Let $R$ be a ring and $M_1,\dots, M_j, M$ be $ \BE_k\otimes_{\BZ}R$-modules which are of finite length as a $ W\otimes_{\BZ}R$-module.
\begin{itemize}
\item[1)] There is a canonical morphism \[\tau:M_1\times\dots\times M_j\to  T(M_1\times\dots\times M_j),\quad (m_1,\dots,m_j)\mapsto [1\otimes m_1\otimes\dots\otimes m_j].\] This morphism belongs to the $R$-module $ L^R(M_1\times\dots\times M_j, T(M_1\times\dots\times M_j))$ and has the following universal property:\\ for every $ \BE_k\otimes_{\BZ}R$-module $N$, the homomorphism \[ \tau^*:\Hom(T(M_1\times\dots\times M_j),N)\to L^R(M_1\times\dots\times M_j,N) \] induced by $ \tau $ is an isomorphism.
\item[2)] Similarly, there is a canonical morphism \[\lambda:M^j\to  T_{\text{weakalt}}(M^j),\quad (m_1,\dots,m_j)\mapsto [1\otimes m_1\otimes\dots\otimes m_j].\] This morphism belongs to the $R$-module $ L_{\text{alt}}^R(M^j, T_{\text{weakalt}}(M^j))$ and has the following universal property:\\ for every $ \BE_k\otimes_{\BZ}R$-module $N$, the homomorphism \[ \lambda^*:\Hom(T_{\text{weakalt}}(M^j),N)\to L_{\text{alt}}^R(M^j,N) \] induced by $ \lambda $ is an isomorphism.
\item[3)] Finally, there exists a universal morphism $ \sigma:M^j\to T_{\text{sym}}(M^j) $ inside the $R$-module $ L_{\text{sym}}^R(M^j, T_{\text{sym}}(M^j))$.
\end{itemize}
\end{rem}

\begin{lem}
\label{lem 12}
Let $R$ be a ring and $P$ an $ \BE_k\otimes_{\BZ}R$-module which is of finite length as a $ W\otimes_{\BZ}R$-module. Assume further that we have two commuting maps $ \phi:\ovset{j}{\unset{W\otimes_{\BZ}R}{\ep}} P\to \ovset{j}{\unset{W\otimes_{\BZ}R}{\ep}} P $ and respectively $ \upsilon: \ovset{j}{\unset{W\otimes_{\BZ}R}{\ep}} P\to  \ovset{j}{\unset{W\otimes_{\BZ}R}{\ep}} P$ which are $ ^{\sigma^{-1}}\otimes\Id$-linear and respectively $ ^{\sigma}\otimes\Id$-linear and make the $F$-diagram associated to $ \phi $ and respectively the $V$-diagram associated to $ \upsilon $ commute. Then there is a natural structure of $\BE_k$-module on $\ovset{j}{\unset{W\otimes_{\BZ}R}{\ep}} P$, where $F$ and $V$ act through $ \phi $ and $ \upsilon$ and we have a canonical $ \BE_k\otimes_{\BZ}R$-linear isomorphism
$$\hat{T}_{\text{weakalt}}(P^j)\cong \ovset{j}{\unset{W\otimes_{\BZ}R}{\ep}} P.$$
\end{lem}

\begin{proof}[\textsc{Proof}]
We first show that $\phi\circ\upsilon=p$. Indeed, we have for all $d_1,\cdots,d_j\in P$
\[\phi\circ\upsilon(d_1\wedge\cdots\wedge d_j)=\phi(Vd_1\wedge\cdots\wedge Vd_j)=FVd_1\wedge d_2\wedge \cdots\wedge d_j=pd_1\wedge d_2\wedge \cdots\wedge d_j\]where the first equality follows from the $V$-diagram and the second equality from the $F$-diagram. As $\upsilon\circ\phi=\phi\circ\upsilon$, we have $\upsilon\circ\phi=\phi\circ\upsilon=p$. Now, since $p$ is different from $2$, the antisymmetry condition in the construction of $T_{\text{antisym}}(P^j)$ (cf. Definition \ref{def 13} (iii) ) means that $ T_{\text{weakalt}}(P^j)\cong T_{\text{antisym}}(P^j) $ and that this module is the quotient of $\BE_k\otimes_W\ovset{j}{\unset{W\otimes_{\BZ}R}{\ep}}P$ by the submodule generated by the relations:
\[(\rho_1)\qquad V\otimes m_1\wedge\cdots\wedge m_j-1\otimes Vm_1\wedge\cdots\wedge Vm_j,\]
\[(\rho_2)\qquad F\otimes m_1\wedge Vm_2\wedge\cdots\wedge Vm_j-1\otimes Fm_1\wedge m_2\wedge\cdots\wedge m_j\]
(note that the other relations follow from these two).
Now, define a morphism $$\theta:\BE_k\otimes_W\ovset{j}{\unset{W\otimes_{\BZ}R}{\ep}} P\to\ovset{j}{\unset{W\otimes_{\BZ}R}{\ep}} P$$ by
\[F^i\otimes x\mapsto \phi^i(x)\quad \text{and}\quad V^i\otimes x\mapsto \upsilon^i(x).\]
Since $\upsilon\circ\phi=\phi\circ\upsilon=p$, this morphism is a well defined $\BE_k$-linear morphism. We claim that this morphism factors through the quotient $T_{\text{antisym}}(P^j)$. i.e., it is zero on the relations $\rho_1$ and $\rho_2$.
\begin{itemize}
\item $(\rho_1)$: We have $\theta(V\otimes m_1\wedge\cdots\wedge m_j-1\otimes Vm_1\wedge\cdots\wedge Vm_j)=$
\[\upsilon(m_1\wedge\cdots\wedge m_j)-Vm_1\wedge\cdots\wedge Vm_j=0\] by the $V$-diagram.
\item $(\rho_1)$: We have $\theta(F\otimes m_1\wedge Vm_2\wedge\cdots\wedge Vm_j-1\otimes Fm_1\wedge m_2\wedge\cdots\wedge m_j)=$
\[\phi(m_1\wedge Vm_2\wedge\cdots\wedge Vm_j)-Fm_1\wedge m_2\wedge\cdots\wedge m_j=0\] by the $F$-diagram.
\end{itemize}
It is straightforward to see that the morphism $\vartheta:\ovset{j}{\unset{W\otimes_{\BZ}R}{\ep}}P\to T_{\text{antisym}}(P^j)$ sending an element $x$ to $[1\otimes x]$ is an inverse of $\bar{\theta}:T_{\text{antisym}}(P^j)\to\ovset{j}{\unset{W\otimes_{\BZ}R}{\ep}}P $ induced by $\theta$. Therefore, $$T_{\text{antisym}}(P^j)\cong\ovset{j}{\unset{W\otimes_{\BZ}R}{\ep}}P.$$ The latter being a finite length module over $W\otimes_{\BZ}R$, we deduce that 
\[\hat{T}_{\text{weakalt}}(P^j)\cong T_{\text{weakalt}}(P^j)\cong T_{\text{antisym}}(P^j)\cong \ovset{j}{\unset{W\otimes_{\BZ}R}{\ep}}P\] and the proof is achieved.
\end{proof}

\begin{lem}
\label{lem 22}
Let $R$ be a ring and $P, Q$ two $\BE_k\otimes_{\BZ}R$-modules which are of finite length as $ W\otimes_{\BZ}R$-modules. Assume further that the multiplication by $V$ on $P$ is an isomorphism. Then there exists a natural structure of $\BE_k$-module on $P\otimes_{\scriptscriptstyle W\otimes_{\BZ}R}Q$, where $F$ acts as $V^{-1}\otimes F$ and $V$ acts as $V\otimes V$. Furthermore, we have a canonical $\BE_k\otimes_{\BZ}R$-linear isomorphism
\[\hat{T}(P\times Q)\cong P\otimes_{\scriptscriptstyle W\otimes_{\BZ}R}Q.\]
\end{lem}

\begin{proof}[\textsc{Proof}]
The compositions $(V^{-1}\otimes F)\circ (V\otimes V)$ and $(V\otimes V)\circ (V^{-1}\otimes F)$ are equal to multiplication by $p$ and so there is a natural structure of $\BE_k$-module on the tensor product $P\otimes_{\scriptscriptstyle W\otimes_{\BZ}R}Q$. It remains to show the stated isomorphism.\\

Since the modules $P$ and $Q$ are of finite length over $W\otimes_{\BZ}R$, we have
\[\hat{T}(P\times Q)\cong T(P\times Q).\]
We define morphisms
\[\theta:T(P\times Q)\to P\otimes_{\scriptscriptstyle W\otimes_{\BZ}R}Q\]
and
\[\eta:P\otimes_{\scriptscriptstyle W\otimes_{\BZ}R}Q\to T(P\times Q)\]
as follows:
$\theta([1\otimes x\otimes y])=x\otimes y, \theta([F^i\otimes x\otimes y])=V^{-i}(x)\otimes F^i(y)$ and $\theta([V^i\otimes x\otimes y])=V^i(x)\otimes V^i(y)$. And for $\eta$, we set $\eta(x\otimes y)=[1\otimes x\otimes y]$, where by elements in brackets, we mean their class in the quotient $T(P\times Q)$. It is now straightforward to check that $\theta$ is well-defined, and these morphism are inverse to each other.
\end{proof}

The following proposition is a direct generalization of the theorem 5.6.2 in \cite{P} and its proof is the same as the proof of theorem 5.6.2 in \cite{P} with the slight and easy modifications due to $R$-linearity and $R$-multilinearity, and therefore we omit the proof of the proposition. Note however that we should assume $ p\neq 2$ in the proposition, which is not the case in theorem 5.6.2 in \cite{P}.

\begin{prop}
\label{prop 10}
Let $G_1, G_2$ and $G$ be pro-$p$ $R$-module schemes over $k$, then the tensor product $G_1\otimes_R G_2$, the symmetric power $ \unset{R}{S}^jG $ and the exterior power $ \unset{R}{\ep}^jG $ exist and are again pro-$p$ $R$-module schemes over $k$, and there are natural isomorphisms
\begin{itemize}
\item $D_*(G_1\otimes_R G_2)\cong \hat{T}(D_*(G_1)\times D_*(G_2)),$
\item $D_*(\unset{R}{S}^jG)\cong \hat{T}_{\text{sym}}(D_*(G)^j),$
\item $D_*(\unset{R}{\ep}^jG)\cong \hat{T}_{\text{weakalt}}(D_*(G)^j).$
\end{itemize}
\qed

\end{prop}


\begin{rem}
\label{rem44}
Let $M$ be a finite $p$-torsion $R$-module scheme over $k$. According to Remark \ref{rem33}, we have an isomorphism 
\begin{myequation}
\label{univ1}
L_{\text{alt}}^R(D_*(M)^j,D_*(\epR^jM))\cong \Alt^R(M^j,\epR^jM).
\end{myequation}
We also explained in Remark \ref{rem44} that the universal morphism $ \lambda:D_*(M)^j\to T_{\text{weakalt}}(D_*(M)^j) $ induces an isomorphism 
\begin{myequation}
\label{univ2}
\Hom^R(T_{\text{weakalt}}(D_*(M)^j),D_*(\epR^jM))\to L_{\text{alt}}^R(D_*(M)^j,D_*(\epR^jM)).
\end{myequation} 
It follows that the isomorphism $ T_{\text{weakalt}}(D_*(M)^j)\cong D_*(\epR^jM) $ given in the previous proposition is mapped to the universal alternating morphism $ M^j\to\epR^jM $, under the composition of the two isomorphism \eqref{univ1} and \eqref{univ2}. In other words, using the isomorphism $ T_{\text{weakalt}}(D_*(M)^j)\cong D_*(\epR^jM) $ we obtain an isomorphism \[ L_{\text{alt}}^R(D_*(M)^j,T_{\text{weakalt}}(D_*(M)^j))\cong \Alt^R(M^j,\epR^jM) \] and under this isomorphism, the universal elements correspond to each other.
\end{rem}

\begin{lem}
\label{lem 11}
Let $R$ be a ring and $G$ a finite $R$-module scheme over $k$ of order a power of $p$. Assume that there are commuting morphisms $\phi:\ovset{j}{\unset{W\otimes_{\BZ}R}{\ep}} D_*(G)\to \ovset{j}{\unset{W\otimes_{\BZ}R}{\ep}} D_*(G)$ respectively $\upsilon:\ovset{j}{\unset{W\otimes_{\BZ}R}{\ep}} D_*(G)\to \ovset{j}{\unset{W\otimes_{\BZ}R}{\ep}} D_*(G)$ which are $^{\sigma^{-1}}\otimes\Id$ respectively $^{\sigma}\otimes\Id$-linear and make the $F$-diagram associated to $ \phi $ and the $V$-diagram associated to $ \upsilon $ commute.
Then $\ovset{j}{\unset{W\otimes_{\BZ}R}{\ep}} D_*(G)$ is the covariant Dieudonn\'e module of $\epR^jG$ with $F$ and $V$ acting respectively through $\phi$ and $\upsilon$ respectively.
\end{lem}

\begin{proof}[\textsc{Proof}]
The existence of $\epR^jG$ is guaranteed by Proposition \ref{prop 10}. So, assume the existence of $\phi$ and $\upsilon$. We know that the Dieudonn\'e module of $G$, $ D_*(G) $, is a finite length $W$-module, and therefore $ \ovset{j}{\unset{W\otimes_{\BZ}R}{\ep}} D_*(G) $ is a finite length module over $ W\otimes_{\BZ}R $. Now, according to Lemma \ref{lem 12}, We have 
\[\ovset{j}{\unset{W\otimes_{\BZ}R}{\ep}} D_*(G) \cong \hat{T}_{\text{weakalt}}(D_*(G)^j)\]which is isomorphic to $ D_*(\unset{R}{\ep}^jG) $ by Proposition \ref{prop 10} and therefore, we have \[\ovset{j}{\unset{W\otimes_{\BZ}R}{\ep}} D_*(G)\cong  D_*(\unset{R}{\ep}^jG)\] and the proof is achieved.
\end{proof}

\begin{rem}
\label{rem45}
It follows from this lemma that the universal morphism  $$D_*(G)^j\to D_*(\epR^jG)\cong \epR^jD_*(G) $$ (cf. Remarks \ref{rem43} and \ref{rem44}) is the ``natural" one, i.e., the one sending $ (x_1,\dots,x_j)$ to $ x_1\wedge\dots\wedge x_j$.
\end{rem}

\begin{lem}
\label{lem 23}
Let $R$ be a ring and $G_1, G_2$ finite $R$-module schemes over $k$, of order a power of $p$ with $G_1$ \'etale. Then the tensor product $G_1\otimes G_2$ exist and its Dieudonn\'e module, $D_*(G_1\otimes G_2)$, is canonically isomorphic to the tensor product $D_*(G_1)\otimes_{\scriptscriptstyle W\otimes_{\BZ}R} D_*(G_2)$, of Dieudonn\'e modules of $G_1$ and $G_2$.
\end{lem}

\begin{proof}[\textsc{Proof}]
The existence of the tensor product follows at once from the first part of Proposition \ref{prop 10} (about the existence of tensor product).\\

In order to show the isomorphism, notice that the Dieudonn\'e modules of $G_1$ and $G_2$ are of finite length over the ring $ W\otimes_{\BZ}R $ and since $G_1$ is \'etale, its Verschiebung morphism is an isomorphism. We can now apply Lemma \ref{lem 22} to obtain a canonical isomorphism
\[\hat{T}(D_*(G_1)\times D_*(G_2))\cong D_*(G_1)\otimes_{\scriptscriptstyle W\otimes_{\BZ}R}D_*(G_2)\]
which together with the first isomorphism of Proposition \ref{prop 10}, gives the desired result.
\end{proof}

\begin{rem}
\label{rem 15}
If the ring $R$ in the previous two lemmas is an $S$-algebra, where $S$ is a subring of the ring of Witt vectors $W$, and if the actions of $R$ on $G_1$, $G_2$ and $G$ are $S$-linear, then the statements of these lemmas remain true, if we replace $\BZ$ by $S$. This is so for example if the ring $R$ is our discrete valuation ring $\CO$ and $S$ is the ring of $p$-adic integers $\BZ_p$. We will mainly use this version of these lemmas in the sequel.
\end{rem}


\chapter{\texorpdfstring{Multilinear Theory of $ \pi $-Divisible Modules}{Multilinear Theory of pi-Divisible Modules}}

In this chapter, $\mathcal{O}$ denotes a complete discrete valuation ring with a fixed uniformizing parameter $\pi$ and finite residue field $\BF_q$ ($q=p^f$). We denote by $K$ the fraction field of $\CO$. If $\CO$ has mixed characteristic, then it is a finite extension of $\BZ_p$ of degree $ef$ where $e$ is the ramification index of the extension $K/\BQ_p$ and $f$ is its residue degree and in this case, we have $ p=u\cdot\pi^e $, with $u$ a unit of $\CO$. In the equal characteristic case, $\CO$ is isomorphic to $\BF_q\lbb\pi\rbb$, the formal power series in $\pi$.\\

\section{First definitions}

In this section, we would like to generalize the notion of a $p$-divisible group. Let us fix some notations.\\

\begin{dfn}
\label{piBT}
Let $S$ be a scheme and $\CM$ an fppf sheaf of $ \CO $-modules over $S$. We call $\CM$ a \emph{$ \pi $-Barsotti-Tate group} or \emph{$\pi$-divisible $ \CO $-module scheme over $S$} if the following conditions are satisfied:
\begin{itemize}
\item[(i)] $\CM$ is $\pi$-divisible, i.e., the homomorphism $\pi:\CM \to \CM$ is an epimorphism.
\item[(ii)] $\CM$ is $\pi$-torsion, i.e., the canonical homomorphism $ \uset{n}{\dirlim}\, \CM[\pi^n]\to \CM $ is an isomorphism.
\item[(iii)] $ \CM[\pi] $ is representable by a finite locally free $ \CO $-module scheme over $S$.
\end{itemize}
The rank of $ \CM[\pi]$ is of the form $ q^h $, where $h:S\to \BQ_{\geq 0}$ is a locally constant function, called \emph{the height} of $\CM$. If the ring $ \CO $ is clear from the context, we may call $ \CM $ simply a $ \pi $-divisible module. We denote by $ \CM_i $ the kernel of multiplication by $ \pi^i $.
\end{dfn}

\begin{rem}
\label{rem 9}
$  $
\begin{itemize}
\item[1)] Using the first property and induction on $j$, we can show that for every $j\in \BN$ we have an exact sequence:\[0\to \CM_1\arrover{\text{inclusion}} \CM_{j+1}\arrover{\pi.} \CM_j\to 0.\]
\item[2)] It follows from this exact sequence that $\CM_j$ is finite locally free (flat) $ \CO $-module scheme over $S$ and has rank (order) equal to $q^{jh}$. 
\item[3)] Now, it can be shown easily that for every $i,j\in \BN$, we have an exact sequence\[0\to \CM_i\arrover{\text{inclusion}} \CM_{j+i}\arrover{\pi^i.} \CM_j\to 0.\]  
\end{itemize}
\end{rem}

The following remark gives another definition of $\pi$-divisible modules:

\begin{rem}
\label{rem 10}
Assume that we have a diagram
\[\CM_1\arrover{\iota_1}\CM_2\arrover{\iota_2}\CM_3\to\cdots\]
where the $\CM_i$ are finite locally free $\mathcal{O}$-module schemes over $S$ with the following properties:
\begin{itemize}
\item the order of $\CM_j$ is equal to $q^{jh}$, with $h$ a fixed locally constant map $ S\to \BQ_{\geq 0} $,
\item the sequences $0\to \CM_j\arrover{\iota_j}\CM_{j+1}\arrover{\pi^j}\CM_{j+1}$ are exact.
\end{itemize}
Then, the limit $\dirlim (\CM_n,\iota_n)$ is a $\pi$-divisible $\CO$-module scheme over $S$, of height $h$ and $\kernel(\pi^j\cdot)\cong \CM_j$ for every $j>0$.
\end{rem}

\begin{rem}
\label{rem 12}
Let $A$ be a Henselian local ring and $\CM$ a $\pi$-divisible formal $\CO$-module over $A$. The same arguments as in Remark \ref{rem 1} 7) show that the connected-\'etale sequence of $\CM$, as a formal group scheme over $A$, \[0\to \CM^0\to \CM\to \CM^{\text{\'et}}\to0\] is in fact a sequence of formal $\CO$-module schemes over $A$. Using the functoriality of this sequence, the multiplicativity of the order of finite flat group schemes with respect to exact sequences and what we know about the orders of $\CM_n$ (with notations as in Definition \ref{piBT}), one can show that connected and \'etale factors of $\CM$ are $\pi$-divisible $\CO$-module schemes over $A$ as well and that the connected-\'etale sequence of $\CM$ is the direct limit (over $n>0$) of the connected-\'etale sequence of $\CM_n$ and in particular we have
\begin{itemize}
\item $(\CM^0)_n=(\CM_n)^0$ and
\item $(\CM^{\text{\'et}})_n=(\CM_n)^{\text{\'et}}.$ 
\end{itemize}
\end{rem}

\begin{prop}
\label{etalepidiv}
Let $ \CM $ be an \'etale $ \pi $-divisible $ \CO $-module scheme over a base scheme $S$. Then there exists a finite \'etale cover $ T\to S $ such that $ \CM_T $ is the constant $ \pi $-divisible module with $ \CM_{n,T}\cong \ul{(\CO/\pi^n)}^h $, where $h$ is the height of $ \CM $. If $S$ is connected, we can take a connected finite \'etale cover $T$.
\end{prop}

\begin{proof}
We can assume that the base scheme is connected, since otherwise, $ \CM $ is a disjoint union of $ \pi $-divisible $ \CO $-module schemes over each connected component of $S$ and having the result for each of them, provides a finite \'etale cover of $S$ satisfying the required property. By ``\'etale dictionary'', the category of finite $ \CO $-module schemes is equivalent to the category of finite $ \CO $-modules with a continuous action of the \'etale fundamental group, $ \pi^{\text{\'et}}_1(S,\sbar) $, at a (fixed) geometric point, $ \sbar $, of $S$. It means that for any given finite $ \CO $-module scheme $M$, there exists a connected finite \'etale cover $S'$ of $S$ such that the action of the \'etale fundamental group at a geometric point of $S'$ (mapping to $ \sbar $) on $ M_{S'} $ is trivial. In other words, $ M_{S'} $ is a constant $ \CO $-module scheme.\\

Let $ T\to S $ be a finite connected \'etale cover such that $ \CM_{1,T} $ is a constant $ \CO $-module scheme. Fix geometric points $ \sbar $ of $S$ and $ \bar{t} $ of $T$, with $ \bar{t} $ mapping to $ \sbar $. We want to show, by induction on $n\geq 1$, that $ \CM_{n,T} $ is the constant $ \CO $-module scheme $ \ul{(\CO/\pi^n)}^h $. If $ n=1 $, the we know that $ \CM_1 $ is a constant $ \CO $-module scheme corresponding to a $ \CO $-module $ M $ with $ q^h $ elements and annihilated by $ \pi $. By the structure theorem of finitely generated modules over principal ideal domains, there is only one possibility for such an $ \CO $-module, namely, $ (\CO/\pi)^h $. Consider the short exact sequence \[ 0\to \CM_{1,T}\to \CM_{n+1,T}\arrover{\pi} \CM_{n,T}\to 0 .\] This gives rise to a short exact sequence \[ 0\to \CM_{1}(\bar{t})\to \CM_{n+1}(\bar{t})\arrover{\pi} \CM_{n}(\bar{t})\to 0 \] of $ \CO $-modules with continuous $ \pi_1^{\text{\'et}}(T,\bar{t}) $-action. Since the action of $ \pi_1^{\text{\'et}}(T,\bar{t})  $ on two terms (the left and the right terms) of this sequence is trivial, it acts trivially on the third (one can use the long exact sequence of group cohomology to deduce this fact), and therefore, $ \CM_{n+1,T} $ is also a constant $ \CO $-module scheme. We know that $ \CM_{1}(\bar{t})\cong (\CO/\pi)^h $ and $  \CM_{n}(\bar{t})\cong (\CO/\pi^n)^h $, and that $ \CM_{n+1}(\bar{t}) $ is a finite $ \CO $-module annihilated by $ \pi^{n+1} $ and of order $ q^{(n+1)h} $. The only $ \CO $-module that fits to the above exact sequence and has these properties is $ (\CO/\pi^{n+1})^h $. This achieves the proof.
\end{proof}

\begin{rem}
\label{rem 11}
In the mixed characteristic case, any $\pi$-divisible module of height $h$ is canonically a $p$-divisible group of height $efh$. Indeed, let $\CM$ be a $\pi$-divisible module. As $p=u\cdot\pi^e$, with $u$ a unit in $\CO$, the morphism $p^n.:\CM\to \CM$ (multiplication by $p^n$) is equal to the composition $\CM\arrover{u^n} \CM\arrover{\pi^{ne}}\CM$ and so with the notation of Definition \ref{piBT}, $\CM_{ne}=\kernel(p\cdot)=:\CM'_n$ for all $n\geq1$. Now:
\begin{itemize}
\item $\pi:\CM\to \CM$ is an epimorphism if and only if  $p=u\cdot\pi^e:\CM\to \CM$ is an epimorphism (multiplication by $u$ is an isomorphism).
\item the subset $\{ne\mid n>0\}$ of $\BN$ is cofinal, so we have
\[\CM=\underset{n}{\bigcup}\, \CM_n=\underset{n}{\bigcup}\, \CM_{ne}=\underset{n}{\bigcup}\, \CM'_n\]
\item according to Remark \ref{rem 9} 2), $\CM'_1=\kernel(p\cdot)=\CM_e$ is finite and the order of $\CM'_1$  is $q^{eh}=p^{efh}$.
\end{itemize}
Hence, $\CM$ is a $p$-divisible group of height $efh$.\\

If the base scheme is the spectrum of a field, then the converse is also true, in other words, every $ p $-divisible group with an $ \CO $-action is canonically a $ \pi $-divisible $ \CO $-module scheme. By what we have said above, the first two conditions of Definition \ref{piBT} are satisfied. Since the kernel of multiplication by $ \pi $ is a subgroup of the kernel of multiplication by $p$, and we are over a field, the third condition is also satisfied.
\end{rem}

\begin{dfn}
\label{def 16}
Let $ \CM $ be a $\pi$ divisible $ \CO $-module scheme over $S$ and denote by $ \CM_n^* $ the Cartier dual of $ \CM_n $, i.e., we have $ \CM_n^*=\innHom_S(\CM_n,\BG_m) $. The inductive system \[ \CM_1^*\to\CM_2^*\to\CM_3^*\to\dots \] induced by the homomorphisms $ \CM_{i+1}\to \CM_i $ (cf. Remark \ref{rem 9}), is called the dual $ \pi $-divisible $ \CO $-module scheme of $ \CM $.
\end{dfn}

\begin{rem}
\label{rem 22}
By functoriality of Cartier duality, the action of $ \CO $ on $ \CM_n $ induces an action on $ \CM_n^* $. The exacts sequences \[ 0\to\CM_n\to\CM_{n+m}\to\CM_m\to0 \] are then transformed to the exact sequences \[ 0\to\CM_n^*\to\CM_{n+m}^*\to\CM_m^*\to0.\] Since the order of $\CM_n^*  $ over $S$ is equal to the order of$ \CM_n $, which is equal to $ p^{nh} $, it follows from Remark \ref{rem 10} that the dual of a $ \pi $-divisible $ \CO $-module scheme is again a $\pi$-divisible $\CO$-module scheme of the same height.
\end{rem}

\section{Some properties}

\begin{lem}
\label{lem 7}
Let $R$ be an integral domain over a field $k$ which is infinite dimensional as $k$-vector space and let $M$ be a finite free $R$-module. Assume that we have a ring automorphism $\sigma:R\to R$ which restricts to a ring automorphism of $k$, i.e., $\sigma(k)=k$ and we have a $\sigma$-linear morphism $\phi:M\to M$, i.e., $\phi(rm)=\sigma(r)\phi(m)$ for all $r\in R$ and $m\in M$. If the dimension (as $k$-vector space) of the cokernel of $\phi$ is finite, then $\phi$ is injective.
\end{lem}

\begin{proof}[\textsc{Proof}]
Take an $R$-basis of $M$, say $m_1,\cdots, m_n$. For every $i=1,\cdots, n$ write $\phi(m_i)=\sum_{j=1}^nr_{ji}m_j$ with $r_{ji}\in R$. Take an element $m\in M$ and write it as $\sum_{i=1}^n x_im_i$. We have 
\[\phi(m)=\phi(\sum_{i=1}^nx_im_i)=\sum_{i=1}^n\phi(x_im_i)=\]
\[\sum_{i=1}^n \sigma(x_i)\phi(m_i)=\sum_{i=1}^n\sigma(x_i)\sum_{j=1}^nr_{ji}m_j=\sum_{j=1}^n(\sum_{i=1}^nr_{ji}\sigma(x_i))m_j.\] It follows that $\phi$ is the following composition of morphisms:
\[M\overset{\cong}{\to} R^n\overset{\sigma^n}{\to} R^n\overset{\rho}{\to} R^n\overset{\cong}{\to} M\]
where the first and last morphisms are the $R$-linear isomorphisms given by the choice of the basis $\{m_1,\cdots, m_n\}$ and $\rho$ is the $R$-linear morphism given by the matrix $(r_{ji})$. The morphism $\sigma^n:R^n\to R^n$ is a $\sigma$-linear isomorphism. It follows from these observations that $\phi$ is injective if and only if $\rho$ is injective and that the cokernel of $\phi$ is isomorphic (as $R$-module) to the cokernel of $\rho$. So, we have reduced the problem to the case, where $\phi$ is an $R$-linear morphism and not only $\sigma$-linear. So, we assume that $\phi:M\to M$ is an $R$-linear morphism.\\

We claim that cokernel of $\phi$ is a torsion $R$-module. Indeed, take a non-zero element $m\in\cokernel(\phi)$ and an infinite set of $k$-linearly independent elements of $R$, say $\{r_1,r_2,\cdots\}$. Since $\dim_k(\cokernel(\phi))<\infty$ we have that $\sum_{i=1}^nx_i(r_im)=0$ for some $n\in\BN$ and a subset $\{x_1,x_2,\cdots, x_n\}\neq\{0\}$ of $k$. Since $r_i$ are linearly independent, the sum $r:=\sum_{i=1}^nx_ir_i$ is not zero but $rm=0$, which shows that $m$ is a torsion element. This proves the claim.\\

Let us denote the fraction field of $R$ by $Q$. Cokernel of $\phi$ being a torsion $R$-module implies that the tensor product $\cokernel(\phi)\otimes_RQ$ is zero. It follows that $\phi\otimes\Id:M\otimes_RQ\to M\otimes_RQ$ is surjective. The $Q$-vector space $M\otimes_RQ$ has finite dimension, which implies that the morphism $\phi\otimes\Id$ is injective too. By flatness of $Q$ over $R$, we have $\kernel(\phi)\otimes_RQ=\kernel(\phi\otimes\Id)=0$. Since $M$ is a torsion-free $R$-module, its submodule $\kernel(\phi)$ is also torsion-free and therefore embeds in $\kernel(\phi)\otimes_RQ$ which is a trivial module. Hence $\kernel(\phi)=0$ and $\phi$ is injective.
\end{proof}

\begin{lem}
\label{lem 8}
Let $\CO=\BF_q\lbb\pi\rbb$ and $\CM$ be a $\pi$-divisible $\CO$-module scheme over a perfect field $k$ containing $\BF_q$. Then the contravariant Dieudonn\'e module of $\CM$, $D(\CM)$, is a finite free module over $\CO\otimes_{\BF_q}k\cong k\lbb\pi\rbb$.
\end{lem}

\begin{proof}[\textsc{Proof}]
Let us write $\CM=\bigcup \CM_n$ where $\CM_n$ are kernels of $\pi^n$ and write $D$ (respectively $D_n$) for $D(\CM)$ (respectively $D(\CM_n)$). Then $D=\invlim D_n$. The Dieudonn\'e module $D_n$ is finite over $W(k)$, the ring of Witt vectors on $k$, but $p\cdot \CM=0$, which implies that $p\cdot D=0$ and thus $D_n$ is finite over $W(k)/p=k$. Let $d_1,\cdots, d_r$ be elements in $D$ whose images in $D_1$ is a basis over $k$ and define a morphism $k\lbb\pi\rbb^r\to D$ by sending basis elements to $d_i$. This morphism induces morphisms $(k\lbb\pi\rbb/(\pi^n))^r\to D/\pi^nD\cong D_n$ which are surjective (since modulo $\pi$ they are surjective) an so, being an inverse limit of surjective morphisms, $k\lbb\pi\rbb^r\to D$ is surjective. This implies that $D$ is a finite module over $k\lbb\pi\rbb$. The action of $\pi$ on $\CM$ is surjective and therefore its action on $D$ is injective. It follows that $D$ is a torsion-free $k\lbb\pi\rbb$-module and hence is free over it, since $k\lbb\pi\rbb$ is a principal ideal domain.
\end{proof}

\begin{thm}
\label{thm 2}
Finite dimensional $\pi$-divisible $ \CO $-module schemes are formally smooth.
\end{thm}

\begin{proof}[\textsc{Proof}]
Note that we may assume that the base scheme is an algebraically closed field $k$.\\

Let $\CM$ be a $\pi$-divisible $\CO$-module scheme over $k$. If $k$ is has characteristic different from $p$, then $ \CM $ is \'etale and so it is smooth. So, we can assume that the characteristic of $k$ is $p$. Since $k$ is perfect, $\CM$ splits into the \'etale and connected factors and so, $\CM$ is smooth if and only if both the \'etale and connected parts are smooth. The \'etale factor is smooth and so we may assume that $\CM$ is connected. As for connected formal schemes, being smooth is equivalent to the Frobenius morphism being an epimorphism, we will show that the Frobenius morphism is an epimorphism.\\

In the mixed characteristic case, by Remark \ref{rem 11}, every $\pi$-divisible module is $p$-divisible. Since the multiplication by $p$ factors through Frobenius, i.e., $p=F_{\CM}\circ V_{\CM}$, and multiplication by $p$ is an epimorphism, we see that $F_{\CM}$ is an epimorphism too.\\

Now, assume that $\CO$ has characteristic $p$ and so $\CO=\BF_q\lbb \pi\rbb$. By Lemma \ref{lem 8}, the contravariant Dieudonn\'e module of $\CM$, $D:=D(\CM)$ is a finite free $k\lbb\pi\rbb$-module. Denote by $\sigma:k\to k$ the Frobenius morphism of $k$, i.e., $\sigma(x)=x^p$ for all $x\in k$. It has a natural extension to $k\lbb\pi\rbb$ by sending $\pi$ to itself, and also denote this extension by $\sigma$. Since the action of $\pi$ on $\CM$ is a morphism of formal group schemes, it commutes with the Frobenius of $\CM$ and thus the Frobenius morphism of $D$ is $\sigma$-linear morphism. The dimension of the tangent space of $\CM$ is equal to the dimension of $\CM$ and the tangent space of $\CM$ is isomorphic to the dual of the cokernel of the Frobenius of $D$. This shows that the cokernel of $D$ has finite dimension over $k$. It follows from Lemma \ref{lem 7} that Frobenius of $D$ is injective and therefore the Frobenius of $\CM$ is an epimorphism. Hence the smoothness of $\CM$.
\end{proof}

\begin{rem}
\label{rem 16}
What we have shown in the last Theorem is that the Frobenius morphism of the contravariant Dieudonn\'e module of a $ \pi$-divisible $\CO$-module scheme is injective. We will see in the next lemma that similarly, the Verschiebung of the covariant Dieudonn\'e module of a $ \pi$-divisible $\CO$-module scheme is injective as well.
\end{rem}

\begin{lem}
\label{lem 13}
Let $ \CM $ be a finite dimensional $ \pi $-divisible $ \CO $-module scheme over a perfect field of characteristic $p$, then the Verschiebung morphism of the covariant Dieudonn\'e module of $ \CM $ is injective.
\end{lem}

\begin{proof}[\textsc{Proof}]
Let $D$ and respectively $D_i$ (for any natural number $i$) denote the covariant Dieudonn\'e module of $\CM$ and respectively of $\CM_i$, the kernel of $ \pi^i:\CM\to\CM $. Let us also denote by  $K_i$ the kernel of the Verschiebung of $D_i$. We know that  $ K_i\cong \CL ie(\CM_i)$, and so $ \dim_kK_i\leq \dim_k\CL ie(\CM)<\infty $. Since the inclusion $ \eta:D_i \into D_{i+j} $ ($j$ a natural number) is compatible with the Verschiebungen, it induces a morphism between kernels of $V$, which we denote also by $ \eta:K_i\into K_{i+j} $. Similarly, the epimorphism $ \zeta:D_{i+j}\onto D_i $ induces the morphism $ \zeta:K_{i+j}\to K_i $, and as we have seen before, the composition  $K_{i+j}\ovset{\zeta}{\to}K_i\ovset{\eta}{\into} K_{i+j}$ is the multiplication by $ \pi^j $ (we have seen it for the composition of the morphisms between Dieudonn\'e modules, but as $ K_i $ is a submodule of $ D_i $, this is also true for the morphisms between these kernels).\\

Now, since the dimension of $K_i$ is bounded above, and we have inclusions $ K_i\into K_{i+1} $, there exists a natural number $n_0$, such that for all $ i\geq n_0 $, we have $ \dim_kK_i=\dim_kK_{n_0} $ and so $ \eta :K_i\into K_{i+j}$ is an isomorphism for any natural number $j$ and any $ i\geq n_0 $. We claim that the morphisms $ \zeta:K_{i+n_0}\to K_{i} $ are zero for all $ i\geq 1 $. Indeed, the composition $  K_{i+n_0}\to K_{i}\into K_{i+n_0} $ is the multiplication by $\pi^{n_0}$, and so, composed with the inclusion $ K_{n_0}\into K_{i+n_0} $ is zero (note that $ K_{n_0}\subset D_{n_0} $ and $ D_{n_0} $ is killed by $ \pi^{n_0} $), which implies that the composition $ K_{n_0}\into K_{i+n_0}\to K_i $ is zero ($ K_{i}\into K_{i+n_0} $ is injective). But, the morphism $ K_{n_0}\into K_{i+n_0}$ is actually an isomorphism by the choice of $ n_0 $, and therefore the morphism $ K_{i+n_0}\to K_i $ is zero and the claim is proved.\\

For every natural number $ i $, we have an exact sequence
$0\to K_i\to D_i\ovset{V}{\to} D_i.$
Taking the inverse limit over $i$ with the transition morphisms $ \zeta $, and recalling that $D$ is the inverse limit of $D_i$, we obtain the exact sequence
$0\to \invlim K_i\to D\ovset{V}{\to} D$ (note that the inverse limit is a left exact functor). But, since for every $i\geq 1$, the transition morphism $ K_{i+n_0}\to K_{i} $ is zero, it follows that the inverse limit $ \invlim K_i $ is trivial, and hence $V:D\to D$ is injective.
\end{proof}

\begin{thm}
\label{thm 3}
Let $S$ be a scheme and $\CM$ a $\pi$-divisible $ \CO $-module scheme over $S$. Then the height $h(\CM):S\to \BQ_{\geq 0}$ takes integer values.
\end{thm}

\begin{proof}[\textsc{Proof}]
Since the height is invariant under base change, we may assume that $S$ is the spectrum of an algebraically closed field $k$. The order of finite group schemes is multiplicative with respect to exact sequences and it follows from Remark \ref{rem 12} that the height of $\CM$ is the sum of the heights of its connected and \'etale parts. So, we prove the statement for \'etale and connected $\pi$-divisible $\CO$-modules.\\

Assume that $\CM$ is \'etale. Then by Remark \ref{rem 12}, $\CM_1$ is also \'etale, and since $k$ is separably closed, $\CM_1$ is a constant group scheme, and so, the order of $\CM_1$ is equal to the order of the group $\CM_1(k)$, which is a module over $\CO/\pi=\BF_q$. The height of $\CM$ is by definition equal to $\log_q|\CM_1|=\log_q| \CM_1(k)|$. But being a vector space over the field $\BF_q$, the order of $\CM_1(k)$ is a power of $q$ and therefore the height of $\CM$ is a natural number.\\

Now, assume that $\CM$ is connected and so $k$ has characteristic $p$. By Theorem \ref{thm 2}, $\CM$ is smooth and therefore it is a commutative formal Lie group.
Again, we consider the problem in mixed and equal characteristic cases separately. If $\CO$ has characteristic zero, then by Remark \ref{rem 11}, $\CM$ is a $p$-divisible group of height $efh(\CM)$. Note that for $p$-divisible groups, the height is always a natural number, which follows from its definition (the residue degree is $1$), i.e., $efh(\CM)\in \BN$. From Theorem 1 of \cite{W2}, we know that the height of $\CM$, regarded as a $p$-divisible group, is divisible by the degree of the extension $K/\BQ_p$, which is $ef$. Therefore, $h(\CM)$ is a natural number.\\

Now, assume that $\CO$ has characteristic $p$ and therefore is isomorphic to $\BF_q\lbb\pi\rbb$. The Dieudonn\'e module of $\CM$, $D:=D(\CM)$, is a finitely generated module over the ring $k\otimes_{\BF_p}\CO$ (see Lemma \ref{lem 8}). This ring is isomorphic to the product $$\underset{\text{Gal}(\BF_q/\BF_p)}{\prod}k\otimes_{\BF_q}\CO=\underset{j\, \text{mod}\, f}{\prod}k\otimes_{\BF_q}\CO,$$ where we identify the Galois group Gal$(\BF_q/\BF_p)$ with the group $\BZ/f\BZ$ by sending the Frobenius to $1$ and the isomorphism is given explicitly by
$$\theta:k\otimes_{\BF_p}\BF_q\lbb\pi\rbb\to\underset{j\, \text{mod}\, f}{\prod}k\otimes_{\BF_q}\BF_q\lbb\pi\rbb=\underset{j\, \text{mod}\, f}{\prod}k\lbb\pi\rbb$$
$$x\otimes(\sum_{i=0}^{\infty}a_i\pi^i)\mapsto (x^{p^j}\otimes(\sum_{i=0}^{\infty}a_i\pi^i))_j=(\sum_{i=0}^{\infty}(a_ix^{p^{j}})\pi^i)_j.$$
This decomposition of $k\otimes_{\BF_p}\CO$ gives a decomposition of $D$, so $D=\underset{j\, \text{mod}\, f}{\prod}C_j$ where each $C_j$ is a module over $k\lbb\pi\rbb$. We know from Lemma \ref{lem 8} that $D$ is a finite free module over $k\lbb\pi\rbb$, it follows that $C_j$ are also finite free modules over $k\lbb\pi\rbb$. Let us denote by $h_j$ the rank of $C_j$ over $k\lbb\pi\rbb$. Then the rank of $D$ is $\sum_{j\, \text{mod}\, f}h_j$. Using the isomorphism $\theta$ one can show that the Frobenius morphism of $D$ restricted to $C_j$ maps to $C_{j-1}$ (note that the addition $j-1$ is modulo $f$), i.e., we have $F:C_j\to C_{j-1}$. As we have seen in the proof of Theorem \ref{thm 2}, Lemma \ref{lem 7} implies that the Frobenius of $D$ is injective, therefore, we have $h_j\leq h_{j-1}$. This being true for every $j$ mod $f$, we conclude that $h_j=h_{j'}$ for all $j$ and $j'$ in $\BZ/f\BZ$. Call this common number $h$. As we said above, we have then that the rank of $D$ as a module over $k\lbb\pi\rbb$ is equal to $fh$. But we know from Dieudonn\'e theory that the rank of $D$ (which is equal to the length of $D(\CM_1)$ over $k$), is equal to $\log_p|\CM_1|=fh(\CM)$. Hence, $h(\CM)=h$ and it is a natural number.
\end{proof}

\section{Exterior powers}

In this section, we define the notion of exterior powers of $ \pi $-divisible $ \CO $-module schemes. These are the objects that we will be interested in and their existence and construction are the main challenge of this work.  

\begin{dfn}
\label{def41}
Let $ \CM_0, \dots, \CM_r, \CM $ and $\CN$ be $\pi$-divisible $ \CO $-module schemes over a scheme $S$.
\begin{itemize}
\item[(i)] A \emph{$ \CO $-multilinear morphism} $ \phi:\CM_1\times\dots\times \CM_r\to \CM_0 $ is a system of $ \CO $-multilinear morphisms $ \{\phi_n:\CM_{1,n}\times\dots\times \CM_{r,n}\to \CM_{0,n}\}_n $ over $S$, compatible with the projections $ \pi.:\CM_{i,n+1}\twoheadrightarrow \CM_{i,n} $ and $ \pi.:\CM_{0,n+1}\twoheadrightarrow \CM_{0,n} $. In other words, it is an element of the inverse limit $$\uset{n}{\invlim}\Mult_S(\CM_{1,n}\times\dots\times \CM_{r,n}, \CM_{0,n}),$$ with transition homomorphisms induced by the projections $ \pi.:\CM_{i,n+1}\twoheadrightarrow \CM_{i,n} $. Denote the group of $ \CO $-multilinear morphisms from $\CM_1\times\dots\times \CM_r$ to $\CM_0$ by $ \Mult_S^{\CO}(\CM_1\times\dots\times \CM_r, \CM_0) $.
\item[(ii)] A \emph{symmetric} $ \CO $-multilinear morphism $ \phi:\CM^r\to \CN $ is a system of symmetric $ \CO $-multilinear morphisms $ \{\phi_n:\CM^r_n\to \CN_n\}_n $ over $S$, compatible with the projections $ \pi.:\CM_{n+1}\twoheadrightarrow \CM_n $ and $ \pi.:\CN_{n+1}\twoheadrightarrow \CN_n $. In other words, it is an element of the inverse limit $\uset{n}{\invlim}\Sym^{\CO}_S(\CM_n^r,\CN_n) $, with transition homomorphisms induced by the projections $ \pi.:\CM_{n+1}\twoheadrightarrow \CM_n $ and $ \pi.:\CN_{n+1}\twoheadrightarrow \CN_n $. Denote the group of symmetric $ \CO $-multilinear morphisms from $\CM^r$ to $\CN$ by $ \Sym^{\CO}_S(\CM^r,\CN) $.
\item[(iii)] An \emph{alternating} $ \CO $-multilinear morphism $ \phi:\CM^r\to \CN $ is a system of alternating $ \CO $-multilinear morphisms $ \{\phi_n:\CM^r_n\to \CN_n\}_n $ over $S$, compatible with the projections $ \pi.:\CM_{n+1}\twoheadrightarrow \CM_n $ and $ \pi.:\CN_{n+1}\twoheadrightarrow \CN_n $. In other words, it is an element of the inverse limit $\uset{n}{\invlim}\Alt^{\CO}_S(\CM_n^r,\CN_n) $, with transition homomorphisms induced by the projections $ \pi.:\CM_{n+1}\twoheadrightarrow \CM_n $ and $ \pi.:\CN_{n+1}\twoheadrightarrow \CN_n $. Denote the group of alternating $ \CO $-multilinear morphisms from $\CM^r$ to $\CN$ by $ \Alt^{\CO}_S(\CM^r,\CN) $.
\end{itemize}
\end{dfn}

\begin{rem}
\label{rem023}
The same arguments as in the proof of Proposition \ref{prop023} show that under this functor, the group of multilinear, respectively symmetric and respectively alternating morphisms of $p$-divisible groups (respectively of truncated Barsotti-Tate groups of level $i$) over $ X $ is isomorphic to the group of multilinear, respectively symmetric and respectively alternating morphisms of $p$-divisible groups (respectively of truncated Barsotti-Tate groups of level $i$) over $ \FX $.
\end{rem}

\begin{dfn}
\label{def42}
Let $ \CM, \CM' $ be $ \pi $-divisible $ \CO $-module schemes over $S$. An alternating $\CO$-multilinear morphism $\lambda:\CM^r\to \CM'$, or by abuse of terminology, the $ \pi $-divisible $ \CO $-module scheme $\CM'$, is called an \emph{$r^{\text{th}}$ exterior power} of $\CM$ over $\CO$, if for all $ \pi $-divisible $ \CO $-module scheme $\CN$ over $S$, the induced morphism $$\lambda^*:\Hom_S(\CM',\CN)\to \Alt_S^{\CO}(\CM^r,\CN),\quad \psi\mapsto \psi\circ\lambda,$$ is an isomorphism. If such $\CM'$ and $\lambda$ exist, we write $\underset{{\CO}}{\bigwedge}^r\CM$ for $\CM'$ and call $ \lambda $ \emph{the universal alternating morphism} defining $\underset{\CO}{\bigwedge}^r\CM$.
\end{dfn}

\section{The main theorem: the \'etale case}

The category of finite \'etale group schemes is equivalent to the category of finite Abelian groups with a continuous action of the \'etale fundamental group of the base. Also, under this equivalence, finite \'etale $ \CO $-module schemes correspond to finite $ \CO $-modules with a continuous. This ``dictionary" allows us to construct the tensor objects and in particular the exterior powers of finite \'etale $ \CO $-module schemes and $\pi$-divisible $ \CO $-module schemes. This is what we do in this section.

\begin{prop}
\label{prop025}
Let $S$ be a base scheme.
\begin{itemize}
\item Let $H$ be a finite \'etale $ \CO $-module scheme over $S$. Then there exists a finite \'etale $ \CO $-module scheme $ \epO^rH $ over $S$ and an alternating morphism $ \lambda:H^r\to \epO^rH $ such that for all $ \CO $-module schemes $X$ over $S$ the induced homomorphism \[ \lambda^*:\Hom^{\CO}_S(\epO^rH,X)\to \Alt_S^{\CO}(H^r,X)\] is an isomorphism. Furthermore, if $T$ is an $S$-scheme, then the canonical homomorphism $\epO^r(H_T)\to (\epO^rH)_T$, induced by the universal property of $ \epO^r(H_T) $ and $ \lambda_T $, is an isomorphism. In other words, the exterior powers of $H$ exist, are again \'etale and their construction commutes with arbitrary base change.
\item Let $G$ be an \'etale $\pi$-divisible $ \CO $-module over $S$ of height $h$. Then there exists a $\pi$-divisible group $ \epO^rG$ over $S$, of height $ \binom{h}{r} $, and an alternating morphism $ \lambda:G^r\to \epO^rG $ such that for all $\pi$-divisible $ \CO $-module $Y$ over $S$ the induced homomorphism \[ \lambda^*:\Hom_S^{\CO}(\epO^rG,Y)\to \Alt_S^{\CO}(G^r,Y)\] is an isomorphism. Furthermore, if $T$ is an $S$-scheme, then the canonical homomorphism $\epO^r(G_T)\to (\epO^rG)_T$, induced by the universal property of $ \epO^r(G_T) $ and $ \lambda_T $, is an isomorphism. In other words, the exterior powers of $G$ exist, are again \'etale and their construction commutes with arbitrary base change.
\end{itemize}
\end{prop}

\begin{proof}
$ $
\begin{itemize}
\item Assume at first that $S$ is connected and choose a geometric point $ \bar{s} $ of $S$. Then the functor \[ J\mapsto J(\bar{s})=\Mor_S(\bar{s},J)\] from the category of finite \'etale $ \CO $-module schemes over $S$ to the category of finite $ \CO $-modules with a continuous action of the \'etale fundamental group at $ \bar{s} $, i.e., $ \pi_1^{\text{\'et}}(S,\bar{s}) $, is an equivalence of categories. Moreover, this functor preserves multilinear and alternating morphisms. The category of $ \CO $-modules with a continuous $\pi_1^{\text{\'et}}(S,\bar{s})$-action is a tensor category and has in particular all exterior powers. It follows that the category of finite \'etale $ \CO $-module schemes over $S$ possesses all exterior powers over $S$. The universal alternating morphism $ H(\bar{s})^r\to \epO^r(H(\bar{s})) $ induces the universal alternating morphism $ \lambda:H^r\to \epO^rH $ with the universal property stated in the proposition.\\

If $S$ is not connected, the finite \'etale group scheme $H$ decomposes into disjoint union of finite \'etale $ \CO $-module schemes over connected components of $S$. The above construction yields exterior powers of each of them over a connected base and these exterior powers glue together to produce exterior powers of the whole $ \CO $-module scheme $H$. One has to verify that these exterior powers satisfy the universal property of exterior powers, but this is true since homomorphisms and multilinear morphisms of $ \CO $-module schemes over $S$ decompose into homomorphisms and multilinear morphisms of $ \CO $-module schemes over each connected component. This proves the existence of the exterior powers.\\

Now assume that $T$ is an $S$-scheme. For the statement on the base change, we can assume that $S$ and $T$ are both connected. Fix a geometric point $ \bar{t} $ of $T$ lying over the fixed geometric point $ \bar{s} $ of $S$. The structural morphism $ f:T\to S $ induces a $ \CO $-module homomorphism $f_*:\pi_1^{\text{\'et}}(T,\bar{t}) \to  \pi_1^{\text{\'et}}(S,\bar{s})$. For every finite \'etale  scheme $J$ over $S$, the map of sets $ J(\bar{s})\to J_T(\bar{t}) $ is a bijection and the action of $ \pi_1^{\text{\'et}}(T,\bar{t}) $ on $ J_T(\bar{t}) $ is the restriction, under the homomorphism $ f_* $, of the action of $ \pi_1^{\text{\'et}}(S,\bar{s}) $ on $ J(\bar{s}) $, after identifying the two sets $ J(\bar{s}) $ and $ J_T(\bar{t}) $. In other words, the base extension from $S$ to $T$ of \'etale group schemes over $S$ corresponds to the restriction of the action of the \'etale fundamental group of $S$ at $ \bar{s}$ to that of $T$ at $ \bar{t} $. It follows at once that the $ \CO $-module scheme homomorphism  $ \epO^r(H_T)\to (\epO^rH)_T $ is an isomorphism.

\item Similarly, assuming $S$ is connected and fixing a geometric point $ \bar{s} $ of $S$, the functor that assigns to an \'etale $\pi$-divisible $ \CO $-module its Tate module (at $ \bar{s} $) defines an equivalence of categories between the category of \'etale $\pi$-divisible $ \CO $-module over $S$ and the category of finite free $ \CO $-modules with a continuous action of the \'etale fundamental group of $S$ at $\bar{s}$. Also, Multilinear (respectively alternating) morphisms are preserved under this functor. The category of finite free $ \CO $-modules with a continuous $  \pi_1^{\text{\'et}}(S,\bar{s}) $-action is a tensor category and therefore possesses all exterior powers. It implies that the category of \'etale $\pi$-divisible $ \CO $-modules over $S$ has all exterior powers in the sense stated in the proposition, with the universal alternating morphism $ \lambda:G^r\to \epO^rG $ obtained from the universal alternating morphism $ T_p(G)^r\to \epO^rT_p(G)\cong T_p(\epO^rG) $.\\

The height of an \'etale $\pi$-divisible group is equal to the rank over $ \CO $ of its Tate module. Thus, the height of $ \epO^rG $ is equal to the rank over $ \CO $ of $ T_p(\epO^rG)\cong \epO^r(T_p(G)) $, which is equal to $ \binom{h}{r} $.\\

The same arguments as in the last case (of finite \'etale $ \CO $-module schemes) show that this construction commutes with arbitrary base change.
\end{itemize}
\end{proof}

\begin{rem}
\label{rem024}
$  $
\begin{itemize}
\item[1)] Note that if $\CM$ is an \'etale $\pi$-divisible $ \CO $-module scheme over a scheme $S$, then for every positive natural number $n$, we have $ \ep^r(\CM_n)\cong (\ep^r\CM)_n $.
\item[2)] For details on the \'etale fundamental group of a scheme and the equivalence of categories mentioned in the above proof, we refer to \cite{SGA1} and \cite{JM}.
\end{itemize}
\end{rem}

\section{The main theorem: over fields of characteristic $p$}

In this section, we want to construct the exterior powers of $ \pi $-divisible modules over fields of characteristic $p$. In this section, unless otherwise specified, $k$ is field of characteristic $p$.\\

Let $\CM$ be a $\pi$-divisible $\CO$-module scheme over $k$ and let us denote as usual $\CM_i$ for the kernel of $\pi^i:\CM\to \CM$. For every $i>0$, we have a natural monomorphism $\iota:\CM_i\into \CM_{i+1}$ and a natural epimorphism $\varPi:\CM_{i+1}\onto \CM_i$ (multiplication by $\pi$) such that the both compositions $ \CM_i\ovset{\iota}{\into}\CM_{i+1}\ovset{\varPi}{\onto}\CM_i $ and respectively $ \CM_{i+1}\ovset{\varPi}{\onto}\CM_i\ovset{\iota}{\into}\CM_{i+1}  $, are just the multiplication by $\pi$ on $\CM_i$ and respectively on $ \CM_{i+1} $ and give rise to the following exact sequences:
\begin{myequation}
\label{cokernel of pi^i}
\quad0\to \CM_i\ovset{\iota}{\to}\CM_{i+1}\ovset{\pi^i}{\to}\CM_{i+1}\quad\text{and}
\end{myequation}

\begin{myequation}
\label{kernel of pi^i}
\quad\CM_{i+1}\ovset{\pi^i}{\to}\CM_{i+1}\ovset{\varPi}{\to}\CM_i\to 0.\qquad
\end{myequation}

Therefore, if $k$ is a perfect field, whether we use the covariant or contravariant Dieudonn\'e theory (and if we denote by $D$ the Dieudonn\'e functor), we have an injection $\eta:D(\CM_i)\into D(\CM_{i+1})$ and a surjection $\zeta:D(\CM_{i+1})\onto D(\CM_i)$ such that the compositions $$D(\CM_i)\ovset{\eta}{\into}D(\CM_{i+1})\ovset{\zeta}{\onto}D(\CM_i)$$ and $$ D(\CM_{i+1})\ovset{\zeta}{\onto}D(\CM_i)\ovset{\eta}{\into}D(\CM_{i+1}) $$ are multiplication by $\pi$ and we have the following exact sequences:
\[D(\CM_{i+1})\ovset{\pi^i}{\to}D(\CM_{i+1})\ovset{\zeta}{\to}D(\CM_i)\to 0\quad \text{and}\]
\[0\to D(\CM_i)\ovset{\eta}{\to}D(\CM_{i+1})\ovset{\pi^i}{\to}D(\CM_{i+1})\qquad\]
(in the covariant case, we use the sequence $(\ref{kernel of pi^i})$ in order to obtain the first sequence and we use $(\ref{cokernel of pi^i})$ to obtain the second one, and in the contravariant case we use the sequence $(\ref{cokernel of pi^i})$ for the first sequence and $(\ref{kernel of pi^i})$ for the second one).

\begin{dfn}
\label{def 12}
Assume that $k$ is perfect and let $ \CM $ be a $ \pi $-divisible $ \CO $-module scheme over $k$.
\begin{itemize}
\item[(i)] We define the \emph{Dieudonn\'e module} of a $\pi$-divisible $\CO$-module scheme $\CM$ to be the inverse limit $\unset{i}{\invlim} (D(\CM_i),\zeta)$. It is called the \emph{covariant Dieudonn\'e module}, if it is the inverse limit of covariant Dieudonn\'e modules and is called the \emph{contravariant Dieudonn\'e module} in the other case.
\item[(ii)] The morphism induced on the Dieudonn\'e module of $\CM$ by the Frobenius morphisms (respectively Verschiebungen) of $D(\CM_i)$ is called \emph{the Frobenius} (respectively \emph{Verschiebung}) and is denoted by $F$ (respectively $V$).
\end{itemize}
\end{dfn}

\begin{cons}
\label{conslimitdieudonne}
Let $ \CM_0,\CM_1,\dots,\CM_r $ be $ \pi $-divisible $ \CO $-module schemes over a perfect field $k$ of characteristic $p$ and for every $ i=0,\dots,r $, denote by $ D_i$ the Dieudonn\'e module of $ \CM_i$. Let $ f:D_1\times\dots\times D_r\to D_0 $ be an $ \CO $-multilinear morphism (of $ W(k)\otimes_{\BZ_p}\CO $-modules) satisfying the $ V $-condition, i.e., \[ Vf(x_1,\dots,x_r)=f(Vx_1,\dots,Vx_r) \] for every $ x_i\in D_i $. For every $ n\geq 1 $, this morphisms induces a morphism $ D_1\times\dots\times D_r\to D_0/\pi^nD_0 $, and using the multilinearity of $f$, we obtain an $ \CO $-multilinear morphism \[ D_1/\pi^nD_1\times\dots\times D_r/\pi^nD_r\to D_0/\pi^nD_0 \] that we denote by $ f_n $. It follows from its construction, that $ f_n $ satisfies the $V$-condition. We claim that it satisfies also the $ F $-conditions, i.e., that for every $i=1,\dots,r$, and every $(x_1,\dots,x_r)\in D_1\times\dots\times D_r$ we have \[ Ff_n(Vx_1,\dots,Vx_{i-1},x_i,Vx_{i+1},\dots,Vx_r)=f_n(x_1,\dots,x_{i-1},Fx_i,x_{i+1},\dots,x_r). \] In fact, we claim that $ f $ itself satisfies the $F$-condition, and therefore, $f_n$ inherits this property. Let $ (x_1,\dots,x_r) $ be an arbitrary element of the product $ D_1\times\dots\times D_r $. We have \[ VFf(Vx_1,\dots,Vx_{i-1},x_i,Vx_{i+1},\dots,Vx_r)=\]\[pf(Vx_1,\dots,Vx_{i-1},x_i,Vx_{i+1},\dots,Vx_r)=\]\[f(Vx_1,\dots,Vx_{i-1},px_i,Vx_{i+1},\dots,Vx_r)=\]\[f(Vx_1,\dots,Vx_{i-1},VFx_i,Vx_{i+1},\dots,Vx_r)=\]\[ Vf(x_1,\dots,x_{i-1},Fx_i,x_{i+1},\dots,x_r)\] and since by Lemma \ref{lem 13} $V$ is injective, we cancel $V$ from both side of the equality and conclude that \[ Ff(Vx_1,\dots,Vx_{i-1},x_i,Vx_{i+1},\dots,Vx_r)=f(x_1,\dots,x_{i-1},Fx_i,x_{i+1},\dots,x_r) \] as claimed. Let us denote by $ \Mult^{\CO}(D_1\times\dots\times D_r,D_0) $ the $ \CO $-module of all $ \CO $-multilinear morphisms from $ D_1\times\dots\times D_r $ to $ D_0 $ that satisfy the $V$-condition. Thus, the construction of $ f_n $ from $f$ defines an $ \CO $-linear morphism \[ \alpha_n:\Mult^{\CO}(D_1\times\dots\times D_r, D_0)\to L^{\CO}(D_{1,n}\times\dots\times D_{r,n},D_{0,n}),\] where we denote by $ D_{i,n} $ the Dieudonn\'e module of $ \CM_{i,n}=\kernel(\pi^n:\CM_i\to\CM_i)$, which is canonically isomorphic to $ D_i/\pi^nD_i $. These morphisms are compatible with the canonical morphisms \[ L^{\CO}(D_{1,n+1}\times\dots\times D_{r,n+1})\to L^{\CO}(D_{1,n}\times\dots\times D_{r,n},D_{0,n}) \] given by the projections $ D_{i,n+1}\to D_{i,n} $ and therefore define an $ \CO $-linear morphism \[\alpha:\Mult^{\CO}(D_1\times\dots\times D_r,D_0)\to \uset{n}{\invlim}\,L^{\CO}(D_{1,n}\times\dots\times D_{r,n}).\]\\

Similarly, we define the $ \CO $-modules $ \Sym^{\CO}(D_1^r,D_0) $ and $ \Alt^{\CO}(D_1^r,D_0) $ and the $ \CO $-linear morphisms \[\Sym^{\CO}(D_1^r,D_0)\to\uset{n}{\invlim}\,L^{\CO}_{\rm sym}(D_{1,n}^r,D_{0,n})\] and \[ \Alt^{\CO}(D_1^r,D_0)\to\uset{n}{\invlim}\,L^{\CO}_{\rm alt}(D_{1,n}^r,D_{0,n})\] which are the restrictions of $ \alpha $.
\end{cons}

\begin{lem}
\label{lem Dieudonné modules}
Let $ \CM_0,\CM_1,\dots,\CM_r $ be $ \pi $-divisible $ \CO $-module schemes over a perfect field $k$ of characteristic $p$. The $ \CO $-linear morphisms \[\alpha:\Mult^{\CO}(D_1\times\dots\times D_r,D_0)\to \uset{n}{\invlim}\,L^{\CO}(D_{1,n}\times\dots\times D_{r,n},D_{0,n}),\]  \[\Sym^{\CO}(D_1^r,D_0)\to\uset{n}{\invlim}\,L^{\CO}_{\rm sym}(D_{1,n}^r,D_{0,n}) \] and \[ \Alt^{\CO}(D_1^r,D_0)\to\uset{n}{\invlim}\,L^{\CO}_{\rm alt}(D_{1,n}^r,D_{0,n})\] constructed above are isomorphisms.
\end{lem}

\begin{proof}
Let us at first show that $ \alpha $ is an isomorphism. Define a map \[ \omega:\uset{n}{\invlim}\,L^{\CO}(D_{1,n}\times\dots\times D_{r,n},D_{0,n})\to \Mult^{\CO}(D_1\times\dots\times D_r,D_0) \] as follows. Take an element $ g=(g_n)_n $ in the inverse limit. By assumption, the following diagram commutes:\[ \xymatrix{D_{1,n+1}\times\dots\times D_{r,n+1}\ar@{->>}[d]\ar[rr]^{\quad g_{n+1}}&& D_{0,n+1}\ar@{->>}[d]\\D_{1,n}\times\dots\times D_{r,n}\ar[rr]_{g_n}&&D_{0,n}},\] where the vertical morphisms are the canonical projections. Let $$ \omega(g):D_1\times\dots\times D_r\to D_0 $$ be the following morphism \[ \big((x_{1,j})_j,(u_{2,j})_j,\dots,(u_{r,j})_j\big)\mapsto \big(g_j(u_{1,j},\dots,u_{r,j})\big)_j,\] where $ (u_{i,j})_j $ is an element of $ D_i=\uset{j}{\invlim}\, D_{i,j} $. The commutativity of the above diagram implies that $ \big(g_j(u_{1,j},\dots,u_{r,j})\big)_j $ is an element of the inverse limit $ \uset{j}{\invlim}\, D_{0,j}=D_0 $, and by construction, $ \omega(g) $ satisfies the $ V $-condition. It is now straightforward to check that the compositions $ \alpha\circ\omega $ and $ \omega\circ \alpha $ are identities, showing that $ \alpha $ is an isomorphism.\\

If $ \CM_1=\CM_2=\dots=\CM_r $ and for all $n\geq 1$, $ g_n $ is symmetric (respectively alternating), then $ \omega(g) $ is symmetric (respectively alternating), which implies that the restriction of $ \alpha $ to $ \Sym^{\CO}(D_1^r,D_0) $ (respectively $ \Alt^{\CO}(D_1^r,D_0) $) induces an isomorphism \[\Sym^{\CO}(D_1^r,D_0)\to\uset{n}{\invlim}\,L^{\CO}_{\rm sym}(D_{1,n}^r,D_{0,n}) \] (respectively \[\Alt^{\CO}(D_1^r,D_0)\to\uset{n}{\invlim}\,L^{\CO}_{\rm alt}(D_{1,n}^r,D_{0,n}) ).\]
\end{proof}

\begin{cor}
\label{cordieudonnpidiv}
Let $ \CM_0,\CM_1,\dots,\CM_r $ be $ \pi $-divisible $ \CO $-module schemes over a perfect field $k$ of characteristic $p$. For every $ i=0,\dots,r $, denote by $ D_i$ the (covariant) Dieudonn\'e module of $\CM_i$. There exist natural isomorphisms \[ \Mult^{\CO}(D_1\times\dots\times D_r,D_0)\cong \Mult^{\CO}_k(\CM_1\times\dots\times\CM_r,\CM_0),\]
\[\Sym^{\CO}(D_1^r,D_0)\cong \Sym^{\CO}_k(\CM_1^r,\CM_0)\] and \[\Alt^{\CO}(D_1^r,D_0)\cong \Alt^{\CO}_k(\CM_1^r,\CM_0)\] functorial in all arguments.
\end{cor}

\begin{proof}
We prove only the first isomorphism; the proofs of the other two are similar. Let us use the notations of the Construction \ref{conslimitdieudonne}. It follows from Corollary \ref{cor03} that for every $ n\geq 1 $, there exists a natural isomorphism \[ L^{\CO}(D_{1,n}\times\dots\times D_{r,n},D_{0,n})\cong \Mult^{\CO}_k(\CM_{1,n}\times\dots\times \CM_{r,n},\CM_{0,n}).\] As these isomorphisms are functorial in all arguments, we obtain an isomorphism \[ \uset{n}{\invlim}\,L^{\CO}(D_{1,n}\times\dots\times D_{r,n},D_{0,n})\cong \uset{n}{\invlim}\,\Mult^{\CO}_k(\CM_{1,n}\times\dots\times \CM_{r,n},\CM_{0,n}). \] Now, applying the previous Lemma  and using Definition \ref{def41}, we obtain the required isomorphism \[ \Mult^{\CO}(D_1\times\dots\times D_r,D_0)\cong \Mult^{\CO}_k(\CM_1\times\dots\times\CM_r,\CM_0),\] functorial in all arguments.
\end{proof}

Let us fix some notations for the rest of this section.\\

\textbf{Notations:}
\begin{itemize}
\item We fix a natural number $ j $.
\item Unless otherwise specified, $k$ is a perfect field of characteristic $p>2$.
\item $\mathcal{M}$ is a $\pi$-divisible $\CO$-module scheme of dimension $1$ and height $h$ over $k$, and for every natural number $i$, $\CM_i$ is the kernel of $\pi^i.:\CM\to \CM$.
\item $W$ is the ring of Witt vectors over $k$ and $L$ is the fraction field of $W$.
\item $D:=D(\CM)$ is the covariant Dieudonn\'e module of $\CM$ and $\ep^jD:=\unset{W\widehat{\otimes}_{\BZ_p}\CO}{\ep^j}D$.
\item For every natural number $i$, $D_i:=D_{*}(\CM_i)$ is the covariant Dieudonn\'e module of $\CM_i$ and $\ep^jD_i:=\unset{W\otimes_{\BZ_p}\frac{\CO}{\pi^i}}{\ep^j}D_i$.
\item Denote by $ \zeta $ the surjection $ \zeta:D\onto D_i $ and by $ \ep^j\zeta $ the surjection $ \ep^jD\onto\ep^jD_i $. Note that $ \zeta$ doesn't have any index (to avoid complexity) and we use the same letter for different indices.
\item Denote by $ \Upsilon $ (respectively $ \upsilon $) the morphism $ \ep^jV:\ep^jD\to \ep^jD $ (respectively $ \ep^jV:\ep^jD_i\to \ep^jD_i $) sending an element $ d_1\wedge\cdots\wedge d_j $ with $ d_1, \cdots, d_j\in D $ (respectively in $ D_i $) to $ Vd_1\wedge\cdots\wedge Vd_j $.
\end{itemize}

\begin{rem}
\label{rem 18}
Note that we have $ \ep^j\zeta\circ\Upsilon=\upsilon\circ\ep^j\zeta $.
\end{rem}

\begin{lem}
\label{lem 9}
Let $k$ be algebraically closed.
\begin{itemize}
\item[a)] There exist a ring $A$, a ring homomorphism $ \BZ_p\to A $ and a decomposition $$W\otimes_{\BZ_p}\CO=\prod_{\BZ/f\BZ}W\otimes_A\CO,$$ where $W\otimes_A\CO$ is a discrete valuation ring with residue field $k$ and maximal ideal generated by $1\otimes\pi$. 
\item[b)] The decomposition in $a)$ gives the following decomposition of the completed tensor product $W\widehat{\otimes}_{\BZ_p}\CO$ as a product of complete discrete valuation rings with maximal ideal generated by $1\widehat{\otimes}\pi$ and residue field equal to $k$: $$W\widehat{\otimes}_{\BZ_p}\CO=\prod_{\BZ/f\BZ}W\widehat{\otimes}_A\CO.$$
\item[c)] Let $N$ be a $W\widehat{\otimes}_{\BZ_p}\CO$-module endowed with a $^{\sigma}$-linear morphism $\phi:N\to N$, i.e., for every $x\in W\widehat{\otimes}_{\BZ_p}\CO$ and $n\in N$, we have $\phi(x\cdot n)=(^{\sigma}\widehat{\otimes}\Id)(x)\cdot\phi(n)$. Then there is a decomposition of $N$ as a product $\prod_{i\in\BZ/f\BZ}N_i$ into $W\widehat{\otimes}_A\CO$-modules, according to the decomposition of $W\widehat{\otimes}_{\BZ_p}\CO$ given above, such that the morphism $\phi$ restricts to morphisms $\phi:N_i\to N_{i-1}$ for all $i\in \BZ/f\BZ$.
\end{itemize}
\end{lem}

\begin{proof}[\textsc{Proof}]
$ $
\begin{itemize}
\item[$a$)] We prove this lemma in equal and mixed characteristic cases separately. 
\begin{itemize}
\item Equal characteristic: Set $A:=\BF_q$ and let $ \BZ_p\to \BF_q$ be the canonical ring homomorphism. In this case, $\CO$ is isomorphic to $\BF_q\lbb\pi\rbb$ and therefore, the tensor product $W\otimes_{\BZ_p}\CO$ is isomorphic to $k\otimes_{\BF_p}\BF_q\lbb\pi\rbb$ which decomposes as $$\prod_{\BZ/f\BZ}k\otimes_{\BF_q}\BF_q\lbb\pi\rbb=\prod_{\BZ/f\BZ}k\lbb\pi\rbb$$ as we have seen in the proof of Theorem \ref{thm 3}. It is then clear that the ring $k\lbb\pi\rbb$ is a discrete valuation ring with residue field $k$ and maximal ideal generated by $\pi.$
\item Mixed characteristic: Let $E$ be the maximal unramified subextension of $K$ (recall that $K$ is the fraction field of $\CO$) and denote by $A$ its ring of integers. We then have a canonical ring extension $ \BZ_p\into A $. Since $k$ is algebraically closed, there is a copy of $A$ inside $W$. As $E$ is the maximal unramified subextension of $K$, the degree of the extension $E/\BQ_p$ is equal to $f$ and therefore we have an $A$-algebra isomorphism $$W\otimes_{\BZ_p}A\cong \prod_{\BZ/f\BZ}W\otimes_AA=\prod_{\BZ/f\BZ}W$$ given on elements by $w\otimes a\mapsto (a\cdot w^{\sigma^{i}})_i$, where $^{\sigma}:A\to A$ is the Frobenius of $A$, induced by the Frobenius of $W$. It follows that 
\[W\otimes_{\BZ_p}\CO\cong W\otimes_{\BZ_p}A\otimes_A\CO\cong\prod_{\BZ/f\BZ}W\otimes_A\CO\] and the Frobenius, i.e., the morphism $^{\sigma}\otimes\Id$ interchanges the factors. This shows the first statement. Now, as $K$ is totally ramified over $E$, $\CO$ is generated over $A$ by an Eisenstein element and since $L$ is unramified over $E$, the same element is again Eisenstein over $W$. Hence, $L\otimes_EK$ is a field and $W\otimes_A\CO$ is the valuation ring in it. Again, since $E/\BQ_p$ is the maximal unramified extension inside $K$, the residue degree of the extension $K/E$ is one and therefore $\CO$ and $A$ have the same residue fields $\BF_q$. Therefore, we have \[\frac{W\otimes_A\CO}{(1\otimes\pi)W\otimes_A\CO}=W\otimes_A\CO/\pi\cong W\otimes_{\BF_q}\BF_q=k\] 
where the first equality follows from flatness of $W$ over $\BZ_p$. This proves the other statement. 
\end{itemize}
\item[$b$)] Since for every $s>0$, $p^s\CO\subset \pi^{s'}\CO$ for some $s'$ (in the equal characteristic case, $s'=1$ and in the mixed characteristic $s'= s$), the submodule $p^s\otimes\CO+W\otimes\pi^{r}\CO$ of $W\otimes_{\BZ_p}\CO$ is equal to $W\otimes\pi^{r'}\CO$ for some $r'$, and therefore the completed tensor product $W\widehat{\otimes}_{\BZ_p}\CO$ is equal to $$\unset{r}{\invlim}\, \dfrac{W\otimes_{\BZ_p}\CO}{W\otimes_{\BZ_p}\pi^r\CO}=\invlim\, W\otimes_{\BZ_p}\frac{\CO}{\pi^r}$$ where the last equality follows from flatness of $W$ over $\BZ_p$. Now using part $a)$, we have
\[\invlim (W\otimes_{\BZ_p}\frac{\CO}{\pi^r})=\invlim \prod W\otimes_A\frac{\CO}{\pi^r}=\prod \invlim (W\otimes_A\frac{\CO}{\pi^r}).\]
As we have seen in $a)$, $W\otimes_A\CO$ is a discrete valuation ring with uniformizer $1\otimes\pi$, and this implies that $\invlim W\otimes_A\frac{\CO}{\pi^r}$ is the $1\otimes\pi$-adic completion of it, which is a complete discrete valuation ring with uniformizer $1\widehat{\otimes}\pi$ and residue field equal to $k$.
\item[$c$)] Let $e_i$ be the primitive idempotent for the $i^{\text{th}}$ factor of the decomposition of $W\widehat{\otimes}_{\BZ_p}\CO$. This decomposition gives a decomposition of $N$, say $N=\prod N_i$, where $N_i=e_iN$. Now, for any $n\in N$, we have $\phi(e_in)=(^{\sigma}\widehat{\otimes}\Id)(e_i)\phi(n)\subset e_{i-1}N=N_{i-1}$. Hence, $\phi(N_i)\subset N_{i-1}$. 
\end{itemize}
\end{proof}

\begin{rem}
\label{rem 24}
$  $
\begin{itemize}
\item[1)] The proof of part $a$) of the lemma in the mixed characteristic case is inspired by the proof of the lemma in \cite{W2}.
\item[2)] Part $b)$ of the lemma implies that  \[\frac{W\widehat{\otimes}_{\BZ_p}\CO}{(1\widehat{\otimes}\pi^i)W\widehat{\otimes}_{\BZ_p}\CO}=\prod \frac{W\widehat{\otimes}_A\CO}{(1\widehat{\otimes}\pi^i)W\widehat{\otimes}_{A}\CO}=\prod W\otimes_A\frac{\CO}{\pi^i}= W\otimes_{\BZ_p}\frac{\CO}{\pi^i}.\]
\end{itemize}
\end{rem}

\begin{lem}
\label{lem 10}
Assume that $k$ is algebraically closed. The Dieudonn\'e module of $\CM$, is a free $W\widehat{\otimes}_{\BZ_p}\CO$-module of rank $h$. If $\CM$ is connected, then there exists an element $\epsilon\in D$ such that the set $\{\eps,V^f\eps,\cdots,V^{(h-1)f}\eps\}$ is a basis of $D$ over $W\widehat{\otimes}_{\BZ_p}\CO$.
\end{lem}

\begin{proof}[\textsc{Proof}]
From Lemma \ref{lem 9} $c)$, we know that there is a decomposition of the Dieudonn\'e module $D=\prod_{i\in\BZ/f\BZ}M_i$, where each $M_i$ is a module over $W\widehat{\otimes}_A\CO$ and that the Verschiebung permutes them cyclically (since it is $^{\sigma^{-1}}$-linear). We want to show that each $M_i$ is a free $W\widehat{\otimes}_A\CO$-module of rank $h$. The Dieudonn\'e module $D_1$ is a $W$-module of finite length, which implies that it is a finite length module over $W\otimes_A\CO$, where $A$ is the ring defined in Lemma \ref{lem 9}, but $\pi D_1=0$ and therefore $D_1$ is a module of finite length over $W\otimes_A\CO/\pi=k$ and we know that its length is $\log_p|\CM_1|=fh$. Take $fh$ elements in $D$ such that their images in $D_1$ generate $D_1$ over $k$, then by Nakayama lemma, they generate $D$ over $W\widehat{\otimes}_A\CO$. Note that the action of $\pi$ on $D$ is free, this follows from the fact that the kernel of $\pi$ on each $D_i$ is the same module $D_1$, and the transition morphisms from $D_{i+1}$ to $D_i$ is multiplication by $\pi$, and therefore the kernel of $\pi$ on $D$ is the inverse limit of $D_1$ with trivial transition morphisms, and hence it is trivial. Therefore $D$ is a finitely generated torsion-free $W\widehat{\otimes}_A\CO$-module and since by Lemma \ref{lem 9} $b)$ $W\widehat{\otimes}_A\CO$ is a discrete valuation ring (and in particular a principal ideal domain), $D$ is free over $W\widehat{\otimes}_A\CO$. The rank of $D$ over $W\widehat{\otimes}_A\CO$ is equal to the length of $D_1$ over $W\otimes_A\CO/\pi=k$, which is $fh$. It follows that $M_i$ are free $W\widehat{\otimes}_A\CO$-modules of finite rank. As $V:D\to D$ is injective by Lemma \ref{lem 13}, and its restriction to $M_i$ is a morphism $M_i\to M_{i+1}$ (for all $i\in \BZ/f\BZ$), the $M_i$ will all have the same rank $h$ over $W\widehat{\otimes}_A\CO$. This shows that $D$ is free of rank $h$.\\

Now, assume that $\CM$ is connected. If we find elements $\eps_i\in M_i$ ($i\in \BZ/f\BZ$) such that the set $\{\eps_i,V^f\eps_i,\cdots,V^{(h-1)f}\eps_i\}$ is a basis of $M_i$ over $W\widehat{\otimes}_A\CO$ (note that $V^{jf}(M_i)\subset M_i$ for every $j\geq 0$) , then the element $\eps:=\eps_1+\eps_2\cdots+\eps_f$ will be the desired element and we are done. Since $W\widehat{\otimes}_A\CO$ is a local ring with maximal ideal generated by $1\widehat{\otimes}\pi$ and since $M_i$ is a free module of rank $h$ over it, in order to find $\eps_i$, it suffices (by Nakayama lemma) to find an element $\overline{\eps_i}\in\cokernel(\pi:M_i\to M_i):=\overline{M}_i$ such that the set $\{\overline{\eps_i}, V^f\overline{\eps_i}, \cdots, V^{(h-1)f}\overline{\eps_i}\}$ is a basis of $\overline{M}_i$ over $W\widehat{\otimes}_A\CO/\pi\cong k$ ($M_i$ being free of rank $h$ over $W\widehat{\otimes}_A\CO$, we have that $\overline{M}_i$ has dimension $h$ over $k$) and then define $\eps_i$ to be a lift of $\ol{\eps_i}$ in $M_i$.\\

From the definition of $\overline{M}_i$ we have that $D_1=\prod_{i}\overline{M}_i$ and that Verschiebung is a morphism $V:\overline{M}_i\to \overline{M}_{i+1}$. Since the dimension of $\CM$ is $1$, the Hopf algebra of $\CM_1$ is isomorphic to $\dfrac{k[x]}{(x^{q^h})}$ (cf. \cite{W1} p.112, \S14.4, Theorem), and so $F_{\CM_1}^{r}=0$ if and only if $r\geq fh$. It follows that $V^r:D_1\to D_1$ is the zero morphism if and only if $r\geq fh$. Set $\phi:=V^f$. As stated above, we have $\phi(\overline{M}_i)\subset \overline{M}_i$, and so we have a $^{\sigma^{-f}}$-linear morphism $\phi:\overline{M}_i\to \overline{M}_i$. We claim that $\phi^{h-1}:\ol{M}_i\to \ol{M}_i$ is not the zero morphism. Indeed, if we have $V^{(h-1)f}|_{\ol{M}_i}=\phi^{h-1}|_{\ol{M}_i}=0$ for some $i$, then for every $j$ and every element $x\in \ol{M_j}$, we have $V^{\ol{i-j}}(x)\in \ol{M}_i$, where $\ol{i-j}\geq 0$ is the class of $i-j$ modulo $f$ and so  $V^{(h-1)f+\ol{(i-j)}}(x)=0$. But $(h-1)f+\ol{(i-j)}<hf$. This implies that $V^{hf-1}$ is the zero morphism on $D_1$, which is in contradiction with what we said above. Now, let $\ol{\eps_i}\in \ol{M}_i$ be an element with $\phi^{h-1}(\ol{\eps_i})\neq 0$. Then the set $\{\ol{\eps_i}, \phi(\ol{\eps_i}),\cdots,\phi^{h-1}(\ol{\eps_i})\}$ is linearly independent over $k$, for if we have a non-trivial relation $\sum_{j=j_0}^{h-1}a_j\phi^{j}(\ol{\eps_i})=0$ with $a_j\in k$ and $a_{j_0}\neq 0$, then $$0=\phi^{(h-1-j_0)}(\sum_{j=j_0}^{h-1}a_j\phi^{j}(\ol{\eps_i}))=\sum_{j=j_0}^{h-1}a_j^{q^{h-1-j_0}}\phi^{h-1-j_0+j}(\ol{\eps_i})=a_{j_0}^{q^{h-1-j_0}}\phi^{h-1}(\ol{\eps_i}),$$ because $\phi^{r}=0$ for $r\geq h$. But $\phi^{h-1}(\ol{\eps_i})$ is not zero, and so $a_{j_0}=0$, which is in contradiction with the choice of $j_0$. As the dimension of $\ol{M}_i$ over $k$ is $h$ and the set $\{\ol{\eps_i}, \phi(\ol{\eps_i}),\cdots,\phi^{h-1}(\ol{\eps_i})\}$ is linearly independent and has $h$ elements, we deduce that this set is in fact a basis of $\ol{M}_i$ over $k$ and the proof is achieved.
\end{proof}

\begin{rem}
\label{rem 14}
$  $
\begin{itemize}
\item[1)] Note that the first part of the Lemma, i.e., that the Dieudonn\'e module is free of rank $h$, is true without assuming that $ \CM $ has dimension 1.
\item[2)] Since $D_i$ is the cokernel of $\pi^i$ on $D$ and the projection from $D$ to $D_i$ commutes with $V$, it follows from Lemma \ref{lem 10} that $D_i$ is a free $W\otimes_{\BZ_p}\frac{\CO}{\pi^i}$-module of rank $h$ and that the set $\{\bar{\eps},V^f\bar{\eps},\cdots,V^{(h-1)f}\bar{\eps}\}$, where $\bar{\eps}$ is the image of $\eps$ in $D_i$, is a basis of $D_i$ over $W\otimes_{\BZ_p}\frac{\CO}{\pi^i}$.
\item[3)] In the above proof, let $i$ be such that restriction of $V^{hf-1}$ to $\ol{M}_i$ is not zero and choose $\eps_i\in \ol{M}_i$ with $V^{hf-1}(\eps_i)\neq 0$. Then for every $0\leq j\leq f-1$, we have $V^{(h-1)f}(V^j\eps_i)\neq 0$. Since for these $j$, we have $V^j\eps_i\in \ol{M}_{i+j}$, we see that we could take $\eps_{i+j}$ to be $V^{j}\eps_i$. This shows that we have a sequence of $^{\sigma^{-1}}$-linear isomorphisms \[\ol{M}_i\uset{\cong}{\arrover{V}}\ol{M}_{i+1}\uset{\cong}{\arrover{V}}\ol{M}_{i+2}\uset{\cong}{\arrover{V}}\dots\uset{\cong}{\arrover{V}}\ol{M}_{i-1}.\] Now, by Nakayama lemma, and the fact that $V$ is injective on $D$, we conclude that $V$ induces $^{\sigma^{-1}}$-linear isomorphisms \[V:M_j\to M_{j+1}\] for every $j\neq i-1$. It follows that $$VD\cong VM_{f-1}\times VM_0\times VM_1\times\dots\times VM_{f-2}\cong$$ \[M_0\times M_1\times\dots\times M_{i-1}\times VM_{i-1}\times M_{i+1}\times\dots\times M_{f-1}.\] Thus, the Lie algebra of $\CM$ is isomorphic to $$D/VD\cong M_i/VM_{i-1}\cong M_i/V^fM_i.$$
\end{itemize}
\end{rem}

\begin{lem}
\label{lem 14}
Assume that $\CM$ is connected. Then the morphism $ \Upsilon:\ep^jD\to \ep^jD $ is injective.
\end{lem}

\begin{proof}[\textsc{Proof}]
Since the extension $ k\into \kbar $ is faithfully flat, we may assume that $k$ is algebraically closed. We know that a semi-linear endomorphism of a free module of finite rank over a (commutative) ring (with $1$) is injective if and only if its determinant is a non-zero divisor.
As $ D $ is a free $ W\widehat{\otimes}_{\BZ_p}\CO $-module of rank $ h $, $ \ep^jD $ is a free $ W\widehat{\otimes}_{\BZ_p}\CO $-module of rank $ \binom{h}{j} $. Now, the determinant of $ \Upsilon=\ep^jV $ is equal to det$(V)^{\binom{h-1}{j-1}} $ and from what we said at the beginning of the proof, it follows that $ V $ is injective if and only if $ \Upsilon $ is injective, and since by Lemma \ref{lem 13} $V$ is injective, $ \Upsilon $ is injective too.
\end{proof}

\begin{rem}
\label{rem 21}
Assume that $ \CM $ is connected. Then, by previous lemma, there exists at most one $ ^{\sigma^{-1}}\otimes\Id $-linear morphism $ \Phi:\ep^jD\to\ep^jD $ such that $ \Upsilon\circ\Phi=p $.
\end{rem}

\begin{dfn}
\label{def 14}
Assume that $\CM$ is connected and $k$ is algebraically closed.
\begin{itemize}
\item[(i)] Denote by $ \Phi $ the morphism $\ep^jD\to \ep^jD$
which is defined on the basis $$ \{V^{f\alpha_1}\eps\wedge\dots\wedge V^{f\alpha_j}\eps\, |\, 0\leq\alpha_1<\alpha_2<\cdots<\alpha_j\leq h-1\} $$ by
\[V^{f\alpha_1}\eps\wedge\dots\wedge V^{f\alpha_j}\eps\mapsto FV^{f\alpha_1}\eps\wedge V^{f\alpha_2-1}\eps\wedge\dots\wedge V^{f\alpha_j-1}\eps \]
and is defined on the whole space, $ \ep^jD $, by $ ^{\sigma^{-1}}\otimes\Id $-linearity.
\item[(ii)] Denote by $ \phi $ the morphism $\ep^jD_i\to \ep^jD_i$
which is defined on the basis  $$  \{V^{f\alpha_1}\bar{\eps}\wedge\dots\wedge V^{f\alpha_j}\bar{\eps}\, |\, 0\leq\alpha_1<\alpha_2<\cdots<\alpha_j\leq h-1\}$$ by 
\[V^{f\alpha_1}\bar{\eps}\wedge\dots\wedge V^{f\alpha_j}\bar{\eps}\mapsto FV^{f\alpha_1}\bar{\eps}\wedge V^{f\alpha_2-1}\bar{\eps}\wedge\dots\wedge V^{f\alpha_j-1}\bar{\eps} \]
and is defined on the whole space, $ \ep^jD_i $, by $ ^{\sigma^{-1}}\otimes\Id $-linearity.
\end{itemize}
\end{dfn}

\begin{rem}
\label{rem 19}
Note that we have $ \ep^j\zeta\circ\Phi=\phi\circ\ep^j\zeta $.
\end{rem}

\begin{lem}
\label{lem 15}
Assume that $\CM$ is connected and $k$ is algebraically closed.
\begin{itemize}
\item[a)] We have $ \Phi\circ\Upsilon=p=\Upsilon\circ\Phi $ and the $ F $-diagram associated to $ \Phi $ and the $ V $-diagram associated to $ \Upsilon $ are commutative.
\item[b)] We have $ \phi\circ\upsilon=p=\upsilon\circ\phi $ and the $ F $-diagram associated to $ \phi $ and the $ V $-diagram associated to $ \upsilon $ are commutative.
\end{itemize}
\end{lem}

\begin{proof}[\textsc{Proof}]
Although the statements $a)$ and $b)$ are very similar, the proofs are different and in fact $b)$ follows from $a)$ and $a)$ can be regarded as an auxiliary statement for the proof of $b)$.
\begin{itemize}
\item[$a$)] First of all we check the equality $ \Upsilon\circ\Phi=p $. It is sufficient to calculate $ \Upsilon\circ\Phi $ on the basis elements $ V^{f\alpha_1}\eps\wedge\dots\wedge V^{f\alpha_j}\eps $:
\[\Upsilon\circ\Phi (V^{f\alpha_1}\eps\wedge\dots\wedge V^{f\alpha_j}\eps)=\Upsilon(FV^{f\alpha_1}\eps\wedge V^{f\alpha_2-1}\eps\wedge\dots\wedge V^{f\alpha_j-1}\eps)=\]
\[VFV^{f\alpha_1}\eps\wedge\dots\wedge V^{f\alpha_j}\eps=pV^{f\alpha_1}\eps\wedge\dots\wedge V^{f\alpha_j}\eps\]
where the first and second equality follow respectively from the definition of $ \Phi $ (cf. Definition \ref{def 14} (i) ) and $ \Upsilon $ (cf. Notations), and the last equality follows from the equality $ VF=p $. Hence the equality $ \Upsilon\circ\Phi=p $.

Now, we calculate $ \Upsilon\circ \Phi \circ (\Id\wedge V\wedge\cdots\wedge V) $:
\[\Upsilon\circ \Phi \circ (\Id\wedge V\wedge\cdots\wedge V)=p\circ (\Id\wedge V\wedge\cdots\wedge V)=p\wedge V\wedge\cdots\wedge V=\]
\[VF\wedge V\wedge\cdots\wedge V=\Upsilon\circ (F\wedge\Id\wedge\cdots\wedge\Id)\]where the first equality follows from the equality $ \Upsilon\circ\Phi=p $ and the other equalities follow from the definition of $ \Upsilon $ and the equality $ VF=p $. But we know from Lemma \ref{lem 14} that $ \Upsilon $ is injective and therefore, we have 
\[\Phi \circ (\Id\wedge V\wedge\cdots\wedge V)=F\wedge\Id\wedge\cdots\wedge\Id.\]
Denoting by $ \lambda $ the universal alternating morphism $ D\times\cdots\times D\to \ep^jD $ sending an element $ (x_1,\cdots,x_j) $ to $ x_1\wedge\cdots\wedge x_j $, the last equality gives rise to the following diagram:
\[\xymatrix{
D\times\cdots\times D\ar[rrr]^{\lambda} & & & \ep^jD\ar[dd]^{\Phi}\\
D\times\cdots\times D\ar[u]^{\id\times V\times\cdots\times V}\ar[d]_{F\times\Id\times\cdots\times\Id}\ar[rr]^{\lambda} & & \ep^jD\ar[ur]^{\Id\wedge V\wedge\cdots\wedge V}\ar[dr]_{F\wedge\Id\wedge\cdots\wedge\Id} &\\
D\times\cdots\times D\ar[rrr]_{\lambda} & & & \ep^jD
}\]
with the right triangle commutative. It follows that the whole diagram is commutative and thus the $ F $-diagram associated to $ \Phi $ is commutative.
The commutativity of the $ V $-diagram associated to $ \Upsilon $ follows from its definition (in fact it is equivalent to the definition of $ \Upsilon $!). 
It remains to show that $ \Phi\circ\Upsilon=p $. We have:
\[\Phi\circ\Upsilon(x_1\wedge\cdots\wedge x_j)=\Phi(Vx_1\wedge\cdots\wedge Vx_j)=FVx_1\wedge x_2\wedge\cdots\wedge x_j=px_1\wedge\cdots\wedge x_j\]
where the second equality follows from the $ V $-diagram associated to $ \Upsilon $, the third one from the $ F $-diagram associated to $ \Phi $ and the last once, again, from the equality $ FV=p $. This completes the proof of $a)$.

\item[$b$)] The compatibility of $ \Upsilon $ and $ \upsilon $, and of $ \Phi $ and $ \phi $ with respect to the epimorphism $ \zeta: D \onto D_i $ (cf. Remarks \ref{rem 18} and \ref{rem 19}) and statement $a)$ of the lemma imply the following properties:
\begin{enumerate}
\item $\upsilon \circ \phi \circ \ep^j\zeta=\upsilon\circ\ep^j\zeta\circ\Phi=\ep^j\zeta\circ\Upsilon\circ\Phi=\ep^j\zeta\circ p=p\circ\ep^j\zeta.$
\item $ \phi\circ \upsilon \circ \ep^j\zeta= \phi\circ\ep^j\zeta\circ\upsilon=\ep^j\zeta\circ \Phi\circ \Upsilon=\ep^j\zeta\circ p=p\circ\ep^j\zeta.$
\item $ \phi\circ\lambda\circ (\Id\times V\times\cdots\times V)\circ \zeta^j=  \phi\circ\lambda\circ\zeta^j\circ (\Id\times V\times\cdots\times V)=\phi\circ\ep^j\zeta\circ\lambda\circ (\Id\times V\times\cdots\times V)=\ep^j\zeta\circ\Phi\circ\lambda\circ (\Id\times V\times\cdots\times V)=\\
\ep^j\zeta\circ\lambda\circ (F\times\Id\times\cdots\times\Id)=\lambda\circ\zeta^j\circ(F\times\Id\times\cdots\times\Id)=\\
\lambda\circ(F\times\Id\times\cdots\times\Id)\circ\zeta^j.$
\end{enumerate}
Since the morphism $ \zeta:D\to D_i $ is surjective, the morphisms $ \ep^j\zeta:\ep^jD\to\ep^jD_i $ and $ \zeta^j:D^j\to D_i^j $ are surjective as well and thus we can cancel them from the right and conclude from properties 1 and 2 that $ \upsilon \circ \phi =p=\phi\circ \upsilon  $ and from 3 that the $ F $-diagram associated to $ \phi $ is commutative. The commutativity of the $ V $-diagram associated to $ \upsilon $ follows once more from the definition of $ \upsilon $. The part $b)$ is now proved.
\end{itemize}
\end{proof}

\begin{prop}
\label{prop 9}
Assume that $k$ is algebraically closed. Then the Dieudonn\'e module of $\underset{\CO}\ep^j\CM_i$ is isomorphic to $\ep^jD_i$ and in particular the order of $\underset{\CO}{\ep}^j\CM_i$ is equal to $q^{i\binom{h}{j}}$. More precisely we have:
\begin{itemize}
\item[a)] if $\CM$ is \'etale, then the module scheme $\underset{\CO}\ep^j\CM_i$ is isomorphic to the constant module scheme $\underset{\CO}\ep^j(\frac{\CO}{\pi^i})^h$, which has order $q^{i\binom{h}{j}}$ and 
\item[b)] if $\CM$ is connected, then the covariant Dieudonn\'e module of $\underset{\CO}\ep^j\CM_i$ is isomorphic to $\ep^jD_i$ with the actions of $F$ respectively $V$ defined by $\phi$ respectively $\upsilon$.
\end{itemize}
\end{prop}

\begin{proof}[\textsc{Proof}]
Before proving $a)$ and $b)$, let us explain how these two parts will imply the first two statements of the proposition:
For the first part (about the Dieudonn\'e module of the exterior power), note that if $\CM$ is \'etale, each $\CM_i$ is \'etale and, since $k$ is algebraically closed, $\CM_i$ are constant $\CO$-module schemes. In $a)$ we show that in fact $\CM_i$ is isomorphic to the constant $\CO$-module scheme $(\frac{\CO}{\pi^i})^h$ and $\epO^j\CM_i$ is isomorphic to the constant $\CO$-module scheme $\underset{\frac{\CO}{\pi^i}}\ep^j(\frac{\CO}{\pi^i})^h$. The Dieudonn\'e module of $\CM_i$ is therefore isomorphic to $$W\otimes_{\BZ_p}(\frac{\CO}{\pi^i})^h\cong (W\otimes_{\BZ_p}\frac{\CO}{\pi^i})^h\cong (W\otimes_{\BZ_p}\frac{\CO}{\pi^i})^h$$ and the Dieudonn\'e module of $\epO^j\CM_i$ is isomorphic to $$W\otimes_{\BZ}\underset{\frac{\CO}{\pi^i}}\ep^j(\frac{\CO}{\pi^i})^h\cong \underset{W\otimes_{\BZ_p}\frac{\CO}{\pi^i}}{\ovset{j}{\ep}}(W\otimes_{\BZ_p}\frac{\CO}{\pi^i})^h\cong \underset{W\otimes_{\BZ_p}\frac{\CO}{\pi^i}}{\ovset{j}{\ep}} D_i.$$ If $\CM$ is connected, $b)$ is exactly what we need to show.
Now, in the general case, write $\CM_i$ as the direct sum, $\CM_i^{\text{\'et}}\oplus \CM_i^0$, of its \'etale and connected parts (which is possible, since $k$ is algebraically closed). Using the universal properties of exterior power, tensor product and direct sum, we obtain a canonical isomorphism:
\[\ep^j\CM_i\cong\bigoplus_{r=0}^j(\ep^r\CM_i^{\text{\'et}}\otimes\ep^{j-r}\CM_i^0).\]

Applying the covariant Dieudonn\'e functor on the both sides of this isomorphism, and using the fact that the Dieudonn\'e functor preserves direct sums, we get the following isomorphism (in the following isomorphisms, we omit the subscript $W\otimes_{\BZ_p}\CO$ from the exterior powers and tensor product in order to avoid heavy notations):
\[D_*(\ep^j\CM_i)\cong \bigoplus_{r=0}^jD_*(\ep^r\CM_i^{\text{\'et}}\otimes \ep^{j-r}\CM_i^0).\]

Now, applying Lemma \ref{lem 23} to $\CM_i^0$ and the \'etale $\CO$-module scheme $\CM_i^{\text{\'et}}$, we can interchange the Dieudonn\'e functor with the tensor product and we obtain:
\[D_*(\ep^j\CM_i)\cong \bigoplus_{r=0}^j(D_*(\ep^r\CM_i^{\text{\'et}})\otimes D_*(\ep^{j-r}\CM_i^0)).\]

Finally, using parts $a)$ and $b)$ of the proposition, we get the following isomorphism 
\[D_*(\ep^j\CM_i)\cong \bigoplus_{r=0}^j(\ep^rD_i^{\text{\'et}}\otimes\ep^{j-r}D_i^0)\]

where $D_i^{\text{\'et}}$ and $D_i^0$ denote respectively the Dieudonn\'e module of $\CM_i^{\text{\'et}}$ and $\CM_i^0$. But the right hand side of the isomorphism is isomorphic to $\ep^j(D_i^{\text{\'et}}\oplus D_i^0)$ which is itself isomorphic to $\ep^jD_i$, again since the Dieudonn\'e functor commutes with direct sums. Hence, the canonical isomorphism 
\[D_*(\ep^j\CM_i)\cong \ep^jD_i.\]

For the statement about the order, using the fact that the Dieudonn\'e module of $\epO^j\CM_i$ is isomorphic to $\ovset{j}{\unset{W\otimes_{\BZ_p}\frac{\CO}{\pi^i}}{\ep}}D_i$ and recalling from Remark \ref{rem 14}, that $D_i$ is a free $W\otimes_{\BZ_p}\dfrac{\CO}{\pi^i}$-module of rank $h$, we deduce that $\ovset{j}{\unset{W\otimes_{\BZ_p}\frac{\CO}{\pi^i}}{\ep}}D_i$ is a free $W\otimes_{\BZ_p}\dfrac{\CO}{\pi^i}$-module of rank $\binom{h}{j}$ and that the order of $ \unset{\frac{\CO}{\pi^i}}{\ep}^j\CM_i $ is equal to $$p^{\ell_{W}(\ep^jD_i)}=p^{\binom{h}{j}\cdot\ell_{W}(W\otimes_{\BZ_p}\frac{\CO}{\pi^i})}.$$ Using Lemma \ref{lem 9} a), we have that $$\ell_W(W\otimes_{\BZ_p}\frac{\CO}{\pi^i})=f\cdot\ell_W(W\otimes_{A}\frac{\CO}{\pi^i}).$$ Now recall from the proof of Lemma \ref{lem 9} that in the equal characteristic case the ring $A$ is $\BF_q$ and in the mixed characteristic case it is the ring of integers of the maximal unramified subextension of $K/\BQ_p$. Therefore, in the equal characteristic case we have $$W\otimes_A\frac{\CO}{\pi^i}\cong k\otimes_{\BF_q}\frac{\BF_q\lbb\pi\rbb}{\pi^i}\cong\frac{k\lbb\pi\rbb}{\pi^i}$$ and in the mixed characteristic case we have $$W\otimes_A\frac{\CO}{\pi^i}\cong \frac{W}{p^i}\otimes_{\frac{A}{p^i}}\frac{\CO}{\pi^i}\cong \frac{W}{p^i}$$ since $\frac{A}{p^i}\cong \frac{\CO}{\pi^i}$. It follows that in either case we have $\ell_W(W\otimes_A\frac{\CO}{\pi^i})=i$. Hence the order of  $ \unset{\frac{\CO}{\pi^i}}{\ep}^j\CM_i $ is equal to $p^{fi\binom{h}{j}}=q^{i\binom{h}{j}}$. Now we prove parts $a)$ and $b)$.

\begin{itemize}
\item[$a$)] Since $k$ is algebraically closed, the finite group schemes $\CM_i$ are constant and by abuse of notation, we will denote by $\CM_i$ the abstract group of $k$-rational points of $\CM_i$. Again, since $k$ is algebraically closed, we have exact sequences
$$(*)\qquad 0\to \CM_n\to \CM_{n+m}\ovset{\pi^n}{\to}\CM_m\to 0$$ for all natural numbers $m$ and $n$ (here we mean the exact sequence of $\CO$-modules and not $\CO$-module schemes). For $i=1$, we have that $\CM_i$ is an $\BF_q$-vector space of dimension $h$ and so it is isomorphic to $(\BF_q)^h\cong (\frac{\CO}{\pi})^h$. For $i=2$, we know that $\CM_2$ is an $ \frac{\CO}{\pi^2}$-module of order equal to the order of $ (\frac{\CO}{\pi^2})^h$ and that it is an extension of $(\frac{\CO}{\pi})^h$ by $(\frac{\CO}{\pi})^h$, more precisely, we know that we have the following exact sequence:
$$0\to (\frac{\CO}{\pi})^h\to \CM_{2}\ovset{\pi}{\to}(\frac{\CO}{\pi})^h\to 0$$
It is now an straightforward calculation to see that the only possibility for such an extension is the following one (we use the structure of finitely generated modules over principal ideal domains):
$$0\to (\frac{\CO}{\pi})^h\ovset{\gamma}{\to}(\frac{\CO}{\pi^2})^h\to(\frac{\CO}{\pi})^h\to 0$$
where $\gamma:(\frac{\CO}{\pi})^h\into (\frac{\CO}{\pi^2})^h$ is the canonical injection (i.e., given by the injection $\frac{\CO}{\pi}\into \frac{\CO}{\pi^2}, x\mapsto \pi x$) and $(\frac{\CO}{\pi^2})^h\onto (\frac{\CO}{\pi})^h$ in the canonical projection (i.e., given by the projection $\frac{\CO}{\pi^2}\onto \frac{\CO}{\pi}$, which is not the multiplication by $\pi$ any more!). 
Proceeding in the same fashion and using the exact sequences $(*)$, we conclude that for every $i\geq 1$, we have $\CM_i\cong (\frac{\CO}{\pi^i})^h$ which is also an isomorphism of $\CO$-module schemes and thus $\underset{\CO}\ep^j\CM_i\cong \underset{\CO}\ep^j\underline{(\frac{\CO}{\pi^i})^h}$ (the underline here is to emphasize that we are dealing with a constant group scheme). By Proposition \ref{prop 8} c), we know that $\underset{\CO}\ep^j\underline{(\frac{\CO}{\pi^i})^h}\cong \underset{\frac{\CO}{\pi^i}}\ep^j\underline{(\frac{\CO}{\pi^i})^h}$. Now the universal property of exterior powers (and some straightforward calculations) imply that $ \underset{\frac{\CO}{\pi^i}}\ep^j\underline{(\frac{\CO}{\pi^i})^h}\cong \underline{\underset{\frac{\CO}{\pi^i}}\ep^j(\frac{\CO}{\pi^i})^h}$ which is isomorphic to $ \underline{(\frac{\CO}{\pi^i})^{\binom{h}{j}}}$. This finishes the proof of $a$).

\item[$b$)] We know from Lemma \ref{lem 15} $b$) that $ \phi $ and $ \upsilon $ are commuting morphisms making the $ F $-diagram associated to $ \phi $ and the $ V $-diagram associated to $ \upsilon $ commute. We can therefore apply Lemma \ref{lem 11} and conclude that the covariant Dieudonn\'e module of $ \unset{\frac{\CO}{\pi^i}}{\ep}^j\CM_i $ is isomorphic to $ \ep^jD_i $ with the actions of $ F $ and respectively $ V $ through the actions of $ \phi $ and respectively $ \upsilon $.
\end{itemize}
\end{proof}

\begin{rem}
\label{rem 17 second}
$  $
\begin{itemize}
\item[1)] In the proof of $a$) we have shown that if $\CM$ is \'etale, then $\CM_i\cong (\frac{\CO}{\pi^i})^h$ and the injections $ \iota:\CM_i\into \CM_{i+1}$ correspond to the canonical injections $(\frac{\CO}{\pi^i})^h\into(\frac{\CO}{\pi^{i+1}})^h$ given by multiplication by $\pi$. It follows that as a constant formal $\CO$-module scheme, $\CM$ is isomorphic to $(K/\CO)^h$.
\item[2)] Note that in the proof of the proposition, we didn't made any assumption on $j$ (i.e., we didn't assume $j\leq h$). If $j>h$, then in the \'etale case, $ \unset{\CO}{\ep}^j\CM_i=\underline{\unset{\CO}{\ep}^j(\frac{\CO}{\pi^i})^{\binom{h}{j}}}=0$ and in the connected case, $M_{ij}=0$ and so in any case, $ \unset{\CO}{\ep}^j\CM_i=0$ and therefore it has order $1=q^{\binom{h}{j}}$.
\end{itemize}
\end{rem}

\begin{cor}
\label{cor41}
Let $k$ be a perfect field of characteristic $p>2$. Then the Dieudonn\'e module of $\underset{\CO}\ep^j\CM_i$ is canonically isomorphic to $\ep^jD_i$ and the order of $\underset{\CO}{\ep}^j\CM_i$ is equal to $q^{i\binom{h}{j}}$.
\end{cor}

\begin{proof}
Fix an algebraic closure $ \kbar $ of $k$. In order to simplify the notations, fix an $i$ and set $ M:=\CM_i $. By Proposition \ref{prop 41} we know that the canonical homomorphism $ \epO^j(M_{\kbar})\to(\epO^jM)_{\kbar} $ is an isomorphism. We then obtain the following series of isomorphisms: \[ D_*(\epO^jM)\otimes_{W(k)}W(\kbar)\cong D_*((\epO^jM)_{\kbar})\cong D_*(\epO^j(M_{\kbar}))\cong \ep^jD_*(M_{\kbar}) \] \[ \ep^j(D_*(M)\otimes_{W(k)}W(\kbar))\cong \ep^jD_*(M)\otimes_{W(k)}W(\kbar) \] where the first and fourth isomorphisms are the base change property of the Dieudonn\'e functor and the third property is given by the previous proposition, and finally, the last isomorphism is the base change property of the exterior powers in the category of modules. It follows from the construction of these isomorphisms that the resulting isomorphism \[\ep^jD_*(M)\otimes_{W(k)}W(\kbar)\cong D_*(\epO^jM)\otimes_{W(k)}W(\kbar)\] is the extension of scalars of the canonical homomorphism \[\vartheta:\ep^jD_*(M)\to D_*(\epO^rM)\cong T_{\text{weakalt}}(M^j)\](cf. the proof of Lemma \ref{lem 12} for the definition of this homomorphism). As the ring homomorphism $ W(k)\to W(\kbar) $ is faithfully flat, and $  \vartheta\otimes_{W(k)}\Id$ is an isomorphism, it follows that $ \vartheta $ is an isomorphism as well. The free $ W(k)\otimes_{\BZ}\CO/\pi^i $-module $ \ep^jD_*(M) $ has rank $\binom{h}{j}$ and therefore, the order of $ \epO^jM $ is equal to $ q^{i\binom{h}{j}} $.
\end{proof}

\begin{prop}
\label{propbasechangeofBT}
Let $k$ be a field of characteristic $p>2$ and let $ \ell/k $ be a field extension. Then the base change homomorphism $$ f:\epO^r(\CM_{n,\ell})\to (\epO^r\CM_n)_{\ell} $$ is an isomorphism.
\end{prop}

\begin{proof}
First assume that $\ell $ is an algebraically closed field and consider the homomorphism \[\lambda^*: \innHom_{\ell}(\epO^r(\CM_{n,\ell}),\BG_m)\to \widetilde{\innAlt}^{\CO}_{\ell}(\CM_{n,\ell}^r,\BG_m) \] obtained by universal property of $ \epO^r(\CM_{n,\ell}) $ and sheafification. This homomorphism induces an isomorphism on the $ \ell $-valued points (this is the universal property of $ \epO^r(\CM_{n,\ell}) $). For every finite group scheme $I$ over $ \ell $, we have a commutative diagram \[ \xymatrix{\Hom_{\ell}(I, \innHom_{\ell}(\epO^r(\CM_{n,\ell}),\BG_m))\ar[rr]^{\Hom(I,\lambda^*)}\ar[d]_{\cong} && \Hom_{\ell}(I,\widetilde{\innAlt}_{\ell}^{\CO}(\CM_{n,\ell}^r,\BG_m))\ar[d]^{\cong}\\ \Hom_{\ell}(\epO^r(\CM_{n,\ell}),\innHom_{\ell}(I,\BG_m))\ar[rr]&& \widetilde{\Alt}_{\ell}^{\CO}(\CM_{n,\ell}^r,\innHom_{\ell}(I,\BG_m)).}\] Since $ I $ is a finite group scheme over $ \ell $, $ \innHom_{\ell}(I,\BG_m) $ is a (finite) group scheme over $ \ell $ and so, by the universal property of $ \epO^r(\CM_{n,\ell}) $, the bottom morphism of the diagram is an isomorphism, which implies that the top morphism is an isomorphism too. We can know apply Lemma \ref{prop02} and conclude that $ \lambda^* $ is an isomorphism. In particular, since by Corollary \ref{cor41}, $ \epO^r(\CM_{n,\ell}) $ is finite over $ \ell $, it follows that $ \widetilde{\innAlt}_{\ell}^{\CO}(\CM_{n,\ell}^r,\BG_m) $ is finite over $ \ell $ as well. From the isomorphism \[ \widetilde{\innAlt}_{k}^{\CO}(\CM_n^r,\BG_m)_{\ell}\cong \widetilde{\innAlt}_{\ell}^{\CO}(\CM_{n,\ell}^r,\BG_m) \] and the finiteness of $ \widetilde{\innAlt}_{\ell}^{\CO}(\CM_{n,\ell}^r,\BG_m) $ over $ \ell $ we deduce that $ \widetilde{\innAlt}_{k}^{\CO}(\CM_n^r,\BG_m) $ is finite over $ k $. Now, let $ \ell $ be any extension of $k$. By Theorem \ref{thm41}, $ \epO^r\CM_n $ is isomorphic to $ \widetilde{\innAlt}_{k}^{\CO}(\CM_n^r,\BG_m)^* $ which is equal to the Cartier dual of $ \widetilde{\innAlt}_{k}^{\CO}(\CM_n^r,\BG_m) $ (since it is finite). So, we have shown that for every field $k$, we have a canonical isomorphism \[ \epO^r\CM_n\cong \widetilde{\innAlt}_{k}^{\CO}(\CM_n^r,\BG_m). \] As the Cartier duality and the construction $ \widetilde{\innAlt}^{\CO} $ commute with base change, we conclude that the base change homomorphism $$ f:\epO^r(\CM_{n,\ell})\to (\epO^r\CM_n)_{\ell} $$ is an isomorphism as desired.
\end{proof}

\begin{rem}
\label{rem 17}
$  $
\begin{itemize}
\item[1)] If the ground field $k$, of characteristic $p$, is not perfect, we still have that the order of $ \epO^j\CM_i $ is equal to $ q^{\binom{h}{j}} $. This follows from Proposition \ref{propbasechangeofBT} and the fact that the order of a group scheme is invariant under field extensions. So, we may assume that $k$ is algebraically closed and apply Corollary \ref{cor41}.
\item[2)] It follows from Corollary \ref{cor41} that the universal alternating morphism $ \CM^j_i\to\epO^j\CM_i $ and the universal alternating morphism \[D_*(\CM_i)^j\to \ep^jD_*(\CM_i)\cong D_*(\epO^j\CM_i)\] correspond to each other under the isomorphism \[ L^{\CO}_{\text{alt}}(D_*(\CM_i)^j, D_*(\epO^j\CM_i))\cong \Alt^{\CO}(\CM_i^j,\epO^j\CM_i) \] explained in Remark \ref{rem33}.
\end{itemize}
\end{rem}


\textbf{Notations.} From now on, unless otherwise specified, $k$ is a field of characteristic $p>2$ (not necessarily perfect) and $ \CM $ is a one dimensional $ \pi $-divisible $ \CO $-module over $k$.\\

We know by Proposition \ref{prop 7} that we have the following exact sequence:

\begin{displaymath}
\epO^j\CM_{n+m} \ovset{\pi^m}{\to} \epO^j\CM_{n+m} \to \epO^j\CM_m \to 0
\end{displaymath}

and since the composition $ \pi^n\circ\pi^m $ is zero on $ \epO^j\CM_{n+m} $, the morphism\\ $ \epO^j\CM_{n+m} \ovset{\pi^n}{\to} \epO^j\CM_{n+m} $, factors through the epimorphism $ \epO^j\CM_{n+m} \onto \epO^j\CM_m $:

\begin{myequation}
\label{diagram eta}
\xymatrix{
\epO^j\CM_{n+m}\ar[rr]^{\pi^n}\ar@{->>}[dr] && \epO^j\CM_{n+m}\\
& \epO^j\CM_m.\ar[ur]_{\eta} &}
\end{myequation}

\begin{lem}
\label{lem 16}
The morphism $ \eta: \epO^j\CM_m \to \epO^j\CM_{n+m}  $ is a monomorphism.
\end{lem}

\begin{proof}[\textsc{Proof}]
We know from Remark \ref{rem 17} 1), that $ \epO^j\CM_{n+m} $ is a finite group scheme, and therefore the morphism $ \epO^j\CM_{n+m} \ovset{\pi^n}{\to} \epO^j\CM_{n+m} $ factors through its image, i.e., we have a commutative diagram:

\begin{myequation}
\label{diagram 1 lemma 16}
\xymatrix{
\epO^j\CM_{n+m}\ar[rr]^{\pi^n}\ar@{->>}[dr] && \epO^j\CM_{n+m}\\
& \image (\pi^n) \ar@{^{(}->}[ur] &}
\end{myequation}

and by Proposition \ref{prop 7}, we have an exact sequence:

\begin{myequation}
\label{diagram 2 lemma 16}
\epO^j\CM_{n+m} \ovset{\pi^n}{\to} \epO^j\CM_{n+m} \to \epO^j\CM_n \to 0.  
\end{myequation} 

It follows from diagram (\ref{diagram 1 lemma 16}) and exact sequence (\ref{diagram 2 lemma 16}) that the sequence
\[0 \to \image(\pi^n) \to \epO^j\CM_{n+m} \to \epO^j\CM_n \to 0 \]
is exact and thus we have the following relation on orders of these group schemes: $$ |\image(\pi^n)|\cdot |\epO^j\CM_n|=|\epO^j\CM_{n+m}|. $$ But we know from Remark \ref{rem 17} 1), that 
$ |\epO^j\CM_n|= q^{n\binom{h}{j}} $ and $ |\epO^j\CM_{n+m}|= q^{(n+m)\binom{h}{j}} $ which imply that $ |\image(\pi^n)|=q^{m\binom{h}{j}} $.\\

The commutativity of diagrams (\ref{diagram eta}) and (\ref{diagram 1 lemma 16}) and the fact that $ \pi^n\circ \pi^m $ is zero on $ \epO^j\CM_{n+m} $ imply that the epimorphism $ \epO^j\CM_{n+m} \onto \image(\pi^n) $ factors through the epimorphism $ \epO^j\CM_{n+m} \onto \epO^j\CM_m  $, and therefore we have a commutative diagram:
\[\xymatrix{
\epO^j\CM_{n+m}\ar@{->>}[rr]\ar@{->>}[dr] && \image(\pi^n)\\
&\epO^j\CM_m  \ar[ur] &}\]
and so $ \epO^j\CM_m \to \image(\pi^n) $ is an epimorphism. But the two group schemes $ \epO^j\CM_m $ and $ \image(\pi^n) $ have the same order $ q^{m\binom{h}{j}} $, and therefore the morphism $ \epO^j\CM_m \to \image(\pi^n) $ is an isomorphism. Now, $ \eta $ being the composition of the isomorphism $ \epO^j\CM_m \to \image(\pi^n) $ and the monomorphism $  \image(\pi^n) \into \epO^j\CM_{n+m}$ (use diagrams (\ref{diagram eta}) and (\ref{diagram 1 lemma 16}), and the fact that $ \epO^j\CM_{n+m} \to \epO^j\CM_m $ is an epimorphism), we conclude that it is a monomorphism as well and the proof is achieved.
\end{proof}

\begin{rem}
\label{rem 23}
The arguments in the proof of the lemma show that for every two natural numbers $m,n$, the following sequence is exact:
\[0 \to \epO^j\CM_m \ovset{\eta}{\to} \epO^j\CM_{n+m} \to \epO^j\CM_n \to 0.\]
Moreover, observing diagram (\ref{diagram eta}) with $ m $ and $ n $ switched, we obtain the following commutative diagram:
\[\xymatrix{
0\ar[r] & \epO^j\CM_m\ar[r]^{\eta} & \epO^j\CM_{n+m}\ar[r]\ar[d]^{\pi^m} & \epO^j\CM_n\ar@{^{(}->}[dl]_{\eta}\ar[r] & 0\\
  &             & \epO^j\CM_{n+m} &             &}
\]
from which we obtain the following exact sequence:
\begin{myequation}
\label{diagram kernel of pi^m}
0 \to \epO^j\CM_m \ovset{\eta}{\longto} \epO^j\CM_{n+m} \ovset{\pi^m}{\longto} \epO^j\CM_{n+m}.
\end{myequation}
If we regard $ \epO^j\CM_n $ as a submodule scheme of $ \epO^j\CM_{n+m} $, using the monomorphism $ \eta: \epO^j\CM_n \into  \epO^j\CM_{n+m}  $, we may as well denote the  epimorphism $ \epO^j\CM_{n+m} \onto \epO^j\CM_n $ by $ \pi^m $, i.e., multiplication by $ \pi^m $ and rewrite the first exact sequence of the remark as:
\begin{myequation}
\label{diagram the short exact sequence of alt. powers}
0 \to \epO^j\CM_m \ovset{\eta}{\longto} \epO^j\CM_{n+m} \ovset{\pi^m}{\longto} \epO^j\CM_n \to 0.
\end{myequation}
\end{rem}
	
\begin{prop}
\label{rem41}
Let $k$ be a field of characteristic $p$ and let $\CM$ be a one dimensional $\pi$-divisible $\CO$-module scheme over $k$, of height $h$. Denote by $ \Lambda^j$ the following inductive system \[  \epO^j\CM_{1} \ovset{\eta}{\into} \epO^j\CM_2 \ovset{\eta}{\into} \epO^j\CM_{3} \into \cdots\] viewed as an ind-object in the category of finite group schemes over $k$. Then $ \Lambda^j $ is a $ \pi $-divisible $\CO$-module scheme over $k$, of height $ \binom{h}{j} $, and together with the system of universal alternating morphisms $ \lambda_n:\CM_n^j\to\epO^j\CM_n $, it is the $ j^{\text{th}} $-exterior power of $ \CM $, i.e., we have $ \Lambda^j\cong\epO^j\CM $.
\end{prop}

\begin{proof}
Fix a natural number $ m $. By Lemma \ref{lem 11}, we know that $ \epO^j\CM_m $ exists as an $ \CO $-module scheme, and its order is equal to $ q^{m\binom{h}{j}} $ by Remark \ref{rem 17} 1). Setting $ n=1 $ in the exact sequence (\ref{diagram kernel of pi^m}), we obtain the following exact sequence:
\[0 \to \epO^j\CM_m \ovset{\eta}{\longto} \epO^j\CM_{m+1} \ovset{\pi^m}{\longto} \epO^j\CM_{m+1}.\]
We have seen in Remark \ref{rem 10}, that these two properties (the equality $ |\,\epO^j\CM_m\,|= q^{m\binom{h}{j}}$ and the exact sequence) imply that the direct limit $ \Lambda^j=\dirlim(\epO^j\CM_n,\eta) $ is a $ \pi $-divisible $ \CO $-module scheme and we have $$ \kernel(\pi^m:\Lambda^j\to\Lambda^j)\cong \epO^j\CM_m. $$ The height of $ \Lambda^j $ is equal to $ \log_q  |\,\epO^j\CM_1\,|=\binom{h}{j}$, as claimed.\\

The projections $ \epO^j\CM_{n+1}\onto\epO^r\CM_n $ are induced by the universal alternating morphisms \[ \lambda_{n+1}:\CM_{n+1}^j\to \epO^j\CM_{n+1} \] and the alternating morphisms $$ \CM_{n+1}^j\ovset{(\pi.)^j}{\onto}\CM_n^j\arrover{\lambda_n}\epO^r\CM_n .$$ Thus, the system of alternating morphisms $ \lambda:=(\lambda_n)_{n\geq 1} $ is an element of the $ \CO $-module $ \Alt_k^j(\CM,\Lambda^j) $ with the following universal property:\\

for every $ \pi $-divisible $ \CO $-module scheme $ \CN $ over $k$ the $ \CO $-linear homomorphism \[ \lambda^*:\Hom_k(\Lambda^j,\CN)\longrightarrow\Alt_k^j(\CM,\CN) \] induced by $ \lambda $ is an isomorphism. This proves the last part of the proposition.
\end{proof}
	
\begin{prop}
\label{prop 11}
Let $\CM$ be a one dimensional $\pi$-divisible $\CO$-module scheme over a perfect field $k$ of characteristic $p>2$. Then the covariant Dieudonn\'e module of $\epO^j\CM$ is isomorphic to $\ep^jD$.
\end{prop}

\begin{proof}[\textsc{Proof}]
By definition, the covariant Dieudonn\'e module of $\epO^j\CM$ is the inverse limit of the system $D_*(\epO^j\CM_{i+1})\ovset{\zeta}{\to}D_*(\epO^j\CM_i)$. By Corollary \ref{cor41}, we know that $D_*(\epO^j\CM_i)=\ep^jD_i$. So, we have only to show that 
\[\unset{W\widehat{\otimes}_{\BZ_p}\CO}{\ep^j}\invlim D_i=\invlim\unset{W\otimes_{\BZ_p}\frac{\CO}{\pi^i}}{\ep^j}D_i.\]
Writing $D$ for the inverse limit $\invlim D_i$ as before, we have natural and compatible morphisms $\ep^j D\to \ep^jD_i$ (for all $i\geq 1$) that we denoted by $\ep^j\zeta$ and were induced by the natural morphisms $\zeta:D\to D_i$. So, by the universal property of the inverse limit, we obtain a morphism $\psi:\ep^jD\to \invlim \ep^jD_i$. Explicitly, $\psi$ sends an element $d_1\wedge\dots\wedge d_j\in \ep^jD$ to the element $(\zeta_i(d_1) \wedge \dots \wedge \zeta_i(d_j))\in \invlim \ep^jD_i $ (here we write down the index $i$ for the morphism $\zeta$ to emphasize that the element $\zeta_i(d_1) \wedge \dots \wedge \zeta_i(d_j)$ is the $i^{\text{th}}$ component). Now, the result is a formal consequence of the fact that $D$ is a finite free $W\widehat{\otimes}_{\BZ_p}\CO$-module (cf. Lemma \ref{lem 10}) and, $D_i=D/\pi^iD$ is a free $W\otimes_{\BZ_p}\CO/\pi^i$-module. Indeed, if we choose a basis $\{e_1,\dotsb,e_h\}$ of $D$ over $W\widehat{\otimes}_{\BZ_p}\CO$, then the images $\bar{e_r}\in D_i, r=1,\dots,h$ form a basis of $D_i$ over $W\otimes_{\BZ_p}\CO/\pi^i$. The morphism $\psi$ sends an element $\sum a_{\nu_1,\dotsb,\nu_j}e_{\nu_1}\wedge\dotsb\wedge e_{\nu_j}$ to $(\sum \bar{a}_{\nu_1,\dots,\nu_j}\bar{e}_{\nu_1}\wedge\dots\wedge\bar{e}_{\nu_j})$ where $\bar{a}_{\nu_1,\dots,\nu_j}$ is the image of $a_{\nu_1,\dotsb,\nu_j}$ in $W\otimes_{\BZ_p}\CO/\pi^i$. As $W\widehat{\otimes}_{\BZ_p}\CO$ is complete, this implies that $\psi$ is bijective and hence an isomorphism.
\end{proof}

\begin{lem}
\label{lem 17}
Let $\CB$ be a discrete valuation ring with valuation $v:\CB\to \BZ\cup\{\infty\}$ and endowed with an automorphism $\omega:\CB\to \CB$. Let $\chi_i:M_i\to N_i$ ($i=1,\dots,n$) be $\omega$-linear morphisms between finite free $\CB$-modules of the same rank. Then, we have 
\[\ell_{\CB}(\cokernel(\chi))=v(\text{det}(\chi_1)\cdot\text{det}(\chi_2)\cdots\text{det}(\chi_n))\]
where $\chi:M_1\oplus M_2\oplus\dots\oplus M_n\to N_1\oplus N_2\oplus\dots\oplus N_n$ is the direct sum of $\chi_i$.
\end{lem}

\begin{proof}[\textsc{Proof}]
Since $M_i$ are free $\CB$-modules, the automorphism $\omega$ induces an $\omega^{-1}$-linear automorphism, $\omega_i$, of $M_i$, more precisely, if we choose a basis $\{m_{\alpha}\,|\,\alpha\in\Lambda\}$ of $M_i$, then $\omega_i(\sum_{\alpha\in\Lambda}b_{\alpha}m_{\alpha})=\sum_{\alpha\in\Lambda}\omega^{-1}(b_{\alpha})m_{\alpha}$. Now, the composite $\chi_i\circ\omega_i:M_i\to N_i$ is $\omega$-linear and we have $\cokernel(\chi)\cong \cokernel (\bigoplus(\chi_i\circ\omega_i))$ and $\text{det}(\chi_i)=\text{det}(\chi_i\circ\omega_i)$ for all $1\leq i\leq n$. Replacing $\chi_i$ by $\chi_i\circ\omega_i$, we may assume that $\chi_i$ are $\CB$-linear.\\

We first prove the lemma when $n=1$. By the elementary divisor theorem we know that there is a basis of $M_1$ and a basis of $N_1$ such that the matrix of $\chi_1$ in this basis is diagonal (Smith normal form), say
\[\begin{pmatrix}a_1&0&\dots&0\\
0&a_2&\dots&0\\
\vdots&\vdots&\ddots&\\
0&0&&a_r\\
\end{pmatrix}\]
with $a_l\in\CB$ for all $1\leq l\leq r$. It follows that the determinant of $\chi_1$ is equal to $a_1\cdot a_2\cdots a_r$ and that the cokernel of $\chi_1$ is isomorphic to $$\dfrac{\CB}{(a_1)}\oplus\dfrac{\CB}{(a_2)}\oplus\dots\oplus\dfrac{\CB}{(a_r)}$$ which implies that the length of $\cokernel(\chi_1)$ is equal to the sum of the lengths of $\dfrac{\CB}{(a_l)}$ ($1\leq l\leq r$). But we have that $\ell_{\CB}(\dfrac{\CB}{(a_l)})=v(a_l)$, and therefore \[\ell_{\CB}(\cokernel(\chi_1))=v(a_1)+v(a_2)+\dots+v(a_r)=v(a_1\cdots a_n)=v(\text{det}(\chi_1))\]
(note that if $r$ is less than the rank of $M_1$ over $\CB$, the both side of the equality are equal to infinity).\\

In the general case, we have that 
\[\cokernel(\chi)=\cokernel(\chi_1)\oplus\cokernel(\chi_2)\oplus\dots\oplus\cokernel(\chi_n).\]
Now, from the result in the case $n=1$, we know that for every $i$, $\ell_{\CB}(\cokernel(\chi_i))=v(\text{det}(\chi_i))$, and so
\[\ell_{\CB}(\cokernel(\chi))=\sum_{i=1}^n\ell_{\CB}(\cokernel(\chi_i))=\sum_{i=1}^nv(\text{det}(\chi_i))=v(\text{det}(\chi_1)\cdots\text{det}(\chi_n)).\]
\end{proof}

\begin{lem}
\label{lem 18}
Let $\CN$ be a finite dimensional $\pi$-divisible $\CO$-module scheme over a perfect field $k$. Then we have
\[\dim\CN=\ell_W(\cokernel(V_{D_*(\CN)}))=\ell_{W\widehat{\otimes}_{A}\CO}(\cokernel(V_{D_*(\CN)}))\]
where $A$ is the ring defined in Lemma \ref{lem 9}.
\end{lem}

\begin{proof}[\textsc{Proof}]
As $\CN$ is finite dimensional, there is a natural number $n_0$, such that the dimension of $\CN$ is equal to the dimension of $\CN_i$ for all $i\geq n_0$, where as usual $\CN_i$ denotes the kernel of $\pi^i$. Now, we have an isomorphism $\kernel(V_{D_*(\CN_i)})\cong\CL ie(\CN_i)$, which implies that 
\begin{myequation}
\label{dimension}
\dim \CN_i=\dim_k \kernel(V_{D_*(\CN_i)})=\ell_W(\kernel(V_{D_*(\CN_i)})).
\end{myequation}

Since $V:D_*(\CN_i)\to D_*(\CN_i)$ is a morphism between finite length $W$-modules, its kernel and cokernel have the same length over $W$, and the inverse limit (over $i$) of the cokernel of $V_{D_*(\CN_i)}$ is isomorphic to the cokernel of the inverse limit of $V_{D_*(\CN_i)}$, which is the Verschiebung of the covariant Dieudonn\'e module of $\CN$. The projections $D_*(\CN_{i+1})\onto D_*(\CN_i)$ induce surjections between the cokernels of $V_{D_*(\CN_{i+1})}$ and $V_{D_*(\CN_i)}$, and since for large enough $i$, cokernels of $V_{D_*(\CN_{i+1})}$ and $V_{D_*(\CN_i)}$ have the same length over $W$ (here we use the fact that $\CN_{i+1}$ and $\CN_i$ have the same dimension and so by (\ref{dimension}), the kernels of $V_{D_*(\CN_{i+1})}$ and $V_{D_*(\CN_i)}$ have the same length over $W$ and finally by what we said above, the cokernels have the same length), the induced surjections are isomorphisms for large $i$. It follows that the cokernel of $V_{D_*(\CN)}$ is isomorphic to the cokernel of $V_{D_*(\CN_i)}$ for large $i$, and in particular, we have $$\ell_W(\cokernel(V_{D_*(\CN)})=\ell_W(V_{D_*(\CN_i)})=\ell_W(\kernel(V_{D_*(\CN_i)}))$$ for large $i$. Putting this together with (\ref{dimension}) and $\dim\CN=\dim\CN_i$, we deduce:
\[\dim\CN=\ell_W(\cokernel(V_{D_*(\CN)})).\]
This is the first equality of the lemma. We should now show that 
$$\ell_W(\cokernel(V_{D_*(\CN)}))=\ell_{W\widehat{\otimes}_{A}\CO}(\cokernel(V_{D_*(\CN)})).$$
As we discussed above, we have $ \cokernel(V_{D_*(\CN)})\cong \cokernel(V_{D_*(\CN_i)}) $ for large enough $i$ and so, it is sufficient to show the equality
$$\ell_W(\cokernel(V_{D_*(\CN_i)}))=\ell_{W\widehat{\otimes}_{A}\CO}(\cokernel(V_{D_*(\CN_i)})).$$
As $ D_*(\CN_i) $ is killed by $ \pi^i $, and $ (W\widehat{\otimes}_A\CO)/\pi^i\cong W\otimes_A\dfrac{\CO}{\pi^i} $, we have
$$\ell_{W\widehat{\otimes}_{A}\CO}(\cokernel(V_{D_*(\CN_i)}))=\ell_{W\otimes_{A}\frac{\CO}{\pi^i}}(\cokernel(V_{D_*(\CN_i)})).$$
Now, as we have seen before (in the proof of Proposition \ref{prop 9}), we have\\
$ W\otimes_A\dfrac{\CO}{\pi^i}\cong \dfrac{W}{p^i} $, and therefore we obtain
\[\ell_{W\otimes_{A}\frac{\CO}{\pi^i}}(\cokernel(V_{D_*(\CN_i)}))=\ell_{\frac{W}{p^i}}(\cokernel(V_{D_*(\CN_i)}))=\ell_W(\cokernel(V_{D_*(\CN_i)})).\]
These equalities together with the last one imply the desired equality and finish the proof of the lemma.
\end{proof}

\begin{thm}
\label{thm 4}
Let $\CM$ be a $\pi$-divisible $\CO$-module scheme of height $h$ and dimension $1$ over a field $k$ of characteristic $p>2$. Fix a natural number $ j $. Then the $ j^{\text{th}} $-exterior power of $ \CM $ in the category of $ \pi $-divisible $ \CO $-module schemes over $k$ exists, and has height $\binom{h}{j}$ and dimension $\binom{h-1}{j-1}$. Moreover, for every positive natural number $n$, we have $ (\epO^j\CM)_n\cong \epO^j(\CM_n) $.
\end{thm}

\begin{proof}[\textsc{Proof}]
We already know by Proposition \ref{rem41} that $ \epO^r\CM $ is a $ \pi $-divisible $ \CO $-module scheme over $k$, of height $ \binom{h}{j} $. So, we should only calculate the dimension of $\epO^j\CM$. Since the dimension is invariant under the base change, by Proposition \ref{propbasechangeofBT} we may assume that $k$ is algebraically closed. Using Lemma \ref{lem 18}, we know that the dimension of $\epO^j\CM $ is equal to the length of the cokernel of the Verschiebung of the covariant Dieudonn\'e module of $\epO^j\CM$, which is by Proposition \ref{prop 11} isomorphic to $ \ep^jD $, and thus, we have to calculate $ \ell_{W\widehat{\otimes}_A\CO}(\cokernel(\ep^jV))$. By Lemma \ref{lem 9}, we have $ D=M_0\times M_1\times\cdots\times M_{f-1} $, where $M_i$ are $W\widehat{\otimes}_A\CO$-modules and the Verschiebung induces $^{\sigma^{-1}}\otimes\Id$-linear morphisms $ V:M_i\to M_{i+1} $ for all $ i\in \frac{\BZ}{f\BZ} $, which we denote by $V_i$. Again, by Proposition \ref{prop 11}, the Verschiebung of $ \ep^jD $ is the morphism $ \Upsilon=\ep^jV $. It is now straightforward to check that $$ \ep^jD=\ep^jM_0\times \ep^jM_1\times\cdots\times\ep^jM_{f-1} $$ and that the morphism $ (\ep^jV)_i:\ep^jM_i\to \ep^jM_{i+1} $ induced by $ \ep^jV $ is equal to $ \ep^j(V_i) $. Now, recalling from Lemma \ref{lem 9} that $ W\widehat{\otimes}_A\CO $ is a discrete valuation ring, and denoting its valuation by $v$, we have by Lemma \ref{lem 17} (here we are using the fact that $D$ and $\ep^jD$ are free $ W\widehat{\otimes}_A\CO $-modules; cf. Lemma \ref{lem 10}):
$$\ell_{W\widehat{\otimes}_A\CO}(\cokernel(\ep^jV))=v(\text{det}(\ep^jV_1)\cdots\text{det}(\ep^jV_f))=$$
$$v(\text{det}(V_1)^{\binom{h-1}{j-1}}\cdots\text{det}(V_f)^{\binom{h-1}{j-1}})=\binom{h-1}{j-1}\cdot v(\text{det}(V_1)\cdots\text{det}(V_f))=$$
\[\binom{h-1}{j-1}\cdot \ell_{W\widehat{\otimes}_A\CO}(\cokernel(V))=\binom{h-1}{j-1}\cdot \dim \CM=\binom{h-1}{j-1}\]
where the second equality is true because $ V_i $ are semi-linear morphisms between finite free modules and we have the relation between the determinant of $ V_i $ and that of $ \ep^jV_i $ which we mentioned in the proof of Lemma \ref{lem 14}, and the one before the last equality follows from Lemma \ref{lem 18}.\\

The last statement follows from the way we constructed $ \epO^j\CM $ (cf. Proposition \ref{rem41}). The proof is now achieved.
\end{proof}

\section{The main theorem: over arbitrary fields}

In this section, we combine the results of the last two sections to prove that the exterior powers of $ \pi $-divisible $ \CO $-module schemes over any base field exist.

\begin{thm}
\label{thmoveranyfield}
Let $ \CM $ be a $ \pi $-divisible $ \CO $-module scheme of height $h$ over a base field $k$. Assume that the dimension of $ \CM $ is at most 1. Fix a positive natural number $j$. Then the $ j^{\text{th}} $-exterior power of $ \CM $ in the category of $ \pi $-divisible $ \CO $-module schemes over $k$ exists, and has height $\binom{h}{j}$. If the dimension of $ \CM $ is 1, then $\epO^j\CM$ has dimension $\binom{h-1}{j-1}$, otherwise, it has dimension zero. Moreover, for every positive natural number $n$, we have $ (\epO^j\CM)_n\cong \epO^j(\CM_n) $.
\end{thm}

\begin{proof}
If the characteristic of $k$ is different from $p$ or the dimension of $ \CM $ is zero, then $ \CM $ is \'etale and the statements of the theorem follow from Proposition \ref{prop025} and Remark \ref{rem024}. So, we can assume that the characteristic of $k$ is $p$ and the dimension of $ \CM $ is 1. The statements of the theorem now follow from Theorem \ref{thm 4}.
\end{proof}

\chapter{Multilinear Theory of Displays}

In this chapter, after recalling basic definitions, constructions and results of the theory of display developed by Zink, we develop a multilinear theory of displays, which, in the next chapter, will be related to the multilinear theory of group schemes developed by Pink. For more details on displays, we refer to \cite{Z1}. Unless otherwise specified, $R$ is a ring and $ F^R:W(R)\to W(R) $ is the Frobenius morphism of the ring of Witt vectors with coefficient in $R$, and $ V^R:W(R)\to W(R) $ its Verschiebung. We denote by $ I_R $ the image of the Verschiebung. We sometimes denote the Frobenius and Verschiebung without the superscript $R$, when it is clear what is meant.

\section{Recollections}

\begin{dfn}
\label{def02}
A $3n$-display over $R$ is a quadruple $\CP=(P,Q,F,V^{-1})$, where $P$ is a finitely generated $ W(R) $-module, $ Q\subseteq P $ is a submodule and $ F, V^{-1} $ are $ F^R$-linear morphisms $ F:P\to P $ and $ V^{-1}:Q\to P $, subject to the following axioms:
\begin{itemize}
\item[(i)] $ I_RP\subseteq Q\subseteq P $  and there is a decomposition of $ P $ into the direct sum of $ W(R) $-modules $ P=L\oplus T $, called \emph{a normal decomposition}, such that $ Q=L\oplus I_RT $.
\item[(ii)] $ V^{-1}:Q\to P $ is an $ F^R $-linear epimorphism (i.e., the $ W(R) $-linearization $ (V^{-1})^{\sharp}:W(R)\otimes_{F^R,W(R)} Q\to P $ is surjective).
\item[(iii)] For any $ x\in P $ and $ w\in W(R) $ we have \[ V^{-1}(V^R(w)x)=wF(x). \]
\end{itemize}
\end{dfn}

\begin{rem}
\label{rem0 15}
Note that from the last axiom, it follows that $ F $ is uniquely determined by $ V^{-1} $. Indeed, we have for every $ x\in P $: \[ F(x)=V^{-1}(V^R(1)x).\] It follows also from this relation and $ F^R $-linearity of $ V^{-1} $, that for every $ y\in Q $, we have \[ F(y)=V^{-1}(V^R(1)y)=F^RV^R(1)V^{-1}(y)=pV^{-1}(y). \]
\end{rem}

\begin{cons}
\label{cons015}
According to lemma 10, p.14 of \cite{Z1}, there exists a unique $ W(R) $-linear map $ V^{\sharp}:P\to W(R)\otimes_{F,W(R)}P $, satisfying the following equations:\[ V^{\sharp}(wF(x))=pw\otimes x,\quad w\in W(R), x\in P \] and \[ V^{\sharp}(wV^{-1}(y))=w\otimes y,\quad w\in W(R), y\in Q.\] If we denote by $ F^{\sharp}:W(R)\otimes_{F,W(R)}P\to P $ the $ W(R)$-linearization of $ F:P\to P $, we have the properties: 
\begin{myequation}
\label{Fr-Ver}
F^{\sharp}\circ V^{\sharp} =p.\Id_P \text{\quad and\quad}  V^{\sharp}\circ F^{\sharp}=p.\Id_{W(R)\otimes_{F,W(R)}P}.
\end{myequation}
Denote by $ V^{n\sharp} $ the composition \[ P\arrover{V^{\sharp}} W(R)\otimes_{F}P\arrover{\Id\otimes_{F}V^{\sharp}}W(R)\otimes_{F^2}P\to\dots\arrover{\Id\otimes_{F^n}V^{\sharp}}W(R)\otimes_{F^n}P. \]
\end{cons}

\begin{cons}
\label{cons016}
Let $ \CP=(P,Q,F,V^{-1}) $ be a $3n$-display over a ring $R$ and let $ \phi:R\to S $ be a ring homomorphism. We are going to construct a $3n$-display, which will be the $3n$-display obtained from $\CP$ by base change with respect to $\phi:R\to S $. Set $ \CP_S:=(P_S,Q_S,F_S,V_S^{-1}) $, where: 
\begin{itemize}
\item $ P_S $ is $ W(S)\otimes_{W(R)}P $,
\item $ Q_S $ is the kernel of the morphism $ W(S)\otimes_{W(R)}P\to S\otimes_R P/Q $,
\item $ F_S:P_S\to P_S $ is the morphism $ F^S\otimes F $ and
\item $ V_S^{-1}:Q_S\to P_S $ is the unique $ W(S) $-linear homomorphism, which satisfies the following properties:\[ V_S^{-1}(w\otimes y)=F^S(w)\otimes V^{-1}(y),\quad w\in W(S), y\in Q \] and \[ V_S^{-1}(V^S(w)\otimes x)=x\otimes F(x),\quad w\in W(S), x\in P.\]
\end{itemize}
If $ P=L\oplus T $ is a normal decomposition of $P$, then one can show that $ P_S=L_S\oplus T_S $ is a normal decomposition of $P_S$, where $ L_S:=W(S)\otimes_{W(R)}L $ and $ T_S=W(S)\otimes_{W(R)}T $ and that we have $Q_S=L_S\oplus I_S\otimes_{W(R)}T$ (note that $ I_ST_S=I_S\otimes_{W(R)}T $). For the details of this construction, in particular to see why this construction produces a $3n$-display, refer to Definition 20 and the discussions following it, p.20 of \cite{Z1}.
\end{cons}

\begin{dfn}
\label{def010}
Let $ \CP=(P,Q,F,V^{-1}) $ be a $3n$-display over $R$. Assume that $p$ is nilpotent in $R$. Then $ \CP $ is called \emph{display} if it satisfies the \emph{nilpotence or $V$-nilpotence condition}, i.e., if there exists a natural number $ N $ such that the morphism $$ V^{N\sharp}:P\to W(R)\otimes_{F^{N},W(R)}P $$ is zero modulo $ I_R+pW(R)$.
\end{dfn}

\begin{rem}
\label{rem020}
$  $
\begin{itemize}
\item[1)] If $p$ is not nilpotent in $R$, but is topologically nilpotent, one defines a display as follows, however, in the sequel, we will only work with displays over rings where $p$ is nilpotent. Let $R$ be a linearly topologized ring, with topology given by a filtration by ideals: \[ R=\Fa_0\supseteq \Fa_1\supseteq\dots\supseteq\Fa_n\supseteq\dots, \] such that $ \Fa_i\Fa_j\subseteq \Fa_{i+j} $. By assumption, $p$ is nilpotent in $ R/\Fa_1 $, and hence in every ring $ R/\Fa_i $. We also assume that $R$ is complete and separated with respect to this filtration. In particular it follows that $R$ is a $p$-adic ring. We call a $3n$-display $ \CP $ a \emph{display}, if the $3n$-display obtained from $ \CP $ by base change over $ R/\Fa_1 $ is a display.
\item[2)] Displays are sometimes called \emph{nilpotent displays}, whereas $3n$-displays are ``\textbf{n}ot \textbf{n}ecessarily \textbf{n}ilpotent". In order to emphasize that displays satisfy $V$-nilpotence condition, we will also sometimes stress the adjective ``nilpotent".
\end{itemize}
\end{rem}

For more details on the following construction, refer to Example 14, p.16 of \cite{Z1}.
\begin{cons}
\label{cons017}
Let $k$ be a perfect field of characteristic $p>0$. A Dieudonn\'e module over $k$ is an $ \BE_k $-module $ M $, which is finitely generated and free as $ W(k) $-module. It is therefore equipped with two operators $ F:M\to M$ and $ V:M\to M $, which are respectively $ F:W(k)\to W(k) $ and $F^{-1}:W(k)\to W(k)$ linear, such that $ FV=VF=p $. We obtain from $M$ a $3n$-display $ \CP=(P,Q,F, V^{-1})$, by setting $ P:=M, Q:=VM $ and $ F $ being $ F:M\to M $ and finally $ V^{-1}:VM\to M $ being the inverse of $ V:M\to VM $ (note that since $ FV=p $ and $ M $ is a free $ W(k) $-module, $ V $ is injective). Moreover, $ \CP $ is a display, if and only if the morphism $ V:M/pM\to M/pM$ is nilpotent. Thus, if $G$ is a $p$-divisible group over $k$, the Dieudonn\'e module, $ D_*(G) $, of $G$ gives rise to a $3n$-display. Since we are working with the covariant Dieudonn\'e theory, we observe that $G$ is connected, if and only if the corresponding $3n$-display is a (nilpotent) display.
\end{cons}

\begin{dfn}
\label{def04}
Let $R$ be a ring and $ \CP=(P,Q,F,V^{-1}) $ a $3n$-display over $R$. The \emph{tangent module} of $ \CP $, denoted by $ \CT(\CP) $, is the $R$-module $ P/Q $.
\end{dfn}

\begin{dfn}
\label{def012}
Let $R$ be a ring and $ \CP=(P,Q,F,V^{-1}) $ a $3n$-display over $R$. The \emph{rank} of $ \CP $, is the rank of its tangent module over $R$ and the \emph{height} of $ \CP $, is the rank of $P$ over $W(R)$.
\end{dfn}

\begin{rem}
\label{rem0 9}
$  $
\begin{itemize}
\item[1)] Using the normal decomposition $ P=L\oplus T $ and $ Q=L\oplus I_RT $, we observe that the tangent module of $ \CP $ is isomorphic to $ T/I_RT $, which is a projective $ R$-module and therefore the rank of the tangent module over $ R $ is equal to the rank of $ T $ over $ W(R)$. It follows that the height of $\CP$ is equal to the sum of ranks of $L$ and $T$ over $W(R)$.
\item[2)] If $R$ is a perfect field of characteristic $p>0$ and $ \CP $ is the display associated to a connected $p$-divisible group $G$, then the tangent module of $ \CP $, which is an $R$-vector space, is canonically isomorphic to the tangent space of the $p$-divisible group $G$. The rank and height of $ \CP $ are respectively equal to the dimension of $G$ and its height.
\end{itemize}
\end{rem}

The following construction is extracted without proofs from example 23, p. 22 of \cite{Z1}.

\begin{cons}
\label{cons02}
Let $\CP=(P,Q,F,V^{-1})$ be a $3n$-display over a ring $R$ with $p.R=0$. Denote by Frob$:R\to R$ the absolute Frobenius morphism of $R$, i.e., Frob$(r)=r^p $ for any $r\in R$. Denote by $ \CP^{(p)}=(P^{(p)},Q^{(p)},F,V^{-1}) $ the base change of $ \CP $ with respect to Frob. More precisely, we have $$ P^{(p)}=W(R)\otimes_{F,W(R)}P $$ and \[ Q^{(p)}=I_R\otimes_{F,W(R)}P+\image(W(R)\otimes_{F,W(R)}Q\to W(R)\otimes_{F,W(R)}P).\] The operators $ F, V^{-1} $ are uniquely determined by the relations:\[ F(w\otimes x)=F(w)\otimes F(x),\quad w\in W(R), x\in P,\] \[ V^{-1}(Vw\otimes x)=w\otimes F(x),\quad w\in W(R), x\in P\] and \[ V^{-1}(w\otimes y)=F(w)\otimes V^{-1}(y),\quad w\in W(R), y\in Q. \]
The map $ V^{\sharp} $ in Construction \ref{cons015}, satisfies $ V^{\sharp}(P)\subseteq Q^{(p)} $ and using the fact that $ P $ is generated as a $ W(R)$-module by elements $ V^{-1}(y) $ with $ y\in Q $, one shows that $ V^{\sharp} $ commutes with $ F $ and $ V^{-1} $ and therefore induces a homomorphism of $3n$-displays \[ Fr_{\CP}:\CP\to \CP^{(p)}.\] Similarly, $ F^{\sharp} $ satisfies $ F^{\sharp}(Q^{(p)})\subseteq I_RP $ and commutes with $ F $ and $ V^{-1} $ and thus induces a homomorphism of $3n$-displays  \[ Ver_{\CP}:\CP^{(p)}\to \CP.\] From the equations (\ref{Fr-Ver}), we obtain the equations \[ Fr_{\CP}\circ Ver_{\CP} =p.\Id_{\CP^{(p)}} \text{\quad and\quad}  Ver_{\CP}\circ Fr_{\CP}=p.\Id_{\CP}.\]
\end{cons}

For the next construction, we fix a $3n$-display $ \CP=(P,Q,F,V^{-1}) $ over a ring $R$, where $p$ is topologically nilpotent, with a fixed normal decomposition $ P=L\oplus T $, and a nilpotent $R$-algebra $\CN$. Set $ S:=R\oplus \CN$. This has a natural structure of an $R$-algebra. This construction is a recapitulation of some of the constructions and results in section 3 of \cite{Z1}.

\begin{cons}
\label{cons012}
Set $ \widehat{P}({\CN}):=\widehat{W}(\CN)\otimes_{W(R)}P $ and $ \widehat{Q}(\CN):=\widehat{P}_{\CN}\cap Q_S$, where $ Q_S $ is the base change of $Q$ (as in $3n$-display), i.e., $$Q_S=\kernel(W(S)\otimes_{W(R)}P\to S\otimes_RP/Q).$$ Then, one sees that $$\widehat{Q}(\CN)=\widehat{W}(\CN)\otimes_{W(R)}L\oplus \widehat{I}_{\CN}\otimes_{W(R)}T .$$ We extend the maps $ F:P\to P $ and $ V^{-1}:Q\to P $ to maps $ F:\widehat{P}(\CN)\to \widehat{P}(\CN) $ and $ V^{-1}:\widehat{Q}(\CN)\to \widehat{P}(\CN) $ as follows. We set $ F:=F^{\CN}\otimes F $, where $ F^{\CN}:\widehat{W}(\CN)\to \widehat{W}(\CN) $ is the Frobenius morphism. We let $ V^{-1} $ act on $ \widehat{W}(\CN)\otimes_{W(R)}L $ as $ F\otimes V^{-1} $ and on $\widehat{I}_{\CN}\otimes_{W(R)}T $ as $ V^{-1}\otimes F $ (since the action of $ V $ on the Witt vectors is injective, the map $ V^{-1}:\widehat{I}_{\CN}\to \widehat{W}(\CN) $ is well-defined). If we want to look at $ \widehat{P}_{\CN} $ and $\widehat{Q}_{\CN}$ as functors on Nil$_R$, we denote $ \widehat{P}_{\CN} $ by $ G_{\CP}^{0}(\CN) $ and $\widehat{Q}_{\CN}$ by $ G_{\CP}^{-1}(\CN) $.\\ 

Denote by $ BT_{\mathcal{P}}(\CN) $ the cokernel of the map $ V^{-1}-\Id: G_{\CP}^{-1}(\CN)\to G_{\CP}^{0}(\CN) $. By theorem 81, p. 77 of \cite{Z1}, the following sequence is exact: \[ 0\longrightarrow G_{\CP}^{-1}(\CN)\arrover{ V^{-1}-\Id} G_{\CP}^{0}(\CN)\longrightarrow BT_{\mathcal{P}}(\CN)\longrightarrow 0   \] and the functor $ BT_{\mathcal{P}}:\text{Nil}_R\to \Ab $ is a finite dimensional formal group. Moreover, if $\CP$ is a display (nilpotent), then $ BT_{\mathcal{P}} $ is a $p$-divisible group (corollary 89, p.83 of \cite{Z1}). By corollary 86, p. 81 of \cite{Z1}, the construction of $ BT_{\mathcal{P}} $ commutes with base change of $3n$-displays. Now, assume that $ pR=0 $. By functoriality, the Frobenius map $ Fr_{\CP}:\CP\to\CP^{(p)} $ gives rise to a map $ BT_{\mathcal{P}}(Fr_{\CP}):BT_{\mathcal{P}}\to BT_{\mathcal{P}^{(p)}}\cong BT_{\mathcal{P}}^{(p)} $, where by $ BT_{\mathcal{P}}^{(p)} $ we mean the base change of the formal group $ BT_{\mathcal{P}} $ with respect to the extension Frob$:R\to R $. Proposition 87, p.81 of \cite{Z1} states that this homomorphism is the Frobenius morphism of the formal group $BT_{\mathcal{P}}$. Similarly, the induced morphism $ BT_{\mathcal{P}}(Ver_{\CP}): BT_{\mathcal{P}}^{(p)}\cong BT_{\mathcal{P}^{(p)}}\to BT_{\mathcal{P}}$ is the Verschiebung of $ BT_{\mathcal{P}} $.
\end{cons}

For future reference, we summarize these results in the following proposition.

\begin{prop}
\label{prop0 20}
Let $ \CP $ be a $3n$-display over a ring $R$, then:
\begin{itemize}
\item $ BT_{\CP} $ is a finite dimensional formal group and the construction $ \CP\rightsquigarrow BT_{\CP} $ commutes with base change, i.e., if $ R\to S $ is a ring homomorphism, then there exists an canonical isomorphism $ (BT_{\CP})_S\cong BT_{\CP_S} $.
\item If $p$ is nilpotent in $R$ and $ \CP $ is nilpotent, then $ BT_{\CP} $ is an infinitesimal $p$-divisible group.
\item If $pR=0$, and $ \CP $ is nilpotent, then the Frobenius and respectively Verschiebung morphisms of the $p$-divisible group $ BT_{\CP} $ are $ BT_{\CP}(Fr_{\CP}) $ and respectively $ BT_{\CP}(Ver_{\CP}) $.
\end{itemize}
\end{prop}

\begin{notation}
\label{notation0 2}
For a nilpotent $R$-algebra $\CN$, we denote by $[b]$ the class of an element $b\in G_{\CP}^0(\CN)$ modulo $ (V^{-1}-\Id)G_{\CP}^{-1}(\CN) $. If $[b]$ is annihilated by $p^n$, we write $[b]_n$ to emphasize the fact that this element belongs to the kernel of $p^n$. In this case, $ p^nb $  belongs to the subgroup $ (V^{-1}-\Id)G_{\CP}^{-1}(\CN) $ of $ G_{\CP}^0(\CN) $, and therefore, since $ V^{-1}-\Id:G_{\CP}^{-1}(\CN)\to G_{\CP}^{0}(\CN)$ is injective, there exists a unique element $ \leftidx{_n}{g}{_{\CP}}(b)\in G_{\CP}^{-1}(\CN) $ with $ (V^{-1}-\Id)(\leftidx{_n}{g}{_{\CP}}(b))=p^nb $.
\end{notation}

\begin{rem}
\label{rem0 7}
It follows from the construction of $ BT_{\CP} $ that for any nilpotent $R$-algebra $ \CN $, any $ w\in \widehat{W}(\CN) $ and any $ x\in P $, we have $ [Fw\otimes x]=[w\otimes Vx] $ and $ [Vw\otimes x]=[w\otimes Fx] $. Indeed, by Construction \ref{cons012}, we know that $ (V^{-1}-\Id)(w\otimes Vx)=Fw\otimes x-w\otimes Vx$ and that $ (V^{-1}-\Id)(Vw\otimes x)=w\otimes Fx-Vw\otimes x$.
\end{rem}

We will need theorem 103, p.94 of \cite{Z1}, which states:

\begin{thm}
\label{thm02}
Let $R$ be an excellent local ring or a ring such that $R/pR$ is of finite type over a field $k$. The functor $\mathcal{P}\mapsto BT_{\mathcal{P}}$ gives an equivalence of categories between the category of (nilpotent) displays over $R$ and infinitesimal $p$-divisible groups over $R$.
\end{thm}

\section{\texorpdfstring{Multilinear morphisms and the map $ \beta $}{Multilinear morphisms and the map beta}}

\begin{dfn}
\label{def05}
Let $ \CP_0,\CP_1,\dots,\CP_r $ and $ \CP $ be $3n$-displays over a ring $R$. 
\begin{itemize}
\item[(i)] A \emph{multilinear} morphism $ \phi:\CP_1\times\dots\times\CP_r\to \CP_0 $ is a $ W(R) $-linear morphism $ \phi:P_1\times\dots\times P_r\to P_0 $ satisfying the following conditions:
\begin{itemize}
\item $ \phi $ restricts to a $ W(R) $-multilinear morphism $ \phi:Q_1\times \dots\times Q_r\to Q_0 $.
\item For any $ q_i\in Q_i $: \[V^{-1}\phi(q_1,\dots, q_r)= \phi(V^{-1}q_1,\dots, V^{-1}q_r).\]
\item For any $ 1\leq i\leq r $, $ x_i\in P_i $ and $ q_j\in Q_j $ ($ j\neq i $):\[ F\phi(q_1,\dots,q_{i-1},x_i,q_{i+1},\dots,q_r)=\]\[\phi(V^{-1}q_1,\dots, V^{-1}q_{i-1},Fx_i,V^{-1}q_{i+1},\dots,V^{-1}q_r).\] The group of multilinear morphisms $ \CP_1\times\dots\times \CP_r\to \CP_0 $ is denoted by $ \Mult(\CP_1\times\dots\times \CP_r,\CP_0) $.
\end{itemize}
\item[(ii)] A \emph{symmetric} multilinear morphism $ \phi:\CP^r\to \CP_0 $, is a multilinear morphism such that the underlying $ W(R) $-multilinear morphism $\phi:P^r\to P_0 $ is symmetric. The group of symmetric morphisms $ \CP^r\to \CP_0 $ is denoted by $ \Sym(\CP^r,\CP_0) $.
\item[(iii)] An \emph{antisymmetric} multilinear morphism $ \phi:\CP^r\to \CP_0 $, is a multilinear morphism such that the underlying $ W(R) $-multilinear morphism $\phi:P^r\to P_0 $ is antisymmetric. The group of antisymmetric morphisms $ \CP^r\to \CP_0 $ is denoted by $ \Anti(\CP^r,\CP_0) $.
\item[(iv)] An \emph{alternating} multilinear morphism $ \phi:\CP^r\to \CP_0 $, is a multilinear morphism such that the underlying $ W(R) $-multilinear morphism $\phi:P^r\to P_0 $ is alternating. The group of alternating morphisms $ \CP^r\to \CP_0 $ is denoted by $ \Alt(\CP^r,\CP_0) $.
\end{itemize}
\end{dfn}

\begin{rem}
\label{rem0 16}
$  $
\begin{itemize}
\item[1)] We call the second and respectively the third property of a multilinear morphism \emph{the $V$-condition} and respectively \emph{the $F$-condition}.
\item[2)] Note that the $F$-condition of a multilinear morphism follows from $ W(R)$-multilinearity and the the first property and the $V$-condition. Indeed, by Remark \ref{rem0 15}, we have \[ F\phi(q_1,\dots,q_{i-1},x_i,q_{i+1},\dots,q_r)\ovset{\ref{rem0 15}}{=}\]\[V^{-1}\big(V(1)\phi(q_1,\dots,q_{i-1},x_i,q_{i+1},\dots,q_r)\big)= \] \[ V^{-1}\phi(q_1,\dots,q_{i-1},V(1)x_i,q_{i+1},\dots,q_r)= \] \[\phi(V^{-1}q_1,\dots, V^{-1}q_{i-1},V^{-1}(V(1)x_i),V^{-1}q_{i+1},\dots,V^{-1}q_r)\ovset{\ref{rem0 15}}{=}\]\[\phi(V^{-1}q_1,\dots, V^{-1}q_{i-1},Fx_i,V^{-1}q_{i+1},\dots,V^{-1}q_r).\]
\end{itemize}
\end{rem}

\begin{cons}
\label{cons013}
Let $\mathcal{P}_1,\dots, \mathcal{P}_r,\mathcal{P}_0$ be $3n$-displays over a ring $R$ and $\phi:\mathcal{P}_1\times\dots\times\mathcal{P}_r\to\mathcal{P}_0$ a multilinear morphism of $3n$-displays. 
\begin{itemize}
\item Let $ R\to S $ be a ring homomorphism. We extend $ \phi $ to a multilinear morphism $\phi_S: \CP_{1,S}\times\dots\times\CP_{r,S}\to \CP_{0,S} $ as follows. For all $ w_i\in W(S) $ and all $ x_i\in P_i $ ($ i=1,\dots,r $), we set \[ \phi_S(w_1\otimes x_1,\dots, w_r\otimes x_r):=w_1\dots w_r\otimes\phi(x_1,\dots,x_r)\] and extend $W(S)$-multilinearly to the whole product $ \mathcal{P}_{1,S}\times\dots\times\mathcal{P}_{r,S} $.
\item Given a nilpotent $R$-algebra $ \CN $  we extend  $ \phi $ to a $\widehat{W}(\CN)$-multilinear morphism $\widehat{\phi}:\widehat{P}_1\times\dots\times\widehat{P}_r\to\widehat{P}_0$ as follows. For all $ \omega_i \in \widehat{W}(\CN) $ and all $ x_i\in P_i $ with $ i=1,\dots,r $, we set: $$\widehat{\phi}(\omega_1\otimes x_1,\dots,\omega_r\otimes x_r):=\omega_1\dots\omega_r\otimes\phi(x_1,\dots,x_r) .$$ Now, we extend $ \widehat{\phi} $ multilinearly to the whole product $ \widehat{P}_1\times\dots\times\widehat{P}_r $.
\end{itemize}
\end{cons}

\begin{lem}
\label{lem0 18}
The multilinear morphisms $ \phi_S $ and $ \widehat{\phi} $ constructed above satisfy the $V$-$F$ conditions.
\end{lem}

\begin{proof}
The proof of the lemma for the two multilinear morphisms are similar and thus, we only prove the lemma for the multilinear morphism $ \widehat{\phi} $. Let $ \CP=(P,Q,F,V^{-1}) $ be a $3n$-display over $R$. Take elements $ w\in \widehat{W}(\CN) $ and $ x\in P $. By construction of $ F $ and $ V $ on $ \widehat{P}_{\CN} $ (cf. Construction \ref{cons012}), we have $$ F(w\otimes x)=F(w)\otimes F(x)= F(w)\otimes V^{-1}(V(1)\cdot x)=$$ $$V^{-1}(w\otimes V(1)\cdot x)=V^{-1}(V(1)w\otimes x).$$ Thus, by the same arguments as in Remark \ref{rem0 16}, it is enough to show that $ \widehat{\phi} $ satisfies the $V$-condition. So, for each $ i\in \lbb1,r\rbb $ take an element $ \hat{q}_i $ in $ \widehat{Q}_{i,\CN} $. As we already know that $ \widehat{\phi} $ is multilinear (by its construction), we can assume that either $ \hat{q_i}\in \widehat{W}(\CN)\otimes L_i$ or $ \hat{q_i}\in \widehat{I}_{\CN}\otimes T_i $, where for each $i$, we have fixed a normal decomposition $ P_i=L_i\oplus T_i $, and that each of $ \hat{q_i} $ is a pure tensor (i.e., of the form $ x\otimes y$). Since the construction of $ \widehat{\phi} $ is symmetric with respect to $i$ and for the sake of simplicity, we can assume that $ \hat{q_j}=w_j\otimes q_j\in \widehat{W}(\CN)\otimes L_j $ for $ 1\leq j\leq s $ and $ \hat{q_j}=V(w_j)\otimes t_j $ for $ s+1\leq j\leq r  $ for some $0\leq s\leq r $. We divide the problem into two cases: when $ s<r $ and when $ s=r $. In the first case, we have: \[ \widehat{\phi}(\hat{q_1},\dots,\hat{q_r})=\widehat{\phi}(w_1\otimes q_1\dots,w_s\otimes q_s,V(w_{s+1})\otimes t_{s+1},\dots,V(w_{r})\otimes t_{r})=\] \[w_1\dots w_sV(w_{s+1})\dots V(w_r)\otimes \phi(q_1,\dots,q_s,t_{s+1},\dots,t_r)=\] \[V\big(F\big(w_1\dots w_s V(w_{s+1})\dots V(w_{r-1})\big)\cdot w_r\big)\otimes \phi(q_1,\dots,q_s,t_{s+1},\dots,t_r).\] The element $ V\big(F\big(w_1\dots w_s V(w_{s+1})\dots V(w_{r-1})\big)\cdot w_r\big) $ being in the ideal $ \widehat{I}_{\CN} $, it follows that $\widehat{\phi}(\hat{q_1},\dots,\hat{q_r})\in \widehat{Q}_{0,\CN} $. In the second case, we have: \[ \widehat{\phi}(\hat{q_1},\dots,\hat{q_r})=\widehat{\phi}(w_1\otimes q_1\dots,w_r\otimes q_r)= w_1\dots w_r\otimes\phi(q_1,\dots,q_r).\] As by assumption $ \phi(q_1,\dots,q_r)\in Q_0$, we conclude again that $\widehat{\phi}(\hat{q_1},\dots,\hat{q_r})\in \widehat{Q}_{0,\CN} $ (note that $ \widehat{Q}_{0,\CN} $ contains elements coming from $ \widehat{W}(\CN)\otimes Q_0 $).
\end{proof}

\begin{cons}
\label{cons05}
Given displays $\mathcal{P}_1,\dots, \mathcal{P}_r,\mathcal{P}_0$ over a ring $R$, with a $V$-$F$ multilinear morphism \[\phi:\mathcal{P}_1\times\dots\times\mathcal{P}_r\to\mathcal{P}_0,\] we construct for any natural numbers $n$, a map \[\beta_{\phi,n}:BT_{\mathcal{P}_1,n}\times\dots\times BT_{\mathcal{P}_r,n}\to BT_{\mathcal{P}_0,n},\] where $BT_{\mathcal{P}_i,n}$ is the kernel of multiplication by $p^n$ on the $p$-divisible group $BT_{\mathcal{P}_i}$.
Take a nilpotent $R$-algebra $ \CN $ and elements $ [x_i]_n\in BT_{\mathcal{P}_i,n}(\CN) $ and set \[\beta_{\phi,n}([x_1]_n,\dots,[x_r]_n):= (-1)^{r-1}\sum_{i=1}^{r}\big[\widehat{\phi}\big(V^{-1}g_1,\dots,V^{-1}g_{i-1},x_i,g_{i+1},\dots,g_{r}\big)\big], \]where for all $ 1\leq j\leq r $, we have abbreviated $ \leftidx{_n}{g}{_{\CP_{j}}}(x_{j}) $ (from Notations \ref{notation0 2}) to $ g_j $. We show in the next lemma that this is a well-defined multilinear morphism.
\end{cons}

\begin{rem}
\label{rem022}
Note that if $r=1$, then $ \beta_{\phi,n}:BT_{\CP_1,n}\to BT_{\CP_0,n} $ is the restriction of the homomorphism $ BT_{\phi}:BT_{\CP_1}\to BT_{\CP_0} $ (using the functoriality of $ BT $) to $ BT_{\CP_1,n} $.
\end{rem}

\begin{prop}
\label{prop0 16}
The maps $\beta_{\phi,n}:BT_{\mathcal{P}_1,n}\times\dots\times BT_{\mathcal{P}_r,n}\to BT_{\mathcal{P}_0,n}$ satisfy the following properties:
\begin{itemize}
 \item[(i)] $\beta_{\phi,n}$ are well-defined multilinear morphisms.
 \item[(ii)] $\beta_{\phi,n}$ are compatible with respect to projections $p.:BT_{\mathcal{P}_i,n+1}\twoheadrightarrow BT_{\mathcal{P}_i,n}$.
 \item[(iii)] If the $3n$-displays $ \CP_1\dots,\CP_r $ are equal, then if $\phi$ is symmetric, antisymmetric or alternating, then $\beta_{\phi,n}$ has the same property. 
 \item[(iv)] The construction of $ \beta_{\phi,n} $ commutes with base change, i.e., if $ R\to S $ is a ring homomorphism, then $ (\beta_{\phi,n})_S $ and $ \beta_{\phi_S,n} $ are equal as multilinear morphisms $ BT_{\mathcal{P}_{1_S},n}\times\dots\times BT_{\mathcal{P}_{r_S},n}\to BT_{\mathcal{P}_{0_S},n} $, using the identification $ BT_{\CP_{i_S}}\cong (BT_{\CP_i})_S $.
 \end{itemize}
\end{prop}

\begin{proof}
We fix a nilpotent $R$-algebra $\CN$.
\begin{itemize}
 \item[(i)] For each $ 1\leq i\leq r $, take elements $ [x_i]_n\in BT_{\mathcal{P}_i,n}(\CN) $. If we show that the element $ \beta_{\phi,n}([x_1]_n,\dots,[x_r]_n) $ does not depend on the representatives $ x_i $ of the class $ [x_i] $ and that the map $ \beta_{\phi,n} $ is multilinear, then it follows that $ \beta_{\phi,n}([x_1]_n,\dots,[x_r]_n) $ lies in the kernel of multiplication by $ p^n $ (note that $ p^n[x_i]=0 $). By multilinearity of $ \widehat{\phi} $, in order to show the independence of $ \beta_{\phi,n}([x_1]_n,\dots,[x_r]_n) $ from the choices of the elements $ x_i $, it is sufficient to show that if one $ x_j $ is in the subgroup $ (V^{-1}-\Id)G_{\CP_j}^{-1} $ of $ G_{\CP_j}^0 $, then the element $ \beta_{\phi,n}([x_1]_n,\dots,[x_r]_n) $ is in the subgroup $ (V^{-1}-\Id)G_{\CP_0}^{-1} $ of $ G_{\CP_0}^0 $. So, assume that $ x_j=(V^{-1}-\Id)(z_j) $ for some $ z_j\in G_{\CP_j}^{-1} $ and for every $i$, set $ g_i:=\leftidx{_n}{g}{_{\CP_j}}(x_i) $. We then have \[(V^{-1}-\Id)(g_j)=p^nx_j=p^n(V^{-1}-\Id)(z_j)=(V^{-1}-\Id)(p^nz_j) \] and since $ V^{-1}-\Id $ is injective, it implies that $ g_j=p^nz_j $. Set also:
 \begin{itemize}
 \item[$\bullet$] $ A:=(-1)^{r-1}\sum_{i=1}^{j-1}\widehat{\phi}(V^{-1}g_1,\dots,V^{-1}g_{i-1},x_i,g_{i+1},\dots,g_{r}) $,
 \item[$\bullet$] $ B:= (-1)^{r-1}\widehat{\phi}(V^{-1}g_1,\dots,V^{-1}g_{j-1},x_j,g_{j+1},\dots,g_{r})$ and
 \item[$\bullet$] $ C:=(-1)^{r-1}\sum_{i=j+1}^{r}\widehat{\phi}(V^{-1}g_1,\dots,V^{-1}g_{i-1},x_i,g_{i+1},\dots,g_{r})$.
 \end{itemize}
 We then have $ \beta_{\phi,n}([x_1]_n,\dots,[x_r]_n)=[A+B+C] $. If we develop each of the terms separately, by replacing $ x_j $ and respectively $ g_j $  with $ (V^{-1}-\Id)(z_j) $ and $ p^nz_j $, we obtain:
 \begin{itemize}
 \item[$\bullet$] $ A=(-1)^{r-1}\sum_{i=1}^{j-1}\widehat{\phi}(V^{-1}g_1,\dots,V^{-1}g_{i-1},x_i,g_{i+1},\dots,\uset{\uparrow}{p^nz_j},\dots,g_{r}) $,\\ where the vertical arrow below  $ p^nz_j $ is to emphasize that only the term at the $ j^{\text{th}} $ place does not follow the pattern of the sequence $ g_{i+1},\dots,g_r $. We will use this convention for the other sums too, in order to avoid heavy notations. By multilinearity, we can pass the coefficient $ p^n $ of $ z_j $ to $ x_i $ and get \[ (-1)^{r-1}\sum_{i=1}^{j-1}\widehat{\phi}(V^{-1}g_1,\dots,V^{-1}g_{i-1},p^nx_i,g_{i+1},\dots,\uset{\uparrow}{z_j},\dots,g_{r}).\] Now, $ p^nx_i $ is equal to $ (V^{-1}-\Id)(g_i) $ and therefore $A$ becomes: \[ (-1)^{r-1}\sum_{i=1}^{j-1}\widehat{\phi}(V^{-1}g_1,\dots,V^{-1}g_{i-1},(V^{-1}-\Id)g_i,g_{i+1},\dots,\uset{\uparrow}{z_j},\dots,g_{r})\] \[=(-1)^{r-1}\sum_{i=1}^{j-1}\widehat{\phi}(V^{-1}g_1,\dots,V^{-1}g_{i-1},V^{-1}g_i,g_{i+1},\dots,\uset{\uparrow}{z_j},\dots,g_{r})- \] \[ (-1)^{r-1}\sum_{i=1}^{j-1}\widehat{\phi}(V^{-1}g_1,\dots,V^{-1}g_{i-1},g_i,g_{i+1},\dots,\uset{\uparrow}{z_j},\dots,g_{r}).\] All but one term in the two sums cancel out (which becomes clear with an index shift in the first sum) and hence, we obtain \[A= (-1)^{r-1}\widehat{\phi}(V^{-1}g_1,\dots,V^{-1}g_{j-1},z_j,g_{j+1},\dots,g_{r})-\]  
 \begin{myequation}
 \label{a}
  (-1)^{r-1}\widehat{\phi}(g_1,\dots,g_{j-1},z_j,g_{j+1},\dots,g_{r}).
 \end{myequation}
 \item[$\bullet$] $ B= (-1)^{r-1}\widehat{\phi}(V^{-1}g_1,\dots,V^{-1}g_{j-1},(V^{-1}-\Id)z_j,g_{j+1},\dots,g_{r})=$ \[ (-1)^{r-1}\widehat{\phi}(V^{-1}g_1,\dots,V^{-1}g_{j-1},V^{-1}z_j,g_{j+1},\dots,g_{r})-\]
 \begin{myequation}
 \label{b}
 (-1)^{r-1}\widehat{\phi}(V^{-1}g_1,\dots,V^{-1}g_{j-1},z_j,g_{j+1},\dots,g_{r}).
 \end{myequation}
 \item[$\bullet$] Finally, performing the same calculations and using the similar arguments as for $A$, we obtain \[C=(-1)^{r-1}\widehat{\phi}(V^{-1}g_1,\dots,V^{-1}g_{j-1},V^{-1}z_j,V^{-1}g_{j+1},\dots,V^{-1}g_{r})-\]
 \begin{myequation}
 \label{c}
 (-1)^{r-1}\widehat{\phi}(V^{-1}g_1,\dots,V^{-1}g_{j-1},V^{-1}z_j,g_{j+1},\dots,g_{r}).
 \end{myequation}
 \end{itemize}
 Now, adding up $A$, $B$ and $C$ and use equations \eqref{a}, \eqref{b} and \eqref{c}, we observe that four terms of the six terms cancel out and we obtain \[  \beta_{\phi,n}([x_1]_n,\dots,[x_r]_n)=[A+B+C]=\]\[[(-1)^{r-1}\widehat{\phi}(V^{-1}g_1,\dots,V^{-1}g_{j-1},V^{-1}z_j,V^{-1}g_{j+1},\dots,V^{-1}g_{r})-\]\[ (-1)^{r-1}\widehat{\phi}(g_1,\dots,g_{j-1},z_j,g_{j+1},\dots,g_{r})]=\] \[ [(V^{-1}-\Id)\big((-1)^{r-1}\widehat{\phi}(g_1,\dots,g_{j-1},z_j,g_{j+1},\dots,g_{r})\big)].\] Since the vector $ (g_1,\dots,g_{j-1},z_j,g_{j+1},\dots,g_{r}) $ belongs to $ \widehat{Q}_{1,\CN}\times \dots\times \widehat{Q}_{r,\CN} $ and therefore by Lemma \ref{lem0 18}, $ \widehat{\phi}(g_1,\dots,g_{j-1},z_j,g_{j+1},\dots,g_{r})\big) $ belongs to $ \widehat{Q}_{0,\CN}=G_{\CP_0}^{-1} $, we conclude that $ \beta_{\phi,n}([x_1]_n,\dots,[x_r]_n) $ is the zero element of the quotient $ BT_{\CP_0} $. This proves the independence from the choices of representatives.\\
 
 It remains to prove the multilinearity. Since the map $ V^{-1}-\Id $ is a homomorphism and is injective, and  $ \widehat{\phi} $ is multilinear, a straightforward calculation shows that $ \beta_{\phi,n} $ is multilinear too. This proves part $ (i) $.
 
\item[(ii)] Take elements $ [x_i]_{n+1}\in BT_{\mathcal{P}_i,n}(\CN) $. If we set $ g_i:=\leftidx{_{n+1}}{g}{_{\CP_i}}(x_{i+1}) $, we have \[ p^n(p^nx_i)=p^{n+1}x_i=(V^{-1}-\Id)g_i \] and therefore $ \leftidx{_{n}}{g}{_{\CP_i}}(x_i)=g_i $. Thus, we have $\beta_{\phi,n}([px_1]_n,\dots,[px_r]_n)=$ \[ (-1)^{r-1}\sum_{i=1}^r\big[\widehat{\phi}(V^{-1}g_1,\dots,V^{-1}g_{i-1},px_i,g_{i+1},\dots,g_r)\big]= \] \[ p(-1)^{r-1}\sum_{i=1}^r\big[\widehat{\phi}(V^{-1}g_1,\dots,V^{-1}g_{i-1},x_i,g_{i+1},\dots,g_r)\big]=\]\[ p\beta_{\phi,n+1}([x_1]_{n+1},\dots,[x_r]_{n+1}) \] where we have used the multilinearity of $ \widehat{\phi} $ for the second equality. This proves part $ (ii) $.

\item[(iii)] Denote by $ \CP $ the equal $3n$-displays $ \CP_1,\dots,\CP_r $. Let $ \sigma\in S_n $ be a permutation of $n$ elements and define a new map $ \psi:\CP^r\to \CP_0 $ by setting $$ \psi(a_1,\dots, a_r):=\epsilon\cdot \sgn(\sigma)\phi(a_{\sigma(1)},\dots, a_{\sigma(r)}), $$ where $ \epsilon\in \{1,-1\} $ is a fixed sign. Since $ \phi $ is a multilinear morphism of $3n$-displays, it follows from the definition that the new map $\psi$ is also a multilinear morphism of $3n$-displays (i.e., multilinear satisfying the $V$-$F$ conditions). We claim that for any natural number $n$, any $ 1\leq i\leq r $, any $ [x_i]\in BT_{\CP,n}(\CN) $ and any permutation $\sigma\in S_n  $ we have \[ \beta_{\psi,n}([x_1],\dots,[x_r])=\epsilon\cdot \sgn(\sigma)\beta_{\phi,n}([x_{\sigma(1)}],\dots,[x_{\sigma(r)}]).\] If we have this result, it follows at once that if $ \phi $ is symmetric (respectively antisymmetric), then $ \beta_{\phi,n} $ is symmetric (respectively antisymmetric). We prove the claim and then show the statement about the alternating morphism. In order to prove the claim, it suffices to assume that $ \sigma $ is a transposition of the form $ (t,t+1) $ with $ t\in\lbb1,r-1\rbb $ (because they generate the group $S_n$). Again, we set $ g_i:=\leftidx{_n}{g}{_{\CP}}(x_i) $. We then have \[\beta_{\psi,n}([x_1],\dots,[x_r])=(-1)^{r-1}\sum_{i=1}^r\big[\widehat{\psi}(V^{-1}g_1,\dots,V^{-1}g_{i-1},x_i,g_{i+1},\dots,g_r)\big]\] \[=(-1)^{r-1}\sum_{i=1}^{t-1}\big[\widehat{\psi}(V^{-1}g_1,\dots,V^{-1}g_{i-1},x_i,g_{i+1},\dots,g_r)+\] \[ (-1)^{r-1}\widehat{\psi}(V^{-1}g_1,\dots,V^{-1}g_{t-1},x_t,g_{t+1},\dots,g_r)+ \] \[  (-1)^{r-1}\widehat{\psi}(V^{-1}g_1,\dots,V^{-1}g_{t},x_{t+1},g_{t+2},\dots,g_r)+ \]  \[ (-1)^{r-1}\sum_{i=t+2}^{r}\widehat{\psi}(V^{-1}g_1,\dots,V^{-1}g_{i-1},x_i,g_{i+1},\dots,g_r)\big]= \] \[-\epsilon(-1)^{r-1}\bigg(\sum_{i=1}^{t-1}\big[\widehat{\phi}(V^{-1}g_1,\dots,V^{-1}g_{i-1},x_i,g_{i+1},\dots\]\[\dots,g_{t-1},g_{t+1},g_t,g_{t+2},g_{t+3},\dots,g_r)+\] \[ \widehat{\phi}(V^{-1}g_1,\dots,V^{-1}g_{t-1},g_{t+1},x_t,g_{t+2},g_{t+3},\dots,g_r)+\] \[\widehat{\phi}(V^{-1}g_1,\dots,V^{-1}g_{t-1},x_{t+1},V^{-1}g_t,g_{t+2},g_{t+3},\dots,g_r)+\]  \[\sum_{i=t+2}^{r}\widehat{\phi}(V^{-1}g_1,\dots,V^{-1}g_{t-1},V^{-1}g_{t+1},V^{-1}g_{t},V^{-1}g_{t+2},V^{-1}g_{t+3},\dots \]\[\dots,V^{-1}g_{i-1},x_i,g_{i+1},\dots,g_r)\big]_n\bigg).\] Now, we calculate $$ -\epsilon\beta_{\phi,n}([x_1],\dots,[x_{t-1}],[x_{t+1}],[x_t],[x_{t+2}],[x_{t+3}],\dots,[x_r]).$$ This is equal to  \[-\epsilon(-1)^{r-1}\bigg(\sum_{i=1}^{t-1}\big[\widehat{\phi}(V^{-1}g_1,\dots,V^{-1}g_{i-1},x_i,g_{i+1},\dots\]\[\dots,g_{t-1},g_{t+1},g_t,g_{t+2},g_{t+3},\dots,g_r)+\] \[\widehat{\phi}(V^{-1}g_1,\dots,V^{-1}g_{t-1},x_{t+1},g_t,g_{t+2},g_{t+3},\dots,g_r)\] \[\widehat{\phi}(V^{-1}g_1,\dots,V^{-1}g_{t-1},V^{-1}g_{t+1},x_t,g_{t+2},g_{t+3},\dots,g_r)+\] \[\sum_{i=t+2}^{r}\widehat{\phi}(V^{-1}g_1,\dots,V^{-1}g_{t-1},V^{-1}g_{t+1},V^{-1}g_{t},V^{-1}g_{t+2},V^{-1}g_{t+3},\dots \]\[\dots,V^{-1}g_{i-1},x_i,g_{i+1},\dots,g_r)\big]_n\bigg).\] Thus, the difference \[\beta_{\psi,n}([x_1],\dots,[x_r])-\epsilon\sgn(\sigma)\beta_{\phi,n}([x_{\sigma(1)}],\dots,[x_{\sigma(r)}])\] is equal to the following (by using the multilinearity of $ \widehat{\phi} $ and the formulae $ (V^{-1}-\Id)g_t=p^nx_t $ and $ (V^{-1}-\Id)g_{t+1}=p^nx_{t+1} $) \[\epsilon(-1)^{r-1}\big[\widehat{\phi}(V^{-1}g_1,\dots,V^{-1}g_{t-1},V^{-1}g_{t+1}-g_{t+1},x_t,g_{t+2},g_{t+3},\dots,g_r)-\] \[\widehat{\phi}(V^{-1}g_1,\dots,V^{-1}g_{t-1},x_{t+1},V^{-1}g_t-g_t,g_{t+2},g_{t+3},\dots,g_r)\big]_n=\] \[ \epsilon(-1)^{r-1}\big[\widehat{\phi}(V^{-1}g_1,\dots,V^{-1}g_{t-1},p^nx_{t+1},x_t,g_{t+2},g_{t+3},\dots,g_r)-\] \[\widehat{\phi}(V^{-1}g_1,\dots,V^{-1}g_{t-1},x_{t+1},p^nx_t,g_{t+2},g_{t+3},\dots,g_r)\big]_n=\]\[ \epsilon(-1)^{r-1}\big[p^n\widehat{\phi}(V^{-1}g_1,\dots,V^{-1}g_{t-1},x_{t+1},x_t,g_{t+2},g_{t+3},\dots,g_r)-\] \[p^n\widehat{\phi}(V^{-1}g_1,\dots,V^{-1}g_{t-1},x_{t+1},x_t,g_{t+2},g_{t+3},\dots,g_r)\big]_n=0.\]

Now, assume that $ \phi $ is alternating. It is therefore also antisymmetric. We have to show that if two components of the vector $ \vec{[x]}:=([x_1],\dots,[x_r])\in BT_{\CP,n}(\CN)^r $ are equal, then $ \beta_{\phi,n}(\vec{[x]})=0 $. Since by the first part, we know that $ \beta_{\phi,n} $ is antisymmetric, without loss of generality, we can assume that the first two components of $ \vec{[x]} $ are equal. Note also that $ \phi $ being alternating, the extended multilinear morphism $ \widehat{\phi} $ is alternating as well. We have \[\beta_{\phi,n}([x_1],[x_1],[x_3],\dots,[x_r])=(-1)^{r-1}\big[\widehat{\phi}(x_1,g_1,g_3,g_4,\dots,g_r)+\] \[\widehat{\phi}(V^{-1}g_1,x_1,g_3,g_4,\dots,g_r)+\]\[\sum_{i=3}^r\widehat{\phi}(V^{-1}g_1,V^{-1}g_1,V^{-1}g_3,V^{-1}g_4,\dots,V^{-1}g_{i-1},x_i,g_{i+1},g_r)\big]_n,\] where as before $  g_i=\leftidx{_n}{g}{_{\CP}}(x_i) $. The last sum is zero, because $ \widehat{\phi} $ is alternating and if we use the fact that $ \widehat{\phi} $ is antisymmetric, the sum of the first two terms will be equal to \[(-1)^{r-1}\big[\widehat{\phi}(V^{-1}g_1,x_1,g_3,g_4,\dots,g_r)-\widehat{\phi}(g_1,x_1,g_3,g_4,\dots,g_r)\big]_n=\] \[(-1)^{r-1}\big[\widehat{\phi}(V^{-1}g_1-g_1,x_1,g_3,g_4,\dots,g_r)\big]_n=\] \[(-1)^{r-1}\big[\widehat{\phi}(p^nx_1,x_1,g_3,g_4,\dots,g_r)\big]_n\]\[=(-1)^{r-1}\big[p^n\widehat{\phi}(x_1,x_1,g_3,g_4,\dots,g_r)\big]_n,\] which is zero, since $ \widehat{\phi} $ is alternating.

\item[(iv)] This follows from the fact that for every nilpotent $S$-algebra $ \CM $, the two groups $ G_{\CP}^0(\CM)=\widehat{W}(\CM)\otimes_{W(R)}P  $ and $ G_{\CP_S}^0(\CM)=\widehat{W}(\CM)\otimes_{W(S)}P_S $ are canonically isomorphism and this isomorphism induces an isomorphism between the subgroups $ G_{\CP}^{-1}(\CM) $ and $ G_{\CP_S}^{-1}(\CM) $, and the canonical isomorphism $ (BT_{\CP})_S\cong BT_{\CP_S}$.
\end{itemize}
\end{proof}

As a direct consequence of this Proposition, we obtain the following Corollary.

\begin{cor}
\label{cor beta}
The construction of $ \beta $ yields homomorphisms \[ \beta:\Mult(\CP_1\times\dots\times\CP_r,\CP_0)\to \Mult(BT_{\CP_1}\times\dots\times BT_{\CP_0},BT_{\CP_0}), \] \[ \Sym(\CP_1^r,\CP_0)\to \Sym(BT_{\CP_1}^r,BT_{\CP_0})\]  and \[ \Alt(\CP_1^r,\CP_0)\to \Alt(BT_{\CP_1}^r,BT_{\CP_0}).\]
\end{cor}

\begin{Ques}
\label{question on mult displays}
Are the morphisms $ \beta $ in the Corollary \ref{cor beta} isomorphisms?
\end{Ques}

\section{Exterior powers}

\begin{cons}
\label{cons03}
Let $\CP=(P,Q,F,V^{-1})$ be a $3n$-display with tangent module of rank one. We want to define a new $3n$-display denoted by $ \ep^r\CP=(\Lambda^rP,\Lambda^rQ,\Lambda^rF,\Lambda^rV^{-1})$. Fix a normal decomposition $$ P=L\oplus T \quad\text{and}\quad Q=L\oplus I_RT. $$ Although we use a normal decomposition for the construction, we will show in the next lemma, that this construction is in fact independent from the choice of a normal decomposition.
\begin{itemize}
\item Define $\Lambda^rP  $ to be the exterior power of $P$, $ \ep^rP$, over the ring $ W(R) $.
\item Define $ \Lambda^rQ $ to be the image of the homomorphism $ \ep^rQ\arrover{\ep^r\iota}\ep^rP $, where $ \iota:Q\into P $ is the inclusion.
\item Since by assumption, $ T $ is projective of rank one, we have $$ \Lambda^rP\cong \ep^rL\oplus\ep^{r-1}L\otimes T  $$ and $$ \Lambda^rQ\cong \ep^rL\oplus\ep^{r-1}L\otimes I_RT.$$ Define $ \Lambda^rF:\Lambda^rP\to \Lambda^rP $ to be $ \ep^{r-1}V^{-1}\wedge F $.
\item Define $ \Lambda^rV^{-1}:\Lambda^rQ\to \Lambda^rP $ to be $ \ep^rV^{-1}:\ep^rQ\to\ep^rP $ restricted to $ \Lambda^rQ $ (note that $ \Lambda^rQ $ is a direct summand of $ \ep^rQ \cong \bigoplus_{i=0}^{r}(\ep^{r-i}L\otimes \ep^iI_RT)$ ).
\end{itemize}
\end{cons}

\begin{lem}
\label{lem0 21}
The construction of $ \ep^r\CP=(\Lambda^rP,\Lambda^rQ,\Lambda^rF,\Lambda^rV^{-1}) $ does not depend on the choice of a normal decomposition of $P$ and defines a $3n$-display structure. The height and rank of $ \ep^r\CP $ are respectively $ \binom{h}{r} $ and $ \binom{h-1}{r-1} $, where $h$ is the height of $ \CP $. If $ \CP $ is a display, then $ \ep^r\CP $ is a display as well. Furthermore, this construction commutes with the base change.
\end{lem}

\begin{proof}
Assume that we have shown that the morphism $ \Lambda^rV^{-1} $ is independent from the choice of a normal decomposition and that this construction defines a $3n$-display structure. Then, as we know that for any $3n$-display, the morphism $F$ is uniquely determined by the morphism $ V^{-1} $ (cf. Remark \ref{rem0 15}), the morphism $\Lambda^rF$ will be independent from the choice of a normal decomposition as well. So, we prove at first the canonicity of $ \Lambda^rV^{-1} $ and then show that with this construction, we obtain a $3n$-display.\\

Set $ N:= \bigoplus_{i=2}^{r}(\ep^{r-i}L\otimes \ep^iI_RT)$. Since $ \ep^r Q\cong\Lambda^rQ\oplus N$ and since the morphism $ \ep^rV^{-1}:\ep^rQ\to \ep^rP $ is independent from the choice of a normal decomposition, if we show that the restriction of this morphism to the submodule $ N $ is the zero morphism, then it follows that the canonical morphism $ \ep^rV^{-1} $ factors through the quotient $ \ep^rQ\onto \Lambda^rQ $. Thus, the resulting morphism $ \Lambda^rQ\to \Lambda^rP $, which is equal to $ \Lambda^rV^{-1} $, is independent from the choice of a normal decomposition. So, it is enough to show that for every $ i>1 $, the morphism $$ \ep^rV^{-1}:\ep^{r-i}L\otimes \ep^iI_RT\to \ep^rP $$ is trivial.  For any $ w\in W(R) $ and $ x\in P $, we have $ V^{-1}(V(w).x)=wF(x) $. It implies that this morphism factors through the image of the morphism $$ \ep^{r-i}V^{-1}\wedge \ep^iF:\ep ^{r-i}L\otimes \ep^iT\to \ep^rP.$$ The module $ \ep^iT $ being trivial for $ i>1 $, we conclude that the morphism $ \ep^rV^{-1} $ restricted to $ \ep^{r-i}L\otimes \ep^iI_R $ is zero, as desired.\\

As $ P $ is a projective $W(R)$-module, its exterior powers are projective too. We have $ \Lambda^rP= \ep^rL\oplus\ep^{r-1}L\otimes T $ and $  \Lambda^rQ= \ep^rL\oplus\ep^{r-1}L\otimes I_RT $, and since $ I_R(\ep^{r-1}L\otimes T)= \ep^{r-1}L\otimes I_RT$, we conclude that the direct sum \[\ep^rL\oplus\ep^{r-1}L\otimes T\] is a normal decomposition of $ \Lambda^rP $. We have to show that the  morphism $ \Lambda^rV^{-1} $ is an $ F^R $-linear epimorphism. But we know that $ V^{-1}:Q\to P $ is an $ F^R $-linear epimorphism and therefore $ \ep^rV^{-1}:\ep^rQ \to \ep^rP$ is an $ F^R $-linear epimorphism as well. As this morphism factors through the quotient $\ep^rQ \onto\Lambda^rQ  $, the morphism $ \Lambda^rV^{-1} $ is also an $ F^R $-linear epimorphism.\\

Now, we show that the morphism $ \Lambda^rF $ has the right properties, i.e., it is $ F^R $-linear and satisfies the relation $ w\Lambda^rF(x)=\Lambda^rV^{-1}(V(w).x) $ for every $ w\in W(R) $ and every $ x\in\Lambda^rP $. The fact that it is $F^R$-linear follows from its construction and the fact that $ V^{-1} $ and $ F $ are $F^R$-linear. Now take an element $ w\in W(R) $ and $ q_1\wedge\dots\wedge q_{r-1}\otimes t $ in the submodule $ \ep^{r-1}L\otimes T $, we have \[ w.\Lambda^rF(q_1\wedge\dots\wedge q_{r-1}\otimes t)=wV^{-1}q_1\wedge\dots\wedge V^{-1}q_{r-1}\wedge F(t)= \] \[ V^{-1}q_1\wedge\dots\wedge V^{-1}q_{r-1}\wedge wF(t)=V^{-1}q_1\wedge\dots\wedge V^{-1}q_{r-1}\wedge V^{-1}(V(1).x)= \] \[ \Lambda^rV^{-1}(V(1).q_1\wedge\dots\wedge q_{r-1}\otimes t ).\] A similar calculation shows that for any $ w\in W(R) $ and any $ x\in \ep^rL $, we have $ w.\Lambda^rF(x)=\Lambda^rV^{-1}(V(1).x) $, and therefore $ \ep^r\CP=(\Lambda^rP,\Lambda^rQ,\Lambda^rF,\Lambda^rV^{-1}) $ is a $3n$-display.\\

By definition, the height of $ \ep^r\CP $ is the rank of the projective $W(R)$-module $\Lambda^rP$, which is equal to $ \binom{h}{r} $, with $ h $ the rank of $P$. The rank of $ \ep^r\CP $ is equal to the rank of the projective $W(R)$-module $ \ep^{r-1}L\otimes T $, which is equal to $ \binom{h-1}{r-1} $, since $ L $ has rank $h-1$ and $T$ has rank one (cf. Remark \ref{rem0 9}).\\

Since the construction of exterior powers of modules commutes with the base change, a straightforward calculation shows that, under the identification \[W(R)\otimes_{F,W(R)}\Lambda^rP\cong \Lambda^r(W(R)\otimes_{F,W(R)}P),\] the morphism $$ (\Lambda^rV)^{N\sharp}:\Lambda^rP\to W(R)\otimes_{F,W(R)}\Lambda^rP $$ (cf. Construction \ref{cons015}) is equal to the morphism $$ \Lambda^r(V^{N\sharp}): \Lambda^rP\to \Lambda^r(W(R)\otimes_{F,W(R)}P).$$ Again, since $ \ep^r$ commutes with base change, $ \Lambda^r(V^{N\sharp}) $ is the zero module $ I_R+pW(R) $, if $ V^{N\sharp} $ is the zero  modulo $ I_R+pW(R) $. This shows that $ \ep^r\CP $ is a display, if $ \CP $ is a display.\\

Finally, we have to show that if $ R\to S $ is a base extension, then there exists a canonical isomorphism $ \ep^r(\CP_S)\cong (\ep^r\CP)_S $. This is a straightforward calculation. We explain why the pairs $ ((\Lambda^rP)_S,(\Lambda^rQ)_S) $ and $ (\Lambda^r(P_S), \Lambda^r(Q_S)) $ are canonically isomorphic, and leave the verification of the equality of the pairs $  ((\Lambda^rF)_S,(\Lambda^rV^{-1})_S) $ and $ (\Lambda^r(F_S), \Lambda^r(V^{-1}_S)) $ to the reader. By definition, we have \[ (\Lambda^rP)_S=W(S)\otimes_{W(R)}\ep^rP=\ep^r(W(S)\otimes_{W(R)}P)=\Lambda^r(P_S)\]
and using a normal decomposition, we have \[(\Lambda^rQ)_S=(W(S)_{W(R)}\ep^rL)\oplus(I_S\otimes_{W(R)}\ep^{r-1}L\otimes_{W(R)}T)\cong\]\[\ep^r(W(S)\otimes_{W(R)}L)\oplus(I_S\otimes_{W(S)}W(S)\otimes_{W(R)}\ep^{r-1}L\otimes_{W(R)} T)\cong\] \[ \ep^rL_S\oplus (W(S)\otimes_{W(R)}\ep^{r-1}L)\otimes_{W(S)}(I_S\otimes_{W(R)}T)\cong \] \[ \ep^rL_S\oplus (\ep^{r-1}L_S)\otimes_{W(S)}I_S(W(S)\otimes_{W(R)}T)=\] \[\ep^rL_S\oplus (\ep^{r-1}L_S\otimes_{W(S)}I_ST_S)=\Lambda^r(Q_S).\] The above isomorphisms are induced by the canonical isomorphism $ (\ep^rP)_S\cong \ep^r(P_S) $, i.e., this isomorphism restricts to an isomorphism $ (\Lambda^rQ)_S\cong \Lambda^r(Q_S) $. Thus, the latter isomorphism does not depend on the choice of a normal decomposition either.
\end{proof}

\begin{prop}
\label{prop ext. disp.}
Let $\CP$ be a $3n$-display of rank one over a ring $R$. The map \[ \lambda:P^r\to\ep^rP,\quad  (x_1,\dots,x_r)\mapsto x_1\wedge\dots\wedge x_r \] defines an alternating morphism of $3n$-displays $ \lambda:\CP^r\to\ep^r\CP $ with the following universal property:\\

For every $3n$-display $ \CP' $ over $R$, the homomorphism \[ \Hom(\ep^r\CP,\CP')\to\Alt(\CP^r,\CP')\] induced by $ \lambda $ is an isomorphism.
\end{prop}

\begin{proof}
It follows from the construction of $ \ep^r\CP=(\Lambda^rP,\Lambda^rQ,\Lambda^rF,\Lambda^rV^{-1}) $ that $ \lambda $ is an alternating morphism of $ 3n $-displays. We therefore only need to show the universal property. Let $ \CP'=(P',Q',F,V^{-1}) $ be a $3n$-display over $R$ and let $ \phi:\CP^r\to\CP' $ be an alternating morphism of $3n$-displays. We ought to show that there exists a unique morphism of $3n$-displays from $\ep^r\CP$ to $ \CP' $, whose composition with $ \lambda $ is $ \phi $. The morphism $ \phi:P^r\to P' $ is an alternating morphism of $R$-modules and the restriction of $ \phi $ to $ Q^r $ is an alternating morphism $ \phi:Q^r\to Q' $. By the universal property of $ \Lambda^rP=\ep^rP $, there exists a unique $R$-modules homomorphism $ \bar{\phi}:\ep^rP\to P' $ such that $ \bar{\phi}\circ \lambda=\phi $ and we claim that this morphism defines a morphism of $3n$-displays from $ \ep^r\CP $ to $ \CP' $. Consider the following diagram: \[ \xymatrix{Q^r\ar[dd]_{\phi}\ar@{->>}[rr]^{\lambda}\ar@{^{(}->}[rd]&&\ep^rQ\ar[rd]&\\&P^r\ar[dd]_{\phi}\ar@{->>}[rr]^{\lambda}&&\ep^rP\ar[ddll]^{\bar{\phi}}\\Q'\ar@{^{(}->}[dr]&&&\\&P'.&&} \] From what we said above and the definition, this diagram commutes. We want to show that the image of $ \Lambda^rQ $ under $ \bar{\phi} $ lies inside $ Q' $. Since by construction, $ \Lambda^rQ $ is the image of the morphism $ \ep^rQ\to \ep^rP $, and since the morphism $ \lambda:Q^r\to \ep^rQ $ is surjective, it is enough to show that the image of the composition $$ Q^r\arrover{\lambda}\ep^rQ\to \ep^rP\arrover{\bar{\phi}} P'$$ lies inside $ Q' $. This follows from the commutativity of the above diagram. Now, we have to show that $ \bar{\phi}\circ \Lambda^rV^{-1}=V^{-1}\circ \bar{\phi} $. Take an element $ q:=q_1\wedge q_2\wedge\dots\wedge q_r\in \ep^rQ $. We have \[ \bar{\phi}\circ \ep^rV^{-1}(q)=\bar{\phi}\big(V^{-1}(q_1)\wedge\dots\wedge V^{-1}(q_r)\big)=\phi\big(V^{-1}(q_1),\dots,V^{-1}(q_r)\big)\]\[=V^{-1}\phi(q_1,\dots,q_r)=V^{-1}\circ\bar{\phi}(q),\] where the third equality follows from the fact that $ \phi $ satisfies the $ V $-condition. This implies that for every $ q\in \Lambda^rQ $, we have $ \bar{\phi}\circ \Lambda^rV^{-1}(q)=V^{-1}\circ \bar{\phi}(q) $ and the claim is proved. By construction of $ \bar{\phi} $, we have $ \bar{\phi}\circ\lambda=\phi $. It remains to show the uniqueness of $ \bar{\phi} $. Since the morphism $ \lambda:P^r\to\ep^rP $ is a surjective map (as sets), any morphism $ \phi_1 :\ep^r\CP\to \CP'$ with $ \phi_1\circ \lambda=\phi $ is equal to $ \bar{\phi} $ as a morphism from $ \ep^rP $ to $ P' $ and therefore is equal to $ \bar{\phi} $ as a morphism of $ 3n $-displays. The proof is now achieved.
\end{proof}

\chapter{Comparisons}

The aim of this chapter is to compare the Cartier module, the Dieudonn\'e module and the display of a connected $p$-divisible group over a perfect field of characteristic $p$. In fact, we would like to show that these three linear algebraic gadgets are isomorphic, a result which floats around and is known to the experts, but lacks a written proof (at least not accessible to the author). According to \cite{B}, the isomorphism between the Cartier module and the Dieudonn\'e module of a connected $p$-divisible group over a perfect field of characteristic $p$ is due to W. Messing. We would also like to emphasize that we need explicit isomorphisms and therefore the knowledge of the existence of such isomorphisms is not sufficient for the purpose we have in mind.\\

In this chapter $G$ denotes a $p$-divisible group over a perfect field $k$, of characteristic $p>0$.

\begin{dfn}
\label{def06}
The \emph{Cartier module} of a formal group $G$, denoted by $M(G)$, is by definition the $ \BE_k$-module $ \Hom(\widehat{W},G) $, with the action of Frobenius and Verschiebung through their action on $ \widehat{W} $.
\end{dfn}

\section{Cartier vs. Dieudonn\'e}

\begin{cons}
\label{cons06}
Assume that $G$ is local-local. We want to construct a homomorphism $ \eta:D_*(G)\to M(G) $. Fix natural numbers $M$ and $m$ with $F_{G_n}^M=0 $ and $ V_{G_n}^m=0 .$ It follows that every $v\in D_*(G_n)=\Hom(\BW,G_n)$ factors through the projection $ \BW\onto W_{m,M} $ and we also denote the induced morphism, $W_{m,M}\to G_n$, by $v$. Now, if we take an element $u\in D_*(G)$ and denote by $[u]_n$ the class of $u$ modulo $p^nD_*(G)$ in $D_*(G_n)$, for $m,M\gg 0$, we can view $[u]_n$ as a morphism $W_{m,M}\to G_n$ or as a morphism $ \BW\to G_n $.\\

The sequence $ W_{m,M}\arrover{F^n}W_{m,M}\arrover{F^{M-n}}W_{m,n} $ is exact and the composition \[ W_{m,M}\arrover{F^n} W_{m,M}\arrover{V^n} W_{m,M}\arrover{[u]_n}G_n \] is zero (since $G_n$ is annihilated by $p^n$, and $V^n\circ F^n=p^n$). Therefore, the composite $ [u]_n\circ V^n:W_{m,M}\to G_n $ factors through $ W_{m,M}\arrover{F^{M-n}} W_{m,M} $, inducing a morphism \[ \eta(u)_n:W_{m,n}\to G_n\into G .\] These morphisms are compatible with respect to inclusions $W_{m,n}\into W_{m',n'}$ induced by $\tau$ (which is not a group homomorphism) and the natural inclusion $W_{m',n}\into W_{m',n'}$, i.e., we send an element $(x_0,\dots,x_{n-1})\in W_{m,n}$ to the element $(x_0,\dots,x_{n-1},0,0,\dots,0)\in W_{m',n'}$. As $ \widehat{W}=\bigcup_{m',n'}\tau(W_{m',n'})$, the morphisms $\eta(u)_n$ induce a unique morphism $\eta(u):\widehat{W}\to G$, extending all $\eta(u)_n$. Hence a map $$ \eta:D_*(G)\to M(G).$$ The following commutative diagram illustrates the constructed maps $ \eta(u)_n $ and $ \eta(u) $: 
\begin{myequation}
\label{Dieudonné-Cartier Diagram}
\xymatrix{W_{m,M}\ar[rr]^{F^{M-n}}\ar[d]_{V^n}&&W_{m,n}\ar[d]_{\eta(u)_n}\ar@{^{(}->}[rr]^{\tau_{m,n}}&&\widehat{W}\ar[d]^{\eta(u)}\\ W_{m,M}\ar[rr]_{[u]_n}&&G_n\ar@{^{(}->}[rr]&&G.}
\end{myequation}
\end{cons}

We will now generalize the above construction to define a homomorphism $ \eta:D_*(G)\to M(G) $ when $G$ is a connected $p$-divisible group (that can contain a multiplicative part).
\begin{cons}
\label{cons018}
Let $G$ be a connected $p$-divisible group over a perfect field of characteristic $p$. By definition, the covariant Dieudonn\'e module of $G$ is the inverse limit $ \uset{n}{\invlim}\,D_*(G_n)=\uset{n}{\invlim}\,D^*(G^*_n) $ where $ D^*(G_n) $ is the contravariant Dieudonn\'e module of the Cartier dual of $ G_n $ and the transition homomorphisms are induced by the inclusions $ G_n^*\into G_{n+1}^* $. Since $G_n^*$ is unipotent, we have $ D^*(G_n^*)=\Hom(G_n^*,CW^u) $, where $ CW^u $ is the group functor of unipotent Witt covectors. The group scheme $G_n^*$ being annihilated by $p^n$, every homomorphism $ G_n^*\to CW^u $ factors through $ CW^u[p^n] $. As $ G_n^* $ is the image of the multiplication by $p$ of $ G_{n+1}^* $, the transition homomorphism $$ \Hom(G_{n+1}^*,CW^u[p^{n+1}])\to\Hom(G_n^*,CW^u[p^n]) $$ is given by the following commutative diagram: \[ \xymatrix{G_{n+1}^*\ar@{->>}[d]\ar[r]&CW^u[p^{n+1}]\ar[d]^{p.}\\G_n^*\ar[r]&CW^u[p^n].} \] Applying the Frobenius morphism to the power $n$ to such a homomorphism, we obtain a homomorphism $ G_n^*\to CW^u[V^n] $, i.e., we have a homomorphism $$ \Hom(G_n^*,CW^u)\arrover{F^n\circ(\_)} \Hom(G_n^*,CW^u[V^n]).$$ The group scheme $ CW^u[V^n] $ is canonically isomorphic to the group of finite Witt vectors of length $n$, i.e., $ W_n $. We have thus for every $n$ a homomorphism $ D^*(G_n^*)\to \Hom(G_n^*,W_n) $. These homomorphisms are compatible with the projections $ G_{n+1}^*\to G_n^* $ and $ W_{n+1}\to W_n $ and therefore induce a homomorphism \[\uset{n}{\invlim}\,D^*(G^*_n)\to \uset{n}{\invlim}\, \Hom(G_n^*,W_n)\cong \Hom(\uset{n}{\invlim}\,G_n^*,\uset{n}{\invlim}\,W_n)\cong \Hom(\uset{n}{\invlim}\,G_n^*,W).\] Now, using the Artin-Hasse exponential in order to obtain a perfect pairing $ \widehat{W}\times W\to \BG_m $, we can take the Cartier duals of homomorphisms $ \uset{n}{\invlim}\,G_n^*\to W $, and obtain an isomorphism $ \Hom(\uset{n}{\invlim}\,G_n^*,W)\cong \Hom(\widehat{W},G)=M(G) $. Composing the homomorphisms and isomorphisms we constructed above, we obtain a homomorphism $ \eta:D_*(G)\to M(G) $.
\end{cons}

\begin{rem}
\label{rem021}
As we noticed before the above construction, the homomorphisms $ \eta:D_*(G)\to M(G) $ defined in Construction \ref{cons06} and Construction \ref{cons018} coincide, when $G$ is local-local. The reason that we considered the special case of $p$-divisible groups of local-local type separately is that we will later need this explicit construction and in this form (cf. Theorem \ref{thm04}).
\end{rem}

Let $H$ be a unipotent group scheme over a perfect field of positive characteristic and let $ v:H\to CW^u $ be an element of the contravariant Dieudonn\'e module of $H$. There exists a natural number $n$ such that $V^nH=0$ and therefore, $v$ factors through the kernel of $V^n$ on $ CW^u $, which is canonically isomorphic to $ W_n $. We can thus consider $ v $ as a group scheme homomorphism $ v:H\to W_n $ and composing this with the morphism $ \tau_n:W_n\into W $, we obtain a morphism of schemes (not a homomorphism!) $ H\to W $ that we denote simply by $ \tau_n v $ or $ \tau v $.

\begin{lem}
\label{lem024}
Take an element $u\in D_*(G)$, a nilpotent $k$-algebra $ \CN $ and an element $ w\in\widehat{W}(\CN) $ annihilated by $ F^n $. Denote by $[u]_n$ the class of $u$ modulo $ p^n $, seen as an element of $ D_*(G_n)=D^*(G_n^*)=\Hom(G_n^*,CW^u) $, and by $E$ the Artin-Hasse exponential. Then, under the identification $ G_n\cong \innHom(G_n^*,\BG_m)$, the homomorphism $$ E(w\cdot \tau_n[F^nu]_n(\_);1):G_n^*\to \BG_m $$ is equal to the element $ \eta(u)(w) \in G_n(\CN)$. 
\end{lem}

\begin{proof}
Since $ F^nw=0 $, the element $ \eta(u)(w) $ lies inside the group $ G_n(\CN) $ and $ E(w\cdot \tau_n[F^nu]_n(\_);1)$ is indeed a homomorphism from $ G_n^* $ to $ \BG_m $ (note that $ [F^nu]_n:G_n^*\to CW^u $ factors through $W_n$). It follows from the construction of $ \eta $ that for every $ g^*\in G_n^*(\CN) $ we have \[ E(w\cdot \tau_n[F^nu]_n(g^*);1)=E(w\cdot \tau_nF^n[u]_n(g^*);1)=g^*(\eta(u)(w))\] and the first equality is true because $ [F^nu]_n $ and $ F^n[u]_n $, seen as homomorphisms $ G_n^*\to CW^u $ are equal (note that the Frobenius of $ G_n^* $ corresponds to Verschiebung of $ G_n $ and the Frobenius on the covariant Dieudonn\'e module is induced by the Verschiebung). Under the identification $ G_n\cong \innHom(G_n^*,\BG_m)$, we have \[g^*(\eta(u)(w))=\big(\eta(u)(w)\big)(g^*).\] These two equalities imply that $ E(w\cdot \tau_n[F^nu]_n(\_);1)=\eta(u)(w) $.
\end{proof}

\begin{lem}
\label{lem025}
Let $S$ be a $k$-scheme, $\ul{x}, \ul{x'}$ elements of $ W(S) $ and $ \ul{y} $ an element of $ \widehat{W}(S) $ such that $ \ul{x} $ and $ \ul{x'} $ have the same image in the group $ W_m(S) $ and $ F^m\ul{y}=0 $ for some natural number $m$. Then we have \[ E(\ul{x}\cdot\ul{y};1)= E(\ul{x'}\cdot\ul{y};1).\]
\end{lem}

\begin{proof}
As $ \ul{x} $ and $ \ul{x'} $ have the same image in the group $ W_m(S) $, there exists an element $ \ul{z}\in W(S) $ such that $ \ul{x}-\ul{x'}=V^m\ul{z} $. We have thus \[ E(\ul{x}\cdot\ul{y};1)= E(\ul{x'}\cdot\ul{y}+V^m\ul{z}\cdot \ul{y};1)=E(\ul{x'}\cdot\ul{y};1)\cdot E(V^m\ul{z}\cdot \ul{y};1).\] We also have $ E(V^m\ul{z}\cdot \ul{y};1)=E(\ul{z}\cdot F^m\ul{y};1) $ which is equal to $1$, because $ F^m\ul{y}=0 $. Hence $E(\ul{x}\cdot\ul{y};1)= E(\ul{x'}\cdot\ul{y};1).$
\end{proof}

\begin{cons}
\label{cons010}
$ $
\begin{itemize}
 \item The projection $G_{n+1}\onto G_n$ induces a homomorphism \[f_n:\Hom(\BW/p^{n+1},G_{n+1})\to \Hom(\BW/p^{n+1},G_n)\cong \Hom(\BW/p^{n},G_n)\] where the isomorphism is due to the fact that $G_n$ is annihilated by $p^n$ and  every map $\BW/p^{n+1}\to G_n$ factors through the quotient $\BW/p^{n+1}\onto \BW/p^n$. Denote by $\uset{f_n}{\invlim}\,\Hom(\BW/p^{n},G_n)$ the corresponding inverse limit.
 \item The map $ p.:\BW/p^n\to \BW/p^{n+1} $ induces a homomorphism $$ g_n: \Hom(\BW/p^{n+1}, G_{n+1}) \to  \Hom(\BW/p^n, G_{n+1})\cong  \Hom(\BW/p^{n},G_n),$$ where the isomorphism is due to the fact that $\BW/p^n $ is annihilated by $ p^n $  and $ G_n $ is the kernel of multiplication by $ p^n $ on $ G_{n+1} $, and therefore any map $\BW/p^n\to G_{n+1}$ factors through the inclusion $ G_n\into G_{n+1} $. We denote by $\uset{g_n}{\invlim}\,\Hom(\BW/p^{n},G_n)$ the corresponding inverse limit.
 \item Likewise, the map $ F:\BW/F^n\to \BW/F^{n+1} $ induces a homomorphism $$ \Hom(\BW/F^{n+1}, G_{n+1}) \to  \Hom(\BW/F^n, G_{n+1})\cong \Hom(\BW/F^n, G_{n}) $$ and we have an isomorphism because $\BW/F^n$ is annihilated by $p^n$. The corresponding inverse limit will be denoted by $\uset{n}{\invlim}\,\Hom(\BW/F^n,G_n).$
\end{itemize}
\end{cons}

\begin{lem}
\label{lem0 13}
The two inverse limits constructed above, $\uset{f_n}{\invlim}\,\Hom(\BW/p^{n},G_n)$ and $\uset{g_n}{\invlim}\,\Hom(\BW/p^{n},G_n)$, are equal, i.e., a sequence $(\gamma_n)$ belongs to the one if and only if it belongs to the other.
\end{lem}

\begin{proof}
We show that for every $n$, the transition homomorphisms $ f_n $ and $ g_n $ are equal. Take an element $ \alpha\in \Hom(\BW/p^{n},G_{n+1}) $. By construction, $ f_n(\alpha):\BW/p^n\to G_n $ is the unique homomorphism making the following diagram commutative
$$\xymatrix{\BW/p^{n+1}\ar[r]^{\alpha}\ar@{->>}[d]&G_{n+1}\ar@{->>}[d]^{p.}\\ \BW/p^n\ar[r]_{f_n(\alpha)}&G_n}$$
and $ g_n(\alpha):\BW/p^n\to G_n $ is the unique homomorphism making the following diagram commutative 
$$\xymatrix{\BW/p^{n}\ar[r]^{g_n(\alpha)}\ar@{^{(}->}[d]_{p.}&G_{n}\ar@{^{(}->}[d]\\ \BW/p^{n+1}\ar[r]_{\alpha}&G_{n+1}.} $$
Putting these two diagrams together, we obtain the following diagram
\[ \xymatrix{\BW/p^{n+1}\ar[r]^{\alpha}\ar@{->>}[d]&G_{n+1}\ar[d]^{p.}\\ \BW/p^{n}\ar@{^{(}->}[d]_{p.}\ar@/_/[r]_{g_n(\alpha)}\ar@/^/[r]^{f_n(\alpha)}&G_n\ar@{^{(}->}[d]\\ \BW/p^{n+1}\ar[r]_{\alpha}&G_{n+1},} \] where the compositions of vertical arrows on left and respectively on right are multiplication by $p$ on $ \BW/p^{n+1} $ and respectively on $ G_{n+1} $. Therefore, composing $ f_n(\alpha) $ and $ g_n(\alpha) $ from right with the monomorphism $ G_n\into G_{n+1} $ and from left with the epimorphism $ \BW/p^{n+1}\onto \BW/p^n $ gives the same homomorphism $ p\alpha $
Consequently, $ f_n(\alpha)=g_n(\alpha) $.
\end{proof}

\begin{rem}
\label{rem0 6}
Since by the previous lemma the two inverse limits are equal, we will drop the subscripts $f_n$ and $g_n$ from them and denote this inverse limit by $\uset{n}{\invlim}\,\Hom(\BW/p^{n},G_n)$. It follows that in order to show that a sequence of maps $(\alpha_n)$ belongs to this inverse limit, it is sufficient to show one of the two compatibilities (commutativity of either of the first two diagrams in the proof of the previous lemma).
\end{rem}

\begin{lem}
\label{lem04}
There is a canonical isomorphism $$ \uset{n}{\invlim}\,\Hom(\BW/p^n,G_n)\cong D_*(G) .$$
\end{lem}

\begin{proof}
By definition, we have $$ D_*(G)=\uset{n}{\invlim}\,D_*(G_n)=\uset{n}{\invlim}\,\Hom(\BW,G_n) $$ and as we have seen before, $ \Hom(\BW,G_n)\cong\Hom(\BW/p^n,G_n) $ and thus $ D_*(G)\cong\uset{n}{\invlim}\,\Hom(\BW/p^n,G_n)$. A short calculation shows that the transition homomorphisms in this inverse limit is the one given in Construction \ref{cons010}.
\end{proof}

\begin{lem}
\label{lem05}
Assume that $G$ is unipotent (has local dual). Then for every $n$, there is a canonical isomorphism  $ \Hom(W[F^n],G_n)\cong \Hom(W_{m,n},G_n) $ for a sufficiently large $m$. Consequently, there is a canonical isomorphism $$\uset{n}{\invlim}\,\Hom(\BW/F^n,G_n)\cong M(G) .$$
\end{lem}

\begin{proof}
As $G$ is unipotent, for every $n$ there exists an integer $m_n$ such that $ V^{m_n}G_n=0 $ and therefore any homomorphism $ W[F^n]\to G_n $ factors through the quotient $W[F^n]\onto W_{m_n,n} $. This proves the first part of the proposition. For the second part, we first show that there is a canonical isomorphism $ \BW/F^n\cong W[F^n] $. For every $ m>n $ and every $i$, we have an exact sequence \[ W_{i,m}\arrover{F^n} W_{i,m}\longto W_{i,n}\longto 0. \] Now taking the inverse limit over all $i$ and $m>n$ we obtain the following exact sequence (note that the Mittag-Leffler condition is satisfied): \[ \uset{i,m}{\invlim}  \,W_{i,m}\arrover{F^n}\unset{i,m}{\invlim}\, W_{i,m}\longto \unset{i}{\invlim}\, W_{i,n}\longto 0 \] which is isomorphic to the following exact sequence: \[ \BW\arrover{F^n}\BW\longto W[F^n]\longto 0 \] and this proves the claim.\\

Now, we show that $ \uset{n}{\invlim}\,\Hom(W[F^n],G_n)\cong M(G) $, which will prove the proposition, using the above isomorphism and the fact that the above isomorphisms are compatible with the transition homomorphisms in the inverse systems $ \{W[F^n]\}_n $ and $ \{\BW/F^n\}_n $. We have from the first statement of the proposition that 
\begin{myequation}
\label{Frob}
\uset{n}{\invlim}\,\Hom(W[F^n],G_n)\cong \uset{n}{\invlim}\,\Hom(W_{m_n,n},G_n).
\end{myequation}
We also have
\[
M(G)=\Hom(\bigcup_{m,n}W_{m,n},G)\cong \uset{n}{\invlim}\,\Hom(\bigcup_m W_{m,n},G)\cong\]
\begin{myequation}
\label{Cartier}
\uset{n}{\invlim}\,\Hom(\bigcup_m W_{m,n},G_n)\cong \uset{n}{\invlim}\,\Hom(W_{m_n,n},G_n),
\end{myequation}
where the second isomorphism follows from the fact that $ W_{m,n} $ is annihilated by $ p^n $ and therefore every homomorphism $ W_{m,n}\to G $ factors through the inclusion $ G_n\into G $ and the last isomorphism follows from the fact that $ V^{m_n}G_n=0 $. It follows from \eqref{Frob} and \eqref{Cartier} that $ \uset{n}{\invlim}\,\Hom(W[F^n],G_n)\cong M(G) $.
\end{proof}

\begin{lem}
For all $n$, the following sequences are exact: \[ G_n\arrover{F^n} G_n^{(p^n)}\arrover{V^n}G_n, \quad G_n^{(p^n)}\arrover{V^n} G_n\arrover{F^n}G_n^{(p^n)}. \] Consequently, we have the following exact sequences: \[ D_*(G)/p^n\arrover{V^n}D_*(G)/p^n\arrover{F^n}D_*(G)/p^n,\] \[D_*(G)/p^n\arrover{F^n}D_*(G)/p^n\arrover{V^n}D_*(G)/p^n.\]
\end{lem}

\begin{proof}
We show the exactness of the sequence $$G_n\arrover{F^n} G_n^{(p^n)}\arrover{V^n}G_n$$ and the exactness of the other sequence is proved similarly. Since $ p^nG_n=0 $ and $ V^n\circ F^n=p^n $, the homomorphism $ F^n:G_n\to G_n^{(p^n)} $ factors through the inclusion $ \kernel(V^n)\into G_n^{(p^n)} $. Now, take a scheme $ X $ over $k$ and an element $ x\in \kernel(V^n)(X) $. The homomorphism $ p^n:G_{2n}\to G_n$ is an epimorphism, and so there exists an fppf cover $ Y \to X $ and an element $ y\in G_{2n}(Y)$ such that $ p^ny=x|_Y $, where by $ |_Y $ we mean the restriction homomorphism $ G_n(X)\to G_n(Y) $. We have $ p^n(V^ny)=V^n(p^ny)=V^n(x|_Y)=V^n(x)|_Y=0 $, which implies that $ V^ny\in G_n(Y) $. As $ x|_Y=p^ny=F^n(V^ny) $, we deduce that the homomorphism $ F^n:G_n\to \kernel(V^n) $ is an epimorphism, and therefore the sequence $$G_n\arrover{F^n} G_n^{(p^n)}\arrover{V^n}G_n$$ is exact.\\

Now, applying the Dieudonn\'e functor on this sequence, and identifying $ D_*(G_n) $ with $ D_*(G)/p^n $, we obtain the exact sequence $$D_*(G)/p^n\arrover{V^n}D_*(G)/p^n\arrover{F^n}D_*(G)/p^n.$$
\end{proof}

\begin{rem}
\label{rem0 8}
Note that since $ \BW/F^n $ is annihilated by $ F^n $, any homomorphism  $h:\BW/F^n\to G_n $ factors through the kernel of $ F^n $, i.e., through $ G_n[F^n] $. But according to the previous lemma, $ G_n[F^n]=V^nG_n^{(p^n)} $, and therefore we can view $h$ as a homomorphism $ \BW/F^n\to V^nG_n^{(p^n)} $.
\end{rem}

\begin{lem}
\label{lem02}
Assume that $G$ is local. Then for all $n$, there exists an $n'>n$, such that $p^{n'-n}(G_{n'}[V^{n'}])=0$.
\end{lem}

\begin{proof}
Since $ G $ is local, for every $n$ there exists an $ n'>n $ such that $ F^{n'}G_n=0 $. Using the previous lemma, we have $ G_{n'}[V^{n'}]=F^{n'}G_{n'}^{(p^{-n'})} $ and so $$ p^{n'-n}(G_{n'}[V^{n'}])=p^{n'-n}F^{n'}G_{n'}^{(p^{-n'})}=F^{n'}p^{n'-n}G_{n'}^{(p^{-n'})}=F^{n'}G_n^{(p^{-n'})} =0 .$$
\end{proof}

\begin{lem}
\label{lem07}
Assume that $G$ is local. Then for every $n$, there exists an $n'>n$ and a homomorphism $ r_{n',n}:V^{n'}G_{n'}\to G_n $ making the following diagram commutative: \[ \xymatrix{G_{n'}\ar@{->>}[rr]^{V^{n'}}\ar@{->>}[dr]_{p^{n'-n}}&&V^{n'}G_{n'}\ar@{-->}[dl]^{r_{n',n}}\\ &G_n.&} \]
\end{lem}

\begin{proof}
Since the homomorphism $ V^{n'}:G_{n'}\to V^{n'}G_{n'}$ is an epimorphism, it is enough to show that there exists an integer $ n'> n $ such that the homomorphism $ p^{n'-n}:G_{n'}\to G_n $ vanishes on the kernel of $ V^{n'} $, i.e., on $ G_{n'}[V^{n'}]$. The existence of such an $ n' $ is guaranteed by previous lemma.
\end{proof}

\begin{rem}
\label{rem0 4}
$  $
\begin{itemize}
 \item[1)] As $p^{n'-n}:G_{n'}\onto G_n$ is an epimorphism, the map $r_{n',n}:V^{n'}G_{n'}\to G_n $ is an epimorphism as well.
 \item[2)] Since the map $V^{n'}:G_{n'}\to V^{n'}G_{n'}$ is an epimorphism, any two maps $V^{n'}G_{n'}\to G_n $, making the diagram in the previous lemma commutative, are equal.
\end{itemize}
\end{rem}

\begin{cons}
\label{cons08}
For every $n$, the morphism $V^n:\BW/F^n\to \BW/p^n$ induces a homomorphism $$ \Hom(\BW/p^n, G_n) \arrover{(\_)\circ V^n}\Hom(\BW/F^n, G_n).$$ Since the diagram \[ \xymatrix{\BW/F^n\ar[r]^{V^n}\ar[d]_F&\BW/p^n\ar[d]^p\\ \BW/F^{n+1}\ar[r]_{V^{n+1}}&\BW/p^{n+1}} \] is commutative for every $n$, we obtain a homomorphism \[\bigvee: \uset{n}{\invlim}\,\Hom(\BW/p^n,G_n) \arrover{(\_)\circ V^n}\uset{n}{\invlim}\,\Hom(\BW/F^n,G_n). \]
\end{cons}

\begin{lem}
\label{lem09}
Assume that $G$ is local and that we have a system of maps $\delta_m: \BW/V^m \to G_m$ (for every $m$), such that for every pair of integers $n> m>0$, the following diagram commutes : \[ \xymatrix{\BW/V^n\ar@{->>}[d]\ar[r]^{\delta_n}&G_n\ar@{->>}[d]^{p^{n-m}}\\ \BW/V^m\ar[r]^{\delta_m}&G_m.} \] Then for every $m$, the map $ \delta_m $ is the zero map.
\end{lem}

\begin{proof}
The group scheme $ \BW/V^n $ is annihilated by $ V^n $ and therefore the homomorphism $ \delta_n $ factors through the inclusion $ G_n[V^n]\into G_n $ and we can rewrite the top square of the above diagram as follows:\[ \xymatrix{\BW/V^n\ar@{->>}[d]\ar[r]^{\delta_n}&G_n[V^n]\ar[d]^{p^{n-m}}\\ \BW/V^m\ar[r]^{\delta_m}&G_m.}\] This holds for every pair of natural numbers $n>m  $. If we fix $m$ and choose $n\gg m$ according to Lemma \ref{lem02}, such that $ p^{n-m}G_n[V^n]=0 $, the composition $$ \BW/V^n\arrover{\delta_n} G_n[V^n]\arrover{p^{n-m}}G_m $$ will be zero and so the composition $$ \BW/V^n\onto\BW/V^m\arrover{\delta_m}G_m$$ is zero as well. It implies that $ \delta_m=0 $, because $ \BW/V^n\onto\BW/V^m $ is an epimorphism.
\end{proof}

\begin{cons}
\label{cons0 9}
Assume that $G$ is local and that we are given an element $(\gamma_i)$ of the group $\uset{i}{\invlim}\,\Hom(\BW/F^i,G_i)$. Fix a natural number $n$. Choose a natural number $n'$ and the map $ r_{n',n}:V^{n'}G_{n'}\to G_n $ satisfying the statement of the Lemma \ref{lem07}. Now, consider the following composition \[\BW/p^{n'}\onto \BW/F^{n'}\arrover{\gamma_{n'}^{(p^{-n'})}}V^{n'}G_{n'}\ontoover{r_{n',n}} G_n\] where $\gamma_{n'}^{(p^{-n'})}$ is the twist of $\gamma_{n'}$ by the Frobenius to the power $-n'$ (note that the base field is perfect). Since the group scheme $G_n$ is annihilated by $p^n$, this composition factors through the quotient $\BW/p^{n'}\onto \BW/p^n$ and therefore we obtain a map $\gamma^{\sharp}_{n',n}:\BW/p^n\to G_n$ which a priori depends on both $n'$ and $n$.
\end{cons}

\begin{lem}
\label{lem0 22}
Let $(\gamma_i)$ be an element of the group $\uset{i}{\invlim}\,\Hom(\BW/F^i,G_i)$. Then for every pair of natural numbers $n>m$, the following diagram commutes:\[ \xymatrix{\BW/p^{n}\ar@{->>}[d]\ar@{->>}[r]&\BW/F^{n}\ar[rr]^{\gamma_{n}^{(p^{-n})}}&&V^{n}G_{n}\ar@{^{(}->}[rr]&&G_{n}^{(p^{-n})}\ar[d]^{F^{n-m}}\\\BW/p^{m}\ar@{->>}[r]&\BW/F^{m}\ar[rr]_{\gamma_{m}^{(p^{-m})}}&&V^{m}G_{m}\ar@{^{(}->}[r]&G_m^{(p^{-m})}\ar@{^{(}->}[r]&G_{n}^{(p^{-m})}.} \]
\end{lem}

\begin{proof}
By assumption, $ (\gamma_i)$ is an element of the group $\uset{i}{\invlim}\,\Hom(\BW/F^i,G_i) $, and thus in the diagram \[ \xymatrix{\BW/F^{n}\ar@{->>}[d]\ar[rr]^{\gamma_{n}^{(p^{-n})}}&&G_{n}^{(p^{-n})}\ar[d]^{F^{n-m}}\\\BW/F^m\ar[r]^{\gamma_{m}^{(p^{-m})}}\ar[d]_{F^{n-m}}&G_{m}^{(p^{-m})}\ar@{^{(}->}[r]&G_{n}^{(p^{-m})}\ar@{=}[d]\\\BW/F^{n}\ar[rr]_{\gamma_{n}^{(p^{-m})}}&&G_{n}^{(p^{-m})},} \] the bottom square commutes. Since the homomorphism $ \BW/F^{n}\onto \BW/F^m $ is an epimorphism and the outer diagram is commutative (due to the fact that $ \gamma_{n}^{(p^{-n})} $ is a group schemes homomorphism and so commutes with Frobenius), the top diagram commutes as well. Finally, the diagram \[ \xymatrix{\BW/p^{n}\ar@{->>}[d]\ar@{->>}[r]&\BW/F^{n}\ar@{->>}[d]\\\BW/p^{m}\ar@{->>}[r]&\BW/F^{m}} \] being commutative, it follows that the diagram in the statement if the lemma is commutative and this finishes the proof.
\end{proof}

\begin{lem}
\label{lem0 11}
Assume that $G$ is local. Then the map $\gamma^{\sharp}_{n',n}:\BW/p^n\to G_n$ does not depend on the choice of $n'$.
\end{lem}

\begin{proof}
It is enough to show that the homomorphisms $\gamma^{\sharp}_{n',n}$ and $\gamma^{\sharp}_{n'+1,n}$ are equal.\\ 

The Frobenius homomorphism $ F:G_{n'+1}^{(p^{-n'-1})}\to G_{n'+1}^{(p^{-n'})} $ restricts to a homomorphism $ V^{n'+1}G_{n'+1}\to V^{n'}G_{n'} $ that we denote again by $F$ (this is true because $ F\circ V^{n'+1}=pV^{n'}=V^{n'}\circ p $ and $ pG_{n'+1}=G_{n'} $). We have then a commutative diagram:
\[ \xymatrix{G_{n'+1}\ar@{->>}[d]_p\ar[rr]^{V^{n'+1}}&&V^{n'+1}G_{n'+1}\ar[d]_F\\G_{n'}\ar@{->>}[d]_{p^{n'-n}}\ar[rr]^{V^{n'}}&&V^{n'}G_{n'}\ar[dll]^{r_{n',n}}\\G_n.&&} \] It follows from Remark \ref{rem0 4} that $ r_{n'+1,n}=r_{n',n}\circ F $. Now, consider the following diagram: \[ \xymatrix{\BW/p^{n'+1}\ar@{->>}[d]\ar@{->>}[rr]&&\BW/F^{n'+1}\ar[rr]^{\gamma_{n'+1}^{(p^{-n'-1})}}&&V^{n'+1}G_{n'+1}\ar[d]^F\\\BW/p^{n'}\ar@{->>}[d]\ar@{->>}[rr]&&\BW/F^{n'}\ar[rr]^{\gamma_{n'}^{(p^{-n'})}}&&V^{n'}G_{n'}\ar[d]^{r_{n',n}}\\\BW/p^n\ar[rrrr]_{\gamma^{\sharp}_{n',n}}&&&&G_n.} \] If we show that the outer diagram commutes, then by construction of $ \gamma^{\sharp}_{n'+1,n} $ and the fact that $r_{n'+1,n}=r_{n',n}\circ F $, we will conclude that $\gamma^{\sharp}_{n'+1,n}$ is equal to $\gamma^{\sharp}_{n',n}$. Since the bottom square commute, it is enough to show that the top square is commutative as well. But the commutativity of the top square is equivalent to the commutativity of the following diagram \[
 \xymatrix{\BW/p^{n'+1}\ar@{->>}[d]\ar@{->>}[r]&\BW/F^{n'+1}\ar[rr]^{\gamma_{n'+1}^{(p^{-n'-1})}}&&V^{n'+1}G_{n'+1}\ar@{^{(}->}[rr]&&G_{n'+1}^{(p^{-n'-1})}\ar[d]^F\\\BW/p^{n'}\ar@{->>}[r]&\BW/F^{n'}\ar[rr]_{\gamma_{n'}^{(p^{-n'})}}&&V^{n'}G_{n'}\ar@{^{(}->}[r]&G_{n'}^{(p^{-n'})}\ar@{^{(}->}[r]&G_{n'+1}^{(p^{-n'})},}\]
whose commutativity is given by Lemma \ref{lem0 22}.
\end{proof}

\begin{notation}
\label{notations03}
With regard to the previous lemma, we will denote the map $\gamma^{\sharp}_{n',n}:\BW/p^n\to G_n$ by $\gamma^{\sharp}_n$.
\end{notation}

\begin{lem}
\label{lem0 12}
Assume that $G$ is local. Then the sequence of maps $(\gamma^{\sharp}_n)_n$ belongs to the group $\uset{i}{\invlim}\,\Hom(\BW/p^i,G_i)$.
\end{lem}

\begin{proof}
Fix a natural number $n$ and choose a large $n'$ such that the maps $ r_{n',n+1} $ and $ r_{n',n} $ are defined (cf. Lemma \ref{lem07}). Now consider the following diagram: \[ \xymatrix{G_{n'}\ar@{->>}[dr]_{p^{n'-n-1}}\ar[rr]^{V^{n'}}&&V^{n'}G_{n'}\ar[dl]^{r_{n',n+1}}&&\BW/F^{n'}\ar[ll]_{\gamma_{n'}^{(p^{-n'})}}&\BW/p^{n'}\ar@{->>}[d]\ar@{->>}[l]\\&G_{n+1}\ar@{->>}[d]_{p}&&&&\BW/p^{n+1}\ar@{->>}[d]\ar[llll]_{\gamma^{\sharp}_{n+1}}\\&G_n&&&&\BW/p^n\ar[llll]^{\gamma^{\sharp}_n}.} \] We need to show that the (bottom) square is commutative and we know that the upper left triangle and the top trapezoid are commutative (cf. construction of $ \gamma_{n+1}^{\sharp} $). By previous lemma, we also know that the outer diagram (i.e., after removing the homomorphism $\gamma_{n+1}^{\sharp}  $ from the diagram) is commutative, because $ p\circ r_{n',n+1}\circ V^{n'}:G_{n'}\to G_n$ is equal to $p^{n'-n} $ and therefore by uniqueness of $ r_{n',n} $ (cf. Remark \ref{rem0 4}), we have $ r_{n',n}=p\circ r_{n',n+1} $. The commutativity of the bottom square follows at once, since the homomorphism $ \BW/p^{n'}\onto \BW/p^{n+1} $ is an epimorphism.
\end{proof}

\begin{prop}
\label{prop017}
Assume that $G$ is local. Then the homomorphism $$ \bigvee: \uset{n}{\invlim}\,\Hom(\BW/p^n,G_n) \to \uset{n}{\invlim}\,\Hom(\BW/F^n,G_n) $$ is an isomorphism.
\end{prop}

\begin{proof}
We show at first the injectivity of this homomorphism. So, take an element $ (\alpha_n)\in \uset{n}{\invlim}\,\Hom(\BW/p^n,G_n) $ such that $ \bigvee(\alpha_n)=0 $, i.e., for every $n$, the composition $$ \BW/F^n\arrover{V^n}\BW/p^n\arrover{\alpha_n}G_n $$ is the zero homomorphism. The cokernel of the homomorphism $\BW/F^n\arrover{V^n}\BW/p^n$ being $ \BW/p^n\onto\BW/V^n $, we obtain, for every $n$, a homomorphism $ \tilde{\alpha_n}:\BW/V^n\to G_n $ whose composition with $ \BW/p^n\onto\BW/V^n $ is $ \alpha_n $. Now consider the following diagram (with $n>m$) \[ \xymatrix{\BW/p^{n}\ar@{->>}[d]\ar@{->>}[r]&\BW/V^n\ar@{->>}[d]\ar[r]^{\tilde{\alpha_n}}&G_n\ar@{->>}[d]_p\\\BW/p^m\ar@{->>}[r]&\BW/V^m\ar[r]_{\tilde{\alpha_m}}&G_m.} \] By assumption, the outer diagram commutes, and since the homomorphism $ \BW/p^n\onto\BW/V^n $ is an epimorphism and the left square is commutative, we conclude that the right square commutes too. It follows from  Lemma \ref{lem09} that all $ \tilde{\alpha_n} $ are zero, and consequently, all $ \alpha_n $ are zero. This shows that $ \bigvee $ is injective.\\

Now, we show the surjectivity of $ \bigvee $. Take an element $ (\gamma_n) \in \uset{n}{\invlim}\,\Hom(\BW/F^n,G_n) $. We have constructed above an element $ (\gamma_n^{\sharp}) \in \uset{n}{\invlim}\,\Hom(\BW/p^n,G_n)$ (cf. Lemma \ref{lem0 12}). We want to show that this element maps to $ (\gamma_n) $ under the homomorphism $ \bigvee $, i.e., we want to show that for every $n$, the composition $$ \BW/F^n\arrover{V^n}\BW/p^n\arrover{\gamma_n^{\sharp}} G_n $$ is equal to the given $\gamma_n$. Consider the following diagram:

\begin{myequation}
\label{twist}
\xymatrix{\BW/p^n\ar@{->>}[r]\ar[d]_{\gamma_n^{\sharp^{(p^n)}}}&\BW/F^n\ar[rr]^{V^n}\ar[dr]^{\gamma_n}&&\BW/p^n\ar[dl]^{\gamma_n^{\sharp}}\\ G_n^{(p^{n})}\ar[rr]_{V^n}&&G_n.&}
\end{myequation}

Since $ \gamma_n^{\sharp} $ is a group schemes homomorphism, it commutes with the Verschiebung and therefore the outer diagram commutes. If we show that the left square is commutative, the fact that the homomorphism $ \BW/p^n\onto \BW/F^n $ is an epimorphism, will imply that the right triangle commutes and the proof is achieved. So, we show the commutativity of the left square.\\

Using definition of $ r_{n',n} $ and the fact that $ F^{n'-n}\circ V^{n'}=p^{n'-n}V^{n} $, we see that the following diagram commutes. \[ \xymatrix{V^{n'}G_{n'}\ar@{^{(}->}[dd]\ar@{->>}[rr]^{r_{n',n}}&&G_n\ar[rr]^{V^n}&&G_n^{(p^{-n})}\ar@{^{(}->}[dd]\\&G_{n'}\ar@{->>}[ul]^{V^{n'}}\ar@{->>}[ur]^{p^{n'-n}}\ar[rr]^{V^n}\ar[dl]^{V^{n'}}&&G_{n'}^{(p^{-n})}\ar@{->>}[ur]^{p^{n'-n}}\ar[dr]_{p^{n'-n}}&\\G_{n'}^{(p^{-n'})}\ar[rrrr]_{F^{n'-n}}&&&&G_{n'}^{(p^{-n})}.} \] In the following diagram, the right square is the above diagram (without the internal groups and homomorphisms). The outer diagram is commutative by Lemma \ref{lem0 22}. It follows that the left square is commutative (note that $ G_n^{(p^{-n})}\into G_{n'}^{(p^{-n})} $ is a monomorphism): \[ \xymatrix{\BW/p^{n'}\ar@{->>}[r]\ar@{->>}[dd]&\BW/F^{n'}\ar[r]^{\gamma_{n'}^{(p^{-n'})}}&V^{n'}G_{n'}\ar@{^{(}->}[r]\ar[d]_{r_{n',n}}&G_{n'}^{(p^{-n'})}\ar[dd]^{F^{n'-n}}\\&&G_n\ar[d]_{V^n}&\\\BW/p^n\ar@{->>}[r]&\BW/F^n\ar[r]_{\gamma_{n}^{(p^{-n})}}&G_n^{(p^{-n})}\ar@{^{(}->}[r]&G_{n'}^{(p^{-n})}.} \] If we rewrite this diagram and insert the homomorphism $ \BW/p^n\arrover{\gamma_n^{\sharp}} G_n $, we obtain the following diagram: \[ \xymatrix{\BW/p^{n'}\ar@{->>}[r]\ar@{->>}[d]&\BW/F^{n'}\ar[rr]^{\gamma_{n'}^{(p^{-n'})}}&&V^{n'}G_{n'}\ar[d]^{r_{n',n}}\\\BW/p^n\ar[rrr]^{\gamma_n^{\sharp}}\ar@{->>}[d]&&&G_n\ar[d]^{V^n}&\\\BW/F^n\ar[rrr]_{\gamma_{n}^{(p^{-n})}}&&&G_n^{(p^{-n})}.} \] Since the outer diagram and the top one are commutative (cf. construction of $ \gamma_n^{\sharp} $), and $ \BW/p^{n'}\onto \BW/p^n $ is an epimorphism, we conclude that the bottom square commutes. Taking the Frobenius twist of this diagram, we obtain a commutative diagram, which is the left square of the diagram (\ref{twist}). This finishes the proof.
\end{proof}

\begin{rem}
\label{rem0 11}
Note that in the proof of the last proposition, we have constructed explicitly an inverse to $ \bigvee $, namely the construction $ (\_)^{\sharp} $, i.e., we have shown that for every element $ (\gamma_i)\in\uset{n}{\invlim}\,\Hom(\BW/p^n,G_n)$ and every $n$, we have $$ \big(\bigvee(\gamma_i)\big)_n^{\sharp}=\gamma_n.$$
\end{rem}

\begin{thm}
\label{thm03}
Let $G$ be a connected $p$-divisible group. The homomorphism $$\eta:D_*(G)\longrightarrow M(G)$$ is an isomorphism of $\BE_k$-modules.
\end{thm}

\begin{proof}
Since the base field $k$ is perfect, the $p$-divisible group $G$ splits into the direct sum of its local-local and  local-\'etale parts and since the Dieudonn\'e functor and the Cartier functor preserve direct sums, we can treat the local-local and local-\'etale case separately. So, assume at first that $G$ is local-local. It is enough to show that under the identifications of Lemmas \ref{lem04} and \ref{lem05},  the maps $ \eta $ (cf. Construction \ref{cons06}) and $ \bigvee $ (cf. Construction \ref{cons08}) are identified, i.e., the following diagram commutes: \[ \xymatrix{ D_*(G)\ar[d]_{\text{Lemma \ref{lem04}}}^{\cong}\ar[rr]^{\eta}&&M(G)\ar[d]^{\text{Lemma \ref{lem05}}}_{\cong}\\ \uset{n}{\invlim}\,\Hom(\BW/p^n,G_n)\ar[rr]_{\bigvee}&&\uset{n}{\invlim}\,\Hom(\BW/F^n,G_n). } \] Take an element $ u\in D_*(G) $. We have to show that for every natural number $n$, the following diagram commutes: \[ \xymatrix{\BW/F^n\ar[dd]_{V^n}\ar[r]^{\cong}&W[F^n]\ar@{->>}[r]&W_{m,n}\ar[dd]^{\eta(u)_n}\\&& \\\BW/p^n\ar[rr]_{[u]_n}&&G_n,} \]where we have also denoted the induced homomorphism $ \BW/p^n\to G_n $ by $ [u]_n $. We insert this diagram (as a face) into the following 3D diagram, with $ m $ and $M$ such that $ F^MG_n=0 $ and $ V^mG_n=0 $: \[ \xymatrix{&\BW/F^n\ar[dd]_{V^n}\ar[r]^{\cong}&W[F^n]\ar@{->>}[r]&W_{m,n}\ar[dd]^{\eta(u)_n}\\\BW\ar[dd]_{V^n}\ar@{->>}[ur]\ar@{->>}[rr]&&W_{m,M}\ar[dd]^{V^n}\ar[ur]_{F^{M-n}}&\\&\BW/p^n\ar[rr]^{[u]_n\qquad}&&G_n\\\BW\ar@{->>}[ur]\ar@{->>}[rr]&&W_{m,M}\ar[ur]_{[u]_n}.&} \] By construction of $ \eta(u)_n $ (cf. Construction \ref{cons08}) and using Lemmas \ref{lem04} and \ref{lem05}, we see that the five other faces of this parallelepiped commute. As the homomorphism $ \BW\onto \BW/F^n $ is an epimorphism, it implies that the other face is commutative too and this finishes the proof.\\

Now assume that $G$ is multiplicative. The homomorphism $ \eta $ is an isomorphism if and only if it is an isomorphism after passing to an algebraically closed field containing $k$. Indeed, this is true since the Dieudonn\'e module and the Cartier module constructions commute with base change. So, we may assume that $k$ is algebraically closed. In this case, the $p$-divisible group $G$ is isomorphic to a direct sum of finite copies of $ \mu_{p^{\infty}} $. Since the Dieudonn\'e functor and the Cartier functor preserve direct sums, we may further assume that $G$ is isomorphic to $ \mu_{p^{\infty}} $. By construction of $ \eta $ (cf. \ref{cons018}), we observe that it is sufficient to show that in this situation, for every $n$, the homomorphism $$ \Hom(G_n^*,CW[p^n])\arrover{F^n\circ(\_)}\Hom(G_n^*,CW[V^n]) $$ is an isomorphism. The group scheme $ G_n^* $ is isomorphic to the constant group scheme $ \BZ/p^n $, and therefore, homomorphisms $ \BZ/p^n\to CW $ are identified with $k$-rational points of $ CW $, i.e., elements of $ CW(k) $ and the above homomorphism $ F^n\circ(\_) $ is identified with the homomorphism $$ F^n:CW[p^n](k)\to CW[V^n](k) .$$ As $k$ is perfect, the Frobenius homomorphism of $ CW(k) $ is an isomorphism and therefore the induced homomorphism $$ F^n:CW[p^n](k)\to CW[V^n](k) $$ is injective. It is also surjective, because $ F^n:CW(k)\to CW(k) $ is surjective and $ p^n=V^nF^n $.
\end{proof}

\section{Cartier modules vs. Displays}

The following proposition is proposition 90, p. 84 of \cite{Z1}:

\begin{prop}
\label{prop09}
Let $ \mathcal{P}=(P,Q,F,V^{-1}) $ be a 3n-Display over a ring $ R $. There is a canonical surjection $$\BE_R\otimes_{W(R)}P\twoheadrightarrow M(BT_{\mathcal{P}}),\quad e\otimes x\mapsto (u\mapsto [ue\otimes x])$$ where $ \BE_R $ is the Cartier ring, the ring opposite to the ring $ \End(\widehat{W}) .$ The kernel of this morphism is the $ \BE_R$-submodule generated by the elements $ F\otimes x-1\otimes Fx $, for $ x\in P $, and $ V\otimes V^{-1}y-1\otimes y $, for $ y\in Q $.
\end{prop}

\begin{prop}
\label{prop010}
Let $\mathcal{P}$ be a Dieudonn\'e display over $k$. Then the following morphism is an isomorphism of $\BE_k$-modules: \[ \mu:P\longrightarrow M(BT_{\CP}), \] \[ x\mapsto (\xi\mapsto [\xi\otimes x]). \]
\end{prop}

\begin{proof}
Let us denote by $I$ the kernel of the morphism in the last proposition (with $R=k$). Then the morphism $$ (\BE_k\otimes_{W(k)}P) /I\to M(BT_{\CP}) $$ sending the class of $ e\otimes x $ modulo $I$ to the morphism $ (u\mapsto [ue\otimes x]) $, is an isomorphism. Now, we have a canonical morphism $$ \phi:P\to (\BE_k\otimes_{W(k)}P) /I $$ sending an element $x$ to $ [1\otimes x] $, the class of $ 1\otimes x $ in the quotient. The composition of this morphism with the above isomorphism is the morphism given in the statement of the proposition. So, it is enough to show that $ \phi $ is an isomorphism. Since by assumption $\CP$ is the 3$n$-display of a $p$-divisible group, it is endowed with a Verschiebung and so we can define a map $$ \psi:(\BE_k\otimes_{W(k)}P) /I \to P $$ by sending $ F^i\otimes x $ to $ F^ix $ and $ V^j\otimes y $ to $ V^jy $ with $i,j$ natural numbers and $x,y$ arbitrary elements of $P$. We claim that this is a well-defined homomorphism of $ \BE_k$-modules and is the inverse to the morphism $ \phi $, defined above. It follows from the definition that $ \psi $ is a homomorphism of $ \BE_k$-modules. Again, by definition, elements of the form $ F\otimes x -1\otimes Fx $ or $ V\otimes V^{-1}y-1\otimes y $ map to zero and since they generate the ideal $I$, we see that our map $ \psi $ , is well-defined.  It is clear that the composition $ \psi\circ \phi $ is the identity of $P$. As in the quotient, elements $ F^i\otimes x $ and $ 1\otimes F^ix $, respectively $ V^j\otimes y $ and $ 1\otimes V^jy $, are  identified, it follows that the composition $ \phi\circ \psi $ is the identity of $(\BE_k\otimes_{W(k)}P) /I$ and this finishes the proof.
\end{proof}

\begin{prop}
\label{prop021}
Let $G$ be a connected $p$-divisible group and denote by $ \CP $ the associated display. Consider the map \[ \chi:BT_{\CP}\to G,\qquad [w\otimes u]\mapsto \eta(u)(w) \] with $ \CN $ a nilpotent $k$-algebra, $ w\in \widehat{W}(\CN) $ and $ u\in D_*(G) $. This is a canonical and functorial isomorphism.
\end{prop}

\begin{proof}
As the Cartier functor is an equivalence of categories, it is sufficient to prove that $ \chi $ is an isomorphism after applying the Cartier functor $ M $ on it. Since the homomorphism $ \mu:D_*(G)\to M(BT_{\CP}) $ defined in the previous proposition is an isomorphism, it suffices to show that the composition $$ M(\chi)\circ\mu:D_*(G)\to M(G) $$ is an isomorphism. By construction, this homomorphism sends element $ u\in D_*(G) $ to the homomorphism $ \eta(u):\widehat{W}\to G $, in other words, this composition is the homomorphism $ \eta:D_*(G)\to M(G) $, which is an isomorphism by Theorem \ref{thm03}.
\end{proof}

\begin{rem}
\label{rem01}
$ $
\begin{itemize}
\item[1)] If we denote the composition \[\mu^{-1}\circ \eta:D_*(G)\arrover{\cong} M(G)\arrover{\cong} P\] by $ \theta $, then for all $u\in D_*(G)$, all $n$, all nilpotent $k$-algebra $ \CN $ and all $\xi\in \widehat{W}(\CN)$ annihilated by $ F^n $, we have \[ [\xi\otimes \theta(u)]=\mu(\theta(u))(\xi)=\eta(u)(\xi)= E(\xi\cdot \tau_n[F^n\theta(u)]_n(\_);1), \] where $  E(\xi\cdot \tau_n[F^n\theta(u)]_n(\_);1) $ is seen as an element of $ G_n(\CN) $ under the identification $\innHom(G_n^*,\BG_m)\cong G_n $ and we have used Lemma \ref{lem024} for the last equality.
\item[2)] If $G$ is local-local, and $ \xi $ belongs to $ W_{m,M}(\CN) $ with $m$ and $M$ such that $ F^MG_n=0 $ and $ V^mG_n=0 $, then we have $$[u]_n(V^n\xi)=[F^{M-n}\tau(\xi)\otimes\theta(u)]_n\in G_n(\CN).$$
\end{itemize}
\end{rem}

\chapter{The Main Theorem for $p$-Divisible Groups}

In this chapter, we prove that the exterior powers of \'etale $p$-divisible groups and $p$-divisible groups of dimension $1$ (at points of characteristic $p$) over arbitrary base exist and that the construction of the exterior power commutes with arbitrary base change.\\

\section{Technical results and calculations}

Throughout this section, unless otherwise specified, $k$ is a perfect field of characteristic $p$ and $G$ is a $p$-divisible group over $k$.

\begin{prop}
\label{prop0 15}
Assume that $G$ is connected. Then the action of the Verschiebung on the Dieudonn\'e module of $G$ is topologically nilpotent, i.e., for every $x\in D_*(G)$, we have $\uset{n\to\infty}{\lim}\, V^nx=0$.
\end{prop}

\begin{proof}
The topology on $ D:=D_*(G) $ is the $p$-adic topology. So we ought to prove that for every $ x\in D $ and every $ n\in \BN $, there exists an $ m_0\in\BN $ such that for all $ m\geq m_0 $, we have $ V^{m}x\in p^nD $. Since $ G $ is local, there exists a natural number $ m_0 $ such that $ G_n $ is annihilated by the Frobenius to the power $ m_0 $ and therefore $ V^{m_0}D_*(G_n)=0 $ (note that we are working with the covariant Dieudonn\'e theory). We have a short exact sequence \[ 0\to p^nD\to D\to D_*(G_n)\to0, \] which commutes with the action of Verschiebung. Since $ V^{m}D_*(G_n)=0 $ for all $ m\geq m_0 $, we have $ V^{m}D\subseteq p^nD $, which finishes the proof.
\end{proof}

\begin{lem}
\label{lem0 10}
Assume that $G$ is connected. Then for every natural number $n$, there exist a natural number $m$ such that $V^mD_*(G)\subset F^nD_*(G)$.
\end{lem}

\begin{proof}
From the previous lemma, we know that the action of Verschiebung is topologically nilpotent. It is therefore enough to show that for every natural number $n$, the submodule $ F^nD_*(G) $ is open in $D:= D_*(G) $. Indeed, if for every $ x\in D $, there exists an $ m $ such that $ V^mx\in F^nD $, since $ D $ is finitely generated over the ring of Witt vectors over the base field, then there exists an $ m$ such that $ V^mD\subseteq F^nD $ (and in fact, we showed in the proof of the previous lemma that there is an $m$ with $ V^mD\subseteq p^nD $). We have $ p^nD=V^nF^nD\subseteq F^nD $. Hence, $ F^nD $ is open in the $p$-adic topology.
\end{proof}

\begin{notation}
\label{notation0 5}
Assume that $G$ is local-local and that we are given an element $ u\in D_*(G) $. We denote by $ [u^{\flat}]_n $ the composition \[ V^n\BW\into \BW\arrover{[u]_n}G_n.\] Fix natural numbers $m$ and $M$ such that $ F^MG_n=0 $ and $ V^mG_n=0 $. The map $ [u]_n:\BW\to G_n $ factors through the quotient $ \BW\onto W_{m,M} $, therefore, the map $ [u^{\flat}]_n $, too, factors through the quotient $ V^n\BW\onto V^nW_{m,M} $ and since $G_n$ is annihilated by $p^n$, this map factors through the quotient $ V^nW_{m,M}\onto V^nW_{m,M}/p^n $, and by abuse of notation, we denote the resulting maps $ V^nW_{m,M}\to G_n $ and $ V^nW_{m,M}/p^n\to G_n $ by $ [u^{\flat}]_n $ as well, and we have the following commutative diagrams:\[ \xymatrix{V^n\BW\ar[r]\ar@/^4pc/[drr]^{[u^{\flat}]_n}\ar@{->>}[d]&\BW\ar@{->>}[d]\ar[dr]^{[u]_n}&\\ V^nW_{m,M}\ar@/_2pc/[rr]_{[u^{\flat}]_n}\ar[r]&W_{m,M}\ar[r]_{[u]_n}&G_n} \] and \[ \xymatrix{V^n\BW\ar@{->>}[r]\ar[dr]_{_{[u^{\flat}]_n}}&V^nW_{m,M}\ar@{->>}[r]\ar[d]_{_{[u^{\flat}]_n}}&V^nW_{m,M}/p^n\ar[dl]^{_{[u^{\flat}]_n}}\\&G_n.} \] Note also that by the functoriality of Verschiebung, the image of the map $ [u^{\flat}]_n $ lies inside the subgroup $ V^{n}G_n^{(p^n)}. $ And once again, by abuse of notation, we will also denote by $ [u^{\flat}]_n $ the induced map going to $ V^{n}G_n^{(p^n)}. $
\end{notation}

\begin{rem}
\label{rem0 12}
Assume that $G$ is local-local and we are given an element $ u\in D_*(G) $. Keeping the above notations, we have the following commutative diagram: \[ \xymatrix{W_{m,M}\ar@{->>}[rr]^{F^{M-n}}\ar@{->>}[dr]^{\cong}\ar[dd]_{V^n}&&W_{m,M}\ar@{^{(}->}[rr]\ar[dd]^{\eta(u)_n}&&\widehat{W}\ar[dd]^{\eta(u)}\\ &V^nW_{m,M}\ar[dr]^{[u^{\flat}]_n}\ar@{^{(}->}[dl]&&&\\ W_{m,M}\ar[rr]_{[u]_n}&&G_n\ar@{^{(}->}[rr]&&G.} \]
\end{rem}

\begin{lem}
\label{lem0 17}
Assume that $G$ is local-local and we are given an element $ u\in D_*(G) $. Then for every $n$ and every $m$ and $M$ with $ F^MG_n=0 $ and $ V^mG_n=0 $, the following diagram is commutative, where $ \bigvee(u)_n $ denotes the image of $ \bigvee(u)$ under the map $\uset{i}{\invlim}\,\Hom(\BW/F^i,G_i) \to \Hom(\BW/F^n,G_n) $: \[ \xymatrix{V^n\BW/p^n\cong \BW/F^n\ar[rr]^{\qquad \bigvee(u)_n}\ar[d]&& V^{n}G_n^{(p^n)}\\ V^nW_{m,M}/p^n\ar[urr]_{[u^{\flat}]_n}.} \]
\end{lem}

\begin{proof}
It is enough to show that the two homomorphisms $$ V^n\BW/p^n\arrover{\bigvee(u)_n} V^{n}G_n^{(p^n)}$$ and $$ V^n\BW/p^n\onto V^nW_{m,M}/p^n\arrover{[u^{\flat}]_n} V^{n}G_n^{(p^n)} $$ are equal after composing with the inclusion $ V^{n}G_n^{(p^n)}\into G_n $. Since the composition of $[u^{\flat}]_n  $ with this inclusion is equal to the composition of the inclusion $ V^n\BW/p^n\into\BW/p^n $ with the homomorphism $ [u]_n $, we are reduced to show that the following diagram commutes: \[ \xymatrix{V^n\BW/p^n\ar@{->>}[d]\ar[r]^{\bigvee(u)_n}&G_n\\ V^nW_{m,M}/p^n\ar@{^{(}->}[r]&W_{m,M}/p^n\ar[u]^{[u]_n}.} \] We know that the composition $$ V^n\BW/p^n\into \BW/p^n\arrover{[u]_n} G_n $$ is equal to $ \bigvee(u)_n $ and since homomorphisms $ [u]_n $ and respectively $ \bigvee(u)_n $ factor through $ W_{m,M}/p^n $ and respectively $ V^nW_{m,M} $, we obtain the commutativity of the above diagram as desired.
\end{proof} 
 
\begin{lem}
\label{lem0 14}
Assume that $G$ is local-local and we are given an element $ u\in D_*(G) $. Then for every $n$, $n'\gg n$ and $m,M$ with $ F^MG_{n'}=0 $ and $ V^mG_{n'}=0 $, the following diagram is commutative: \[ \xymatrix{W_{m,M}\ar@{->>}[rr]^{V^{n'}}\ar[drr]_{[u]_n}&&V^{n'}W_{m,M}\ar[rr]^{[u^{\flat}]_{n'}^{(p^{-n'})}}&&V^{n'}G_{n'}\ar@{->>}[dll]^{r_{n',n}}\\ && G_n&&G_{n'}\ar@{->>}[ll]^{p^{n'-n}}\ar@{->>}[u]_{V^{n'}}.} \]
\end{lem}

\begin{proof}
It is enough to show that this diagram commutes after composing it from left with the epimorphism $ \BW\onto W_{m,M} $ (the right triangle commutes by definition of $ r_{n',n} $) and since the homomorphisms $ [u]_n $ and $[u^{\flat}]_{n'}^{(p^{-n'})}  $ land in $ G_{n'} $, we can replace $ W_{m,M} $ and respectively $ V^{n'}W_{m,M} $ with $ \BW/p^{n'} $ and respectively  $ V^{n'}\BW/p^{n'} $. From the last lemma we know that the composition $$ V^{n'}\BW/p^{n'}\onto V^{n'}W_{m,M}/p^{n'}\arrover{[u^{\flat}]_{n'}^{(p^{-n'})}} V^{n'}G_{n'} $$ is equal to $ \bigvee(u)_{n'}^{(p^{-n'})} $ and therefore, we are reduced to show that the following diagram commutes: \[ \xymatrix{\BW/p^{n'}\ar[r]^{V^{n'}}\ar[d]_{[u]_n}&V^{n'}\BW/p^{n'}\ar[d]^{\bigvee(u)_{n'}^{(p^{-n'})}}\\G_n&V^{n'}G_{n'}\ar[l]^{r_{n',n}}.} \] The commutativity of this diagram follows from Remark \ref{rem0 11} or Proposition \ref{prop017}.
\end{proof}  
 
\begin{lem}
\label{lem0 15}
Let $\CP$ be a Dieudonn\'e display over $k$ and write $G$ for the $p$-divisible group $ BT_{\CP} $. Given an element $ u\in D_*(G)$, a nilpotent $k$-algebra $\CN$ and  an element $ \xi\in \BW(\CN) $, we have for all $n'\gg n$ and $m,M$ with $ F^MG_{n'}=0 $ and $ V^mG_{n'}=0 $:
\begin{itemize}
\item[(i)] $ [u^{\flat}]_{n'}^{(p^{-n'})}(V^{n'}\xi)=[\bar{\xi}\otimes 1\otimes V^{M-n'}\theta(u)]_{n'} $ and
\item[(ii)] $ [u]_n(\xi)=p^{n'-n}[\bar{\xi}\otimes z(u)]_{n'}=[p^{n'-n}\bar{\xi}\otimes z(u)]_{n}\in G_{n}(\CN),$
\end{itemize}
where $ \bar{\xi} $ is the image of $ \xi $ under the composition $$ \BW(\CN) \onto W_{m,M}(\CN)\into \widehat{W}(\CN) $$ and  $ z(u)\in P $ is any element such that $ F^{n'}z(u)=V^{M-n'}\theta(u) $.
\end{lem}

\begin{proof}
\begin{itemize}
\item[(i)] By Remark \ref{rem01}, we have $$ [u^{\flat}]_{n'}(V^{n'}\xi)=[u]_{n'}(V^{n'}\xi)=[F^{M-n'}\bar{\xi}\otimes \theta(u)] $$ and since we want to calculate the twisted homomorphism $ [u^{\flat}]_{n'}^{(p^{-n'})} $ on $ (V^{n'}\xi) $, we obtain $$ [u^{\flat}]_{n'}^{(p^{-n'})}(V^{n'}\xi)=[u]_{n'}^{(p^{-n'})}(V^{n'}\xi)=[F^{M-n'}\bar{\xi}\otimes 1\otimes  \theta(u)] $$ (cf. Construction \ref{cons02}). Finally, by Remark \ref{rem0 7}, we know that \[ [F^{M-n'}\bar{\xi}\otimes 1\otimes  \theta(u)] = [\bar{\xi}\otimes 1\otimes V^{M-n'}\theta(u)]\] and we obtain the desired equality \[ [u^{\flat}]_{n'}^{(p^{-n'})}(V^{n'}\xi)=[\bar{\xi}\otimes 1\otimes V^{M-n'}\theta(u)]_{n'}. \]
\item[(ii)] By the last equality and the previous lemma, we have \[ [u]_n(\xi)=r_{n',n}\big([u^{\flat}]_{n'}^{(p^{-n'})}(V^{n'}\xi)\big)=r_{n',n}\big([\bar{\xi}\otimes 1\otimes V^{M-n'}\theta(u)]_{n'}\big). \] In order to calculate the latter, we have to find an element in $ G_{n'}(\CN) $ such that when we apply $ V^{n'} $ on it, we obtain the element $ [\bar{\xi}\otimes 1\otimes V^{M-n'}\theta(u)]$, and then $ r_{n',n}\big([\bar{\xi}\otimes 1\otimes V^{M-n'}\theta(u)]_{n'}\big) $ will be $ p^{n'-n} $ times that element. We claim that $ [\bar{\xi}\otimes z(u)]_{n'} $ is such an element. Indeed, we have by the construction of Verschiebung on $ BT_{\CP} $ (cf. Construction \ref{cons02}) that \[ V^{n'}[\bar{\xi}\otimes z(u)]=Ver_{\CP}[\bar{\xi}\otimes z(u)]=[\bar{\xi}\otimes 1\otimes F^{n'}z(u)]=[\bar{\xi}\otimes1\otimes V^{M-n'}\theta(u)]. \] Hence the equality \[[u]_n(\xi)=r_{n',n}\big([\bar{\xi}\otimes 1\otimes V^{M-n'}\theta(u)]_{n'}=p^{n'-n}[\bar{\xi}\otimes z(u)]_{n'}=[p^{n'-n}\bar{\xi}\otimes z(u)]_{n}. \]
\end{itemize}
\end{proof}

\begin{rem}
\label{rem0 5}
Let us use the notations of the previous lemma. From the construction of $ V^{-1} $ given in Construction \ref{cons012}, we have the following equalities inside the module $ \widehat{P}$:  $$(V^{-1}-\Id)(-\sum_{i=0}^{M-1}F^{i}\bar{\xi}\otimes V^{M-i}\theta(u))=$$ \[\sum_{i=0}^{M-1}F^{i}\bar{\xi}\otimes V^{M-i}\theta(u)-\sum_{i=0}^{M-1}F^{i+1}\bar{\xi}\otimes V^{M-i-1}\theta(u)=\] \[ \bar{\xi}\otimes V^M\theta(u) - F^M\bar{\xi}\otimes \theta(u)= \bar{\xi}\otimes V^M\theta(u),\] where the latter equality follows from the fact that $ F^M\bar{\xi}=0$. Further, we have $$ \bar{\xi}\otimes V^M\theta(u)= \bar{\xi}\otimes V^{n'}V^{M-n'}\theta(u)=\bar{\xi}\otimes V^{n'}F^{n'}z(u)=\bar{\xi}\otimes p^{n'}z(u)= p^{n'}\bar{\xi}\otimes z(u).$$ It follows that 
\begin{myequation}
\label{g}
\leftidx{_n}{g}{_{\CP}}(p^{n'-n}\bar{\xi}\otimes z(u))= -\sum_{i=0}^{M-1}F^{i}\bar{\xi}\otimes V^{M-i}\theta(u).
\end{myequation}
\end{rem}

\begin{lem}
\label{lem0 19}
Let $\mathcal{P}_0,\mathcal{P}_1,\dots, \mathcal{P}_r$ be Dieudonn\'e displays over $k$ and \[\phi:\mathcal{P}_1\times\dots\times\mathcal{P}_r\to\mathcal{P}_0\] a multilinear morphism satisfying the $ V$-$F $ conditions. Fix natural numbers $N$ and $M$ and a vector $ (d_1,\dots, d_r)\in \BN^r $. Assume that for all $ i=1,\dots,r $ we have elements $ y_i, z_i\in P_i $ such that $ F^{N}z_i=V^My_i $. Then  \[F^N\big(p^{(r-1)N}\phi(V^{d_1}z_1,\dots,V^{d_r}z_r)\big)=V^M\phi(V^{d_1}y_1,\dots,V^{d_r}y_r).\]
\end{lem}

\begin{proof}
Set $ z:= p^{(r-1)N}\phi(V^{d_1}z_1,\dots,V^{d_r}z_r)$. We have \[ V^N(F^Nz)=p^Nz=p^{rN}\phi(V^{d_1}z_1,\dots,V^{d_r}z_r)=\phi(p^NV^{d_1}z_1,\dots,p^NV^{d_r}z_r)= \] \[ \phi(V^{d_1}V^NF^Nz_1,\dots,V^{d_r}V^NF^Nz_r)=\phi(V^{d_1}V^NV^My_1,\dots,V^{d_r}V^NV^My_r)= \] \[ V^N(V^M\phi(V^{d_1}y_1,\dots,V^{d_r}y_r)), \] where the third equality follows from the fact that $ \phi $ is multilinear and the last one from the fact that $ \phi $ satisfies the $V$ condition. Since $ V:P_0\to P_0 $ is injective, it follows that $ F^Nz=V^M\phi(V^{d_1}y_1,\dots,V^{d_r}y_r). $
\end{proof}

\begin{cons}
\label{cons011}
Let $\mathcal{P}_0, \mathcal{P}_1,\dots, \mathcal{P}_r$ be Dieudonn\'e displays over $k$ and \[\phi:\mathcal{P}_1\times\dots\times\mathcal{P}_r\to\mathcal{P}_0\] a multilinear morphism satisfying the $ V$-$F $ conditions. For all $ 0\leq i\leq r $, set $G_i:=BT_{\mathcal{P}_i}$. The map $ \phi $ induces a multilinear map $ P_1\times\dots\times P_r\to P_0/p^n$ and since it is linear in each factor, we obtain a multilinear map $$  P_1/p^n\times\dots\times P_r/p^n\to P_0/p^n .$$ As $ P_i/p^n\cong D_*(G_{i,n}) $, we have a $ V$-$F $ multilinear map $$\tilde{\phi_n}: D_*(G_{1,n})\times\dots\times D_*(G_{r,n})\to D_*(G_{0,n}) $$ i.e., an element of the group $ L(D_*(G_{1,n})\times\dots\times D_*(G_{r,n}), D_*(G_{0,n})) $ which is isomorphic  to the group $ \Mult(G_{1,n}\times\dots\times G_{r,n},G_{0,n}) $ by corollary \ref{cor03}. Hence, we obtain a multilinear map \[ \nabla^{-1}\circ\Delta (\tilde{\phi_n}) :G_{1,n}\times\dots\times G_{r,n}\to G_{0,n}.\] where we have abbreviated $ \Delta_{(G_{1,n},\dots,G_{r,n};G_{0,n})} $ and respectively $ \nabla_{(G_{1,n},\dots,G_{r,n};G_{0,n})} $ to $ \Delta $ and respectively $ \nabla $.
\end{cons}

\begin{cons}
\label{cons014}
Let us fix a positive natural number $M$. We set $ \CS_{i,r}:=\lbb1,M-1\rbb^{i-1}\times \{0\}\times \lbb0,M-1\rbb^{r-i-1}\subset \BN^r  $. Then the sets $ \CS_{i,r} $ and $\CS_{j,r}$ are disjoint if $ i\neq j $ and their union is the set $\BZ^{r}_{0,<M}$. We define a map $ \udel:\BZ^{r}_{0,<M}\to \BZ^{r}_{0,<M}$ as follows. Take an element $ \ul{d}=(d_1,\dots,d_r)\in \BZ^r_{0,<M} $ and set $d:=\max{\ul{d}}$. Define $\udel(\ul{d}):=(d-d_1,\dots, d-d_r)$.
\end{cons}

\begin{lem}
\label{lem0 20}
The map $ \udel $ is well-defined and is an involution, i.e., is its own inverse.
\end{lem}

\begin{proof}
We show at first that this map is well-defined, i.e., we show that $ \udel(\ul{d})\in \BZ^r_{0,<M}$ (for all $ \ul{d} $). Take an element $ \ul{d} $ in $\BZ^{r}_{0,<M}$ and set $ d:=\max\ul{d} $. As $ d$ is the maximum of all $ d_j $ and it is smaller than $ M $ it follows that all components of $\udel(\ul{d})$ are in $ \lbb0,M-1\rbb $ and at least one of them is zero. This shows that $ \udel $ is well-defined.\\

Now we show the second statement. Since the set $ \BZ^{r}_{0,<M} $ is finite, it is enough to show that the composition $ \udel\circ \udel $ is the identity of $ \BZ^{r}_{0,<M} $. So, take an element $ \ul{d} $ and let $ d $ be the maximum of the $ d_j $. Since at least of the $ d_j $ is zero, the maximum of the vector $ \udel(\ul{d})=(d-d_1,\dots,d-d_r) $ is again equal to $d$ and thus $$ \udel(\udel(\ul{d}))=(d-(d-d_1),\dots, d-(d-d_r))=(d_1,\dots, d_r)=\ul{d} .$$
\end{proof}

\begin{thm}
\label{thm04}
Let $\mathcal{P}_0, \mathcal{P}_1,\dots, \mathcal{P}_r$ be (nilpotent) displays over $k$ and \[\phi:\mathcal{P}_1\times\dots\times\mathcal{P}_r\to\mathcal{P}_0\] a multilinear morphism satisfying the $ V$-$F $ conditions and set $ G_i:=BT_{\CP_i} $. Then the two morphisms $ \nabla^{-1}\circ\Delta (\tilde{\phi_n})  $ and $ \beta_{\phi,n} $ are equal.
\end{thm}

\begin{proof}
If $r=1$, then $ \beta_{\phi,n} $ is the restriction of $ BT_{\phi}:BT_{\CP_1}\to BT_{\CP_0} $ to a morphism $ BT_{\CP_1,n} \to BT_{\CP_0,n} $ and this is theorem is just a restatement of the Theorem \ref{thm02}, which states that the functor $ BT $ is an equivalence of categories.\\

So, we assume that $ r\geq 2 $. The $p$-divisible groups $ G_i$ are connected, because they correspond to nilpotent displays. Assume that there exists an $i\in\lbb1,r\rbb$ such that $ G_i $ is of multiplicative type. Then, by lemma 4.5.6, p.51 of \cite{P}, for all positive natural numbers $n$, the group $ \Mult(G_{1,n}\times\dots\times G_{r,n},G_{0,n}) $ is the trivial group. Indeed, we have \[\Mult(G_{1,n}\times\dots\times G_{r,n},G_{0,n})\cong \Mult(G_{1,n}\times\dots\times G_{r,n}\times G_{0,n}^*,\BG_m),\] which is the trivial group by the aforementioned lemma. As the two morphisms $ \nabla^{-1}\circ\Delta (\tilde{\phi_n})  $ and $ \beta_{\phi,n} $ belong to the group $ \Mult(G_{1,n}\times\dots\times G_{r,n},G_{0,n}) $, they are both the zero morphism and hence equal. We can therefore assume that for every $ 1\leq i\leq r $, the $p$-divisible group $ G_i$ has no multiplicative part and therefore has connected dual. We denote by $ D_i $ (respectively by $ D_{i,n} $) the Dieudonn\'e module of $ G_i $ (respectively of $ G_{i,n)} $. Fix a positive natural number $n$ and choose $ n'\gg n $ such that $ r_{n',n}:V^{n'}G_{i,n'}\to G_{i,n} $ is defined for every $ i=1,\dots,r $ (cf. Lemma \ref{lem07}). Also fix $M>n'$ and $ m\geq M $ such that the group schemes $ G_{1,n'}\dots,G_{r,n'} $ are annihilated by $ F^M $ and $ V^m $ and for every $ i=0,\dots,r $ we have $ V^{M-n'}D_i\subseteq F^{n'}D_i $ (cf. Lemma \ref{lem0 10}) and $ F^MG_{0,n}=0 $. We prove that the two maps $ \Delta (\tilde{\phi_n}) $ and $ \nabla(\beta_{\phi,n}) $, from $ D_{1,n}\times \dots\times D_{r,n} $ to $ \Mult(\BW^r,G_{0,n}) $, are equal. Take for every $ i=1,\dots, r $, arbitrary elements $ u_i\in D_i $,  $ \xi^{(i)}\in \BW $ and chose $ z_i\in P_i $ such that $ F^{n'}z_i=V^{M-n'}\theta(u_i) $ (cf. Lemma \ref{lem0 10}). For every $ j\in\BN $, denote by $ \xi^{(i)}_j $ the projection of $ \xi^{(i)} $ under $ \pi_j:\BW\onto W[F^j] $ and by $\overline{\xi}_j^{(i)} $ its projection under $$\BW\ontoover{\pi_j}W[F^j]\ontoover{r_m} W_{m,j} .$$ So, for every $ s\leq j $, we have

\begin{myequation}
\label{eq2}
F^s \xi^{(i)}_j= \xi^{(i)}_{j-s},\quad F^s \overline{\xi}_j^{(i)} = \overline{\xi}_{j-s}^{(i)}\quad \text{ and }\quad F^s \overline{\xi}_s^{(i)}=0.
\end{myequation}

Set $ z_{\ul{d}}:=p^{(r-1)n'}\phi(V^{d_1}z_1,\dots,V^{d_r}z_r) $, $ \theta_i:=\theta(u_i) $ and $ g_i:=\leftidx{_n}{g}{_{\CP_i}}(p^{n'-n}\overline{\xi}_M^{(i)}\otimes z_i). $

Using Lemma \ref{lem0 19} we know that

\begin{myequation}
\label{eq1}
F^{n'}z_{\ul{d}}=V^{M-n'}\phi(V^{d_1}\theta_1,\dots,V^{d_r}\theta_r).
\end{myequation}

By Remark \ref{rem0 5} and using the equations \eqref{eq2}, we have

\begin{myequation}
\label{eq4}
g_i= -\sum_{j=0}^{M-1}F^{j}\overline{\xi}_M^{(i)}\otimes V^{M-j}\theta_i = -\sum_{\delta_i=0}^{M-1}\overline{\xi}_{M-\delta_i}^{(i)}\otimes V^{M-\delta_i}\theta_i
\end{myequation}

and 

\begin{myequation}
\label{eq5}
V^{-1}g_i= -\sum_{j=0}^{M-1}F^{j+1}\overline{\xi}_{M-j}^{(i)}\otimes V^{M-j-1}\theta_i = -\sum_{\delta_i=1}^{M-1}\overline{\xi}_{M-\delta_i}^{(i)}\otimes V^{M-\delta_i}\theta_i.
\end{myequation}
 
Let us at first calculate $ \nabla:=\nabla(\beta_{\phi,n})([u_1]_n,\dots,[u_r]_n)(\xi^{(1)},\dots,\xi^{(r)}) $. By definition, this is equal to \[ \beta_{\phi,n}\big([u_1]_n(\xi^{(1)}),\dots,[u_r]_n(\xi^{(r)})\big)=  \beta_{\phi,n}\big([p^{n'-n}\overline{\xi}_M^{(1)}\otimes z_1],\dots,[p^{n'-n}\overline{\xi}_M^{(r)}\otimes z_r]\big),\] where we are using Lemma \ref{lem0 15} for the last equality. Now, by definition of $ \beta_{\phi,n} $, the latter is equal to \[ (-1)^{r-1}\sum_{i=1}^{r}\big[\widehat{\phi}(V^{-1}g_1,\dots,V^{-1}g_{i-1},p^{n'-n}\overline{\xi}_M^{(i)}\otimes z_i,g_{i+1},\dots,g_r)\big]= \] \[ (-1)^{r-1}\sum_{i=1}^{r}p^{n'-n}\big[\widehat{\phi}(V^{-1}g_1,\dots,V^{-1}g_{i-1},\overline{\xi}_M^{(i)}\otimes z_i,g_{i+1},\dots,g_r)\big].\] Using identities \eqref{eq4} and \eqref{eq5}, this sum becomes: \[\sum_{i=1}^{r}p^{n'-n}\bigg[\widehat{\phi}\big(\sum_{\delta_1=1}^{M-1}\overline{\xi}_{M-\delta_i}^{(1)}\otimes V^{M-\delta_1}\theta_1,\dots,\sum_{\delta_{i-1}=1}^{M-1}\overline{\xi}_{M-\delta_{i-1}}^{(i-1)}\otimes V^{M-\delta_{i-1}}\theta_{i-1},\overline{\xi}_M^{(i)}\otimes z_i,\] \[\sum_{\delta_{i+1}=0}^{M-1}\overline{\xi}_{M-\delta_{i+1}}^{(i+1)}\otimes V^{M-\delta_{i+1}}\theta_{i+1},\dots,\sum_{\delta_{r}=0}^{M-1}\overline{\xi}_{M-\delta_{r}}^{(r)}\otimes V^{M-\delta_{r}}\theta_{r} \big)\bigg] = \]
\begin{myequation}
\label{eq6}
\sum_{i=1}^{r}p^{n'-n}\bigg[\sum_{\ul{\delta}\in\CS_{i,r}}\overline{\xi}_{M-\delta_{1}}^{(1)}\dots\overline{\xi}_{M-\delta_{r}}^{(r)}\otimes\phi(V^{M-\delta_1}\theta_1,\dots,\uset{\uparrow}{z_i},\dots,V^{M-\delta_{r}}\theta_{r})\bigg] ,
\end{myequation}

where the vertical arrow under $z_i$ is to emphasize that the $i^{\text{th}}$-entry doesn't follow the pattern of the other entries (cf. notations at the beginning of the thesis). We claim that this sum is equal to the following sum:

\begin{myequation}
\label{eq3}
\sum_{\ul{d}\in\BZ^r_{0,<M}}p^{n'-n}\big[\overline{\xi}_{M+d_1}^{(1)}\dots\overline{\xi}_{M+d_r}^{(r)}\otimes p^{(r-1)n'}\phi(V^{d_1}z_1,\dots,V^{d_r}z_r)\big].
\end{myequation}

We know by Lemma \ref{lem0 20} that the two index sets of these sums are in bijection and we want to show that in fact, under the bijection given in the aforementioned lemma, the corresponding summands are equal. Take an index $ \ul{d}\in\BZ^r_{0,<M} $ and assume that $ \ul{\delta}:=\udel(\ul{d}) $ belongs to $ \CS_{i,r} $ (i.e., $ d_i $ is the first maximum occurring in $ \ul{d} $). In the summand corresponding to the index $\ul{d}$ of the sum \eqref{eq3}, using the multilinearity of $ \phi $, distribute the factor $ p^{(r-1)n'} $ into $ \phi $, except at the $i^{\text{th}}$-place. The term $$\big[\overline{\xi}_{M+d_1}^{(1)}\dots\overline{\xi}_{M+d_r}^{(r)}\otimes p^{(r-1)n'}\phi(V^{d_1}z_1,\dots,V^{d_r}z_r)\big] $$  becomes: \[ \big[\overline{\xi}_{M+d_1}^{(1)}\dots\overline{\xi}_{M+d_r}^{(r)}\otimes \phi(V^{d_1}p^{n'}z_1,\dots,\uset{\uparrow}{V^{d_i}z_{i}},\dots,V^{d_r}p^{n'}z_r)\big] .\] Writing $ p^{n'} $ as $ V^{n'}F^{n'} $ and using the identity $ F^{n'}z_{j}=V^{M-n'}\theta_j $, this term becomes: \[\big[\overline{\xi}_{M+d_1}^{(1)}\dots\overline{\xi}_{M+d_r}^{(r)}\otimes \phi(V^{d_1+M}\theta_1,\dots,\uset{\uparrow}{V^{d_i}z_{i}},\dots,V^{d_r+M}\theta_r)\big].\] As by assumption $ \phi $ satisfies the $ V$-$F$ conditions, we can factor out $ V^{d_i} $ and using Remark \ref{rem0 7}, we obtain the term \[\big[F^{d_i}\big(\overline{\xi}_{M+d_1}^{(1)}\dots\overline{\xi}_{M+d_r}^{(r)}\big)\otimes \phi(V^{M-(d_i-d_1)}\theta_1,\dots,\uset{\uparrow}{z_{i}},\dots,V^{M-(d_i-d_r)}\theta_r)\big].\] Now, using the fact that Frobenius is a ring homomorphism and the second equality of \eqref{eq2}, we obtain \[ \big[\overline{\xi}_{M-(d_i-d_1)}^{(1)}\dots\overline{\xi}_{M-(d_i-d_r)}^{(r)}\otimes \phi(V^{M-(d_i-d_1)}\theta_1,\dots,\uset{\uparrow}{z_{i}},\dots,V^{M-(d_i-d_r)}\theta_r)\big]. \] By definition of $ \udel(\ul{d}) $, we have $ d_i-d_j=\delta_j $ and therefore, this term is equal to \[ \big[\overline{\xi}_{M-\delta_1}^{(1)}\dots\overline{\xi}_{M-\delta_r}^{(r)}\otimes \phi(V^{M-\delta_1}\theta_1,\dots,\uset{\uparrow}{z_{i}},\dots,V^{M-\delta_r}\theta_r)\big]. \] This term multiplied by $ p^{n'-n} $ is exactly equal to the the summand corresponding to $ \ul{\delta}\in\CS_{i,r} $ in the sum \eqref{eq6} and thus \eqref{eq3} and \eqref{eq6} are equal. This proves the claim. Thus the element $ \nabla $ of $ G_{0,n}(\CN) $ is equal to \[ \sum_{\ul{d}\in\BZ^r_{0,<M}}p^{n'-n}\big[\overline{\xi}_{M+d_1}^{(1)}\dots\overline{\xi}_{M+d_r}^{(r)}\otimes z_{\ul{d}}\big]=\sum_{\ul{d}\in\BZ^r_{0,<M}}\big[p^{n'-n}\overline{\xi}_{M+d_1}^{(1)}\dots\overline{\xi}_{M+d_r}^{(r)}\otimes z_{\ul{d}}\big].\] Set $ w_{\ul{d}}:=\xi_{M+d_1}^{(1)}\dots\xi_{M+d_r}^{(r)}$ and let $  \overline{w}_{\ul{d}}$ be its image under the morphism $$ W\onto W_{m,M}\into \widehat{W} .$$ Set also $\phi(\theta_{\ul{d}}):= \phi(V^{d_1}\theta_1\dots,V^{d_r}\theta_r) $ and $ N:=M-n'+n$. As $ F^M\overline{w}_{\ul{d}}=0 $, we have that $ F^N(p^{n'-n}\overline{w}_{\ul{d}})=0 $, and thus, by Remark \ref{rem01}, the element $ [p^{n'-n}\overline{w}_{\ul{d}}\otimes z_{\ul{d}}] $ is equal to \[E(p^{n'-n}\overline{w}_{\ul{d}}\cdot \tau_N[F^Nz_{\ul{d}}]_N(\_);1)=E(\overline{w}_{\ul{d}}\cdot p^{n'-n}\tau_N[F^Nz_{\ul{d}}]_N(\_);1)=\] \[ E(\overline{w}_{\ul{d}}\cdot V^{n'-n}\tau_N[F^{n'-n+N}z_{\ul{d}}]_N(\_);1)=E(\overline{w}_{\ul{d}}\cdot V^{n'-n}\tau_N[F^{M-n'}F^{n'}z_{\ul{d}}]_N(\_);1) \ovset{\eqref{eq1}}{=}\] \[ E(\overline{w}_{\ul{d}}\cdot V^{n'-n}\tau_N[F^{M-n'}V^{M-n'}\phi(\theta_{\ul{d}})]_N(\_);1) =\]
\begin{myequation}
\label{eq7}
E(\overline{w}_{\ul{d}}\cdot V^{n'-n}\tau_N[p^{M-n'}\phi(\theta_{\ul{d}})]_N(\_);1).
\end{myequation}

We claim that $V^{n'-n}\tau_N[p^{M-n'}\phi(\theta_{\ul{d}})]_N(\_)$ and $\tau_M[\phi(\theta_{\ul{d}})]_n(\_)$ are equal as morphisms $ G_{0,n}^*\to W  $ (note that the former is a morphism $ G_{0,N}^*\to W $, and we are restricting it to the subgroup scheme $ G_{0,n}^* $). It is enough to show that the compositions of these two morphisms with the projection $ \pi:G_{0,n+1}^*\onto G_{0,n}^* $ are equal. Take a section $ g $ of $ G_{0,N}^* $. The element $ x:=[\phi(\theta_{\ul{d}})]_n(\pi(g)) $ belongs to $ W_M $, because by assumption, $ F^MG_{0,n}=0 $. The element $ y:=[p^{M-n'}\phi(\theta_{\ul{d}})]_N(g) $ belongs to $ W_N $ and we know that $x$ and $y$ are equal as elements in $ CW^u $. Thus, $ V^{n'-n}y=V^{M-N}y=x $ and so $ V^{n'-n}\tau_N(y)=\tau_M(x) $. This proves the claim. It follows from the claim and equation \eqref{eq7} that $[p^{n'-n}\overline{w}_{\ul{d}}\otimes z_{\ul{d}}] $ is equal to $E(\overline{w}_{\ul{d}}\cdot \tau_M[\phi(\theta_{\ul{d}})]_n(\_);1)$. As $r>1$ and for every $ i\in\lbb1,r\rbb $, we have $ F^m\overline{\xi}^{(i)}_{M+d_i}=0 $, we can use Lemma \ref{lem025} twice and deduce that the element $E(\overline{w}_{\ul{d}}\cdot \tau_M[\phi(\theta_{\ul{d}})]_n(\_);1)$ is equal to $E(w_{\ul{d}}\cdot \tau_M[\phi(\theta_{\ul{d}})]_n(\_);1)$ (note that $ \overline{\xi}^{(i)}_{M+d_i} $ and $ \xi^{(i)}_{M+d_i} $ have the same image inside $ W_m $). The latter is by definition equal to $ \Phi_{\ul{d}}(\xi^{(1)},\dots,\xi^{(r)},[\phi(\theta_{\ul{d}})]_n(\_))$. Recalling that $ [\phi(\theta_{\ul{d}})]_n $ is equal to $ \tilde{\phi_n}(V^{d_1}[u_1]_n,\dots,V^{d_r}[u_r]_n)$, the above calculations show that $\nabla$ is equal to the sum \[ \sum_{\ul{d}\in\BZ^r_{0,<M}}\Phi_{\ul{d}}(\xi^{(1)},\dots,\xi^{(r)}, \tilde{\phi_n}(V^{d_1}[u_1]_n,\dots,V^{d_r}[u_r]_n)(\_)),\] where  $$ \Phi_{\ul{d}}(\xi^{(1)},\dots,\xi^{(r)}, \tilde{\phi_n}(V^{d_1}[u_1]_n,\dots,V^{d_r}[u_r]_n)(\_)):G_{0,n}^*\to \BG_m  $$ is seen as a section of $ G_{0,n} $. This equality means that the two multilinear morphisms $ \nabla(\beta_{\phi,n})([u_1]_n,\dots,[u_r]_n) $ and $ \Delta(\tilde{\phi_n})([u_1]_n,\dots,[u_r]_n) $ from $ \BW^r $ to $ G_{0,n} $ are equal. As every element of $ D_*(G_{i,n}) $ is of the form $ [u_i]_n $ for some $ u_i\in D_*(G_i) $, it implies the equality of $ \Delta (\tilde{\phi_n}) $ and $ \nabla(\beta_{\phi,n}) $.
\end{proof}

\begin{cor}
\label{cortechresult}
Let $ \CP_0,\CP_1,\dots,\CP_r $ be (nilpotent) displays over $k$. The homomorphisms \[ \beta:\Mult(\CP_1\times\dots\times\CP_r,\CP_0)\to \Mult_k(BT_{\CP_1}\times\dots\times BT_{\CP_r},BT_{\CP_0}),\] \[ \Sym(\CP_1^r,\CP_0)\to \Sym(BT_{\CP_1}^r,BT_{\CP_0})\]  and \[ \Alt(\CP_1^r,\CP_0)\to \Alt(BT_{\CP_1}^r,BT_{\CP_0}),\] given in Corollary \ref{cor beta} are isomorphisms.
\end{cor}

\begin{proof}
As usual, we only prove the first isomorphism, and leave the similar proofs of the other two. For every $ i=0,\dots,r $, we set $ G_i:=BT_{\CP_i} $, the $p$-divisible group associated to $ \CP_i $ and denote by $ D_i $ the (covariant) Dieudonn\'e module of $ G_i$. Using the previous Theorem, we obtain a commutative diagram \[ \xymatrix{\Mult(\CP_1\times\dots\times\CP_r,\CP_0)\ar[rd]^{\beta}\ar[dd]_{\cong}\\&\Mult(G_1\times\dots\times G_r,G_0)\\Mult(D_1\times\dots\times D_r,D_0)\ar[ur]_{\cong},&}\] where the vertical isomorphism is given by the identifications of displays and Dieudonn\'e modules and the oblique isomorphism is given by Corollary \ref{cordieudonnpidiv}. It follows at once that $ \beta $ is an isomorphism.
\end{proof}

\begin{rem}
\label{remisoofbetageneralbase}
The author believes that using this isomorphism and a similar argument as that given in \cite{Z1} (to prove that the functor $ BT $ is an equivalence of categories), one can prove that the morphism $ \beta $ is an isomorphism over any excellent local ring or a ring $R$ such that $R/pR$ is of finite type over a field. In other words, that the answer to Question \ref{question on mult displays} is affirmative.
\end{rem}

\section{The affine base case}

In this section, we show the existence of the exterior powers of $p$-divisible groups over complete local Noetherian rings with residue field of characteristic $p$, whose special fiber are connected $p$-divisible groups of dimension $1$. We also calculate the height of these exterior powers and their dimension at the closed point of the base. Furthermore, we show that these exterior powers commute with arbitrary base change. The prime number $p$ is assumed to be different from $2$. 

\begin{cons}
\label{cons07}
Assume that $p$ is nilpotent in $R$. Let $\mathcal{P}$ be a display over $R$, with tangent module of rank at most 1 and denote by $\Lambda_R^r$ the $p$-divisible group associated to $\bigwedge^r\mathcal{P}$. The universal alternating morphism $\lambda:\mathcal{P}^r\to \bigwedge^r\mathcal{P}$ (cf. Proposition \ref{prop ext. disp.}) induces an alternating morphism $\beta_{\lambda,n}:G_n^r\to \Lambda^r_{R,n}$ which gives rise to a homomorphism $$\lambda_n^*(X):\Hom_{R}(\Lambda^r_{R,n},X)\to \Alt_{R}^r(G_n,X)$$ for every group scheme $X$. Sheafifying this morphism, we obtain a sheaf homomorphism \[\underline{\lambda^*_n}(X):\innHom_{R}(\Lambda^r_{R,n},X)\to \innAlt_{R}^r(G_n,X).\]
\end{cons}

\begin{rem}
\label{rem0 19}
Note that by Lemma \ref{lem0 21} and Proposition \ref{prop0 20}, the construction of $ \ep^r\CP $, and therefore the formation of $ \Lambda_R^r $ commutes with the base change, i.e., if $A$ is any $R$-algebra, then we have canonical isomorphisms $ (\ep^r\CP)_A\cong \ep^r(\CP_A) $  and $ (\Lambda^r_R)_A\cong \Lambda^r_A $ (note that since $p$ is nilpotent in $R$ it is so also in $A$ and therefore $ \Lambda^r_A $ is a $p$-divisible group).
\end{rem}

\begin{thm}
\label{thm05}
If $ R $ is a perfect field of characteristic $p$, then for every group scheme $X$ over $ R$, the morphism $$\lambda_n^*(X):\Hom_{R}(\Lambda^r_{R,n},X)\to \Alt_{R}^r(G_n,X)$$ is an isomorphism. Consequently, we have a canonical and functorial isomorphism  \[\Lambda^r_{R,n}\cong\bigwedge^r(G_{n})\] for all positive natural numbers $n$.
\end{thm}

\begin{proof}
We know that for each $n$, the exterior power $ \ep^rG_n $ exists (Proposition \ref{prop 10}), is finite and  its Dieudonn\'e module is isomorphic to $ \ep^rD_*(G_n) $ (corollary \ref{cor41}) which is isomorphic to $ \big(\ep^rD_*(G)\big)/p^n $. This shows that the canonical homomorphism $ \ep^rG_n\to\Lambda^r_{R,n} $ (induced by the universal property of $ \ep^rG_n $) is an isomorphism. We have also shown (Remark \ref{rem 17}) that the universal alternating morphism $\tau_n: G_n^r\to \ep^rG_n $ corresponds via the isomorphism $$L(D_*(G_n)^r,\ep^rD_*(G_n))\isoto \Mult(G_n^r,\ep^rG_n)$$ (cf. corollary \ref{cor03} and Remark \ref{rem33}), to the universal alternating morphism $ D_*(G_n)^r\to \ep^rD_*(G_n) $, which is the reduction of the morphism $ \lambda:D_*(G)^r\to \ep^rD_*(G) $ modulo $ p^n $. It follows from the previous theorem, and after identifying the two group schemes $ \ep^rG_n  $ and $ \Lambda^r_{R,n} $, that the two alternating morphisms $ \tau_n $ and $ \beta_{\lambda,n} $ are equal, i.e., that the morphism $ \beta_{\lambda,n} $ is the universal alternating morphism. This means exactly that for every group scheme $X$ over $R$, the homomorphism \[ \lambda_n^*(X):\Hom_{R}(\Lambda^r_{R,n},X)\to \Alt_{R}^r(G_n,X) \] is an isomorphism. 
\end{proof}

\begin{prop}
\label{prop011}
Assume that $p$ is nilpotent in $R$. For every group scheme $X$ over $R$, and every ring homomorphism $R\to L$, with $L$ a perfect field, the morphism $\underline{\lambda^*_n}(X)$ is an isomorphism on the $L$-rational points, i.e., the homomorphism \[\lambda_n^*(X)(L):\Hom_L(\Lambda^r_{L,n},X_L)\to \Alt_L^r(G_{L,n},X_L)\] is an isomorphism.
\end{prop}

\begin{proof}
This follows from Remark \ref{rem0 19} and the previous theorem, noting that $L$ has characteristic $p$, since $p$ is nilpotent in $R$.
\end{proof}

\begin{prop}
\label{prop012}
Let $R$ be a perfect field of characteristic $p$. Then, for every group scheme $X$ over $R$, the morphism $\underline{\lambda^*_n}(X)$ is an isomorphism.
\end{prop}

\begin{proof}
Let $I$ be a finite group scheme over $R$. Then the sheaf of Abelian groups $ \innHom_R(I,X) $ is representable and we have a commutative diagram \[ \xymatrix{\Hom_R(I,\innHom_R(\Lambda^r_{R,n},X))\ar[rrr]^{\Hom(I,\underline{\lambda^*_n}(X))}\ar[d]_{\cong}&&&\Hom_R(I,\innAlt^r_R(G_n,X))\ar[d]^{\cong}\\\Hom_R(\Lambda^r_{R,n},\innHom_R(I,X))\ar[rrr]_{\lambda_n^*(\innHom_R(I,X))}&&&\Alt_R^r(G_n,\innHom_R(I,X)).} \] The bottom homomorphism of this diagram is an isomorphism by the Theorem \ref{thm05} and therefore the top homomorphism is an isomorphism as well. We also know from the previous proposition that the homomorphism $ \underline{\lambda^*_n}(X) $ is an isomorphism on the $L$-valued points, for every perfect field $L$, and in particular for the algebraic closure of $R$. It follows from the Proposition \ref{prop02}, that this homomorphism is an isomorphism. 
\end{proof}

\begin{prop}
\label{prop013}
Assume that $p$ is nilpotent in $R$. The homomorphism \[\ul{\lambda_n^*}(\BG_m):\innHom_{R}(\Lambda^r_{R,n},\BG_m)\to \innAlt_{R}^r(G_n,\BG_m)\] is an isomorphism.
\end{prop}

\begin{proof}
Let $ L $ be a perfect field and $s$ an $L$-valued point of the scheme $\Spec(R)$. By Remark \ref{rem0 19}, the group scheme $ (\Lambda^r_{R,n})_L $ is canonically isomorphic to the group scheme $ \Lambda^r_{L,n} $ and therefore, the fiber of the homomorphism $ \ul{\lambda_n^*}(\BG_m) $ over $s$ is the homomorphism $$ \ul{\lambda_n^*}(\BG_m)_s:\innHom_{L}(\Lambda^r_{L,n},\BG_{m,L})\to \innAlt_{L}^r(G_{L,n},\BG_{m,L}),$$ which is an isomorphism by the previous proposition. Since $ \innHom_{R}(\Lambda^r_{R,n},\BG_m) $, being the Cartier dual of the finite flat group scheme $ \Lambda^r_{R,n} $, is a finite flat group scheme over $R$, and the group scheme $ \innAlt_{R}^r(G_n,\BG_m) $ is affine and of finite type over $ \Spec(R) $ (cf. Remark \ref{rem 26}), we can apply Remark \ref{rem0 17} and Proposition \ref{prop018} and conclude that the homomorphism $\ul{\lambda_n^*}(\BG_m)$ is an isomorphism.
\end{proof}

\begin{prop}
\label{prop014}
Assume that $p$ is nilpotent in $R$. For every finite and flat group scheme $X$ over $R$, the morphism $$\ul{\lambda}_n^*(X):\innHom_{R}(\Lambda^r_{R,n},X)\to \innAlt_{R}^r(G_n,X)$$ is an isomorphism. Consequently, $\beta_{\lambda,n}:G_n^r\to \Lambda^r_{R,n} $ is the $ r^{\text{th}} $-exterior power of $ G_n $ in the category of finite and flat group schemes over $R$.
\end{prop}

\begin{proof}
As $X$ is finite and flat over $R$, there exists a canonical isomorphism $ X\cong \innHom_R(X^*,\BG_m)$, where $ X^* $ is the Cartier dual of $X$. We then obtain a commutative diagram
\[ \xymatrix{\innHom_R(\Lambda^r_{R,n},X)\ar[rrr]^{\ul{\lambda}_n^*(X)}\ar[d]_{\cong}&&&\innAlt_R^r(G_n,X)\ar[d]^{\cong}\\\innHom_R(\Lambda^r_{R,n},\innHom_R(X^*,\BG_m))\ar[rrr]^{\ul{\lambda}_n^*(\innHom_R(X^*,\BG_m))}\ar[d]_{\cong}&&&\innAlt^r_R(G_n,\innHom_R(X^*,\BG_m))\ar[d]^{\cong}\\\innHom_R(X^*,\innHom_R(\Lambda^r_{R,n},\BG_m))\ar[rrr]_{\innHom_R(X^*,\ul{\lambda_n^*}(\BG_m))}&&&\innHom_R(X^*,\innAlt_R^r(G_n,\BG_m)).} \] Since by the previous proposition, the homomorphism $ \ul{\lambda_n^*}(\BG_m) $ is an isomorphism, the bottom homomorphism of this diagram is an isomorphism as well, and thus also the homomorphism $ \ul{\lambda}_n^*(X) $. Taking the global sections of $\ul{\lambda}_n^*(X)$ (i.e., taking the $R$-valued points), we conclude that the homomorphism $ \lambda_n^*(X) $ is an isomorphism, and therefore $\beta_{\lambda,n}:G_n^r\to \Lambda^r_{R,n} $ is the $ r^{\text{th}} $-exterior power of $ G_n $ in the category of finite and flat group schemes over $R$.
\end{proof}

\begin{Ques}
\label{thm06}
Is the morphism $$\lambda_n^*(X):\Hom_{R}(\Lambda^r_{R,n},X)\to \Alt_{R}^r(G_n,X)$$ an isomorphism for every group scheme $X$ over $R$?
\end{Ques}


\begin{prop}
\label{prop026}
Let $R$ be a complete local Noetherian ring with residue characteristic $p$ and $G$ a $p$-divisible group over $R$ such that the special fiber of $G$ is a connected $p$-divisible group of dimension $1$. Then there exists a $p$-divisible group $ \bigwedge^rG $ over $R$ and an alternating morphism $ \tau:G^r\to \bigwedge^rG$, such that for every $p$-divisible group $H$ over $R$ the induced group homomorphism \[\tau^*:\Hom_S(\ep^rG,H)\to \Alt_S(G^r,H)\] is an isomorphism. Furthermore, for all $n$, the canonical homomorphism \[\ep^r(G_n)\to (\ep^rG)_n,\] induced by the universal property of $ \ep^r(G_n) $ is an isomorphism. Finally, the height of $ \ep^rG $ is equal to $ \binom{h}{r} $ and its dimension at the closed point of $R$ is equal to $ \binom{h-1}{r-1} $.
\end{prop}

\begin{proof}
First assume that $ R $ is a local Artin ring. Then $p$ is nilpotent in $R$ and $G$ is infinitesimal. Set $ \ep^rG:=\Lambda^r_R $. By Proposition \ref{prop014}, the alternating morphism $ \beta_{\lambda,n}:G_n^r\to (\ep^rG)_n $ is the $ r^{\text{th}} $-exterior power of $ G_n $ over $R$ and therefore, the canonical homomorphism $ (\ep^rG)_n\to \ep^r(G_n) $ is an isomorphism and the induced homomorphism \[ \Hom_R(\ep^rG_n,H_n)\to \Alt_R^r(G_n,H_n) \] is an isomorphism. Taking the inverse limit of this isomorphism and noting that by definition, $$ \Alt^r_R(G,H)=\uset{n}{\invlim}\Alt^r_R(G_n,H_n) $$ and $$ \Hom_R(G,H)=\uset{n}{\invlim}\Hom_R(G_n,H_n) $$ we deduce that the canonical homomorphism \[ \Hom_R(G,H)\to\Alt^r_R(G,H) \]induced by the system $ \{\beta_{\lambda,n}:G_n^r\to (\ep^rG)_n\}_n $ is an isomorphism.\\

In the general case, set $ X:=\Spec(R) $, $ \FX:=\Spf(R) $ and for all $i$, $ \FX_i:=\Spec(R/\Fm^i) $, where $\Fm$ is the maximal ideal of $R$. Let $ G(i) $ denote the base change of $ G $ to $ \FX_i $. From above, we know that $ \ep^rG(i) $ exists for all $i$ and we have a universal alternating morphism $ \lambda(i):G(i)^r\to \ep^rG(i) $. We also know that the construction of the exterior power commutes with base change (note that $ G(i) $ is infinitesimal), and thus the universal alternating morphism $$\lambda(i+1):G(i+1)^r\to \ep^rG(i+1)$$ restricts over $ \FX_i $ to $\lambda(i):G(i)^r\to \ep^rG(i) $. By Proposition \ref{prop023} and Remark \ref{rem023}, there exists a $p$-divisible group $ \ep^rG $ over $X$ and an alternating morphism $ \lambda:G^r\to \ep^rG $ which restricts over each $ \FX_i $ to $ \lambda(i) $. It follows from the universal property of $ \ep^rG(i) $ and Remark \ref{rem023} that the alternating morphism $ \lambda $ is the universal alternating morphism making $ \ep^rG $ the $ r^{\text{th}} $-exterior power of $ G $ over $ X $.\\

The same arguments show that the truncated Barsotti-Tate groups $ \ep^r(G(i)_n) $ (of level $n$) form a compatible system and therefore define a truncated Barsotti-Tate group of level $n$ over $ \FX $. Using again Proposition \ref{prop023} and Remark \ref{rem023}, we obtain a truncated Barsotti-Tate group of level $n$ over $X$, denoted $ \ep^r(G_n) $ and an alternating morphism $ \lambda_n:G_n^r\to\ep^r(G_n) $. Like above it is the universal alternating morphism making $\ep^r(G_n)$ the $ r^{\text{th}} $-exterior power of $ G_n $ over $ X $. Since for all $i$  the canonical homomorphism $ \ep^r(G(i)_n)\to (\ep^rG(i))_n$ is an isomorphism, it follows from Proposition \ref{prop023} that the canonical homomorphism $ \ep^r(G_n)\to (\ep^rG)_n$ is an isomorphism as well.\\

The dimension of a $p$-divisible group over a field is invariant under field extensions and the height of a $p$-divisible group (over any base scheme) is invariant under any base change. We also know that the construction of the exterior powers of a $ p $-divisible group (of dimension $1$) over a field of characteristic $p$ commutes with field extensions (cf. Remark \ref{rem0 19}). Thus, in order to determine the height of $ \ep^rG $, and its dimension at the closed point of $S$, we can assume that the residue field of $R$ is algebraically closed (note that $ (\ep^rG)_k\cong \ep^r(G_k) $). By Remark \ref{rem0 9}, the dimension and respectively the height of a $p$-divisible group is equal to the rank and respectively to the height of the corresponding display. The result on the dimension and height of $ \ep^rG $ follows  at once from Lemma \ref{lem0 21}.
\end{proof}

\begin{rem}
\label{rem025}
Note that by construction of $ \ep^rG $ (respectively of $ \ep^r(G_n) $, the canonical homomorphism $ \ep^r(G_k)\to(\ep^rG)_k $ (respectively $ \ep^r(G_{n,k})\to(\ep^rG_n)_k $) is an isomorphism, where $k$ denotes the residue field of $R$.
\end{rem}

Now, we would like to show that the exterior powers constructed in the above proposition commute with arbitrary base change, that is, the base change of the exterior powers of $G$ are the exterior powers of the base change of $G$.

\begin{notation}
\label{notations02}
Let $R$ be a complete local Noetherian ring with residue field $k$ of characteristic $p$. Let $G$ be a $p$-divisible group over $R$, of height $h$, such that the dimension of $G$ at points of $S:=\Spec(R)$ of characteristic $p$ is $1$ and that the special fiber of $G$ is a connected $p$-divisible group. Fix a positive natural number $n$ and set: $H:=G_n$, $X:=\innHom_R(\ep^rH,\BG_m) $ and $ Y:=\innAlt_R^r(H,\BG_m) $. Denote by $ \alpha $ the canonical homomorphism $\alpha:X\to Y $ induced by the universal alternating morphism $ \lambda:H^r\to \ep^rH $.
\end{notation}

Note that all (rational) integers prime to $p$ are invertible in $R$, i.e., $R$ is $ \BZ_{(p)} $-algebra. Indeed, if $n$ is an integer prime to $p$, then $n$ is invertible $k$, so, it is not contained in the maximal ideal of $R$. Since $R$ is local, it is therefore invertible in $R$. This implies that for all ring homomorphisms $ R\to L $ with $L$ a field, the characteristic of $L$ is either $p$ or zero.

\begin{lem}
\label{lem026}
Let $\bar{s}$ be a geometric point of $ S $. Then the group scheme $ Y_{\bar{s}} $ is a finite group scheme of order $ p^{n\binom{h}{r} }$ over $ \bar{s} $.
\end{lem}

\begin{proof}
Assume that $ \sbar=\Spec(\Omega) $. We know that the exterior powers of $ H_{\Omega} $ exist and we have a canonical homomorphism \[f:\innHom_{\Omega}(\ep^r(H_{\Omega}),\BG_{m,\Omega})\to \innAlt_{\Omega}^r(H_{\Omega},\BG_{\Omega})=Y_{\Omega}.\] If $ \Omega $ has characteristic $p$, then by assumption $ G_{\Omega} $ has dimension $1$ and therefore by Proposition \ref{prop013} $ f $ is an isomorphism. By Proposition \ref{prop026} we know that the order of $ \ep^r(H_{\Omega}) $ is equal to $ p^{n\binom{h}{r} } $, and therefore its Cartier dual, which is isomorphic to $ Y_{\Omega} $ has order $ p^{n\binom{h}{r} } $ as well. Now assume that $ \Omega $ has characteristic zero. Then $ \innHom_{\Omega}(\ep^r(H_{\Omega}),\BG_{m,\Omega}) $ is a finite \'etale group scheme over $ \Omega $. By definition of $ \ep^r(H_{\Omega}) $ the homomorphism \[ f(\Omega):\Hom_{\Omega}(\ep^r(H_{\Omega}),\BG_{m,\Omega})\to\Alt_{\Omega}^r(H_{\Omega},\BG_{m,\Omega})=Y_{\Omega}(\Omega) \] is an isomorphism. The group of homomorphisms $\Hom_{\Omega}(\ep^r(H_{\Omega}),\BG_{m,\Omega}) $, being the $ \Omega $-valued points of a finite group scheme over $ \Omega $, is finite. Since the group scheme $ Y_{\Omega} $ is of finite type over $ \Omega $ and has finitely many $ \Omega $-valued points, it is a finite group scheme over $ \Omega $. It is thus \'etale. Since $ \Omega $ is algebraically closed, the two finite \'etale group schemes $ X_{\Omega} $ and $ Y_{\Omega} $ are constant. So $f$ is an isomorphism, because it is so on the $ \Omega $-valued points. Again, the order of $ Y_{\Omega} $ is equal to the order of $ \ep^r(H_{\Omega}) $, which is equal to $ p^{n\binom{h}{r} } $ by Proposition \ref{prop025}.
\end{proof}

\begin{prop}
\label{prop027}
The homomorphism $ \alpha $ is an isomorphism.
\end{prop}

\begin{proof}
Set $ A:=\CO(X) $ and $ B:=\CO(Y) $. We know that $ A $ is a finite flat $R$-module of rank $ p^{n\binom{h}{r} } $ and $ B $ is a finitely generated $R$-algebra. Denote by $ f:B\to A $ the ring homomorphism corresponding to $ \alpha $ and by $C$ the cokernel of $f$. By Remark \ref{rem025} and Proposition \ref{prop013}, $ \alpha_k $ is an isomorphism, which means that $ C\otimes_Rk=0 $. Since $A$ is finite over $R$, $C$ is also finite over $R$ and applying Nakayama's lemma, we conclude that $C=0$. This implies that $f$ is an epimorphism. Let $\bar{s}=\Spec(\Omega)$ be a geometric point of $ S $.  By previous lemma $ B\otimes_R\Omega $ has dimension, over $ \Omega $, equal to $ p^{n\binom{h}{r} } $, which is equal to the dimension, over $ \Omega $,  of $ B\otimes_R\Omega $. As $f$ is an epimorphism and therefore also $ f\otimes_R\Omega $, we conclude that $ f\otimes_R\Omega $ is an isomorphism. This proves that $ \alpha_{\sbar} $ is an isomorphism. It follows by Remark \ref{rem0 17} and Proposition \ref{prop018} that $ \alpha $ is an isomorphism.
\end{proof}

\begin{prop}
\label{prop028}
Let $T$ be an $S$-scheme and $Z$ a finite flat group scheme over $T$. The canonical homomorphism \[ \ul{\lambda}^*_Z:\innHom_T((\ep^rH)_T,Z)\to\innAlt_T^r(H_T,Z) \] induced by the alternating morphism $ \lambda_T:H_T^r\to(\ep^rH)_T $ is an isomorphism. In particular, $ (\ep^rH)_T $ is the $ r^{\text{th}} $-exterior power of $ H_T $ in the category of finite flat group schemes over $T$ and $ \lambda_T $ is the universal alternating morphism.
\end{prop}

\begin{proof}
As $Z$ is finite and flat over $T$, there exists a canonical isomorphism $ Z\cong \innHom_T(Z^*,\BG_{m,T})$, where $ Z^* $ is the Cartier dual of $Z$. We then obtain a commutative diagram
\[ \xymatrix{\innHom_T((\ep^rH)_T,Z)\ar[rr]^{\ul{\lambda}^*_Z}\ar[d]_{\cong}&&\innAlt_T^r(H_T,Z)\ar[d]^{\cong}\\\innHom_T((\ep^rH)_T,\innHom_T(Z^*,\BG_{m,T}))\ar[rr]\ar[d]_{\cong}&&\innAlt^r_T(H_T,\innHom_T(Z^*,\BG_{m,T}))\ar[d]^{\cong}\\\innHom_T(Z^*,\innHom_T((\ep^rH)_T,\BG_{m,T}))\ar[rr]_{\quad\innHom_T(Z^*,\alpha_T)}&&\innHom_T(Z^*,\innAlt_T^r(H_T,\BG_{m,T})).} \] Since $ \alpha $ is an isomorphism by the previous proposition, $ \alpha_T $ and thus also the homomorphism $ \innHom_T(Z^*,\alpha_T) $ are isomorphisms. It follows that $ \ul{\lambda}^*_Z $ is an isomorphism as well. Taking the global sections of the homomorphism $ \ul{\lambda}^*_Z  $, we obtain an isomorphism \[\Hom_T((\ep^rH)_T,Z)\to\Alt_T^r(H_T,Z), \] which makes $ \lambda_T:H_T^r\to (\ep^rH)_T $ the universal alternating morphism.
\end{proof}

\begin{cor}
\label{cor04}
Let $T$ be an $S$-scheme. The base change to $T$ of the alternating morphism $ \tau:G^r\to \ep^rG $ given by Proposition \ref{prop026} is the universal alternating morphism, i.e., for every $p$-divisible group $ G'$ over $T$, the induced homomorphism \[ \tau^*_T:\Hom_T((\ep^rG)_T,G')\to\Alt^r_T(G_T,G') \] is an isomorphism.
\end{cor}

\begin{proof}
By definition, we have $ \Hom_T((\ep^rG)_T,G')=\uset{n}{\invlim}\,\Hom_T((\ep^rG_n)_T,G'_n) $ and $ \Alt^r_T(G_T,G')=\uset{n}{\invlim}\,\Alt^r_T(G_{n,T},G'_n) $ and the homomorphism \[\tau^*_{n,T} \Hom_T((\ep^rG_n)_T,G'_n)\to\Alt^r_T(G_{n,T},G'_n) \] is induced by the alternating morphisms $ \tau_{n,T}:G_{n,T}^ r\to(\ep^rG_n)_T$. By previous proposition, $ \tau_{n,T} $ is the universal alternating morphism, and therefore the homomorphisms $ \tau^*_{n,T} $ are isomorphisms. Hence  $ \tau^*_T $ is an isomorphism as well.
\end{proof}

\begin{rem}
\label{rem026}
In virtue of the previous corollary, the $r^{\text{th}} $-exterior power of $G_T$ exists and we can write $ \ep^rG_T $ instead of $ (\ep^rG)_T $ and $ \ep^r(G_T) $. The same holds (by Proposition \ref{prop028}) for the truncated Barsotti-Tate groups $ G_{n,T} $.
\end{rem}

\section{The general case}

In this section we would like to prove the main theorem over any base scheme. The prime number $p$ is again different from $2$. We first show a result which will serve as an auxiliary tool to transfer the question of the existence of exterior powers over an (almost) arbitrary base, to the question over a special complete local Noetherian base, where we know the answer. We then prove some faithfully flat descent properties, which together with the mentioned proposition and the results from the last chapters, will provide a proof of the main theorem.\\

The following proposition and its proof are due to Lau. We include the proof for the sake of completeness. Since the proof is not due to the author and its contents are not used (even partially) elsewhere in this writing, we will not prepare its ingredients. We refer to \cite{CA}, \cite{I} and \cite{TW} for more details.

\begin{prop}
\label{prop024}
Let $ G_0 $ over $ \BF_p $ be a connected $p$-divisible group of dimension 1 and height $h$, and $ \CG $ over $ R:=\BZ_p\lbb x_1,\dots,x_{h-1}\rbb $ the universal deformation of $G_0$. Let $H$ be a truncated Barsotti-Tate group of level $n\geq 1$ and of height $h$ over a $ \BZ_{(p)} $-scheme $X$. We assume that the fibers of $H$ in points of characteristic $p$ of $X$ have dimension $1$. Then there exist morphisms \[ X\ovset{\phi}{\longleftarrow} Y \ovset{\psi}{\longrightarrow} \Spec R\] with $ \phi $ faithfully flat and affine, such that $ \phi^*H\cong \psi^*\CG_n.$
\end{prop}

\begin{proof}
Let $ \CY $ over $ \BZ_{(p)} $ be the algebraic stack of truncated Barsotti-Tate groups of level $n$ and height $h$ and let $ \CY_0\subset \CY $ be the substack of truncated Basrsotti-Tate groups as in the proposition. The inclusion $ \CY_0\into \CY $ is an open immersion because for a morphism $ X\to \CY $ corresponding to a truncated Barsotti-Tate group $H$ over $X$, the points of characteristic $p$ of $X$ in which the dimension of $H$ is not equal to $1$ form an open and closed subscheme $X_1$ of $ X\times \Spec\BF_p $, and we have $ X\times_{\CY}\CY_0=X\setminus X_1 $. The group $ \CG_n $ over $R$ defines a morphism \[ \alpha:\Spec R\to \CY_0.\] We claim that $ \alpha $ is faithfully flat and affine. Then for $H$ over $X$ as in the proposition, which corresponds to a morphism $ X\to \CY_0 $, we can take $ Y=X\times_{\CY}\Spec R $.\\

The morphism $ \alpha $ is affine, because $ \Spec R $ is affine and the diagonal of $ \CY $ is affine. It is easy to see that $ \alpha $ is surjective on geometric points. Indeed, let $ k $ be an algebraically closed field. If $k$ has characteristic zero, then $ \CY(k)=\CY_0(k) $ has only only one isomorphism class. If $k$ has characteristic $p$, then $ \CY_0(k) $ has precisely $ h-1 $ isomorphism classes corresponding to the \'etale rank. These isomorphism classes all occur in the fibers of $ \CG_n $ over $R$. It remains to show that $ \alpha $ is flat.\\

Let $ T $ be the following functor on $ \BZ_{(p)} $-schemes: $ T(X) $ is the set of isomorphism classes of pairs $ (H,a) $ where $ \pi:H\to X $ is an open object of $ \CY_0 $ and where $ a:\CO_X^{p^{nh}}\cong \pi_*\CO_H $ is an isomorphism of $ \CO_X $-modules. Then $ T $ is representable by a quasi-affine scheme of finite type over $ \BZ_{(p)} $; see \cite{TW}. The morphism $ T\to \CY_0 $ defined by forgetting $ a $ is a $ GL_{p^{nh}} $-torsor and thus smooth. By \cite{I}, the algebraic stack $ \CY $ is smooth over $ \BZ_{(p)} $. Hence the same is true for $ \CY_0 $ and $ T $.\\

Let $ t\in T $ be a closed point with residue field $ \BF_p $ such that the associated group over $ \BF_p $ is $ G_{0,n} $. The image of $t$ in $ \CY(\BF_p) $ is also denoted by $ t $. The homomorphism of tangent spaces $ T_{T,t}\to T_{\CY,t} $ is surjective, because $ T\to \CY $ is smooth. Let $ Z\to T $ be a regular immersion (thus $ Z $ is smooth) with $ t\in Z $ such that $ T_{Z,t}\to T_{\CY,t} $ is bijective.\\

After shrinking $ Z $ we can assume that $ Z\to \CY $ is smooth. Indeed, let $ U=Z\times_{\CY}T $ and let $ u\in U $ be a closed point with residue field $ \BF_p $ lying over $ (t,t)\in Z\times T $. Then the second projection $ T_{U,u}\to T_{T,t} $ is surjective. Since $ U $ and $ T $ are smooth, the projection $ U\to T $ is smooth in $ u $. Thus there is an open subscheme $ U_0 $ of $ U $ containing $ u $ such that $ U_0\to T $ is smooth. If we replace $Z$ by the image of $ U_0\to Z $, which is open, because this map is smooth, then $ Z\to \CY $ is smooth.\\

Let $ S=\hat{\CO}_{Z,t} $ and let $H$ over $S$ be the truncated Barsotti-Tate group corresponding to the given morphism $ \Spec S\to \CY $. The special fiber of $H$ is isomorphic to $ G_{0,n} $. By \cite{I} there is a $p$-divisible group $ G' $ over $S$ with special fiber $ G_0 $ such that $ G'_n $ is isomorphic to $H$. Since the first-order deformations of $ G_0 $ and of $ G_{0,n} $ coincide, it follows that $ G' $ is a universal deformation of $ G_0 $. Thus the morphism $ \alpha $ can be written as a composition \[ \Spec R\cong \Spec S\to Z\to \CY.\] Here $ Z\to \CY $ is smooth and thus flat, and $ \Spec S\to Z $ is flat, because it is a completion of a Noetherian ring. thus the composition is flat as well.
\end{proof}

Now we prove the faithfully flat descent of the universal alternating morphism and the exterior powers.

\begin{lem}
\label{lem027}
Let $ f:T\to S $ be a faithfully flat morphism of schemes and let $ H $ (respectively $G$) be a finite flat group scheme (respectively a $p$-divisible group) over $S$.
\begin{itemize}
\item[1)]  Assume that we are given an alternating morphism $ \tau:H^r\to \Lambda $ (respectively $ \tau:G^r\to \Lambda $), with $ \Lambda $ a finite flat group scheme (respectively a $p$-divisible group) over $S$, such that for all morphisms $ g:T'\to T $, the pullback $ g^*f^*\tau:g^*f^*H^r\to g^*f^*\Lambda $ (respectively $g^*f^*\tau:g^*f^*G^r\to g^*f^*\Lambda$) is the universal alternating morphism in the category of finite flat group schemes (respectively $p$-divisible groups) over $T'$. Then $ \tau $ is the universal alternating morphism in the category of finite flat group schemes (respectively $p$-divisible groups) over $S$, i.e., $ \Lambda=\ep^rH $ (respectively  $ \Lambda=\ep^rG $).

\item[2)] Assume that we have an alternating morphism $ \tau':f^*H^r\to \Lambda' $ (respectively $ \tau':f^*G^r\to \Lambda' $), with $ \Lambda' $ a finite and flat group scheme (respectively a $p$-divisible group) over $T$, such that for all morphisms $ g:T'\to T $, the pullback $ g^*\tau':g^*f^*H^r\to g^*\Lambda' $ (respectively $ g^*\tau':g^*f^*G^r\to g^*\Lambda' $) is the universal alternating morphism in the category of finite flat group schemes (respectively $p$-divisible groups) over $T'$. Then there exists a finite flat group scheme (respectively a $p$-divisible group) $ \Lambda $ over $S$ and an alternating morphism $ \tau:H^r\to \Lambda $ (respectively $ \tau:G^r\to \Lambda $), such that for every morphism $ h:S'\to S $, the pullback $ h^*\tau $ is the universal alternating morphism in the category of finite flat group schemes (respectively $p$-divisible groups) over $S'$, i.e., $h^*\Lambda=\ep^r(h^*H) $ (respectively $h^*\Lambda=\ep^r(h^*G) $). In particular, $ \Lambda=\ep^rH $ (respectively $ \Lambda=\ep^rG $).
\end{itemize}
\end{lem}

\begin{proof}
Since alternating morphisms and homomorphisms of $p$-divisible groups are defined as compatible systems of alternating morphisms and homomorphisms of finite flat group schemes (their truncated Barsotti-Tate groups), we will only prove the lemma for truncated Barsotti-Tate groups, and the result for the $p$-divisible groups follows at once.\\

For the proof of both parts of the lemma, we use faithfully flat descent.
\begin{itemize}
\item[1)] Take a finite flat group scheme $X$ over $S$. We have to show that the canonical homomorphism \[ \tau^*:\Hom_S(\Lambda,X)\to\Alt_S^r(H,X) \] induced by $ \tau $ is an isomorphism. So, take an alternating morphism $ \phi:H^r\to X $. Letting $ g $ be the identity morphism of $ T $, we see that in particular, $ f^*\tau:f^*H^r\to f^*\Lambda$ is the universal alternating morphism, and therefore there exists a unique group scheme homomorphism $ a':f^*\Lambda\to f^*X $ such that $ a'\circ f^*\tau=f^*\phi $. We want to descend the homomorphism $ a' $ to a homomorphism $ a:\Lambda\to X $. Then since $ f $ is faithfully flat, we have $ a\circ\tau=\phi $ and $a$ with this property is unique. Set $ T':=T\times_ST $ and let $ p_i:T'\to T $ ($i=1,2$) be the two projections. We have to show that $ p_1^*a'=p_2^*a' $. But this is true, since $ p_1^*f^*=p_2^*f^* $ and by assumption, for $ i=1,2 $, we know that $ p_i^*f^*\tau:p_i^*f^*H^r\to p_i^*f^*\Lambda $ is the universal alternating morphism and we have $ p_i^*a'\circ p_i^*f^*\tau=p_i^*f^*\phi $.

\item[2)] We want to descend the finite flat group scheme $ \Lambda' $ to a group scheme over $ S $. Set $ T':=T\times_S T $ and let $ p_i:T'\to T $ ($i=1,2$) be the two projections. We should prove that the base changes of $ \Lambda' $ via $ p_1 $ and $ p_2 $ are canonically isomorphic. By assumption, we know that $ p_i^*\Lambda'= \ep^r(p_i^*f^*H)$ ($ i=1,2 $). The two compositions $ p_1^*f^* $ and $  p_2^*f^*$ are equal and thus, there exists a unique isomorphism $ p_1^*\Lambda'\cong p_2^*\Lambda' $ by the uniqueness of the exterior powers. This shows that the group scheme $ \Lambda' $ descends to a group scheme over $S$. Since $f$ is faithfully flat and $ \Lambda' $ is finite flat over $T$, we conclude that $ \Lambda $ is finite flat over $S$. The same arguments show that the alternating morphism $ \tau' $ descends to a morphism $ \tau:H^r\to \Lambda $, and again by the faithfully flatness of $f$, it should be alternating.\\

Now let $ T' $ be the base change $ T\times_S S' $ and denote by $h'$ and respectively $ f' $ the projection $ T'\to T $ and respectively $ T'\to S' $. The morphism $ f' $, being the base change of $f$, is faithfully flat. By construction of $ \tau $ ($ f^*\tau=\tau' $) and the assumptions on $ \tau' $, we observe that all the hypotheses of the first part of the lemma are satisfied for the alternating morphism $ h^*\tau:h^*H^r\to h^*\Lambda $ (note that we are considering the lemma for the faithfully flat morphism $ f':T'\to S' $). Consequently, the alternating morphism $ h^*\tau $ is the universal alternating morphism in the category of the finite flat group schemes over $S'$.
\end{itemize}
\end{proof}

Now, we would like to show the existence of the exterior powers over arbitrary base. We need two lemmas before.

\begin{lem}
\label{lem028}
Let $ S $ be a scheme over $ \BZ_{(p)} $ and $H$ a truncated Barsotti-Tate group of level $n$ and height $h$ over $S$, such that the dimension of the fibers at points of characteristic $p$ of $S$ is $1$. Then there exists a truncated Barsotti-Tate group $\Lambda_n$ over $S$ of level $n$ and height $ \binom{h}{r} $, and an alternating morphism $ \lambda_n:H^r\to\Lambda_n $ such that for all morphisms $h:S'\to S$ and all finite flat group schemes $X$ over $S'$, the induced homomorphism \[ \Hom_{S'}(h^*\Lambda_n,X)\to \Alt_{S'}^r(h^*H,X) \]  is an isomorphism, i.e., $ \Lambda_n\cong\ep^rH $ and $ h^*\ep^rH\cong\ep^r(h^*H) $. Moreover fibers of $ \Lambda_n $ at points of characteristic $p$ have dimension $ \binom{h-1}{r-1} $.
\end{lem}

\begin{proof}
Let $G_0$, $R$ and $\CG$ be as in the statement of the Proposition \ref{prop024}. The assumptions of Proposition \ref{prop026} and corollary \ref{cor04} are satisfied (cf. Notations \ref{notations02}. Thus, the $p$-divisible group $ \ep^r\CG $ exists over $R$ and we have the universal alternating morphism $ \tau:\CG^r\to \ep^r\CG $. Furthermore, there exists a canonical isomorphism $ \ep^r(\CG_n)\cong (\ep^r\CG)_n $, induced by the alternating morphism $ \tau_n:\CG_n^r\to (\ep^r\CG)_n $ (which is then the universal one). Also, for every morphism $\phi: T'\to \Spec R $, we have $ \phi^*\ep^r\CG_n\cong \ep^r\phi^*\CG_n $ (cf. Proposition \ref{prop028}). By Proposition \ref{prop026}, the height of $ \ep^r\CG $ is equal to $ p^{\binom{h}{r}} $, and its dimension at the closed point of $R$ is equal to $ \binom{h-1}{r-1}$. So, the order of $ \ep^r\CG_n $ over $R$ is equal to $ p^{n\binom{h}{r}} $, in other words, $ \ep^r\CG_n $ is a truncated Barsotti-Tate group of level $n$ and height $ \binom{h}{r} $, and its dimension at the closed point of $R$ is $ \binom{h-1}{r-1}$.\\

The group scheme $H$ satisfies the hypotheses of the Proposition \ref{prop024} and therefore, there exists a faithfully flat and affine morphism $ f:T\to S $ and a morphism $ g:T\to \Spec R $ such that $ f^*H\cong g^*\CG_n $. By the above discussions, we have an alternating morphism $$ g^*\tau_n:f^*H^r\cong g^*\CG_n^r\to g^*\ep^r\CG_n $$ such that for all morphisms $ g':T'\to T $, the pullback $ g'^*g^*\tau_n $ is the universal alternating morphism $$ g'^*f^*H^r\cong g'^*g^*\CG_n^r\to g'^*g^*\ep^r\CG_n.$$ It follows from the second part of the Lemma \ref{lem027} that there exists a finite flat group scheme $ \Lambda_n $ over $S$ and an alternating morphism $ \lambda_n:H^r\to \Lambda_n $, which has the desired properties stated in the lemma. Since $ f $ is faithfully flat and $ f^*\Lambda_n\cong g^*\ep^r\CG_n $ is a truncated Barsotti-Tate group of level $n$ and height $ \binom{h}{r} $, the group scheme $ \Lambda_n $ is also a truncated Barsotti-Tate group of level $n$ and height $ \binom{h}{r} $. The dimension of $ \Lambda_n $ at points of characteristic $p$ of $S$ is equal to $ \binom{h-1}{r-1}$.
\end{proof}

\begin{lem}
\label{lem029}
Let $ S $ be a scheme over $ \BZ_{(p)} $ and $ G $ a $p$-divisible group over $S$ of height $h$, such that the dimension of $G$ at points of characteristic $p$ of $S$ is $1$. Let $\ep^rG_n$ be the truncated Barsotti-Tate group of level $n$ over $S$ provided by the previous lemma (applying it to $ G_n $). Then there exist natural monomorphisms $ i_n:\ep^rG_n\into \ep^rG_{n+1} $, which make the inductive system $ (\ep^rG_n)_{n\geq 1} $ a Barsotti-Tate group over $S$ of height $ \binom{h}{r} $ and dimension $ \binom{h-1}{r-1}$ at points of $S$ of characteristic $p$.
\end{lem}

\begin{proof}
For every $n$ we have an exact sequence \[G_{n+1}\arrover{p^n} G_{n+1}\arrover{\xi_n} G_n\to 0.\] by Proposition \ref{prop 7}, the induced sequence \[ \ep^rG_{n+1}\arrover{p^n} \ep^rG_{n+1}\arrover{\ep^r\xi_n} \ep^rG_n\to 0 \] is exact as well. Since by previous lemma $ \ep^rG_{n+1} $ is a truncated Barsotti-Tate group of level $n+1$, we have $ p^{n+1}\ep^rG_{n+1}=0$. It implies that there exists a unique homomorphism $ i_n:\ep^rG_n\to\ep^rG_{n+1} $ making the following diagram commutative: \[ \xymatrix{\ep^rG_{n+1}\ar[r]^{p^n}&\ep^rG_{n+1}\ar[r]^{\ep^r\xi_n}\ar[d]_{p}&\ep^rG_n\ar[r]\ar[dl]^{i_n}&0\\&\ep^rG_{n+1}.&&} \] We would like to show that $ i_n $ is a monomorphism and it identifies $ \ep^rG_n $ with $ (\ep^rG_{n+1})[p^n] $. Since $ p^n\ep^rG_n=0 $, there exists a homomorphism $ j_n:\ep^rG_n\to(\ep^rG_{n+1})[p^n] $ whose composition with the inclusion $\iota:(\ep^rG_{n+1})[p^n]\into\ep^rG_{n+1} $ is equal to $ i_n $ (look at the diagram below). Also, as $ \ep^rG_{n+1} $ is a truncated Barsotti-Tate group of level $n+1$, the image of multiplication by $p$ is equal to the kernel of multiplication by $ p^n $ and therefore the homomorphism $ p:\ep^rG_{n+1}\to \ep^rG_{n+1} $ factors through the inclusion $ \iota:(\ep^rG_{n+1})[p^n]\into\ep^rG_{n+1} $ and induces an epimorphism $ q_n:\ep^rG_{n+1}\to (\ep^rG_{n+1})[p^n]$: \[ \xymatrix{0\ar[r]&(\ep^rG_{n+1})[p^n]\ar[rr]^{\iota}&&\ep^rG_{n+1}\ar[r]^{p^n}&\ep^rG_{n+1}\\&&\ep^rG_n\ar[ur]^{i_n}\ar[ul]_{j_n}&&\\&&\ep^rG_{n+1}\ar@{->>}[u]^{\ep^r\xi_n}\ar@/_1pc/[uur]_{p}\ar@/^1pc/@{->>}[uul]^{q_n}.&& }\] The composition of  $ \ep^r\xi_n\circ j_n $ and $ q_n $ with $ \iota $ are equal and since $ \iota $ is a monomorphism, we have $ \ep^r\xi_n\circ j_n=q_n $. So, the $ j_n:\ep^rG_{n}\to(\ep^rG_{n+1})[p^n] $ is an epimorphism ($ q_n $ is an epimorphism). The two group schemes $ \ep^rG_{n}$ and $(\ep^rG_{n+1})[p^n] $ have the same order over $S$ (note that $(\ep^rG_{n+1})[p^n] $ is a truncated Barsotti-Tate group of level $n$) and thus the epimorphism $ j_n $ is in fact an isomorphism. Hence $ i_n $ is a monomorphism identifying $ \ep^rG_n $ with $(\ep^rG_{n+1})[p^n]$. By previous lemma, the order of $ \ep^rG_n $ is equal to $ p^{n\binom{h}{r}} $. This proves that the inductive system \[ \ep^rG_1\ovset{i_1}{\longinto}\ep^rG_2\ovset{i_2}{\longinto}\ep^rG_3\ovset{i_3}{\longinto}\dots \] is a Barsotti-Tate group over $S$ of height $  \binom{h}{r}$. The statement on the dimension follows from Theorem \ref{thm 4}.
\end{proof}


\begin{thm}[The Main Theorem]
\label{thm07}
Let $S$ be a base scheme and $G$ a $p$-divisible group over $S$ of height $h$, such that the fibers of $G$ at points of $S$ of characteristic $p$ have dimension $1$. Then, there exists a $p$-divisible group $\ep^rG$ over $S$ of height $\binom{h}{r}$, and an alternating morphism $ \lambda:G^r\to \ep^rG $ such that for every morphism $ f:S'\to S $ and every $p$-divisible group $H$ over $S'$, the homomorphism \[\Hom_{S'}(f^*\ep^rG,H)\to\Alt^r_{S'}(f^*G,H) \] induced by $ f^*\lambda $ is as isomorphism. In other words, the $ r^{\text{th}} $-exterior power of $ G $ exists and commutes with arbitrary base change. Moreover, the dimension of $ \ep^rG$ at points of $S$ of characteristic $p$ is $ \binom{h-1}{r-1}$.
\end{thm}

\begin{proof}
We have flat morphisms  \[\Spec\BZ[\frac{1}{p}]\ovset{\iota}{\longinto}\Spec\BZ \ovset{\rho}{\longleftarrow}\Spec\BZ_{(p)}\] which define a faithfully flat morphism $$\iota\coprod\rho:\Spec\BZ[\frac{1}{p}]\coprod \Spec\BZ_{(p)}\to \Spec\BZ .$$ Let $ S[\frac{1}{p}]$ and respectively $ S_{(p)} $ be the pullbacks of $ S\to\Spec\BZ $, via $ \iota $ and respectively $ \rho $. The two morphisms $ S[\frac{1}{p}]\to S $ and $ S_{(p)}\to S $ induce a morphism $\pi:S[\frac{1}{p}]\coprod S_{(p)}\to S $, which is the base change of $ \iota\coprod\rho $ via $ S\to\BZ $. It is therefore faithfully flat. We define $ G[\frac{1}{p}] $ over $ S[\frac{1}{p}] $ and $ G_{(p)} $ over $ S_{(p)} $ in the same fashion.\\

The Barsotti-Tate group $ (\ep^rG_{(p),n})_{n\geq 1} $ provided by Lemma \ref{lem029} gives rise to a $p$-divisible group $ \ep^rG_{(p)} $ over $S$ of height $ \binom{h}{r} $, such that $ (\ep^rG_{(p)})_n=\ep^rG_{(p),n} $. The universal alternating morphisms $ \lambda_{(p),n}:G_{(p),n}^r\to \ep^rG_{(p),n} $ provided by Lemma \ref{lem028} are compatible with the projections $ \ep^rG_{(p),n+1}\onto\ep^rG_{(p),n} $. Indeed, the projections $$ \ep^r\xi_n:\ep^rG_{(p),n+1}\onto\ep^rG_{(p),n} $$ are induced by the universal property of $ \ep^rG_{(p),n+1} $ applied to the alternating morphism \[ G_{(p),n+1}^r\to G_{(p),n}^r\arrover{\lambda_{(p),n}}\ep^rG_{(p),n}.\] Therefore, the system $ (\lambda_{(p),n})_{n\geq 1} $ give rise to an alternating morphism $ \lambda_{(p)}:G_{(p)}^r\to\ep^rG_{(p)} $. It follows from Lemma \ref{lem028} that for every morphism $ f:S'\to S_{(p)} $ and every $p$-divisible group $H$ over $S'$, the induced homomorphism \[\Hom_{S'}(f^*\ep^rG_{(p)},H)\to\Alt_{S'}(f^*G_{(p)},H) \] is as isomorphism.\\

Since $p$ is invertible on $ S[\frac{1}{p}] $, all $p$-divisible groups over it are \'etale. Thus, we can apply Proposition \ref{prop025} and obtain a $p$-divisible group $ \ep^rG[\frac{1}{p}] $ over $ S[\frac{1}{p}] $ of height $ \binom{h}{r} $, and an alternating morphism $ \lambda[\frac{1}{p}]:G[\frac{1}{p}]^r\to\ep^rG[\frac{1}{p}] $ such that for every morphism $ f:S'\to S[\frac{1}{p}] $ and every $p$-divisible group $H$ over $S'$, the induced homomorphism \[\Hom_{S'}(f^*\ep^rG[\frac{1}{p}],H)\to\Alt_{S'}(f^*G[\frac{1}{p}],H) \] is as isomorphism.\\

The same arguments as in the proof of the Proposition \ref{prop025}, show that the disjoint union of the two universal alternating morphisms $ \lambda[\frac{1}{p}] $ and $ \lambda_{(p)} $ glue to an alternating morphism \[\lambda[\frac{1}{p}]\coprod \lambda_{(p)}:G[\frac{1}{p}]^r\coprod G_{(p)}^r\cong(G[\frac{1}{p}]\coprod G_{(p)})^r\longrightarrow \ep^rG[\frac{1}{p}]\coprod \ep^rG_{(p)}\] and this morphism is the universal alternating morphism over $S[\frac{1}{p}]\coprod S_{(p)}$, i.e., \[\ep^rG[\frac{1}{p}]\coprod \ep^rG_{(p)}\cong \ep^r(G[\frac{1}{p}]\coprod G_{(p)}).\] Since both $ \lambda[\frac{1}{p}] $ and $ \lambda_{(p)} $ stay the universal alternating morphism after any base change, their disjoint union has the same property. The pullback of $ G $ via the faithfully flat morphism $ \pi$ is $ G[\frac{1}{p}]\coprod G_{(p)} $. It follows from Lemma \ref{lem027} that the $p$-divisible group $ \ep^r(G[\frac{1}{p}]\coprod G_{(p)}) $ and respectively the disjoint union $ \lambda[\frac{1}{p}]\coprod\lambda_{(p)} $ descend to a $p$-divisible group $ \ep^rG $ over $S$ and respectively to an alternating morphism $ \lambda:G^r\to \ep^rG $ over $S$, and this is the universal alternating morphism over $S$ and after any base change of $S$, as stated in the theorem.\\

The statement on the dimension follows from the previous lemma.\\

\textbf{\emph{Quod Erat Demonstrandum.}}
\end{proof}

\begin{rem}
\label{bypassLau}
When the base scheme is locally Noetherian, there is a rather elementary way to bypass Lau's result (Proposition \ref{prop024}). We will pursue this way in the next chapter, when dealing with $ \pi $-divisible $ \CO $-module schemes, where we don't possess a generalization of Lau's result. For more details, we refer to the results following Corollary \ref{corfiniteflat} to the end of chapter 9.
\end{rem}

\chapter{\texorpdfstring{The Main Theorem for $ \pi $-Divisible Modules}{The Main Theorem for pi-Divisible Modules}}

Let $ \CO $ be a mixed characteristic complete discrete valuation ring, with finite residue field $ \BF_q $ ($q=p^f$) and $ \pi $ a fixed uniformizer. We denote by $K$ the fraction field of $ \CO $. In this chapter we would like to generalize the main theorem of the previous chapter to the case of $ \pi $-divisible $ \CO $-module schemes. In fact, we are going to explain how the results from chapters 6, 7 and 8, that led to the main theorem, can be modified (generalized) so as to imply the generalized version of the theorem. Recall that a main ingredient of these chapters is display, or more precisely the equivalence of categories between the category of $p$-divisible groups and the category of displays. We used displays in order to find potential candidates for the exterior powers of a $p$-divisible group. Therefore, if we want to use the same methods, we should have a variant of displays for $\pi$-divisible modules. These are ramified displays. They have been studied in the Ph.D. thesis of T. Ahsendorf (cf. \cite{TA}). We also need a refined version of Dieudonn\'e theory adapted to $ \pi $-divisible modules. The rest of the generalization can be carried on quite easily. From now on, we only consider $\pi$-divisible modules that are defined over schemes over $\CO$ and assume that the action of $\CO$ on their tangent space is given by scalar multiplication.\\

\section{Ramified displays}

We begin with ramified Witt vectors and ramified displays and use the notations above.

\begin{dfn}
\label{ramwitt}
Let $R$ be an $ \CO $-algebra. The set of \emph{ramified Witt vectors}, denoted by $ W_{\CO}(R) $, is the set \[ W_{\CO}(R):= \{(x_0,x_1,\dots)\,\mid\,x_i\in R\}=R^{\BN}. \] The map \[ \mathrm{w}_n:W_{\CO}(R)\to R,\qquad \ul{x}:=(x_0,x_1,\dots)\mapsto x_0^{q^n}+\pi x_1^{q^{n-1}}+\dots+\pi^nx_n \] is called the $ n^{\text{th}} $ Witt polynomial.
\end{dfn}

\begin{rem}
\label{remramwitt}
The association $ R\mapsto W_{\CO}(R) $ is functorial on the category of $ \CO $-algebras.
\end{rem}

We state the following results without proof, and refer to \cite{TA} for details.

\begin{thm}
\label{thmramwitt}
For any $ \CO $-algebra $ R $, there exists a unique $ \CO $-algebra structure on $ W_{\CO}(R) $ with following properties:
\begin{itemize}
\item[a)] The Witt polynomials $\w_n:W_{\CO}(R)\to R $ are $ \CO $-algebra homomorphisms.
\item[b)] For every $ \CO $-algebra homomorphism $ R\to S $, the induced map $ W_{\CO}(R)\to W_{\CO}(S) $ is an $ \CO $-algebra homomorphism.
\end{itemize} 
\end{thm}

\begin{rem}
\label{remramwitt2}
It follows from the theorem that $ W_{\CO} $ is a functor from the category of $ \CO $-algebras to itself. Also, if we denote by $ \Id_{\CO\text{-alg}} $ the identity functor on the category of $ \CO $-algebras, the Witt polynomials define natural transformations of functors $ \w_n:W_{\CO}\to \Id_{\CO\text{-alg}} $.
\end{rem}

\begin{prop}
\label{propramwitt}
For every $ \CO $-algebra $R$, there are $ \CO $-linear endomorphisms $ F_{\pi} $ and $ V_{\pi} $ on $ W_{\CO}(R) $, called respectively \emph{Frobenius} and \emph{Verschiebung}, with the properties:
\begin{itemize}
\item[1)] for every $ \ul{x}\in W_{\CO}(R) $, we have $$ \w_0(V_{\pi}\ul{x})=0, \w_n(V_{\pi}\ul{x})=\pi \w_{n-1}(\ul{x}) \quad \text{and} \quad \w_n(F_{\pi}\ul{x})=\w_{n+1}(\ul{x}) .$$
\item[2)] $ F_{\pi} $ is an $ \CO $-algebra homomorphism.
\item[3)] $ F_{\pi}V_{\pi}=V_{\pi}F_{\pi}=\pi $ and for every $ \ul{x}, \ul{y}\in W_{\CO}(R) $, we have $$ V_{\pi}(F_{\pi}(\ul{x})\ul{y})=\ul{x}V_{\pi}(\ul{y}) .$$
\end{itemize}
\end{prop}

The following result will be used later. It is proved in \cite{TA} and therefore we omit its proof.

\begin{prop}
\label{propmorramwitt}
There exists a unique natural transformation of functors $ \mu:W\to W_{\CO} $ such that $ \w_n\circ \mu=\w_{fn} $ for all natural numbers $n$. If $ R $ is an $ \CO $-algebra, we have for all $ a\in R $ and all $ w\in W(R) $:
\begin{itemize}
\item $ \mu([a])=[a] $
\item $ \mu(F^fw)=F_{\pi}(\mu(w))$
\item $ \mu(Vw)= (\frac{p}{\pi})V_{\pi}\mu(F^{f-1}w).$
\end{itemize}

\end{prop}

\begin{rem}
\label{remramwitt3}
$ $
\begin{itemize}
\item[1)] Note that the first property above determines uniquely the morphisms $ F_{\pi} $ and $ V_{\pi} $ and the other properties follow from the first one. It is not hard to see that $ V_{\pi}(x_0,x_1,\dots)=(0,x_0,x_1,\dots) $, thus we have $ I_R:=\image (V_{\pi})=\kernel (\w_0) $.
\item[2)] If $ \CO $ is the ring of $p$-adic integers, then we obtain the usual ring of Witt vectors.
\end{itemize}
\end{rem}

\textbf{Notations.} Let $r=p^n$ be a power of $p$. We denote by $ \BZ_r $ the ring $ W(\BF_r) $.

\begin{prop}
\label{propDr}
Let $k$ be a perfect field of characteristic $p$.
\begin{itemize}
\item[1)] The $ \CO $-algebra $ W_{\CO}(k) $ is a complete discrete valuation ring with residue field $k$ and maximal ideal generated by $ \pi $.
\item[2)] If $k$ contains $ \BF_q $, then there exists a canonical $ \CO $-algebra isomorphism \[ \CO\widehat{\otimes}_{\BZ_q}W(k)\cong W_{\CO}(k).\]
\end{itemize}
\end{prop}

\begin{proof}
The first statement is a standard one, stated e.g. in \cite{DR}. For the second statement, note that $ W_{\CO}(k) $ is an $ \CO $-algebra and contains also $ W(k) $ as subring. There exists therefore a canonical $ \CO $-algebra homomorphism $$ \CO\otimes_{\BZ_q}W(k)\to W_{\CO}(k) .$$ The subring $ \BZ_q $ of $ \CO $ is the ring of integers of the maximal unramified subextension of $ K $. This is the ring $A$ in Lemma \ref{lem 9}. As we have seen in that lemma, the completion $ \CO\widehat{\otimes}_{\BZ_q}W(k) $ is a $ \pi $-adically complete discrete valuation ring with residue field $k$. Since by the first assertion $ W_{\CO}(k) $ is also $ \pi $-adically complete, the above homomorphism extends to an $ \CO $-algebra homomorphism \[ \CO\widehat{\otimes}_{\BZ_q}W(k)\to W_{\CO}(k).\] As both completed discrete valuation rings have the same residue field and the same uniformizer, this homomorphism is an isomorphism (note that every element in the codomain can be written as a power series in $ \pi $ and with coefficients in a system of representatives of elements of $k$).
\end{proof}

\begin{cor}
\label{corramwitt}
Let $k$ be perfect field of characteristic $p$, containing $\BF_q$. There exists a canonical $ \CO $-linear decomposition $$\CO \widehat{\otimes}_{\BZ_p}W(k)\cong\prod_{\BZ/f\BZ}W_{\CO}(k).$$
\end{cor}

\begin{proof}
This follows from the last proposition and Lemma \ref{lem 9}.
\end{proof}

Now we define the ramified displays.

\begin{dfn}
\label{dfnram3ndisplay}
Let $R$ be an $ \CO $-algebra. A ramified $3n$-display over $R$ is a quadruple $\CP=(P,Q,F,V^{-1})$, where $P$ is a finitely generated $ W_{\CO}(R) $-module, $ Q\subseteq P $ is a submodule and $ F, V^{-1} $ are $ F_{\pi}$-linear morphisms $ F:P\to P $ and $ V^{-1}:Q\to P $, subject to the following axioms:
\begin{itemize}
\item[(i)] $ I_RP\subseteq Q\subseteq P $  and there is a decomposition of $ P $ into the direct sum of $ W(R) $-modules $ P=L\oplus T $, called \emph{a normal decomposition}, such that $ Q=L\oplus I_RT $.
\item[(ii)] $ V^{-1}:Q\to P $ is an $ F_{\pi}$-linear epimorphism (i.e., the $ W_{\CO}(R) $-linearization $ (V^{-1})^{\sharp}:W_{\CO}(R)\otimes_{F_{\pi},W_{\CO}(R)} Q\to P $ is surjective).
\item[(iii)] For any $ x\in P $ and $ w\in W_{\CO}(R) $ we have \[ V^{-1}(V_{\pi}(w)x)=wF(x). \]
\end{itemize}
\end{dfn}

\begin{rem}
\label{remram3ndisplay}
$ $
\begin{itemize}
\item[1)] Note that from the last axiom, it follows that $ F $ is uniquely determined by $ V^{-1} $. Indeed, we have for every $ x\in P $: \[ F(x)=V^{-1}(V_{\pi}(1)x).\] It follows also from this relation and $ F_{\pi} $-linearity of $ V^{-1} $, that for every $ y\in Q $, we have \[ F(y)=V^{-1}(V_{\pi}(1)y)=F_{\pi}V_{\pi}(1)V^{-1}(y)=\pi V^{-1}(y). \]
\item[2)] Since $ W_{\CO}(R) $ is an $ \CO $-algebra, $ P $ and $ Q $ have a natural $ \CO $-module structure and the morphisms $ F $ and $ V^{-1} $ are $ \CO $-linear (note that $ F_{\pi} $ is $ \CO $-linear).
\end{itemize}
\end{rem}

Similar constructions, remarks and propositions, as in chapter 6, hold for ramified $3n$-displays. Because of these similarities, we will only mention them in a list, without giving the details. More details can be found in \cite{TA}. In the following, $R$ is an $ \CO $-algebra and $ \CP=(P,Q,F,V^{-1}) $ is a $3n$-display over $R$.

\begin{itemize}
\item[1)] The tangent module, rank and height of a ramified $3n$-display are defined analogously (cf. Definitions \ref{def04} and \ref{def012}.)
\item[2)] Similar to Construction \ref{cons015}, we have a $ W_{\CO}(R) $-linear morphism \[ V^{\sharp}:P\to W_{\CO}(R)\otimes_{F_{\pi},W_{\CO}(R)}P ,\] satisfying the following equations:\[ V^{\sharp}(wF(x))=\pi w\otimes x,\quad w\in W_{\CO}(R), x\in P \] and \[ V^{\sharp}(wV^{-1}(y))=w\otimes y,\quad w\in W_{\CO}(R), y\in Q.\] If we denote by $ F^{\sharp}:W(R)\otimes_{F_{\pi},W_{\CO}(R)}P\to P $ the $ W_{\CO}(R)$-linearization of $ F:P\to P $, we have the properties: \[
\label{F-Ver}
F^{\sharp}\circ V^{\sharp} =\pi.\Id_P \text{\quad and\quad}  V^{\sharp}\circ F^{\sharp}=\pi.\Id_{W_{\CO}(R)\otimes_{F_{\pi},W_{\CO}(R)}P}.\] We define $ V^{n\sharp} $ similarly.
\item[3)] Like Construction \ref{cons016}, we construct the base change of a $3n$-display, with respect to a ring homomorphism $ R\to S $.
\item[4)] Assume that $ pR=0 $. Denote by $ \CP^{(q)} $ the base change of $ \CP $ with respect to the ring homomorphism Frob$^f:R\to R  $, sending $ r $ to $ r^{q} $. Analogous to Construction \ref{cons02}, we construct morphisms of ramified $3n$-displays \[ Fr_{\CP}:\CP\to\CP^{(q)} \text{\quad and\quad} Ver_{\CP}:\CP^{(q)}\to \CP,\] such that \[ Fr_{\CP}\circ Ver_{\CP} =\pi.\Id_{\CP^{(p)}} \text{\quad and\quad}  Ver_{\CP}\circ Fr_{\CP}=\pi.\Id_{\CP}.\]
\end{itemize}

Now we can define ramified (nilpotent) displays:

\begin{dfn}
\label{dfnramdisplay}
Let $ \CP=(P,Q,F,V^{-1}) $ be a ramified $3n$-display over $R$. Let $\pi$ be nilpotent in $R$. Then $ \CP $ is called \emph{ramified display} if it satisfies the \emph{nilpotence or $V$-nilpotence condition}, i.e., if there exists a natural number $ N $ such that the morphism $$ V^{N\sharp}:P\to W_{\CO}(R)\otimes_{F_{\pi}^{N},W_{\CO}(R)}P $$ is zero modulo $ I_R+\pi W_{\CO}(R)$.
\end{dfn}

\begin{dfn}
\label{dfnramdieudonnémodule}
Let $k$ be a perfect field of characteristic $p$, which is an $ \CO $-algebra. A ramified Dieudonn\'e module over $k$ is a finite free $ W_{\CO}(k) $-module $M$ endowed with an $F_{\pi}$-linear morphism $F : M \to M$  and an $ F_{\pi}^{-1} $-linear morphism $V :M\to M$, such that $FV=\pi=VF$.
\end{dfn}

We continue our list:

\begin{itemize}
\item[5)] As in the classical case, there is an equivalence of categories between the category of ramified $3n$-displays and the category of ramified Dieudonn\'e modules. Under this equivalence, nilpotent displays correspond to ramified Dieudonn\'e modules on which $ V $ is topologically (in the $ \pi $-adic topology) nilpotent (cf. \cite{TA} for more details).
\item[6)] Let $ \CN $ be a nilpotent $ R $-algebra. We construct $  \widehat{P}({\CN}), \widehat{Q}(\CN), G_{\CP}^{0}(\CN) $ and $G_{\CP}^{-1}(\CN)$ as in Construction \ref{cons012}. Also, we define the morphism $ V^{-1}-\Id: G_{\CP}^{-1}(\CN)\to G_{\CP}^{0}(\CN) $ and set $ BT_{\CP}(\CN) $ to be the cokernel of this morphism.
\end{itemize}

The following results are proved in \cite{TA}:

\begin{itemize}
\item[7)] For every nilpotent $R$-algebra $\CN$, we have an exact sequence \[  0\longrightarrow G_{\CP}^{-1}(\CN)\arrover{ V^{-1}-\Id} G_{\CP}^{0}(\CN)\longrightarrow BT_{\mathcal{P}}(\CN)\longrightarrow 0. \]
\item[8)] The functor $ BT_{\CP} $ from the category of nilpotent $R$-algebras to the category of $\CO$-modules is a finite dimensional formal $\CO$-module. The construction $ \CP\rightsquigarrow BT_{\CP} $ commutes with base change, i.e., if $ R\to S $ is a ring homomorphism, then there exists an canonical isomorphism $ (BT_{\CP})_S\cong BT_{\CP_S} $.
\item[9)] If $ \pi $ is nilpotent in $R$ and $ \CP $ is nilpotent, i.e., is a display, then $ BT_{\CP} $ is an infinitesimal $\pi$-divisible $ \CO $-module scheme.
\item[10)] If $pR=0$, then the Frobenius and Verschiebung morphisms of the $\pi$-divisible module $ BT_{\CP} $ are $ BT_{\CP}(Fr_{\CP}) $ and respectively $ BT_{\CP}(Ver_{\CP}) $.
\item[11)] Assume that $ \pi $ is nilpotent in $R$ and that $R$ is a Noetherian ring. Then the functor $ BT $, from the category of (nilpotent) displays over $R$ to the category of infinitesimal $ \pi $-divisible modules is an equivalence of categories.
\end{itemize}

The following are again analogue constructions and results from chapter 6:

\begin{itemize}
\item[12)] $ \CO $-Multilinear, symmetric, antisymmetric and alternating morphisms of ramified $3n$-displays are defined as in Definition \ref{def05} with the obvious additional requirement due to the presence of $ \CO $. Let $ R\to S $ be a ring homomorphism and $\CN$ a nilpotent $R$-algebra. Let $ \phi:\CP_1\times\dots\times \CP_r\to \CP_0 $ be a $ \CO $-multilinear morphisms of ramified $ 3n $-displays. The base change, $$ \phi_S:\CP_{1,S}\times\dots\times \CP_{r,S}\to \CP_{0,S} $$ and the $ \CO $-multilinear morphism $  \widehat{\phi}:\widehat{P}_1\times\dots\times\widehat{P}_r\to\widehat{P}_0$ are constructed similarly to Construction \ref{cons013}.
\item[13)] The $ \CO $-linear \[ \beta:\Mult^{\CO }(\CP_1\times\dots\times \CP_r,\CP_0)\to \Mult^{\CO}(BT_{\CP_1}\times\dots\times BT_{\CP_r},BT_{\CP_0})\] is constructed as in Construction \ref{cons05}. It maps alternating morphisms to alternating morphisms and commutes with base change (cf. Proposition \ref{prop0 16}). Note that the $ \CO $-linearity follows from the construction of $ \beta $ and the fact that we are considering $ \CO $-multilinear morphisms.
\item[14)] Now let $ \CP $ be of rank $1$. Exterior powers of $ \CP $ are constructed as in Construction \ref{cons03}. They enjoy the same properties as the exterior powers of a $3n$-display, stated in Lemma \ref{lem0 21}. That is to say, the construction is independent of the choice of a normal decomposition, and commutes with base change. If $ \CP $ is nilpotent, then the exterior powers of it are also nilpotent. If $ \CP $ has height $h$, then $ \epO^r\CP $, the $ r^{\text{th}} $-exterior power of $ \CP $, has height $ \binom{h}{r} $ and rank $ \binom{h-1}{r-1} $.
\end{itemize}

\section{The main theorem}

In this section $k$ denotes an algebraically closed field of characteristic $p>2$.



\begin{cons}
\label{consramdieudonne1}
We have the isomorphism \[ W\widehat{\otimes}_{\BZ_p}\BZ_p\longto \prod_{i\in \BZ/f\BZ} W\otimes_{\BZ_q,\sigma^{i}}\BZ_q\] sending an element $ (w\otimes a) $ to the element $ (w\otimes a)_i=(wa^{\sigma^{-i}} \otimes 1)_i$. The automorphism $ \sigma\otimes\Id $ on the left hand side induces an automorphism on the right hand side, permuting the factors. More precisely, if we denote by $e_i$ the primitive idempotent for the $ i^{\text{th}} $ factor, then, under this automorphism, $ e_i $  is mapped to $ e_{i-1} $ (for every $ i\in \BZ/f\BZ $). Tensoring with $ \CO $ over $ \BZ_q $, we obtain an isomorphism $  W\widehat{\otimes}_{\BZ_p}\CO\cong \prod W\widehat{\otimes}_{\BZ_q,\sigma^{i}}\CO $, with $ \sigma\otimes\Id $ permuting the factors. Let $ D $ be a Dieudonn\'e module over $k$, endowed with an action of $ \CO $, which acts on the tangent space of $D$  (i.e. on $D/VD$) through the scalar multiplication. Let us call such an action \emph{a scalar action}. The Dieudonn\'e module $D$ is a $   W\widehat{\otimes}_{\BZ_p}\CO$-module and therefore decomposes into a product $ M_0\times M_1\times \dots\times M_{f-1} $, where each $ M_i $ is a $ W\widehat{\otimes}_{\BZ_q,\sigma^{i}}\CO  $-module. The Verschiebung of $D$ is a $ \sigma^{-1}\otimes\Id $-linear morphism and so, for every $ i\in \BZ/f\BZ $, induces a semi-linear monomorphism $ V:M_i\to M_{i+1} $ (cf. Lemma \ref{lem 13}). As in the proof of Lemma \ref{lem 9}, we can show that each $ M_i $ is a free $ W\widehat{\otimes}_{\BZ_p}\CO $-module of rank $h$. Let us denote $ M_0 $ by $ \BH(D) $. It is a free $ W_{\CO}(k) $-module of rank $h$ and is endowed with an injective $^{F_{\pi}^{-1}} $-linear $ \CO $-module homomorphism $ V_{\pi}:=V^f $. If $D$ is the Dieudonn\'e module of a $ \pi $-divisible module, we call $ \BH(D) $ the \emph{ramified Dieudonn\'e module} of $\CM$ and denote it by $ \BH(\CM) $.
\end{cons}

\begin{lem}
\label{lemramdieudonne2}
Let $ D $ be a Dieudonn\'e module over $k$, with a scalar $ \CO $-action. There exists an $ ^{F_{\pi}} $-linear $ \CO $-module homomorphism $ F_{\pi}:\BH(D)\to \BH(D) $ such that $ F_{\pi}\circ V_{\pi}=V_{\pi}\circ F_{\pi}=\pi $.
\end{lem}

\begin{proof}
We show at first that the tangent space of $D$ is canonically isomorphic to $ \BH(D)/V_{\pi}\BH(D) $. With the notations of the above construction, we have $ D\cong  M_0\times\dots \times M_{f-1} $ and therefore, $$ VD\cong VM_{f-1}\times VM_0\times VM_1\times\dots\times VM_{f-2}.$$ Thus, the tangent space of $ D $ is isomorphic to the product $$ \dfrac{M_0}{VM_{f-1}}\times \dfrac{M_1}{VM_0}\times\dots\times \dfrac{M_{f-1}}{VM_{f-2}} .$$ Since by assumption the action of $ \CO $ on this $k$-vector space is via scalar multiplication, for every $ i\in \BZ/f\BZ\setminus\{0\} $, the quotient $ \dfrac{M_i}{VM_{i-1}} $ is trivial and so $ VM_{i-1}=M_i $. Consequently, $ V^fM_0=VM_{f-1} $ and so, the tangent space of $ D $ is isomorphic to $ M_0/V^fM_0 $, which is by definition equal to $ \BH(D)/V_{\pi}\BH(D) $. This proves the claim.\\

Since $ \pi $ goes to zero in $k$ and the the action of $ \CO $ on the tangent space of $D$ is by scalar multiplication, we have $ \pi \big(\BH(D)/V_{\pi}\BH(D)\big) =0$, i.e., $ \pi  \BH(D)\subseteq V_{\pi}\BH(D)$. Set $$F_{\pi}:=V_{\pi}^{-1}\pi:\BH(D)\to \BH(D).$$ This is a well defined $ ^{F_{\pi}} $-linear $ \CO $-module homomorphism and by definition, we have  $ F_{\pi}\circ V_{\pi}=V_{\pi}\circ F_{\pi}=\pi $.
\end{proof}

\begin{rem}
\label{remramdieudonne}
We have seen in the proof of the previous lemma that  $ VM_{i-1}=M_i $ for every $ i\in \BZ/f\BZ\setminus\{0\} $. Since by Lemma \ref{lem 13}, $V$ is injective, we conclude that $ V:M_{i-1}\to M_i $ is a semilinear isomorphism, for every $ i\in \BZ/f\BZ\setminus\{0\} $ and we have $$ D\cong\BH(D)\times V\BH(D)\times V^{2}\BH(D)\times\dots\times V^{f-1}\BH(D) .$$
\end{rem}

\begin{cor}
\label{corramdieudonne}
The quadruple $ \CP=(\BH(D),V_{\pi}\BH(D),F_{\pi},V_{\pi}^{-1}) $ is a ramified $ 3n $-display over $k$ of height $h$ and its rank is equal to the dimension of the $k$-vector space $ D/VD $.
\end{cor}

\begin{proof}
This follows from the fact that $ \BH(D) $ is a ramified Dieudonn\'e module over $k$ and the equivalence of ramified Dieudonn\'e modules over  $k$ and ramified $3n$-displays over $k$ (cf. point $5)$ from the list in the previous section).
\end{proof}




\begin{cons}
\label{consequivdieudonne}
In Construction \ref{consramdieudonne1}, we have constructed a functor, $ \BH $, from  the category of Dieudonn\'e modules over $k$ with a scalar $ \CO $-action to the category of ramified Dieudonn\'e modules over $k$ with scalar $ \CO $-action. Now, we want to construct a functor in the other direction, which will be a an inverse to $ \BH $. Let $ H $ be a ramified Dieudonn\'e module over $k$ such that the action of $ \CO $ on it is scalar. For every $i=0,\dots,f-1$, set $H_i:=W(k)\otimes_{\sigma^{-i},W(k)}H$ and let $\BD(H)$ be the finite free $W(k)\widehat{\otimes}_{\BZ_p}\CO$-module \[H_0\times H_1\times\dots\times H_{f-1}.\]
\end{cons}

\begin{lem}
\label{thefunctorD}
 There exists operators $V$ and $F$ on $\BD(H)$ that make $\BD(H)$ a Dieudonn\'e module with scalar $\CO$-action.
\end{lem}

\begin{proof}
Define operators $V$ and $F$ on $\BD(H)$ as follows: \[V(x_0,1\otimes x_1,\dots,1\otimes x_{f-1}):=(V_{\pi}(x_{f-1}),1\otimes x_0
\dots,1\otimes x_{f-2})\] and \[F(x_0,1\otimes x_1,\dots,1\otimes x_{f-1}):= (px_1,1\otimes px_2,\dots,1\otimes px_{f-1},1\otimes V_{\pi}^{-1}(px_0)).\] These are well-defined maps for the following reasons. Every element of $ H_i $ can be written as $ 1\otimes x_i $ with $ x_i\in H $, because $ \sigma:W(k)\to W(k) $ is an automorphism. As $pH\subseteq \pi H\subseteq V_{\pi}H$ and $ V_{\pi} $ is injective on $ H $ (this follows from the fact that $H$ is a free $ W_{\CO}(k) $-module, that $ F_{\pi}\circ V_{\pi}=\pi $ and  that $ \pi $ is a non-zero divisor of $ W_{\CO}(k) $), the element $ V_{\pi}^{-1}(px_0) $ is well-defined. It is now straightforward to check that $F$ and $V$ are respectively $ ^{\sigma} $ and $ ^{\sigma^{-1}} $-linear and that $F\circ V=p=V\circ F$. Therefore, we have a Dieudonn\'e module endowed with an $\CO$-action. By definition, the tangent space of $\BD(H)$ is isomorphic to $\BD(H)/V\BD(H)\cong H/V_{\pi}H$, since by assumption, the action of $\CO$ on this vector space is by scalar multiplication, we conclude that the action of $\CO$ on $\BD(H)$ is scalar.
\end{proof}

\begin{lem}
\label{equivalence}
 The functors $ \BD $ and $ \BH $ are inverse one to the other.
\end{lem}

\begin{proof}
Let $D$ be a Dieudonn\'e module over $k$ with scalar $\CO$-action. Then, as we explained before, we have \[D\cong M_0\times M_1\times \dots\times M_{f-1}\] with $V:M_{i-1}\to M_{i}$ a $^{\sigma^{-1}}$-linear isomorphism for every $i\in \BZ/f\BZ\setminus\{0\}$. Therefore, we have linear isomorphisms $$ W(k)\otimes_{\sigma^{-1},W(k)}M_{i-1}\to M_i$$ induced by $V$ and so isomorphisms $W(k)\otimes_{\sigma^{-i},W(k)}M_0\cong M_i$ for every $i$. Since by definition, we have $\BH(D)=M_0$, we conclude that $M_i\cong W(k)\otimes_{\sigma^{-i},W(k)}\BH(D)$. This shows that $\BD(\BH(D))\cong D$.\\

Now, let $H$ be a ramified Dieudonn\'e module over $k$ with scalar $\CO$-action. Then, by construction, $\BD(H)=H_0\times H_1\times\dots\times H_{f-1}$, where $H_i=W(k)\otimes_{\sigma^{-i},W(k)}H$. If we decompose the Dieudonn\'e module $\BD(H)$ as the product $M_0\times\dots\times M_{f-1}$, as in Construction \ref{consramdieudonne1}, then by construction, we have $M_i=H_i$. In particular, $\BH(\BD(H))=M_0=H_0=H$.
\end{proof}

\begin{rem}
\label{equivdisp}
The above arguments and the fact that the category of $3n$-displays (respectively ramified $3n$-displays) over $k$ is equivalent to the category of Dieudonn\'e modules (respectively ramified Dieudonn\'e modules) over $k$, show that we have functors $ \BH $ and $ \BD $ which define an equivalence of categories between the category of $3n$-displays over $k$ with a scalar $ \CO $-action and the category of ramified $ 3n $-displays over $k$ with a scalar $ \CO $-action.
\end{rem}

\begin{lem}
\label{isomult}
Let $ D_0,\dots, D_r $ (respectively $ \CP_0,\dots,\CP_r $) be Dieudonn\'e modules (respectively $3n$-displays) over $k$ with scalar $ \CO $-action. There exist canonical isomorphisms \[\Mult^{\CO}(\BH(D_1)\times\dots\times \BH(D_r),\BH(D_0))\cong  \Mult^{\CO}(D_1\times\dots\times D_r, D_0)\] and \[\Mult^{\CO}(\BH(\CP_1)\times\dots\times \BH(\CP_r),\BH(\CP_0))\cong  \Mult^{\CO}(\CP_1\times\dots\times \CP_r, \CP_0)\] functorial in all arguments. 
\end{lem}

\begin{proof}
Since the category of $3n$-displays (respectively ramified $3n$-displays) over $k$ is equivalent to the category of Dieudonn\'e modules (respectively ramified Dieudonn\'e modules) over $k$, it is enough to show the result for Dieudonn\'e modules). Let us denote $ \BH(D_i) $ by $ H_i $ and let $ \phi:H_1\times\dots\times H_r\to  H_0 $ be a $ W_{\CO}(k) $-multilinear morphism satisfying the $V_{\pi}$-condition.  Define $$ \chi(\phi):D_1\times\dots\times D_r\to D_0$$ as follows. For every $ i=1,\dots, r $, take an  element $ x_i \in D_i$. We know that every element of $ D_i $ can be written in a unique way as a sum $ a_0+Va_1+\dots+V^{f-1}a_{f-1} $ with $ a_j\in \BH(D_i) $ (cf. Remark \ref{remramdieudonne}). Therefore, we may assume that each $ x_i $ is of the form $ V^{\alpha_i}y_i $ with $ y_i\in \BH(D_i) $ and $ 0\leq \alpha_i\leq f-1 $. Set \[ \chi(\phi)(V^{\alpha_1}y_1,\dots,V^{\alpha_r}y_r):=\left\{ \begin{array}{rl} V^{\alpha}\phi(y_1,\dots,y_r) & \text{if } \alpha_1=\alpha_2=\dots=\alpha_r=\alpha\\ 0\qquad & \text{otherwise} \end{array} \right. \] It follows from the construction that $ \chi(\phi) $ is a $ W(k)\otimes_{\BZ_p}\CO $-multilinear morphism. We have to check that it satisfies the $ V $-condition. If $ \alpha $ is strictly smaller that $ f-1 $, then $ \alpha+1\leq f-1 $, and by definition, we have \[ V\chi(\phi)(V^{\alpha}y_1,\dots,V^{\alpha}y_r)=V^{\alpha+1}(\phi)(y_1,\dots,y_r)=\chi(\phi)(V^{\alpha+1}y_1,\dots,V^{\alpha+1}y_r).\] If $ \alpha=f-1 $, then we have $ V \chi(\phi)(V^{f-1}y_1,\dots,V^{f-1}y_r)=V^{f}\phi(y_1,\dots,y_r)$. Since $ \phi $ satisfies the $ V_{\pi} $-condition, and on $ \BH(D_0) $, $ V_{\pi} $ is equal to $ V^f $, we conclude that $$ V^{f}\phi(y_1,\dots,y_r)=\phi(V^fy_1,\dots,V^fy_r)=\chi(\phi)(V^fy_1,\dots,V^fy_r) .$$ This shows that $ \chi(\phi) $ belongs to the $ \CO $-module $  \Mult^{\CO}(D_1\times\dots\times D_r, D_0) $ and that we have an $ \CO $-linear homomorphism \[\chi:\Mult^{\CO}(\BH(D_1)\times\dots\times \BH(D_r),\BH(D_0))\to \Mult^{\CO}(D_1\times\dots\times D_r, D_0).\]

Now, we want to define a homomorphism, $ \Xi $, in the other direction, which will be the inverse of $ \chi $. Let $ \psi:D_1\times\dots\times D_r\to D_0 $ be a $ W\otimes_{\BZ_p}\CO $-multilinear morphism satisfying the $ V $-condition. We then obtain a $ W_{\CO}(k) $-multilinear morphism $$\Xi(\psi): \BH(D_1)\times\dots\times \BH(D_r)\to \BH(D_0) $$ just by restricting $ \psi $ on the first components of $ D_i $. As $ \psi $ satisfies the $ V $-condition, it also satisfies $ V^f $-condition, and thus, $ \Xi(\psi) $ satisfies the $ V_{\pi} $-condition. Consequently, we have a $ \CO $-linear homomorphism \[ \Xi: \Mult^{\CO}(D_1\times\dots\times D_r, D_0)\to \Mult^{\CO}(\BH(D_1)\times\dots\times \BH(D_r),\BH(D_0)).\] By construction, the composition $ \Xi\circ\chi $ is the identity. As for the composition $ \chi\circ\Xi $, we have $$ \chi\circ\Xi(\psi)(V^{\alpha}y_1,\dots,V^{\alpha}y_r)= V^{\alpha}\Xi(\psi)(y_1,\dots,y_r)=$$ $$V^{\alpha}\psi(y_1,\dots,y_r)=\psi(V^{\alpha}y_1,\dots,V^{\alpha}y_r)$$ where the last equality follows from the fact that $ \psi $ satisfies the $ V $-condition. This implies that the composition  $ \chi\circ\Xi $ is also the identity. The proof is now achieved.
\end{proof}

\begin{rem}
\label{remChi}
$  $
\begin{itemize}
\item[1)] Let $ \CP_0,\dots,\CP_r $ be $3n$-displays over $k$ with scalar $ \CO $-action. Then the $ \CO $-linear homomorphism $ \chi $ given in the above Lemma is given by the following formula. Take elements $ \phi\in \Mult^{\CO}(\BH(\CP_1)\times\dots\times \BH(\CP_r),\BH(\CP_0))$ and  $\textbf{x}:=(\vec{x}_1,\dots,\vec{x}_r)\in P_1\times\dots\times P_r $. Write $ \vec{x}_i=(x_{i,0},\dots,x_{i,f-1}) $ according to the decomposition $ P_i=P_{i,0}\times\dots\times P_{i,f-1} $. Then $$ \chi(\phi)(\textbf{x})=\big(\phi(x_{1,0},\dots,x_{r,0}),V\phi(V^{-1}x_{1,1},\dots,V^{-1}x_{r,1}),\dots$$ $$,V^{f-1}\phi(V^{-f+1}x_{1,f-1},\dots,V^{-f+1}x_{r,f-1})\big)$$ note that for every $ j>0 $, we have $ P_{i,j}=V^jP_{i,0}$ and the formula makes sense.
\item[2)] Let $ D $ be a nilpotent Dieudonn\'e module over $k$ (i.e., the corresponding $p$-divisible group is connected) with scalar $ \CO $-action and such that the dimension of its tangent space is 1. Then, by results of chapter 5 (cf. Proposition \ref{prop 9}), the exterior power $ \unset{W\widehat{\otimes}_{\BZ_p}\CO}{\ep^r}D $ is again a nilpotent Dieudonn\'e module with scalar $ \CO $-action (we are using the equivalence of $p$-divisible groups with $ \CO $-actions and $ \pi $-divisible modules). We have a canonical isomorphism \[ \BH(\unset{W\widehat{\otimes}_{\BZ_p}\CO}{\ep^r}D)\cong \unset{W_{\CO}(k)}{\ep^r}\BH(D).\] Indeed, if $ D=M_0\times\dots \times M_{f-1} $, then $$ \unset{W\widehat{\otimes}_{\BZ_p}\CO}{\ep^r}D=\unset{W_{\CO}(k)}{\ep}^rM_0\times\dots\times \unset{W_{\CO}(k)}{\ep^r}M_{f-1} .$$ It follows that under the isomorphism of the previous lemma, the universal alternating morphism $ \lambda:D^r\to \unset{W\widehat{\otimes}_{\BZ_p}\CO}{\ep^r}D $ corresponds to the universal alternating morphism $ \lambda:\BH(D)^r\to  \unset{W_{\CO}(k)}{\ep^r}\BH(D)$.
\end{itemize}
\end{rem}

\begin{cons}
\label{constrace}
Let $ D $ be a Dieudonn\'e module over $k$ with a scalar $ \CO $-action. Define a map $  \Tr:D\to \BH(D) $ as follows. Take an element  $x$ of  $D$. It can be written uniquely as a sum $x_0+Vx_1+\dots+V^{f-1}x_{f-1}$ with $ x_i\in \BH(D) $. Now set $\Tr(x):=x_0+x_1+\dots+x_{f-1}. $ Since the Verschiebung is an $ \CO $-linear homomorphism, the map $ \Tr $ is also an $ \CO $-linear homomorphism. In the same fashion, if $ \CP=(P,Q,F,V^{-1})$ is a $3n$-display over $k$, we obtain an $ \CO $-linear homomorphism $ \Tr:\CP\to \BH(\CP) $ (cf. Remark \ref{equivdisp}). Now, let $ \CN $ be a nilpotent $ k $-algebra, we want to define similarly, an $ \CO $-linear homomorphism $ \Tr:  \widehat{P}(\CN)\to \widehat{\BH(\CP)}(\CN)$. We have decompositions \[ P=P_0\times\dots P_{f-1} \] and \[ Q=Q_0\times\dots\times Q_{f-1} \] and we know that in fact, for every $ i=1,\dots,f-1 $, the two $ W_{\CO}(k) $-modules $ P_i $ and $ Q_i $ are equal (and are equal to $ V^{i}P_0 $). We also know that $ Q_0=V^fP_0 $. Thus, we have a decomposition \[ \widehat{W}({\CN})\otimes_{W(k)}P=\widehat{W}({\CN})\otimes_{W(k)}P_0\times \widehat{W}({\CN})\otimes_{W(k)}Q_1\times \dots \times \widehat{W}({\CN})\otimes_{W(k)}Q_{f-1}.\] Note that each component $\widehat{W}({\CN})\otimes_{W(k)}Q_i $ ($i=1,\dots, f-1$) is a submodule of $ \widehat{Q}_{\CN} $ and on it the morphism $ V^{-i} $ is defined. So, we can define an $ \CO $-linear homomorphism $ \Tr:\widehat{P}_{\CN}\to \widehat{\BH(P)}_{\CN} $ by sending an element $ (x_0,q_1,\dots, q_{f-1}) $ to the element $ (\mu\otimes\Id) (x_0+V^{-1}q_1+\dots+V^{-f+1}q_{f-1}) $, where $ u:\widehat{W}({\CN})\to \widehat{W}_{\CO}({\CN}) $ is the canonical morphism given in Proposition \ref{propmorramwitt}.
\end{cons}

\textbf{Notations.} Let $ \CP $ be a ramified $3n$-display over $k$ with an $ \CO $-action. In the following proposition and the theorem that follows it, we denote by $ BT_{\CP}^{\CO} $ the formal group associated to $ \CP $ (in order to stress the ring $ \CO $). So, if $ \CP $ is a $3n$-displays in the classical sense, we denote its formal group by $ BT_{\CP}^{\BZ_p} $.

\begin{prop}
\label{isoBTgroups}
Let $ \CP $ be a $3n$-display over $k$ with a scalar $ \CO $-action. The homomorphism $ \Tr $ induces an isomorphism $$\Tr: BT_{\CP}^{\BZ_p}\arrover{\cong} BT_{\BH(\CP)}^{\CO} $$ of formal $ \CO $-modules.
\end{prop}

\begin{proof}
We have to show that \[ \Tr((V^{-1}-\Id)\widehat{Q}_{\CN})\subseteq (V_{\pi}^{-1}-\Id)(\widehat{\BH(Q)}_{\CN}).\] So, take an element $ (q_0,\dots,q_{f-1})\in \widehat{Q}_{\CN} $. We have \[ \Tr((V^{-1}-\Id)(q_0,\dots,q_{f-1}))=\]\[\Tr(V^{-1}q_1-q_0,V^{-1}q_2-q_1,\dots,V^{-1}q_{f-1}-q_{f-2},V^{-1}q_0-q_{f-1})=\] \[ (\mu\otimes\Id)(V^{-1}q_1-q_0+V^{-1}q_2-q_1+\dots+V^{-1}q_{f-1}-q_{f-2}+V^{-f}q_0-q_{f-1})=\]\[(\mu\otimes\Id)(V^{-f}q_0-q_0)=(\mu\otimes\Id)(V^{-f}q_0)-(\mu\otimes\Id)(q_0).\] It is thus sufficient to show that 
\begin{myequation}
\label{eqVersch}
(\mu\otimes\Id)(V^{-f}q_0)=V_{\pi}^{-1}(\mu\otimes\Id)(q_0).
\end{myequation} 
We can assume that either $ q_0=\xi\otimes y_0 $ with $\xi\in \widehat{W}(\CN), y_0\in Q_0 $ or $ q_0=V^f\xi\otimes x_0 $ with $\xi\in \widehat{W}(\CN), y_0\in P_0 $. In the first case, we have \[ (\mu\otimes\Id)(V^{-f}q_0)=(\mu\otimes\Id)(F^f\xi\otimes V^{-f}y_0)=\mu(F^f\xi)\otimes V_{\pi}^{-1}y_0=\]\[F_{\pi}\mu(\xi)\otimes V_{\pi}^{-1}y_0=V_{\pi}^{-1}(\mu\otimes\Id)(\xi\otimes y_0)=V_{\pi}^{-1}(\mu\otimes\Id)(q_0)\] where we have used Proposition \ref{propmorramwitt} for the third equality. In the second case, we have \[ (\mu\otimes\Id)(V^{-f}q_0)=(\mu\otimes\Id)(\xi\otimes F^fx_0)=\mu(\xi)\otimes F^fx_0= \mu(\xi)\otimes V^{-f}(V^{f}F^fx_0)=\] \[\mu(\xi)\otimes V_{\pi}^{-1}(p^fx_0)=p^{f-1}\mu(\xi)\otimes V_{\pi}^{-1}(px_0)=  \mu(p^{f-1}\xi)\otimes V_{\pi}^{-1}(px_0)=\] \[\mu(F^{f-1}V^{f-1}\xi)\otimes V_{\pi}^{-1}(px_0)= \mu(F^{f-1}V^{f-1}\xi)\otimes V_{\pi}^{-1}\pi(\frac{p}{\pi}x_0)=\] \[ \mu(F^{f-1}V^{f-1}\xi)\otimes F_{\pi}(\frac{p}{\pi}x_0)= V_{\pi}^{-1}(V_{\pi}\mu(F^{f-1}V^{f-1}\xi)\otimes \frac{p}{\pi}x_0) =\]\[V_{\pi}^{-1}(\frac{p}{\pi}V_{\pi}\mu(F^{f-1}V^{f-1}\xi)\otimes x_0)=V_{\pi}^{-1}(\mu(V^f\xi)\otimes x_0)=V_{\pi}^{-1}(\mu\otimes\Id)(q_0)\] where for the penultimate equality we have used Proposition \ref{propmorramwitt}. This proves the claim (equality \ref{eqVersch}) and it follows that $ \Tr $ induces an $ \CO $-linear homomorphism $$\Tr: BT_{\CP}^{\BZ_p}\to BT_{\BH(\CP)}^{\CO} $$ and we have to show that it is an isomorphism. Let us construct $ \widehat{P}_{0,\CN}, \widehat{Q}_{0,\CN} $ and the monomorphism $ V^{-f}-\Id:\widehat{Q}_{0,\CN}\to \widehat{P}_{0,\CN} $ in the same way that we constructed $ \widehat{P}_{\CN}, \widehat{Q}_{\CN} $ and $ V^{-1}-\Id: \widehat{Q}_{\CN}\to \widehat{P}_{\CN}$. Let us denote by $ BT_{\CP_0}^{\BZ_p}(\CN) $ the cokernel of $ V^{-f}-\Id:\widehat{Q}_{0,\CN}\to \widehat{P}_{0,\CN} $. It follows that the homomorphism $ \Tr $ is equal to the composition $$ \widehat{P}_{\CN}\arrover{\Tr'}\widehat{P}_{0,\CN}\arrover{\mu\otimes \Id} \widehat{\BH(P)}_{\CN} $$ where $ \Tr'$ is the homomorphism sending $ (x_0,q_1,\dots, q_{f-1}) $ to the sum $ x_0+V^{-1}q_1+\dots+V^{-f+1}q_{f-1} $. It also follows from the above calculations that $$ \Tr'((V^{-1}-\Id)\widehat{Q}_{\CN})\subseteq (V^{-f}-\Id)(\widehat{Q}_{0,\CN}) $$ and $$ (\mu\otimes\Id)((V^{-f}-\Id)(\widehat{Q}_{0,\CN}))\subseteq (V^{-1}_{\pi}-\Id)(\widehat{\BH(Q)}_{\CN}) .$$ In other words, we have a commutative diagram:\[ \xymatrix{\widehat{P}_{\CN}\ar[r]^{\Tr'}\ar@{->>}[d]\ar@/^2pc/[rr]^{\Tr} & \widehat{P}_{0,\CN}\ar[r]^{\mu\otimes \Id}\ar@{->>}[d] & \widehat{\BH(P)}_{\CN}\ar@{->>}[d] \\ BT_{\CP}^{\BZ_p}(\CN)\ar[r]_{\Tr'}\ar@/_2pc/[rr]_{\Tr} & BT_{\CP_0}^{\BZ_p}(\CN)\ar[r]_{\overline{\mu\otimes\Id}} & BT_{\BH(\CP)}^{\CO}(\CN) .} \] It therefore suffices to show that $ \Tr' $ and $ \overline{\mu\otimes\Id} $ are isomorphisms.\\

The homomorphism $ \Tr':\widehat{P}_{\CN}\to \widehat{P}_{0,\CN} $ is surjective, since $ \Tr'(x_0,0,\dots,0)=x_0 $ for every $ x_0\in \widehat{P}_{0,\CN} $. Thus, $ \Tr' $ is surjective as well. Let $ (x_0,x_1,\dots,x_{f-1}) $ be an element in the kernel of $ \Tr' $. It means that there is an element $ q_0\in \widehat{Q}_{0,\CN} $ such that $$ \Tr' (x_0,x_1,\dots,x_{f-1})=x_0+V^{-1}x_1+\dots+V^{-f+1}x_{f-1}=(V^{-f}-\Id)(q_0).$$ Define, recursively, elements $ y_i\in \widehat{Q}_{\CN} $, with $ i=0,\dots,f-1 $, as follows: $ y_0:=q_0 $ and for $ 0<i<f-1 $, set $ y_i:=V^{-1}y_{i+1}-x_{i} $ where we calculate $ i+1 $ modulo $ f $, e.g., $ y_{f-1}=V^{-1}y_0-x_{f-1} $ and so on. It follows that the element $ (y_0,y_1,\dots,y_{f-1}) $ belongs to  $ \widehat{Q}_{\CN} $ and we have $$ (V^{-1}-\Id)(y_0,y_1,\dots,y_{f-1})=$$ $$(V^{-1}y_1-y_0,\dots,V^{-1}y_{f-1}-y_{f-2},V^{-1}y_0-y_{f-1})=(x_0,x_1,\dots,x_{f-1}) $$ and so $ (x_0,x_1,\dots,x_{f-1}) $ represents the zero element in $ BT_{\CP}^{\BZ_p}(\CN) $. It implies that $ \overline{'} $ is injective too, ans hence an isomorphism.\\

It remains to prove that $ \overline{\mu\otimes\Id}:BT_{\CP_0}^{\BZ_p}(\CN)\to BT_{\BH(\CP)}^{\CO}(\CN) $ is an isomorphism. We can reduce to the case, where $ \CN^2=0 $. Under this condition, there exist commutative diagrams (for details, we refer to \cite{TA} or \cite{Z1}):\[ \xymatrix{0\ar[r]&\widehat{Q}_{0,\CN}\ar@{^{(}->}[rr]\ar@{=}[d]&&\widehat{P}_{0,\CN}\ar[rr]\ar[d]_{V^{-1}-\Id}&&\CN\otimes_k\CT(\CP_0)\ar[r]\ar[d]_{\exp}&0 \\ 0\ar[r]&\widehat{Q}_{0,\CN}\ar[rr]_{V^{-1}-\Id}&&\widehat{P}_{0,\CN}\ar[rr]&&BT_{\CP_0}^{\BZ_p}(\CN)\ar[r]&0} \] and 

\[ \xymatrix{0\ar[r]&\widehat{\BH(Q)}_{\CN}\ar@{^{(}->}[rr]\ar@{=}[d]&&\widehat{\BH(P)}_{\CN}\ar[rr]\ar[d]_{V^{-1}-\Id}&&\CN\otimes_k\CT(\BH(\CP))\ar[r]\ar[d]_{\exp}&0 \\ 0\ar[r]&\widehat{\BH(Q)}_{\CN}\ar[rr]_{V^{-1}-\Id}&&\widehat{\BH(P)}_{\CN}\ar[rr]&&BT_{\BH(\CP)}^{\CO}(\CN)\ar[r]&0} \] where the vertical morphisms $ V^{-1}-\Id $ in the middle are extensions of the usual (horizontal) $ V^{-1}-\Id $ and the exponential morphism, which is $ \CO$-linear, is given by these diagrams. What is proved in \cite{TA} and \cite{Z1} is that the vertical $ V^{-1}-\Id $ are isomorphisms and so are the exponential morphisms.\\

Note that the tangent spaces $ \CT(\CP_0) $ and $ \CT(\BH(\CP)) $ are equal ($ P_0=\BH(P) $ and $ Q_0=\BH(Q)$) and the equality \ref{eqVersch} implies that we have a commutative diagram: \[ \xymatrix{\CN\otimes_k\CT(\CP_0)\ar[r]^{\exp}\ar[d]_{\Id} & BT_{\CP_0}^{\BZ_p}(\CN)\ar[d]^{\overline{\mu\otimes\Id}} \\ \CN\otimes_k\CT(\BH(\CP))\ar[r]_{\exp} & BT_{\BH(\CP)}^{\CO}(\CN).} \] Since the exponential morphisms are isomorphisms, we conclude that $ \overline{\mu\otimes\Id} $ is also an isomorphism. This achieves the proof.
\end{proof}

\begin{rem}
\label{remAhsen}
The proof of the above proposition is partly inspired by a similar result in \cite{TA}.
\end{rem}

\begin{rem}
\label{remrelpidivpdivmult}
Let $ \CM_0,\dots,\CM_r$ be $ \pi $-divisible $ \CO $-module schemes over a base scheme $S$. Let $ \phi:\CM_1\times\dots\times\CM_r\to \CM_0 $ be an $ \CO $-multilinear morphism, when $ \CM_i $ are considered as  $ p $-divisible groups with $ \CO $-action. Then $ \phi $ is not an $ \CO $-multilinear morphism of $ \pi $-divisible modules, but an appropriate ``twist" of it will be. Indeed, if we set $ \phi^{\sharp}_{ne}:=\frac{1}{u^{n(r-1)}}\phi_n $, then $$ \pi^e\phi^{\sharp}_{e(n+1)}(g_1,\dots,g_r)=\pi^e\frac{1}{u^{(n+1)(r-1)}}\phi_{n+1}(g_1,\dots,g_r)=$$ $$ \frac{1}{u^{(n+1)(r-1)+1}}p\phi_{n+1}(g_1,\dots,g_r)=\frac{1}{u^{(n+1)(r-1)+1}}\phi_n(pg_1,\dots,pg_r)=$$ \[\frac{1}{u^{(n+1)(r-1)+1}}\phi_n(u\pi^eg_1,\dots,u\pi^eg_r)=\frac{1}{u^{n(r-1)}}\phi_n(\pi^eg_1,\dots,\pi^eg_r)=\]\[\phi^{\sharp}_{ne}(\pi^eg_1,\dots,\pi^eg_r). \] Thus, the system $ \{\phi^{\sharp}_{ne}\}_n $ belongs to the inverse limit $$ \uset{n}{\invlim}\,\Mult^{\CO}(\CM_{1,ne}\times\dots\times\CM_{r,ne},\CM_{0,ne})\cong  \uset{n}{\invlim}\,\Mult^{\CO}(\CM_{1,n}\times\dots\times\CM_{r,n},\CM_{0,n})=$$ $$\Mult^{\CO}(\CM_1\times\dots\times \CM_r,\CM_0).$$ So, in the sequel, if we identify the $ \CO $-module of multilinear morphisms of $\pi$-divisible modules with the $ \CO $-module of such morphisms of the corresponding $p$-divisible groups with $ \CO $-action, we are implicitly using the above twist.
\end{rem}

The following lemma will be used in the proof of the next proposition.

\begin{lem}
\label{uglysum}
Let $A$ be a ring and $\alpha\in A$ a non-zero divisor. Let $ M_0,\dots,M_r $ be $A$-modules and $ \phi:M_1\times \dots\times M_r\to M_0 $ an $A$-multilinear morphism. Let $n$ be a positive natural number and $y_{i,0},y_{i,1},\dots,y_{i,n-1}, w_{i,0},w_{i,1},\dots, w_{i,n}\in M_i $ ($ i=1,\dots, r $) be elements subject to the following relations \[ \forall i=1,\dots,r\quad\text{and}\quad \forall j=0\dots n-1,\quad w_{i,j+1}=w_{i,j}+\alpha y_{i,j}. \] Then the following equality holds $$\sum_{i=1}^r\sum_{j=0}^{n-1}{\phi}(w_{1,j+1},\dots,w_{i-1,j+1},y_{i,j},w_{i+1,j},\dots,w_{r,j})=$$ $$\sum_{i=1}^r\sum_{j=0}^{n-1}{\phi}(w_{1,n},\dots,w_{i-1,n},y_{i,j},w_{i+1,0},\dots,w_{r,0}).$$
\end{lem}

\begin{proof}
For every $ i=0,\dots,r $, let $ \widetilde{M}_i $ be the free $A$-module with basis $\{e_{m_i}\mid m_i\in M_i \}$, in other words, $ \widetilde{M}_i=A^{M_i}\cong\bigoplus_{m_i\in M_i}Ae_{m_i} $ and define an $A$-multilinear morphism $ \widetilde{\phi}: \widetilde{M}_1\times\dots\times\widetilde{M}_r\to\widetilde{M}_0$ by setting on a basis element $ (e_{m_1},\dots,e_{m_r}) $\[  \widetilde{\phi} (e_{m_1},\dots,e_{m_r})=e_{\phi(m_1\dots,m_r)}\] and extend $ \widetilde{\phi} $ to the whole product $ \widetilde{M}_1\times\dots\times\widetilde{M}_r $ multilinearly. Also, for every $ i=0,\dots,r $, define $A$-module homomorphisms $ h_i:\widetilde{M}_i \to M_i$ by sending a basis element $ e_{m_i} $ to $ m_i $. The $A$-linear homomorphisms $h_i$ are surjective and from their definition and the definition of $ \widetilde{\phi} $, we have \[ h_0\circ  \widetilde{\phi}= \phi\circ (h_1\times\dots\times h_r)\] i.e., the following diagram is commutative \[ \xymatrix{\widetilde{M}_1\times\dots\times\widetilde{M}_r\ar[rr]^{\qquad\widetilde{\phi}}\ar[d]_{h_1\times\dots\times h_r} && \widetilde{M}_0\ar[d]^{h_0}\\ M_1\times\dots\times M_r\ar[rr]_{\qquad{\phi}}&& M_0.} \] For every $ i=1,\dots,r $ and every $ j=0,\dots,n-1 $ take elements $ \tilde{w}_{i,0},\tilde{y}_{i,j}\in \widetilde{M}_i $ such that $ h_i(\tilde{w}_{i,0})=w_{i,0} $ and $ h_i(\tilde{y}_{i,j})=y_{i,j} $ and set recursively $ \tilde{w}_{i,j+1}=\tilde{w}_{i,j} + \alpha \tilde{y}_{i,j}$. Then, because of the choice of $ \tilde{w}_{i,0}$ and $\tilde{y}_{i,j}$ and the relations among $ {w}_{i,j}$ and ${y}_{i,j}\in M_i$, we have \begin{myequation}\label{tildaversion}
 h_i(\tilde{w}_{i,j})=w_{i,j}
\end{myequation} for every $ i=1,\dots,r $ and every $ j=0,\dots,n $.\\

Assume that we have the ``tilde" version of the equality in the statement of the lemma, i.e., $$\sum_{i=1}^r\sum_{j=0}^{n-1}\widetilde{\phi}(\tilde{w}_{1,j+1},\dots,\tilde{w}_{i-1,j+1},\tilde{y}_{i,j},\tilde{w}_{i+1,j},\dots,\tilde{w}_{r,j})=$$ $$\sum_{i=1}^r\sum_{j=0}^{n-1}\widetilde{\phi}(\tilde{w}_{1,n},\dots,\tilde{w}_{i-1,n},\tilde{y}_{i,j},\tilde{w}_{i+1,0},\dots,\tilde{w}_{r,0}).$$ Then, applying the $A$-linear homomorphism $h_0$ on this equality, using the fact that the above diagram commutes and the relations \ref{tildaversion}, we obtain the desired equality of the lemma. So, we may assume that $ M_0 $ is a free $A$-module. Call the left hand side of the equality $ \FL $ and its right hand side $ \FR $. Using the multilinearity of $ \phi $ and the relations given in the lemma, we have \[ \alpha \FL= \sum_{j=0}^{n-1}\sum_{i=1}^r{\phi}(w_{1,j+1},\dots,w_{i-1,j+1},\alpha y_{i,j},w_{i+1,j},\dots,w_{r,j})= \] \[ \sum_{j=0}^{n-1}\sum_{i=1}^r{\phi}(w_{1,j+1},\dots,w_{i-1,j+1},w_{i,j+1}-w_{i,j},w_{i+1,j},\dots,w_{r,j})=\] \[ \sum_{j=0}^{n-1}\sum_{i=1}^r\big({\phi}(w_{1,j+1},\dots,w_{i-1,j+1},w_{i,j+1},w_{i+1,j},\dots,w_{r,j})- \] \[ {\phi}(w_{1,j+1},\dots,w_{i-1,j+1},w_{i,j},w_{i+1,j},\dots,w_{r,j})\big).\] Taking the sum over $i$, the terms cancel each other, except for the first and last terms (it is a telescopic sum) and the resulting sum will be \[ \sum_{j=0}^{n-1}\big({\phi}(w_{1,j+1},\dots,w_{r,j+1}) -  {\phi}(w_{1,j},\dots,w_{r,j})\big) =\] \[ {\phi}(w_{1,n},\dots,w_{r,n}) -  {\phi}(w_{1,0},\dots,w_{r,0}) \] again, because the terms cancel each other, except for the first and last terms. Similarly, we have \[\alpha \FR= \sum_{i=1}^r\sum_{j=0}^{n-1}{\phi}(w_{1,n},\dots,w_{i-1,n},\alpha y_{i,j},w_{i+1,0},\dots,w_{r,0})=\] \[ \sum_{i=1}^r\sum_{j=0}^{n-1}{\phi}(w_{1,n},\dots,w_{i-1,n},w_{i,j+1}-w_{i,j},w_{i+1,0},\dots,w_{r,0}) =\] \[ \sum_{i=1}^r\sum_{j=0}^{n-1}\big({\phi}(w_{1,n},\dots,w_{i-1,n},w_{i,j+1},w_{i+1,0},\dots,w_{r,0})-\] \[  {\phi}(w_{1,n},\dots,w_{i-1,n},w_{i,j},w_{i+1,0},\dots,w_{r,0})\big)=\] \[ \sum_{i=1}^r \big({\phi}(w_{1,n},\dots,w_{i-1,n},w_{i,n},w_{i+1,0},\dots,w_{r,0})\]\[ - {\phi}(w_{1,n},\dots,w_{i-1,n},w_{i,0},w_{i+1,0},\dots,w_{r,0})\big)=\] \[ {\phi}(w_{1,n},\dots,w_{r,n}) -  {\phi}(w_{1,0},\dots,w_{r,0})\] note that in this case we first sum over $j$ and then over $i$. It follows that $ \alpha \FL=\alpha \FR $ and since $M_0$ is a free $A$-module and $ \alpha $ is a non-zero divisor of $A$, we conclude that $ \FL=\FR $ and this completes the proof.
\end{proof}

\begin{rem}
\label{remuniversalizing}
Note that the assumption that $\alpha$ is a non-zero divisor is not necessary and a similar ``universalization" technique implies the equality of the two sums for an arbitrary element $ \alpha\in A $. However, we will use this lemma in the case where $ A $ is the ring $ \CO $ and $ \alpha $ is a non-zero element.
\end{rem}

\begin{prop}
\label{betadiagramramdisplay}
Let $ \CP_0,\dots,\CP_r $ be $3n$-displays over $k$ with scalar $ \CO $-action. Let us denote by $ G_i $ (respectively $ \tilde{G}_i $) the formal $ \CO $-module $ BT_{\CP_i}^{\BZ_p} $ (respectively $ BT_{\BH(\CP_i)}^{\CO} $). Then the following diagram is commutative \[ \xymatrix{\Mult^{\CO}(\CP_1\times\dots\times\CP_r,\CP_0)\ar[r]^{\Xi\qquad}\ar[d]_{\beta} & \Mult^{\CO}(\BH(\CP_1)\times\dots\times\BH(\CP_r),\BH(\CP_0))\ar[d]^{\beta} \\ \Mult^{\CO}(G_1\times\dots\times G_r,G_0)\ar[r]_{\Tr^*} & \Mult^{\CO}(\tilde{G}_1\times\dots\times \tilde{G}_r,\tilde{G}_0) } \] where $ \Xi $ and $ \Tr^* $ are respectively the $ \CO $-linear homomorphism given in Lemma \ref{isomult} and the $ \CO $-linear homomorphism induced by the isomorphisms $ \Tr:G_i\to\tilde{G}_i $ given in Proposition \ref{isoBTgroups}.
\end{prop}

\begin{proof}
Take an $ \CO $-multilinear morphism $ \phi\in \Mult^{\CO}(\BH(\CP_1)\times\dots\times\BH(\CP_r),\BH(\CP_0)) $, a nilpotent $k$-algebra $ \CN $,  and a vector $ \textbf{x}=(\vec{x}_1,\dots,\vec{x}_r)\in G_{1,n}(\CN)\times\dots\times G_{r,n}(\CN) $. Using Remark \ref{remrelpidivpdivmult}, we have to show that \[ \frac{1}{u^{n(r-1)}}\Tr\big(\beta(\chi(\phi))(\textbf{x})\big)=\beta(\phi)(\Tr(\vec{x}_1),\dots,\Tr(\vec{x}_r))\] where $ \chi $ is the isomorphism given in the proof of Lemma \ref{isomult} and the equality should hold in the $ \CO $-module $ G_{0,ne}(\CN) $. By definition (cf. Construction \ref{cons05}), we have \[\beta(\chi(\phi))(\textbf{x})=(-1)^{r-1}\sum_{i=1}^r[\widehat{\chi(\phi)}(V^{-1}\vec{g}_1,\dots,V^{-1}\vec{g}_{i-1},\vec{x}_i,\vec{g}_{i+1},\dots, \vec{g}_r)]\] where elements $ \vec{g}_i $ are given by the formula $ p^n\vec{x}_i=(V^{-1}-\Id)(\vec{g}_i).$ By definition of $ \chi $ (cf. Lemma \ref{isomult} and Remark \ref{remChi}) \[\Tr\big([\widehat{\chi(\phi)}(V^{-1}\vec{g}_1,\dots,V^{-1}\vec{g}_{i-1},\vec{x}_i,\vec{g}_{i+1},\dots, \vec{g}_r) ]\big)=\] 
\[\overline{\mu\otimes\Id}[\sum_{j=0}^{f-1}(\widehat{\phi}(V^{-j-1}g_{1,j+1},\dots,V^{-j-1}g_{i-1,j+1},V^{-j}x_{i,j},V^{-j}g_{i+1,j},\dots,V^{-j}g_{r,j}))]
\]
where $ \vec{x}_i=(x_{i,0},\dots,x_{i,f-1}) $ and $ \vec{g}_i=(g_{i,0},\dots,g_{i,f-1}) $. Therefore, we have $$ \frac{1}{u^{n(r-1)}}\Tr\big(\beta(\chi(\phi))(\textbf{x})\big)= \frac{(-1)^{r-1}}{u^{n(r-1)}}\overline{\mu\otimes\Id}[\sum_{i=1}^r\sum_{j=0}^{f-1}(\widehat{\phi}(V^{-j-1}g_{1,j+1},\dots $$  
\begin{myequation}
\label{eqforbeta}
,V^{-j-1}g_{i-1,j+1},V^{-j}x_{i,j},V^{-j}g_{i+1,j},\dots,V^{-j}g_{r,j}))].
\end{myequation}

Again, by the definition of $ \beta $, we have $\beta(\phi)(\Tr(\vec{x}_1),\dots,\Tr(\vec{x}_r))=$ \[ (-1)^{r-1}\sum_{i=1}^r[\widehat{\phi}(V_{\pi}^{-1}h_1,\dots,V_{\pi}^{-1}h_{i-1},\Tr(\vec{x}_i),h_{i+1},\dots,h_r)]\] where $ \pi^{ne}\Tr(\vec{x}_i)=(V_{\pi}^{-1}-\Id)(h_i) $. Since $ p^n\vec{x}_i=(V^{-1}-\Id)(\vec{g}_i)$, for every $ j=0,\dots,f-1 $, we have
\begin{myequation}
\label{formulaforX}
\pi^{ne}x_{i,j}=\frac{1}{u^{n}}(V^{-1}g_{i,j+1}-g_{i,j})
\end{myequation}
 and therefore $$\pi^{ne}\Tr(\vec{x}_i)=\frac{1}{u^{n}}(\mu\otimes\Id)(V^{-f}g_{i,0}-g_{i,0})=\frac{1}{u^{n}}(V_{\pi}^{-1}-\Id)(\mu\otimes\Id)g_{i,0} $$ which implies that $ h_i=(\mu\otimes\Id)\frac{1}{u^{n}}g_{i,0} $. Thus, \[[\widehat{\phi}(V_{\pi}^{-1}h_1,\dots,V_{\pi}^{-1}h_{i-1},\Tr(\vec{x}_i),h_{i+1},\dots,h_r)]=\] \[[\widehat{\phi}(V_{\pi}^{-1}h_1,\dots,V_{\pi}^{-1}h_{i-1},(\mu\otimes\Id)\sum_{j=0}^{f-1}V^{-j}x_{i,j},h_{i+1},\dots,h_r)]=\]\[ \frac{1}{u^{n(r-1)}}\overline{\mu\otimes\Id}\sum_{j=0}^{f-1}[\widehat{\phi}(V^{-f}g_{1,0},\dots,V^{-f}g_{i-1,0},V^{-j}x_{i,j},g_{i+1,0},\dots,g_{r,0})].\]
 
Consequently, we have $\beta(\phi)(\Tr(\vec{x}_1),\dots,\Tr(\vec{x}_r))=$ \[ \frac{(-1)^{r}}{u^{n(r-1)}}\overline{\mu\otimes\Id}\sum_{i=1}^r\sum_{j=0}^{f-1}[\widehat{\phi}(V^{-f}g_{1,0},\dots,V^{-f}g_{i-1,0},V^{-j}x_{i,j},g_{i+1,0},\dots,g_{r,0})]. \] So, in order to show that \[ \frac{1}{u^{n(r-1)}}\Tr\big(\beta(\chi(\phi))(\textbf{x})\big)=\beta(\phi)(\Tr(\vec{x}_1),\dots,\Tr(\vec{x}_r))\] it is sufficient to show the following equality 
$$\sum_{i=1}^r\sum_{j=0}^{f-1}\widehat{\phi}(V^{-j-1}g_{1,j+1},\dots,V^{-j-1}g_{i-1,j+1},V^{-j}x_{i,j},V^{-j}g_{i+1,j},\dots,V^{-j}g_{r,j})=$$

\begin{myequation}
\label{equalityofsigmas}
\sum_{i=1}^r\sum_{j=0}^{f-1}\widehat{\phi}(V^{-f}g_{1,0},\dots,V^{-f}g_{i-1,0},V^{-j}x_{i,j},g_{i+1,0},\dots,g_{r,0}).
\end{myequation}

For every $ i=1,\dots,r $ and $ j=0,\dots,f $, set $ w_{i,j}:=V^{-j}g_{i,j} $ and $ y_{i,j}:=V^{-j}x_{i,j} $. Since $ V^{-1} $ is an $ \CO $-linear homomorphism, it follows from the relations \ref{formulaforX} that $$ u^n\pi^{ne}y_{i,j}=V^{-j}x_{i,j}=V^{-j-1}g_{i,j+1}-V^{-j}g_{i,j}=w_{i,j+1}-w_{i,j} .$$ If we define $ \alpha:= u^n\pi^{ne} $, then $ \alpha $ is a non-zero divisor of $ \CO $ and we have $w_{i,j+1}= \alpha y_{i,j}+w_{i,j} $. With these new notations, the equality \ref{equalityofsigmas} becomes $$\sum_{i=1}^r\sum_{j=0}^{f-1}\widehat{\phi}(w_{1,j+1},\dots,w_{i-1,j+1},y_{i,j},w_{i+1,j},\dots,w_{r,j})=$$ $$\sum_{i=1}^r\sum_{j=0}^{f-1}\widehat{\phi}(w_{1,f},\dots,w_{i-1,f},y_{i,j},w_{i+1,0},\dots,w_{r,0})$$ which we know holds thanks to Lemma \ref{uglysum}.
\end{proof}


\begin{cor}
\label{corbetaisoramdis}
Let $ \CP_0,\dots,\CP_r $ be ramified displays over $k$. The $ \CO $-linear homomorphism \[ \beta:\Mult^{\CO}(\CP_1\times\dots\times\CP_r,\CP_0) \to \Mult^{\CO}(BT_{\CP_1}\times\dots\times BT_{\CP_r},BT_{\CP_0}) \] is an isomorphism.
\end{cor}

\begin{proof}
For every $ i=0,\dots,r $, let $ \tilde{\CP}_i $ be a display (non-ramified) over $k$ such that $ \BH(\tilde{\CP}_i)\cong \CP_i $ (cf. Remark \ref{equivdisp}). We know by Corollary \ref{cortechresult} that \[ \beta:\Mult(\tilde{\CP}_1\times\dots\times\tilde{\CP}_r,\tilde{\CP}_0)\to \Mult(BT_{\tilde{\CP}_1}\times\dots\times BT_{\tilde{\CP}_r},BT_{\tilde{\CP}_0}) \] is an isomorphism. Since $ \beta $ preserves the $ \CO $-module structure, it induces an isomorphism \[ \Mult^{\CO}(\tilde{\CP}_1\times\dots\times\tilde{\CP}_r,\tilde{\CP}_0)\arrover{\cong} \Mult^{\CO}(BT_{\tilde{\CP}_1}\times\dots\times BT_{\tilde{\CP}_r},BT_{\tilde{\CP}_0}).\] Now, consider the diagram in Proposition \ref{betadiagramramdisplay}. As $ \Xi, \Tr^* $ and $ \beta $ on the left hand side are isomorphisms, it follows that \[ \beta:\Mult^{\CO}(\CP_1\times\dots\times\CP_r,\CP_0) \to \Mult^{\CO}(BT_{\CP_1}\times\dots\times BT_{\CP_r},BT_{\CP_0}) \] is an isomorphism as well.
\end{proof}

One of the main results of section 8.3 is the existence of the exterior powers of truncated Barsotti-Tate groups of dimension 1 over local Artin rings with residue characteristic $p$ (cf. Propositions \ref{prop013} and \ref{prop014}). The proof of this result is based on the fact that $\beta$ is an isomorphism when the base is a perfect field of characteristic $p$. Now that we have the above isomorphism, we can prove in the same way, the similar statement. We therefore omit its proof and the auxiliary statements leading to it.\\

Let $R$ be a local Artin $\CO$-algebra with residue characteristic $p$ and $\CM$ a $\pi$-divisible $\CO$-module scheme over $R$ of height $h$, whose special fiber is a connected $\pi$-divisible module of dimension 1. Let us denote by $\CP$ the ramified display associated to $\CM$. The exterior power $\ep^r\CP$ exists (cf. point 14 from the list in the previous section) and we denote by $\Lambda_R^r$ the $\pi$-divisible $\CO$-module associated to $\ep^r\CP$ (cf. point 11). The universal alternating morphism $\lambda:\CP^r\to\ep^r\CP$ induces an alternating morphism $\beta_{\lambda}:\CM^r\to\Lambda_R^r$. So, we obtain an alternating morphism $\beta_{\lambda,n}:\CM_n^r\to\Lambda_{R,n}^r$. As in Construction \ref{cons07}, for every group scheme $X$, we obtain the morphism \[\underline{\lambda^*_n}(X):\innHom_{R}(\Lambda^r_{R,n},X)\to \widetilde{\innAlt}^{\CO}_{R}(\CM_n^r,X).\] Note that $\innHom_{R}(\Lambda^r_{R,n},X)$ is the $\CO$-module of group scheme homomorphisms. If $R$ is a perfect field, then by Remark \ref{remChi} 2), Corollary \ref{corbetaisoramdis}, Corollary \ref{cordieudonnpidiv} and Remark \ref{rem 17}, $ \Lambda_{R,n}^r $ is the $ r^{\text{th}} $-exterior power of $ \CM_n $ and $ \beta_{\lambda,n} $ is the universal alternating morphism.

\begin{prop}
\label{prop014ram}
For every finite and flat group scheme $X$ over $R$, the morphisms $$\ul{\lambda}_n^*(X):\innHom_{R}(\Lambda^r_{R,n},X)\to \widetilde{\innAlt}^{\CO}_{R}(\CM^r_n,X)$$ and \[\ul{\lambda}_n^*(\BG_m):\innHom_{R}(\Lambda^r_{R,n},\BG_m)\to \widetilde{\innAlt}^{\CO}_{R}(\CM^r_n,\BG_m)\] are isomorphisms. Consequently, $\beta_{\lambda,n}:\CM_n^r\to \Lambda^r_{R,n} $ is the $ r^{\text{th}} $-exterior power of $ \CM_n $ in the category of finite and flat group schemes over $R$ (in the sense of Definition \ref{def 7}).
\end{prop}

\begin{cor}
\label{corfiniteflat}
The $\CO$-module scheme $\widetilde{\innAlt}^{\CO}_{R}(\CM^r_n,\BG_m)$ is finite and flat over $R$ and its order is equal to $q^{n\binom{h}{r}}$.
\end{cor}

We will reduce the general case to the situation over a local Artin ring, but before, we need few technical lemmas.

\begin{lem}
\label{lemsplitting1}
Let $S$ a base scheme, $A$ a ring and $b$ an element of $A$ such that $A/b$ is finite (as a set). Let \[0\to N\to M\arrover{pr} \ul{A/b}\to 0\] be a short exact sequence of finite $A$-module schemes over $S$, where $\ul{A/b}$ denotes the constant $A$-module scheme over $S$ corresponding to the $A$-module $A/b$. Assume that $b$ annihilates $M$ and there exists an scheme-theoretic section $s:S\to M$ such that its composition with the projection $pr:M\onto \ul{A/b}$ is the closed embedding $S\into \ul{A/b}$ corresponding to the element $1\in A/b$ (note that $\ul{A/b}$ is the disjoint union $ \coprod_{x\in A/b}S$). Then there exists an $A$-module scheme section, $\ul{A/b}\into M$, of the projection $pr:M\onto \ul{A/b}$ and so the above short exact sequence is split.
\end{lem}

\begin{proof}
We have \[ \Hom_S^{A}(\ul{A/b},M)\cong \Hom^{A}(A/b,M(S))\cong M(S)[b]=M(S) \] where for the last equality, we used the fact that $ M $ is annihilated by $ b $. Let $ g:\ul{A/b}\to M $ be the $ A $-linear homomorphism mapping to the section $ s $ under this isomorphism. We have a commutative diagram \[ \xymatrix{\Hom_S^{A}(\ul{A/b},M)\ar[r]^{pr^*}\ar[d]_{\cong} & \Hom_S^{A}(\ul{A/b},\ul{A/b}) \ar[d]^{\cong}\\ M(S)\ar[r]_{pr_S} & \ul{A/b}(S)} \] where the horizontal morphisms are induced by the projection $ pr $, the left vertical isomorphism is the one given above and the right vertical isomorphism is given similarly. By assumption, $ pr_S(s) $ is the closed embedding $ S\into \ul{A/b} $ corresponding to the element $1\in \ul{A/b}$ and thus, under the isomorphism $ \ul{A/b}(S)\cong \Hom_S^{A}(\ul{A/b},\ul{A/b})  $, it is mapped to the identity morphism. It follows from the definition of $ g $ and the commutativity of the above diagram, that $ pr\circ g $ is the identity of $ \ul{A/b} $.
\end{proof}

\begin{lem}
\label{lemsplitting2}
Let $S$ be a scheme whose underlying set consists of one point. Let $A$ be a ring and $b_1,\dots,b_n$ elements of $A$ such that the quotient ring $A/b_i$ is a finite set (for every $i=1,\dots,n$). Let \[(\xi)\qquad 0\to N\to M\arrover{pr} \bigoplus_{i=1}^n\ul{A/b_i}\to 0\] be a short exact sequence of finite flat $A$-module schemes, such that for every $i=1,\dots,n$, $b_i$ annihilates $M$. Then there exists a finite and faithfully flat morphism $T\to S$ such that the above short exact sequence splits after extending the base to $T$.
\end{lem}

\begin{proof}
For every $ j=1,\dots, n $, let us denote by $ N_j $ the kernel of the projection $$ M\ontoover{pr} \bigoplus_{i=1}^n\ul{A/b_i}\ontoover{pr_j} \bigoplus_{i=1}^j\ul{A/b_i}$$ and set $ N_0:=M $. Then, for every $ j=0,\dots, n-1 $, we have a short exact sequences \[ (\xi_j):\qquad 0\to N_{j+1}\to N_j\to \ul{A/b_{j+1}}\to 0.\] Assume that for every $ j=0,\dots, n-1 $, we can find a finite and faithfully flat morphism $ T_i\to S $ such that the short exact sequence $ (\xi_j) $ is split over $ T_j $ and set $ T=T_0\times_ST_1\times\dots\times_ST_{n-1} $. Then, since $ (\xi_j)_{T_j} $ is split, $(\xi_j)_T\cong ((\xi_j)_{T_j})_T $ splits too. These splittings, provide a splitting of $ (\xi)_T $. As each $ T_j $ is finite and faithfully flat over $S$, their product, $ T $, is finite and faithfully flat over $S$ as well. So, it is enough to show the existence of $ T_j $.\\

We show the following assertion: let \[ 0\to P\to Q\arrover{pr} \ul{A/b}\to 0 \] be a short exact sequence of finite flat $ A $-module schemes over $S$, where $ b\in A $ is an element annihilating $ Q $ and such that $ A/b $ is a finite set, then there exists a finite and faithfully flat morphism $ T\to S $ such that this short exact sequence splits over $T$.\\

We can write $ Q $ as a disjoint union $ \coprod_{x\in A/b}Q_x $ of closed submodule schemes, with $ Q_0=P $. Set $ T:=Q_1 $. Since $ Q $ is finite and flat over $S$, all $ Q_x $ and in particular $ T $ are also finite and flat  over $S$. By assumption, $ S $ has a single point and so, $ T\to S $ is surjective. This shows that $ T $ is faithfully flat and finite over $S$. The embedding $ T\into Q $ induces a morphism $ T\to Q_T $ over $T$, i.e., a section of the structural morphism $ Q_T\to T $. By definition, this section has the property that its composition with the projection $ Q_T\onto \ul{A/b}_T=\coprod_{x\in A/b}T $ is the embedding corresponding to the element $ 1\in A/b $. Now, we are in the situation of the previous lemma, and applying it, we deduce that the short exact sequence  \[ 0\to P_T\to Q_T\arrover{pr} \ul{A/b}_T\to 0 \] is split. This proves the above assertion.\\

By assumption, $ b_ {j+1}$ annihilates $ M $ and since $ N_j $ is a submodule scheme of $ M $, it annihilates $ N_j $ too. We apply the above assertion to the sequence $ (\xi_j) $ and obtain the desired $ T_j $.
\end{proof}

\begin{lem}
\label{lemsplitting3}
Let $\CM$ be a $\pi$-divisible $\CO$-module scheme over a local Artin ring $R$. Then, for every positive natural number $n$, there exists a finite and faithfully flat morphism $T\to\Spec(R)$ such that $\CM_{n,T}$ becomes isomorphic a the direct sum $(\CM_n^0)_T\oplus (\ul{\CO/\pi^n})^{h^{\text{\'et}}}$, where $\CM^0$ is the connected component of $\CM$ and $h^{\text{\'et}}$ denotes the height of the \'etale part of $\CM$.
\end{lem}

\begin{proof}
Fix a positive natural number $n$ and set $ S:=\Spec(R) $. By Proposition \ref{etalepidiv}, there exists a connected finite \'etale cover $ S'\to S $ such that $ \CM_{n,S'}^{\text{\'et}} $ is the constant $ \CO $-module scheme $ \ul{(\CO/\pi^n)}^{h^{\text{\'et}}}  $ over $S'$. Consider the short exact sequence \[ 0\to (\CM_n^0)_{S'}\to \CM_{n,S'}\to \ul{(\CO/\pi^n)}^{h^{\text{\'et}}} \to 0 \] induced by the connected-\'etale sequence of $ \CM_n $ over $S$ (cf. Remark \ref{rem 12}). The $ \CO $-module $ \CM_{n,S'} $ is annihilated by $ \pi^n $. Since $S'$ is finite over $S$, it is the spectrum of a finite $R$-algebra $R'$. Thus, the Krull dimension of $R'$ is equal to the Krull dimension of $R$, which is zero. Since $R$ is Noetherian and $R'$ is finite over it, $R'$ is also Noetherian. This shows that $R'$ is an Artin ring. Further, $S'=\Spec(R')$ is connected, which means that $R'$ is local. Consequently, $ S' $ is a scheme with 1 point. We can now apply the previous lemma and find a finite faithfully flat morphism $ T\to S' $ that splits the above short exact sequence. Hence $ \CM_{n,T}\cong (\CM_n^0)_T\oplus (\ul{\CO/\pi^n})^{h^{\text{\'et}}} $. Since $ T\to S' $ is finite and faithfully flat and $ S'\to S $ is finite, surjective (it is a cover) and \'etale, the composition $ T\to S'\to S $ is finite and faithfully flat.
\end{proof}

From now on, we assume that $\CM$ is a $\pi$-divisible $\CO$-module scheme over a base scheme $S$ (defined over $\Spec(\CO)$) and of height $h$ and dimension at most 1.

\begin{lem}
\label{lemextArtin}
The $\CO$-module scheme $\widetilde{\innAlt}^{\CO}_{S}(\CM^r_n,\BG_m)$ is finite and flat over $S$, when $S$ is the spectrum of a local Artin $\CO$-algebra. Moreover, its order is the constant function $q^{n\binom{h}{r}}$.
\end{lem}

\begin{proof}
If the dimension of the special fiber of $ \CM $ is zero, then it is an \'etale $ \pi $-divisible module and we know by Proposition \ref{prop025} that the exterior power $ \epO^r\CM_n $ exists and is a finite \'etale $ \CO $-module scheme of order $ q^{n\binom{h}{r}} $. We also know that the construction of this exterior power commutes with arbitrary base change. Thus  \[ \widetilde{\innAlt}^{\CO}_{S}(\CM^r_n,\BG_m)\cong \innHom_S(\epO^r\CM_n,\BG_m) \] The latter, being the Cartier dual of $ \epO^r\CM_n $, is a finite flat $ \CO $-module scheme of order $ q^{n\binom{h}{r}} $. So, we may assume that the dimension of $ \CM $ at the closed point of $S$ is one and so, the closed point of $S$ has characteristic $p$.\\

If we can find a finite and faithfully flat morphism $ T\to S $ such that the base change of $\widetilde{\innAlt}^{\CO}_{S}(\CM^r_n,\BG_m)$ to $T$ is finite and flat over $T$, then $\widetilde{\innAlt}^{\CO}_{S}(\CM^r_n,\BG_m)$ is finite and flat over $S$ (a module over a ring is finite and flat if and only if it is so after base change to a faithfully flat ring extension). So, it is enough to find such a $T$. Let $T\to S$ be a finite faithfully flat morphism such that $$ (\star)\qquad \CM_{n,T}\cong  (\CM_n^0)_T\oplus (\ul{\CO/\pi^n})^{h^{\text{\'et}}}$$ (provided by the previous lemma) and to simplify the notations, write $ \Gamma:= (\ul{\CO/\pi^n})^{h^{\text{\'et}}}$. Set $ Y:=\widetilde{\innAlt}^{\CO}_S((\CM_n^0)^r,\BG_m) $. Since $ \CM^0 $ is a connected $ \pi $-divisible module of dimension 1, by Corollary \ref{corfiniteflat}, $ Y $ is finite and flat over $S$. Thus $$ Y_T=\widetilde{\innAlt}^{\CO}_S((\CM_n^0)^r,\BG_m)_T\cong \widetilde{\innAlt}^{\CO}_T((\CM_n^0)_T^r,\BG_m)  $$ is finite and flat over $T$. By isomorphism $ (\star) $ and adjunction formulas given in Proposition \ref{prop 4}, we have \[ \widetilde{\innAlt}^{\CO}_{T}(\CM^r_{n,T},\BG_m) \cong  \widetilde{\innAlt}^{\CO}_{T}((\CM_n^0)_T^r\oplus \Gamma^r,\BG_m)\cong\]\[ \widetilde{\innAlt}^{\CO}_{T}\big(\Gamma^r, \widetilde{\innAlt}^{\CO}_{T}((\CM_n^0)_T^r,\BG_m)\big)\cong \widetilde{\innAlt}^{\CO}_{T}(\Gamma^r, Y_T).\] By Proposition \ref{prop025} (and its proof), the exterior power $ \epO^r\Gamma $ exists, and is isomorphic to the constant $ \CO $-module scheme $ (\ul{\CO/\pi^n})^{\binom{h^{\text{\'et}}}{r}} $. We therefore have $$ \widetilde{\innAlt}^{\CO}_{T}(\Gamma^r, Y_T)\cong \innHom_{T}^{\CO}(\epO^r\Gamma,Y_T)\cong \bigoplus\innHom_T^{\CO}(\ul{\CO/\pi^n},Y_T) .$$ Thus, it is sufficient to show that $ \innHom_T^{\CO}(\ul{\CO/\pi^n},Y_T) $ is finite and flat over $T$. We claim that this $ \CO $-module scheme is isomorphic to $ Y_T $. Indeed, since $ \CM_n $ is annihilated by $ \pi^n $, its connected component $ \CM_n^0 $ is annihilated by $ \pi^n $ too, and therefore $ Y_T $ is annihilated by $ \pi^n $ as well. It follows that for every $T$-scheme $ T' $, we have isomorphisms  $$ \innHom_T^{\CO}(\ul{\CO/\pi^n},Y_T)(T')\cong \Hom_{T'}^{\CO}(\ul{\CO/\pi^n},Y_{T'})\cong $$  $$\Hom^{\CO}(\CO/\pi^n,Y(T'))\cong \Hom^{\CO}(\CO,Y(T'))\cong Y(T')=Y_T(T').$$ This proves the claim. Hence the finiteness and flatness of $ \innHom_T^{\CO}(\ul{\CO/\pi^n},Y_T) $.\\

To prove the statement on the order of $ \widetilde{\innAlt}^{\CO}_{S}(\CM^r_n,\BG_m) $, it suffices to assume that $S$ is the spectrum of an algebraically closed field. The statement then follows from the fact that the exterior power $ \epO^r\CM_n $ has order $ q^{n\binom{h}{r}} $ (cf. Theorem \ref{thm 4}) and therefore its Cartier dual $ \innHom_S(\epO^r\CM_n,\BG_m)\cong \widetilde{\innAlt}^{\CO}_{S}(\CM^r_n,\BG_m) $ has order $ q^{n\binom{h}{r}} $.
\end{proof}

\begin{lem}
\label{propextclNr}
The $\CO$-module scheme $\widetilde{\innAlt}^{\CO}_{S}(\CM^r_n,\BG_m)$ is finite and flat over $S$, when $S$ is the spectrum of a complete local Noetherian $\CO$-algebra. 
\end{lem}

\begin{proof}
By Remark \ref{rem 26}, we know that $\widetilde{\innAlt}^{\CO}_{S}(\CM^r_n,\BG_m)$ is an affine scheme of finite type over $S$, and so is separated over $S$. Let $R$ be a local Artin $ \CO $-algebra and $ f:\Spec(R)\to S $ an $ \CO $-scheme morphism. We have $$ \widetilde{\innAlt}^{\CO}_{S}(\CM^r_n,\BG_m)_R\cong \widetilde{\innAlt}^{\CO}_{R}(\CM^r_{R,n},\BG_m).$$ By previous lemma, $ \widetilde{\innAlt}^{\CO}_{R}(\CM^r_{R,n},\BG_m) $ is a finite flat $ \CO $-module scheme over $R$. We can now apply Lemma \ref{lemhensel} and conclude that $ \widetilde{\innAlt}^{\CO}_{S}(\CM^r_n,\BG_m) $ is finite and flat over $S$.
\end{proof}

\begin{lem}
\label{propextlNs}
The $\CO$-module scheme $\widetilde{\innAlt}^{\CO}_{S}(\CM^r_n,\BG_m)$ is finite and flat over $S$, when $S$ is a locally Noetherian $\CO$-scheme. Moreover, its order is the constant function $q^{n\binom{h}{r}}$.
\end{lem}

\begin{proof}
Set $ X:= \widetilde{\innAlt}^{\CO}_{S}(\CM^r_n,\BG_m)$. We show at first that $ X $ is flat over $S$. We can assume that $S$ is the spectrum of a local ring $R$. By assumption, $R$ is a Noetherian ring and therefore, the completion $ R\to\widehat{R} $ is a faithfully flat morphism. Thus, $X$ is flat over $R$ if and only if $ X_{\widehat{R}} $ is flat over $ \widehat{R} $. But $$ X_{\widehat{R}}=  \widetilde{\innAlt}^{\CO}_{S}(\CM^r_n,\BG_m)_{\widehat{R}}\cong \widetilde{\innAlt}^{\CO}_{\widehat{R}}(\CM^r_{\widehat{R},n},\BG_m)$$ which is finite and flat over $ \widehat{R} $ by previous lemma. This shows the flatness of $X$ over $S$.\\

We now prove that $X$ is finite over $S$. We can assume that $S$ is affine, say $ S=\Spec(R) $ and then $X$ is affine too, say $ X=\Spec(A) $. Let $ L $ be a field and $ f:\Spec(L)\to S $ be a morphism. Again, we have $$ X_L\cong  \widetilde{\innAlt}^{\CO}_{S}(\CM^r_n,\BG_m)_L\cong \widetilde{\innAlt}^{\CO}_{L}(\CM^r_{L,n},\BG_m)$$ which is finite over $L$ by Lemma \ref{lemextArtin}. It follows that $ X $ is quasi-finite over $S$. If we show that $ X $ is proper over $S$, then it will follow that it is finite over $S$. We use the valuative criterion of properness. So, let $D$ be a valuation ring and $E$ its fraction field. Denote also by $ \widehat{D} $ and respectively $ \widehat{E} $ the completions of $ D $ and $E$. Assume that we have a  commutative ``solid" diagram 
\begin{myequation}
\label{properdiag}
\xymatrix{R\ar[d]\ar[r] & D\ar@{^{(}->}[d] \\ A\ar[r]_{g} \ar@{-->}[ur]^{\tilde{g}}& E}
\end{myequation}where the vertical ring homomorphisms are the obvious ones, and we want to lift $g$ to a homomorphism $ \tilde{g}:A\to D $ filling the diagram (note that $X$ is separated over $S$ and therefore if such a morphism exists, it is unique). We can complete this diagram to the following diagram \[  \xymatrix{R\ar[d]\ar[r] & D\ar@{^{(}->}[d] \ar@{^{(}->}[r]& \widehat{D}\ar@{^{(}->}[d]\\ A\ar[r]_{g} & E\ar@{^{(}->}[r] & \widehat{E}.}\] If we can find a homomorphism $ \tilde{g}:A\to \widehat{D} $ making the above diagram commutative, then since $ E\cap \widehat{D}=D $, the image of $ \tilde{g} $ will be inside $ D $ and we are done. So, we may replace $ D $ and respectively $ E $ by $ \widehat{D} $ and $ \widehat{E} $ and assume that $ D $ is a complete valuation ring. Then, by assumption $ A\otimes_RD $ is finite over $D$ and therefore $ X_D $ is proper over $D$. Consider the following diagram \[ \xymatrix{R\ar[rr] \ar[d]&& D\ar@{^{(}->}[dd]\\ A \ar@{-->}[urr]^{\tilde{g}}\ar[drr]_{g}\ar[d]&& \\ A\otimes_RD \ar@{-->}[uurr]_{\tilde{g}_D}\ar[rr]_{g_D}&& E} \] induced by base change. The valuative criterion of properness, applied to $ A\otimes_RD $, implies that there exists a unique $ \tilde{g}_D:A\otimes_RD\to D $ making the above diagram commutative. Let $ \tilde{g} $ be the composition $ A\to A\otimes_RD\arrover{\tilde{g}_D}D $, then $ \tilde{g} $ fills the diagram \ref{properdiag} and this proves that $ X $ is proper over $S$.\\

As the order of a finite flat group scheme is a locally constant function on the base, which is preserved under base change, we may assume that $S$ is the spectrum of a field. The statement on the order now follows from Lemma \ref{lemextArtin}.
\end{proof}



\begin{prop}
\label{propextlNstBT}
Let $S$ be a locally Noetherian $\CO$-scheme. The exterior power $\epO^r\CM_n$ exists in the category of finite flat group schemes over $S$, and commutes with arbitrary base change. Moreover, its order is the constant function $q^{n\binom{h}{r}}$. Furthermore, for every $S$-scheme $T$, the canonical base change homomorphism $ \epO^r(\CM_{n,T})\to (\epO^r\CM_n)_T $ is an isomorphism.
\end{prop}

\begin{proof}
Set $ \Lambda^r_n:=\innHom_S(\widetilde{\innAlt}^{\CO}_{S}(\CM^r_n,\BG_m),\BG_m) $. According to the previous lemma, being the Cartier dual of $ \widetilde{\innAlt}^{\CO}_{S}(\CM^r_n,\BG_m) $, this is a finite flat $ \CO $-module scheme over $S$. By Cartier duality, we have a canonical isomorphism \[ \alpha:\innHom_S(\Lambda^r_n,\BG_m) \arrover{\cong} \widetilde{\innAlt}^{\CO}_{S}(\CM^r_n,\BG_m)\] and by adjunction formulas (cf. Proposition \ref{prop 4}), we have an alternating morphism $ \lambda_n:\CM_n^r\to \Lambda^r_n $. More precisely, we have isomorphisms \[ \widetilde{\Alt}_{S}^{\CO}(\CM_n^r,\Lambda^r_n)\cong \widetilde{\Alt}_{S}^{\CO}\big(\CM_n^r,\innHom_S(\widetilde{\innAlt}^{\CO}_{S}(\CM^r_n,\BG_m),\BG_m)\big)\cong \] \[ \Hom_S^{\CO}\big(\widetilde{\innAlt}^{\CO}_{S}(\CM^r_n,\BG_m),\widetilde{\innAlt}^{\CO}_{S}(\CM^r_n,\BG_m)\big) \] and $ \lambda_n $ corresponds, via this isomorphism, to the identity morphism. Let $G$ be a finite flat group scheme over $S$. We have a canonical isomorphism $ G\cong\innHom_S(G^*,\BG_m) $, where as usual, $ G^* $ denotes the Cartier dual of $ G $. Now consider the following diagram \[ \xymatrix{\Hom_S(\Lambda^r_n,G)\ar[rrr]^{\lambda_n^*(G)}\ar[d]_{\cong}&&&\widetilde{\Alt}^{\CO}_S(\CM_n^r,G)\ar[d]^{\cong}\\\Hom_S(\Lambda^r_n,\innHom_S(G^*,\BG_m))\ar[rrr]^{\lambda_n^*(\innHom_S(G^*,\BG_m))}\ar[d]_{\cong}&&&\widetilde{\Alt}^{\CO}_S(\CM^r_n,\innHom_S(G^*,\BG_m))\ar[d]^{\cong}\\\Hom_S(G^*,\innHom_S(\Lambda^r_n,\BG_m))\ar[rrr]_{\Hom_S(G^*,\alpha)}&&&\Hom_S(G^*,\widetilde{\innAlt}^{\CO}_S(\CM_n,\BG_m)).} \]  Since $ \alpha $ is an isomorphism, $ \Hom_S(G^*,\alpha) $ is an isomorphism too, which implies that $ \lambda_n^*(G) $ is an isomorphism. Thus, $ \Lambda^r_n $ is the $ r^{\rm th} $-exterior power of $ \CM_n $ in the category of finite flat group schemes over $S$, and we can write $ \epO^r\CM_n\cong\Lambda^r_n $. As $ \epO^r\CM_n $ is the Cartier dual of $ \widetilde{\innAlt}^{\CO}_{S}(\CM^r_n,\BG_m) $ and by previous lemma, $ \widetilde{\innAlt}^{\CO}_{S}(\CM^r_n,\BG_m) $ has order equal to $q^{n\binom{h}{r}}$, we deduce that $ \epO^r\CM_n $ has order equal to the constant function $q^{n\binom{h}{r}}$.\\

What we have proved above is that if $ \widetilde{\innAlt}^{\CO}_{S}(\CM^r_n,\BG_m) $ is finite and flat over $S$ (for any base $ \CO $-scheme $S$), then $ \epO^r\CM_n $ exists and it is isomorphic to the Cartier dual of $  \widetilde{\innAlt}^{\CO}_{S}(\CM^r_n,\BG_m) $. Now, let $T$ be an $ S $-scheme. Since by assumption, $S$ is locally Noetherian, $ \widetilde{\innAlt}^{\CO}_{S}(\CM^r_n,\BG_m) $ is finite and flat over $S$ by previous lemma and therefore, the base change $$ \widetilde{\innAlt}^{\CO}_{S}(\CM^r_n,\BG_m) _T\cong\widetilde{\innAlt}^{\CO}_{T}(\CM^r_{n,T},\BG_m) $$ is finite and flat over $T$. It follows that $ \epO^r(\CM_{n,T}) $ exists and since the Cartier duality commutes with base change, the two $ \CO $-module schemes $ \epO^r(\CM_{n,T}) $ and $ (\epO^r\CM_{n})_T  $ are canonically isomorphic.
\end{proof}

\begin{prop}
\label{propextlNspiBT}
Let $S$ be a locally Notherian $\CO$-scheme. There exist natural monomorphisms $ i_n:\epO^r\CM_n\into \epO^r\CM_{n+1} $, which make the inductive system $ (\epO^r\CM_n)_{n\geq 1} $ a $\pi$-Barsotti-Tate group over $S$ of height $ \binom{h}{r} $ and dimension a locally constant function \[\dim:S\to \{0,\binom{h-1}{r-1}\},\qquad s\mapsto \begin{cases}0 & \text{if}\quad \CM_s\quad \text{is \'etale}\\ \binom{h-1}{r-1}& \text{otherwise.}\end{cases}\]
\end{prop}

\begin{proof}
By previous proposition, $ \epO^r\CM_n $ are finite flat $ \CO $-module schemes over $S$. Since $ \CM_n $ is annihilated by $ \pi^n $, the $ \CO $-module scheme $ \widetilde{\innAlt}_S^{\CO}(\CM_n^r,\BG_m) $ is also annihilated by $ \pi^n $. Thus, its Cartier dual, $ \epO^r\CM_n $ is annihilated by $ \pi^n $ as well. We also know by previous proposition that $ \epO^r\CM_n $ has order equal to the constant function $q^{n\binom{h}{r}}$. Now, the similar arguments as in the proof of Lemma \ref{lem029} provide monomorphisms $ i_n:\epO^r\CM_n\into \epO^r\CM_{n+1} $ and imply this proposition. We therefore omit the proof.
\end{proof}

\begin{rem}
\label{remonlambda}
Since $ \lambda_n $ are the universal alternating morphisms, they are compatible with respect to the projections $ \CM_{n+1}^r\to \CM_n^r $ and $ \epO^r\CM_{n+1}\to \epO^r\CM_n $ and therefore define an alternating morphisms $ \lambda:\CM^r\to\epO^r\CM $.
\end{rem}

\begin{thm}[The Main Theorem for $\pi$-divisible $\CO$-module schemes]
\label{thm07ram}
Let $S$ be a locally Notherian $\CO$-scheme and $\CM$ a $\pi$-divisible $\CO$-module scheme over $S$ of height $h$ and dimension at most $1$. Then, there exists a $\pi$-divisible $\CO$-module scheme $\epO^r\CM$ over $S$ of height $\binom{h}{r}$, and an alternating morphism $ \lambda:\CM^r\to \epO^r\CM $ such that for every morphism $ f:S'\to S $ and every $\pi$-divisible group $\CN$ over $S'$, the homomorphism \[\Hom^{\CO}_{S'}(f^*\epO^r\CM,\CN)\to\widetilde{\Alt}^{\CO}_{S'}((f^*\CM)^r,\CN) \] induced by $ f^*\lambda $ is as isomorphism. In other words, the $ r^{\text{th}} $-exterior power of $ \CM $ exists and commutes with arbitrary base change. Moreover, the dimension of $ \epO^r\CM$ is a locally constant function \[\dim:S\to \{0,\binom{h-1}{r-1}\},\qquad s\mapsto \begin{cases}0 & \text{if}\quad \CM_s\quad \text{is \'etale}\\ \binom{h-1}{r-1}& \text{otherwise.}\end{cases}\]
\end{thm}

\begin{proof}
By Proposition \ref{propextlNstBT}, $ f^*\epO^r\CM_{n} $ is the $ r^{\rm th} $-exterior power of $ f^*\CM_{n} $ in the category of finite and flat group schemes over $S'$. So, the homomorphism $$ \Hom_{S'}(f^*\epO^r\CM_n,\CN_n)\to \widetilde{\Alt}_{S'}^{\CO}((f^*\CM_{n})^r,\CN_n) $$ induced by $ f^*\lambda_n:(f^*\CM_n)^r\to f^*\epO^r\CM_n $ is an isomorphism. Taking the inverse limit of these isomorphisms, we conclude that $$\Hom_{S'}(f^*\epO^r\CM,\CN)\to \widetilde{\Alt}_{S'}^{\CO}((f^*\CM)^r,\CN) $$ is an isomorphism. The statement on height and dimension follows from the previous proposition.\\

\textbf{\emph{Quod Erat Demonstrandum.}}
\end{proof}

\chapter{Examples}

In this chapter, $p$ is an odd prime number, $f$ is a positive natural number and $q=p^f$.

\begin{ex}
Let $\CO$ be the ring of integers of a non-Archimedean local field with residue field $\BF_q$ and $S$ an $\CO$-scheme. Let $F$ be a $\CO$-Lubin-Tate group of height $h$ and dimension 1 over $S$. When $S$ is the spectrum of a field or $\CO$ is an extension of $\BQ_p$, then by the results of chapters 5 and 9 (cf. Theorems \ref{thm 4} and \ref{thm07ram}), the exterior power $\epO^rF$ exists and is an $\CO$-Lubin-Tate group of height $\binom{h}{r}$ and dimension $\binom{h-1}{r-1}$. In particular, the highest exterior power $\epO^hF$ is an $\CO$-Luni-Tate group of height and dimension 1. Now, assume that $\CO$ is an extension of $\BQ_p$ and fix an $\CO$-Lubin-Tate group $F_0$ of height $h$ and dimension 1 over $\BF_q$. Let $ k/\BF_q $ be a perfect field. Define the \emph{deformation functor} that assigns to any local Artin $ \CO $-algebra $R$ whose residue field is an overfield of $ k $, the set \[ \Def\, (F_0)(R)=\{(F,\alpha)\}/\text{isomorphisms} \] where $F$ is a 1 dimensional formal $ \CO $-module over $R$ and \[ \alpha:F\times_RR/\Fm_R\arrover{\cong} F_0\times_{\BF_q}R/\Fm_R \] is an isomorphism of formal $ \CO $-modules. Then the functor $ \Def\, (F_0) $ is represented by $ \Spf(W_{\CO}(k)\lbb t_1,\dots,t_{h-1}\rbb) $, called the Lubin-Tate (moduli) space of $ F_0 $ (cf. \cite{{MR1387691}} or \cite{{MR1263712}}). Taking the highest exterior power induces in a natural way a morphism \[ \det: \Spf(W_{\CO}(k)\lbb t_1,\dots,t_{h-1}\rbb)\to \Spf(W_{\CO}(k)) \] of formal $ \CO $-schemes. Indeed, let $ R $ be as above and let $ (F,\alpha) $ be a deformation of $ F_0 $ over $R$. Then $ (\epO^hF,\epO^h\alpha) $ is a deformation of $ \epO^hF_0 $ over $R$, of dimension 1. This is so, because the construction of exterior powers commutes with base change: \[ (\epO^hF)\times_RR/\Fm_R\cong \epO^h(F\times_RR/\Fm_R)\cong \epO^h(F_0\times_{\BF_q}R/\Fm_R)\cong (\epO^hF)\times_{\BF_q}R/\Fm_R.\]
\end{ex}

\begin{ex}
Let $C$ be a smooth geometrically connected projective curve over $\BF_q$ and $\infty$ a closed point of $C$. Let $A$ be the ring of functions on $C$ which are regular outside $\infty$, i.e., $A=\Gamma(C\setminus\{\infty\},\CO_C)$. Let $L/\BF_q$ be a field with an $A$-structure, i.e., an $\BF_q$-algebra homomorphism $A\to L$. Finally, let $\rho:A\to\End_{\BF_q}(\BG_{a,L})$ be a Drinfeld module. This defines an $A$-module scheme structure on $\BG_{a,L}$. Take a prime ideal $\Fp$ of $A$. The closed subscheme $\rho[\Fp^n]$ of $\BG_a$ is a finite and flat $A/\Fp^n$-module scheme of order $p^{nf\rank (\rho)}=q^{n\rank(\rho)}$. If we fix a uniformizer $\pi$ of the completion $\widehat{A}_{\Fp}$ of $A$ with respect to $\Fp$, then for every $n>0$, we have short exact sequences \[0\to \rho[\Fp]\to \rho[\Fp^{n+1}]\arrover{\pi} \rho[\Fp^n]\to 0\] (cf. \cite{{MR1804937}} or \cite{{MR2020270}}). It follows that we have a $\pi$-Barsotti-Tate module or a $\pi$-divisible $\widehat{A}_{\Fp}$-module, that we denote by $\rho[\Fp^{\infty}]$. Its height is $\rank(\rho)$ and has dimension 1. Therefore, we can construct the exterior power $\uset{\widehat{A}_{\Fp}}{\ep^r}\rho[\Fp^{\infty}]$. Its height and dimension are respectively equal to $\binom{\rank(\rho)}{r}$ and $\binom{\rank(\rho)-1}{r-1}$.
\end{ex}

\begin{ex}
Let $L/\BQ$ be an imaginary quadratic field extension and denote by $\CO_L$ its ring of integers and let $p>2$ be a prime number that splits in $\CO_L$, say $p\CO_L=\Fq\cdot \Fp$, where $\Fp$ and $\Fq$ are different prime ideals of $\CO_L$. Let $S$ be an $\CO_L$-scheme defined over $ \Spf(\BZ_p) $ and $\CA$ an Abelian scheme over $S$ of relative dimension $g$. Assume that $\CO_L$ acts on $\CA$, i.e., we have a ring homomorphism $\CO_L\to \End_{S}(\CA)$ and the induced action on the relative Lie algebra of $\CA$ has signature $ (g-1,1) $. More precisely, the decomposition $ \CO_L\otimes_{\BZ}\CO_S \cong \CO_S\times \CO_S$, according to the splitting of $p$, induces a decomposition of the relative Lie algebra of $\CA$ into a product of two locally free $ \CO_S $-modules and the component corresponding to $ \Fp $ has rank $1$ and the other component, corresponding to $ \Fq $, has rank $g-1$. The $ p $-divisible group, $ \CA[p^{\infty}] $, associated to $ \CA $ has a natural structure of $ \CO_L\otimes_{\BZ}\BZ_p $-module. We can decompose $ \CO_L\otimes_{\BZ}\BZ_p $ into a product $ \CO_{L,\Fq}\times \CO_{L,\Fp} $, where $ \CO_{L,\Fq} $ and $ \CO_{L,\Fp} $ are respectively the completions of $ \CO_L $ with respect to $ \Fq $ and $ \Fp $. Since $ p $ is split in $ \CO_L $, they are both isomorphic to $ \BZ_p $. So, we have $ \CO_L\otimes_{\BZ}\BZ_p\cong\BZ_p\times\BZ_p $. This decomposition induces a decomposition $ \CA[p^{\infty}]\cong \CA[\Fq^{\infty}]\times \CA[\Fp^{\infty}] $ of $ \CA[p^{\infty}] $ into $ p $-divisible groups of height $g$ over $S$. By assumption on the action of $ \CO_L $ on $ \CA $, $ \CA[\Fp^{\infty}] $ has dimension $1$ and $ \CA[\Fq^{\infty}]$ has dimension $ g-1 $. So, for every $r>0$, the exterior power $ \ep^r\CA[\Fp^{\infty}] $ exists and its height and dimension are respectively $ \binom{g}{r} $ and $ \binom{g-1}{r-1} $. Note that there exists a natural number $m\leq n$ such that $ \CA[\Fp^{\infty}] $ has slopes zero and $ \frac{1}{n-m} $ with multiplicities respectively  $m$ and  $ n-m $, and $ \CA[\Fq^{\infty}] $ has slopes zero and $ \frac{n-1}{n-m} $ with multiplicities respectively $ m $ and $ n-m $.
\end{ex}

\begin{ex}
Let $ \CE $ be an elliptic scheme (i.e., Abelian scheme of relative dimension 1) over a base scheme $ S $. Then, the associated $ p $-divisible group $ \CE[p^{\infty}] $ has rank 2 and dimension 1 at points of characteristic $p$. Thus, the second exterior power $ \ep^2\CE[p^{\infty}] $ is a $p$-divisible group of height 1 and dimension 1 at points of characteristic $p$. This means that at these points, $ \ep^2\CE[p^{\infty}] $ is a multiplicative group of height and dimension 1 and so is isomorphic to $\mu_{p^{\infty}}$. At points of characteristic zero, $ \ep^2\CE[p^{\infty}] $ is an \'etale $p$-divisible group of height 1 and if we pass to an algebraic closure, we obtain the constant $p$-divisible group $\BQ_p/\BZ_p$, which is again isomorphic to $\mu_{p^{\infty}}$. Thus, at all geometric points,  $ \ep^2\CE[p^{\infty}] $ is isomorphic to $\mu_{p^{\infty}}$.\\

The Weil pairing $ \omega_n:\CE[p^n]\times\CE[p^n]\to \mu_{p^n} $ is a perfect pairing and induces an alternating morphism $ \omega:\CE[p^{\infty}] \times \CE[p^{\infty}] \to \mu_{p^{\infty}} $. We want to show that $\omega$ is in fact the universal alternating morphism and $ \ep^2\CE[p^{\infty}]  $ is canonically isomorphic to $ \mu_{p^{\infty}} $. Indeed, by the universal property of $ \ep^2\CE[p^{\infty}]  $, we have a homomorphism $ \tilde{\omega}:\ep^2\CE[p^{\infty}] \to \mu_{p^{\infty}}$ such that $\omega=\tilde{\omega}\circ \lambda $ (where $\lambda$ is the universal alternating morphism of $\ep^2\CE[p^{\infty}] $) and we have to show that $\tilde{\omega}$ is an isomorphism. This homomorphism is an isomorphism, if and only if it is so over every geometric point, and therefore, we may assume that $S$ is the spectrum of an algebraically closed field. Then, as we explained above, $\ep^2\CE[p^{\infty}]$ is isomorphic to $\mu_{p^{\infty}}$. So we have a homomorphism $\tilde{\omega}:\mu_{p^{\infty}}\to \mu_{p^{\infty}}$ and we want to show that it is an isomorphism. By Cartier duality, we can consider the dual of this homomorphism, namely $\tilde{\omega}^*:\BQ_p/\BZ_p\to \BQ_p/\BZ_p $. If this is not an isomorphism, it factors through multiplication by $p$. Thus, $\tilde{\omega}$ factors through multiplication by $p$. It follows that on $p$-torsion points, $\omega_1=\tilde{\omega}_1\circ \lambda_1$ is the zero morphism, which is in contradiction with the fact that the Weil pairing is a perfect pairing.
\end{ex}

\cleardoublepage
\phantomsection

\end{document}